\documentclass[]{arxiv_version}              

\newcommand{\bbZ}{\mathbb{Z}}
\newcommand{\bbR}{\mathbb{R}}
\newcommand{\zinfty}{\bar{\zeta}_\infty}
\newcommand{\qinfty}{\bar{\xi}_\infty}
\newcommand\E[1]{\mathbb{E}\left[#1\right]}
\newcommand\Var[1]{\textrm{Var}\left[#1\right]}
\renewcommand\P[1]{\mathbb{P}\left(#1\right)}

\usepackage{bbm}
\newcommand{\step}{\chi}
\newcommand{\calF}{\mathcal{F}}
\newcommand{\calT}{\mathcal{T}}
\usepackage{soul}
\newcommand{\barz}{\bar{z}}
\newcommand{\bara}{\bar{a}}

\newcommand{\calS}{\mathcal{S}}
\newcommand{\calD}{\mathcal{D}}
\newcommand{\calC}{\mathcal{C}}
\newcommand\revcolor[1]{{#1}}
\newcommand{\pfix}[1]{{#1}}
\usepackage{enumitem}
\newcommand\numberthis{\addtocounter{equation}{1}\tag{\theequation}}
\renewcommand\Var[1]{\textrm{Var}\left[#1\right]}
\newcommand{\calU}{\mathcal{U}}

\usepackage{tikz}
\newcommand{\bbC}{\mathbb{C}}
\newcommand{\drift}{d}
\newcommand{\moment}{M}
\usetikzlibrary{shapes}
\usetikzlibrary{decorations}
\usetikzlibrary{backgrounds}
\usetikzlibrary{calc}
\usetikzlibrary{shapes}
\usetikzlibrary{arrows}
\usetikzlibrary{automata,positioning}
\usetikzlibrary{decorations.pathreplacing, calligraphy}
\newcommand{\barq}{\bar{q}}
\newcommand{\bars}{\bar{s}}
\newcommand{\baru}{\bar{u}}
\usepackage{nicematrix}
\allowdisplaybreaks 
\usepackage{xcolor}
\newsavebox{\selvestebox}

\usepackage{natbib}
 \NatBibNumeric
 \def\bibfont{\small}%
 \bibpunct[, ]{[}{]}{,}{n}{}{,}%

\usepackage[colorlinks=true,breaklinks=true,bookmarks=true,urlcolor=blue,
     citecolor=blue,linkcolor=blue,bookmarksopen=false,draft=false]{hyperref}

\def\EMAIL#1{\href{mailto:#1}{#1}}
\def\URL#1{\href{#1}{#1}}         
\theoremstyle{TH}
\newtheorem{condition}{Condition}
\TheoremsNumberedThrough     

\EquationsNumberedThrough    

\begin{document}


\RUNAUTHOR{Varma and Maguluri}

\RUNTITLE{Heavy Traffic in Matching Queues}

\TITLE{A Heavy Traffic Theory of Matching Queues}

\ARTICLEAUTHORS{%
\AUTHOR{Sushil Mahavir Varma}
\AFF{Industrial and Operations Engineering, University of Michigan Ann Arbor \EMAIL{sushilv@umich.edu}, \URL{https://sites.google.com/view/sushil-varma/home}}
\AUTHOR{Siva Theja Maguluri}
\AFF{Industrial and Systems Engineering, Georgia Institute of Technology, \EMAIL{siva.theja@gatech.edu}, \URL{https://sites.google.com/site/sivatheja/}}
} 

\ABSTRACT{%
Motivated by emerging applications in online matching platforms and marketplaces, we study a matching queue. Customers and servers that arrive in a matching queue depart as soon as they are matched. It is known that a matching queue without an external control is unstable and so we study its behavior for a general state-dependent control. While state-dependent control is an effective lever to regulate the throughput and delay, it often comes at a cost for matching platforms in practice. Optimizing this fundamental trade-off motivates the use of small amounts of control, so we study a matching queue in an asymptotic regime where the state-dependent control decreases to zero. Unlike the heavy traffic regime in classical queues, there are two different ways the control can be sent to zero, via a magnitude scaling parameter $\epsilon$ that goes to zero and a time scaling parameter $\tau$ that goes to infinity.

Depending on the cost of control, we show that the rates of $\epsilon$ and $\tau$ that optimize the trade-off between delay and cost of control could correspond to three different regimes. As we traverse these regimes, we observe a phase transition in the limiting distribution of the matching queue. We show that a low cost of control corresponds to the regime $\epsilon \tau \rightarrow 0$ and we call it the delay-driven regime. The limiting behavior in this regime is an asymmetrical Laplace distribution. On the other hand, $\epsilon \tau \rightarrow \infty$ is the cost-driven regime corresponding to a high cost of control where the limiting behavior is either a uniform or a truncated exponential distribution. We christen the in-between regime of $\epsilon\tau \rightarrow (0, \infty)$ the hybrid regime where the limiting behavior is a Gibbs distribution. These results are obtained by novel generalizations of the transform method, where each regime requires new ideas. The hybrid regime employs inverse Fourier transforms while the other two regimes engineer multiple complex exponential test functions.

}%


\MSCCLASS{60K25 (primary), 60F05 (secondary)}
\ORMSCLASS{Primary: Queues – Limit Theorems; secondary: Probability – Markov processes}

\maketitle
\section{Introduction}
Since the seminal work of \citet{erlang1909theory} in the context of telecommunication systems more than a century ago, queueing theory has emerged as a well-established discipline that has had an impact on a large number of applications including wired and wireless networks, cloud computing, manufacturing systems, transportation systems, etc. The central building block of queueing theory is a single server queue, which has a fixed server, customers that wait until they are served in a first-come-first-served (FCFS) manner, and then depart immediately thereafter. In addition, there is a queue or a waiting space for the customers to wait, and a stochastic model of the arrivals and services. 
While the single server queue is well-understood when the arrivals and service are memory-less,
there is no closed-form expression for the stationary distribution of the queue length for general distributions.
Therefore, queueing systems are studied in various asymptotic regimes, including the heavy-traffic regime, where the arrival rate approaches the service rate. 

More precisely, suppose  $\epsilon$  denotes the difference between the service rate and arrival rate, in heavy traffic one studies the queue in the limit when $\epsilon \rightarrow 0$. When $\epsilon=0$, the queue becomes unstable (null-recurrent). However,  \citet{kingman} showed that the limiting distribution of the queue length multiplied by $\epsilon$ is exponential. Moreover, the mean of the exponential depends only on the variance of the inter-arrival and service distributions, but not on the whole distribution. Recent work by \citet{atilla, hurtado2020transform, Walton_SteinHT} has also characterized the rate of convergence to the exponential, thus enabling us to approximate the stationary queue length when $\epsilon$ is not small.

Recent developments in online platforms and matching markets such as ride-hailing, food delivery services, etc. have led to an interest in the study of matching queues, also referred to as two-sided queues. In this paper, we will use the two terms interchangeably. In a matching queue, \textit{both} servers and customers arrive, wait until they are matched, and then immediately depart the system. 
The behavior of matching queues is different from that of classical queues. In particular, a matching queue is never stable without external control. To see this, note that if the arrival rates on the two sides don't match, the system is clearly unstable. But when the rates match, it is analogous to a symmetric random walk on a line, which is null-recurrent as discussed by \citet{sushil_blockchain}. Therefore, matching queues have to be always studied under an external control that modulates the arrival rates in a \textit{state-dependent} manner with levers such as prices in online platforms. In contrast to a large amount of literature on classical queues, there is comparatively very little work on matching queues. 

The goal of this paper is to develop a heavy-traffic theory of matching queues, that will enable us to completely characterize the queueing behavior in an appropriately defined asymptotic regime. Analogous to classical queues, except in special cases, it is hard to obtain the exact stationary distribution of queue length (when it exists) in a matching queue. Further difficulties arise exclusively in the case of matching queues. We point out two of them here. First, as discussed in the previous paragraph, a matching queue is naturally unstable. Due to this difficulty, most of the previous work involving matching queues either characterizes the transient behavior as done by \citet{matchingqueues, dynamictypematchinghu}, or considers a simplified model as done by \citet{caldentey2009fcfs, adan2012exact}. We overcome this challenge by considering state-dependent arrivals which results in a stable system. Secondly, there is no natural notion of `heavy traffic' here. To quote from  \citet{matchingqueues}, ``As there are no processing resources, there is no obvious notion of heavy traffic''. We overcome this challenge as follows. Suppose that the uncontrolled arrivals have an equal rate on both sides. Now, we study the system in the regime when the state-dependent control goes to zero. As mentioned before, in the regime when the control is zero, the system is null-recurrent, and the regime is reminiscent of the heavy-traffic regime of a single-server queue. 

The practical motivation for such an asymptotic regime is the trade-off observed in matching platforms between the efficient system performance as a result of state-dependent control and the cost associated with implementing such control. We outline two examples here and discuss them in detail in Section \ref{sec: motivation}. Firstly, in the context of ride-hailing systems, e.g., see \citet{kim2017value,ridehailing_sigmetrics,sushil_selfish_arxiv,banerjeeridehailing}, pricing policies that are either static or a small perturbation of it are shown to be near-optimal as they optimize the trade-off between the profit obtained by the system operator (cost of control) and the delay experienced by the customers and servers (system performance). The second example is payment channel networks studied by \citet{sushil_blockchain, spider_nsdi}, wherein, agents are connected by payment channels with limited capacity and attempt to transfer funds among each other via these channels. To ensure the payment channels do not run out of capacity in the long run (system performance), the more expensive blockchain is employed once in a while to complete the transaction requests (cost of control). Thus, our analysis is of relevance to understanding the fundamental trade-off between the cost of control and system performance, which is ubiquitous in matching platforms.

In matching queues, there are two ways in which the state-dependent control can be sent to zero. The first is clearly to scale the magnitude of the control by $\epsilon$ which is sent to zero. The second is to apply the control only at larger and larger values of queues, which we control using a parameter $\tau$ that is sent to $\infty$. These parameters are more precisely defined in \eqref{eq: general_pricing_policy}. 
As discussed in Section~\ref{sec: motivation}, depending on the cost of control, the rates of $\epsilon$ and $\tau$ that optimize the trade-off between system performance and cost of control could correspond to three different limiting values of $\epsilon\tau$. 
The optimal regime with a low cost of control corresponds to $\epsilon\tau \rightarrow 0$ and we call it the delay-driven regime. On the other hand, $\epsilon \tau \rightarrow \infty$ is the optimal regime for a high cost of control and we call it the cost-driven regime. The in-between regime when $\epsilon\tau$ goes to a positive number is called the hybrid regime. The primary contribution of this paper is to demonstrate that there is a phase transition in the limiting behavior of the matching queue as we traverse from the delay-driven regime to the cost-driven regime, via the hybrid regime.



\subsection{Main Contributions}
The main contributions of this paper are the following.

\begin{itemize}
    \item To illustrate the phase-transition behavior, we first, consider the simplest matching queue with the simplest possible control, viz., the one with Bernoulli arrivals that is controlled with a two-price policy. Analogous to an $\pfix{\text{M/M/1}}$ (more precisely, a $\pfix{\text{Geo/Geo/1}}$) queue, one can obtain an exact stationary distribution here. By explicitly taking the asymptotic limit of this distribution, we show that the appropriately scaled queue length converges to a Laplace distribution in the delay-driven regime, a Uniform distribution in the cost-driven regime, and a hybrid of the above two in the hybrid regime. This is presented in Section \ref{sec: bernoulli}.
    \item 
    The motivation to study the asymptotic behavior is, of course, its utility when we can't find the exact stationary distribution. Thus, in Section \ref{sec: main_theorem}, we study a general matching queue with general arrival distributions and control policies resulting in a much richer limiting behavior. The stationary distribution of the imbalance generalizes from a Laplace to an asymmetric Laplace in the delay-driven regime, from a Uniform to either a Uniform, truncated exponential, or a Dirac measure in the cost-driven regime, and lastly, the distribution in the hybrid regime generalizes to a Gibbs distribution. These limiting distributions crucially depend on the (pre-limit) state-dependent control even though it vanishes in the heavy-traffic limit, highlighting that our results embody the first-order effect of the control on the stationary distribution.
    These results are summarized in Table \ref{tab: phase_transition}.
    \item We prove these results by developing generalizations of the transform method of \citet{hurtado2020transform} to handle state-dependent control. In particular, directly applying the transform method fails as one does not obtain a closed-form expression of the transform but an implicit equation instead, owing to the state-dependent control. Each of the three regimes requires non-trivial ideas to tackle the associated challenges which we outline below. 
    \begin{itemize}
        \item \textit{Hybrid Regime:} We solve the functional equation by taking its inverse Fourier transform\footnote{For any random variable with a ``nice enough'' PDF, the characteristic function is the Fourier transform of the PDF\pfix{.}}, resulting in a differential equation in the PDF of the limiting imbalance. One can then solve the differential equation to obtain the imbalance distribution. However, one cannot assume the existence of the PDF. 
        We tackle this issue by interpreting the imbalance distribution as a generalized function operating on Schwartz functions, motivated by the theory of distributions pioneered by  \citet{halperin1952introduction}. The spirit of the proof is then the same as when the PDF exists, albeit with additional technical overhead.
        \item \textit{Delay-Driven Regime:} The novelty in this regime is to circumvent the implicit equation by engineering \emph{multiple} complex exponential Lyapunov functions, together resulting in a closed-form expression of the characteristic function. The idea is to exploit symmetry in the underlying process. The first Lyapunov function shows that (appropriately scaled) absolute imbalance has an exponential distribution. The second Lyapunov function then separately establishes symmetry. Combining the two, we conclude the limiting distribution is Laplace.
        \item \textit{Cost-Driven Regime:} Similar to the delay-driven regime, we circumvent the implicit equation by engineering multiple Lyapunov functions---the novelty is in engineering a \emph{combination of real and complex} exponential Lyapunov functions to establish the characteristic function. Such an approach is contrary to the previous work that restricts to just complex or just real exponential Lyapunov functions to obtain the characteristic function and MGF respectively. We make connections to a $\text{G/G/1}$ queue with a finite waiting area (Section~\ref{sec: finite_buffer}) highlighting intuition for the Lyapunov functions used.
    \end{itemize}
    
    \item 
    The above results show that the heavy-traffic behavior in matching queues is much richer than that of a classical single-server queue. This is due to the state-dependent control in a matching queue, while classical heavy-traffic theory focuses on the case when the arrival rates in a single server queue are fixed. In Section \ref{sec: discussion} we show that even the classical single server queue exhibits a phase-transition behavior if the arrival rates of customers are modulated in a state-dependent manner. We obtain these results by using the generalizations of the transform method that we develop which shows the generality of the technique.
\end{itemize}
\begin{table}[bth!]
\TABLE{A Summary of Heavy-Traffic Phase Transitions in Matching Queues.
    \label{tab: phase_transition}}{
    \begin{tabular}{>{\centering\arraybackslash}p{90pt}| >{\centering\arraybackslash}p{105pt}|>{\centering\arraybackslash}p{100pt}|>{\centering\arraybackslash}p{95pt}}
    \hline
         &  Delay-Driven Regime &  Hybrid Regime &  Cost-Driven Regime \\
         &   $\epsilon \tau\rightarrow 0$ &   $\epsilon \tau \rightarrow (0,\infty)$ & $\epsilon \tau \rightarrow  \infty$ \\
         \hline
         & & & \\
         Bernoulli Arrivals, Two Price Policy (Proposition \ref{prop: bernoulli})& Laplace & Hybrid & Uniform \\ [-25pt]
     & \includegraphics[width=0.2\textwidth]{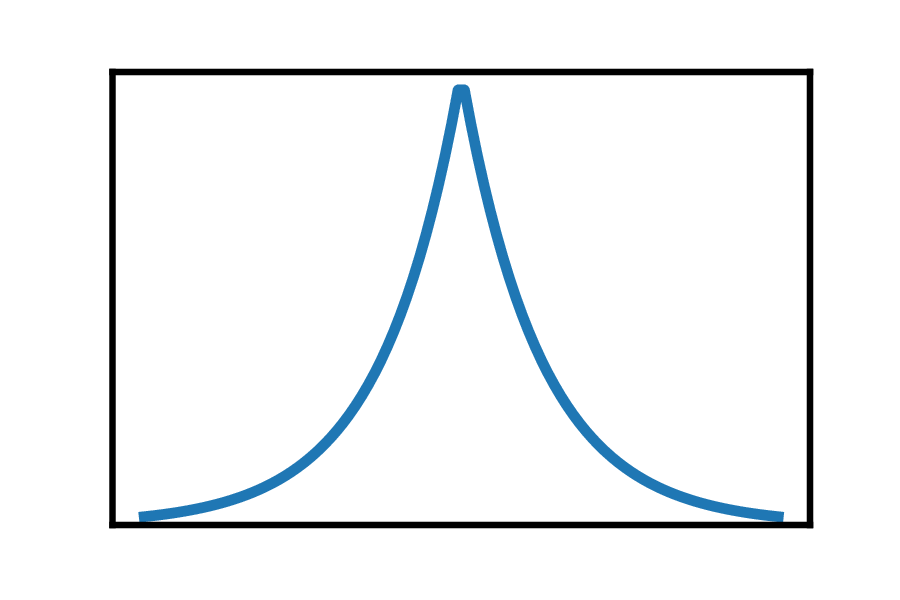} &\includegraphics[width=0.2\textwidth]{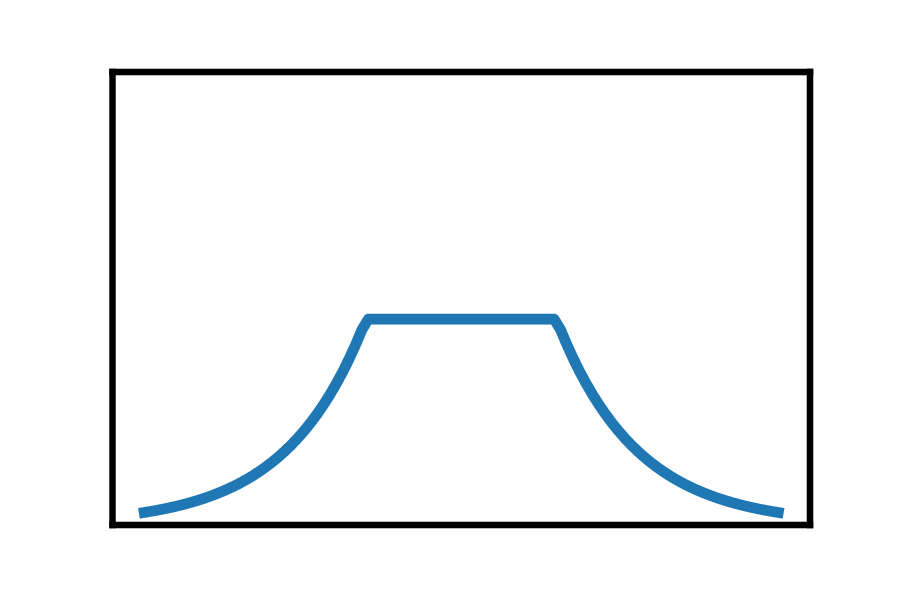} & \includegraphics[width=0.2\textwidth]{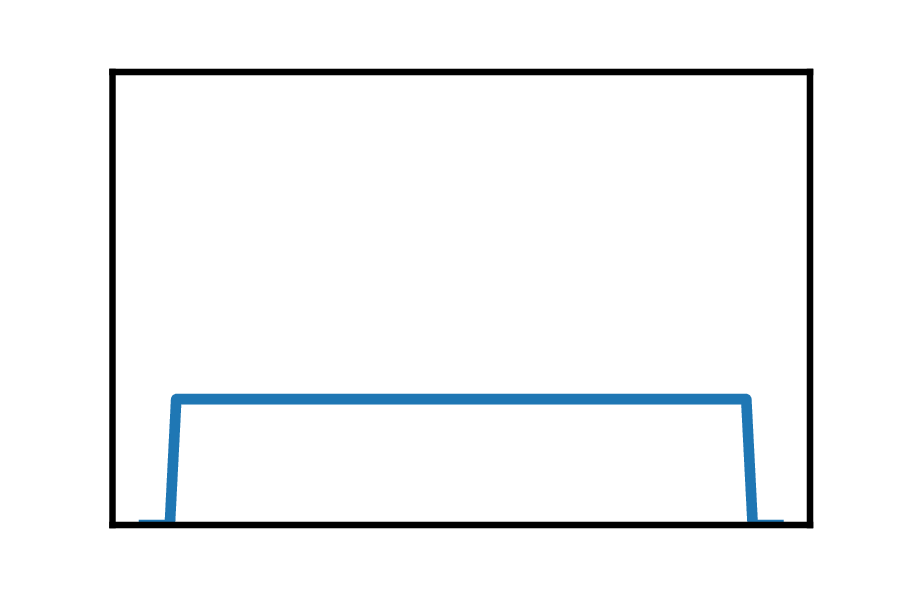} \\ \hline 
    General Arrival and Pricing & Asymmetric Laplace & Gibbs Distribution &  Uniform/Trun Expo   \\   [-10pt]
    & Theorem \ref{theo: case_1} & Theorem \ref{theo: case_2_itm} & Theorem \ref{theo: profit_driven_regime_3} \\ 
  \hline
  Proof Technique. (Complex exponential Lyapunov functions)& (i) Exploiting symmetry around zero (ii) Auxiliary Lyapunov functions & Inverse Fourier transform & (i) Exploiting bounded support (ii) Auxiliary Lyapunov functions \\ \hline
    \end{tabular}}{}
\end{table}
\subsection{Literature Review}
In this section, we present several lines of work that either relate to the applications of our model or our proof techniques.

\subsubsection{Applications of Matching Queues}
In this section, we first present literature specifically on matching queues and position our paper. Then, we present a brief overview of the vast literature on pricing and revenue management. 

There has been a surge of interest in matching queues in recent years with applications that include ride-hailing and online markets studied by \citet{banerjeeridehailing,banerjee2018state,kanoria2019backpressure,sushil_selfish_arxiv,ridehailing_sigmetrics}, kidney exchange by \citet{yashkanoriabarter,akbarpour2017thickness}, and matching markets by \citet{dynamictypematchinghu}, etc. In addition, several models in the literature are closely related to two-sided queues like the dynamic matching models considered by \citet{adan2012exact,caldentey2009fcfs,adan2018reversibility,cadas2020flexibility, weiss2020directed}, assemble-to-order systems studied by \citet{matchingqueues,plambeck2006optimal,reiman2015asymptotically,dougru2010stochastic,song1998order, song2002performance}, and several similar formulations of matching queues as considered by \citet{dynamictypematchinghu,ondemandservers,aveklouris2021matching,ozkan2020joint,amywardpricingmatching,blanchet2021asymptotically}, etc. Most of the outlined work focuses on the transient analysis of a matching queue as it is inherently unstable without external control. Steady-state analysis under external control has been done by \citet{ondemandservers,blanchet2021asymptotically,ridehailing_sigmetrics,sushil_selfish_arxiv} but the focus is on finding upper and lower bounds on the mean delay and profit. In addition, \citet{castro2020matching} proves a product form distribution of the delay for a special case of matching queues with reneging under FCFS. On the other hand, our focus is on pricing or equivalently state-dependent arrivals as the external control. To the best of our knowledge, we are the first ones to provide the limiting distribution of a matching queue under state-dependent arrivals.

As our analysis can be \pfix{adapted} to analyze pricing policies in the context of online marketplaces, we provide a brief literature survey on revenue management. The book by \citet{talluri2006theory} provides a thorough summary of the classical results in revenue management. Pricing in the context of queueing was studied by \citet{low1974optimal,pricinginqueuing2001, Tsitsiklis2000congestion,ks2020optimal} and several structural results of the optimal pricing policy were proved. Recently, \citet{kim2017value} considers pricing in a single server queue and presents near-optimal pricing policies. Pricing in ride-hailing systems was considered by \citet{besbes2021surge, bimpikis2019spatial, castillo2017surge, guda2019your, cachon2017role} in a static/non-queueing setting, and insights into surge pricing were outlined. Furthermore, \citet{banerjeeridehailing,banerjee2018state,kanoria2019backpressure} studies the ride-hailing system as a closed queueing network and proposes near-optimal static and dynamic pricing policies. Other controls are also analyzed in the literature, like matching by \citet{banerjee2018state}, relocation by \citet{braverman2019empty, hosseini2021dynamic}, and joint pricing and matching by \citet{amywardpricingmatching,ozkan2020joint,SIGMETRICS_strategic_servers,sushil_selfish_arxiv}. Most of the past work on pricing focused on either structural properties of the optimal policy or characterizing a near-optimal pricing policy with objectives like throughput, revenue, or social benefits. Such an analysis generally involves characterizing the mean delay and bounding the tail delay. On the other hand, our focus is to characterize the complete distribution of the delay for a broad class of pricing policies.

\subsubsection{Technical Novelty}
In this section, we first discuss phase transition and the queueing models that exhibit phase transition. Next, we discuss the proof techniques to analyze queueing systems and position our methodological contributions in the literature.

Phase transition is of course a widespread phenomenon in many systems in general and a large class of queueing systems. One example is the behavior of many server queues in heavy traffic, where the famous Halfin-Whitt phase transition was presented by  \citet{HalfinWhitt_Regime}. Load balancing systems also are known to exhibit phase transitions in the many-servers-heavy-traffic regime by \citet{LiuLei_JSQ_UniversalScaling, Hurtado_JSQ_alpha_discrete}, even though the behavior is not yet completely characterized. In the case of a single server queue with state-independent control, we observe a trivial phase transition. In particular, the queue is stable when under-loaded and unstable when over-loaded. Now, consider a single server queue with state-dependent control which can be achieved by either considering abandonment or pricing. A single server queue with abandonment is stable even when it is overloaded. Such a system is known to exhibit a phase transition in the limiting distribution of queue length, as it moves from under-loaded to over-loaded, studied in the literature by \citet{ward2005diffusion, huang2018beyond, he2013one, lee2020stationary, baccelli1984single}. We show that phase transition is also observed when pricing is used as a state-dependent control. \pfix{We also show such a behavior in a matching queue.} To the best of our knowledge, such a phase transition with pricing or equivalently, state-dependent arrivals as the external control hasn't been observed in the literature.

Single server queues in heavy traffic have been extensively studied in the literature. A popular approach is to use diffusion limits and study the resultant Brownian control problem. This was first done in literature by \citet{kingman}. Even though they analyze the waiting time in a $\pfix{\text{G/G/1}}$ queue in continuous time, it is equivalent to the queue lengths in a single server queue in discrete time. Later, this method was generalized to further analyze settings like heterogeneous customers by \citet{harrison1988brownian}, parallel servers by \citet{harrison1998heavy,mandelbaum_stolyar2004cmu}, generalized switch by \citet{stolyar2004maxweight}, and generalized Jackson networks by \citet{gamarnik2006validity}, etc. More recently, direct methods that work with the original system, as opposed to the diffusion limit have been developed. One of these methods is the drift method introduced by \citet{atilla} which analyzes the single server queue. This method was further generalized to analyze switch by \citet{MagSri_SSY16_Switch}, flexible load balancing by \citet{zhou2018flexible}, and generalized switch by \citet{Hurtado-gen-switch-SIGMETRICS}, etc. Other methods that directly work with the original system include, the transform method of \citet{hurtado2020transform}, basic adjoint relationship (BAR) method of \citet{braverman_BAR}, and Stein's method of \citet{gurvich2014diffusion}. In addition, a single-server queue with state-dependent control was studied by \citet{kim2017value}.
In this paper, we adopt the transform method introduced by \citet{hurtado2020transform} and use it in the context of a stochastic processing network comprising classical queues. This method provides an explicit formula for the limiting moment generating function (MGF) for classical queues. However, in the case of matching queues, we get an implicit equation involving the MGF due to the state-dependent control. Solving this implicit equation is a major challenge and we develop the inverse Fourier transform method to address this difficulty in the hybrid regime. In the delay-driven and cost-driven regimes, we generalize the characteristic function method by using multiple Lyapunov functions exploiting the underlying symmetry to obtain a closed-form expression of the characteristic function. Thus, the aforementioned methods generalize the transform method and so, they are widely applicable.
\subsection{Notation} \label{sec:notation}
We denote the set of non-negative integers (including 0) by $\bbZ_+$. The imaginary number $\sqrt{-1}$ is denoted by $j$. For a real number $x \in \bbR$, we denote its positive part by $[x]^+\overset{\Delta}{=} \max\{x,0\}$ and negative part by $[x]^-\overset{\Delta}{=} \max\{-x,0\}$. In addition, we denote the smallest integer greater than or equal to $x$ by $\lceil x \rceil$ and the largest integer smaller than or equal to $x$ by $\lfloor x \rfloor$. For a set $A \subseteq \bbR$, we denote its indicator function $\mathbbm{1}_{x \in A}$ by $\mathbbm{1}\{A\}$. Also, we define the sign function with domain $\bbR$ and range $\{-1, 1\}$  by $\textrm{sgn}(x)$ which is equal to 1 for all $x \geq 0$ and -1 otherwise. Moreover, we define the clip function with domain $\{(x, x_{\max}, x_{\min}) \in \bbR^3: x_{\min} \leq x_{\max}\}$ and range $\bbR$ as $\text{clip}(x, x_{\max}, x_{\min}) \overset{\Delta}{=} \max\{\min\{x, x_{\max}\}, x_{\min}\}$. If $x_{\min} = -x_{\max}$, then, we simply write $\text{clip}(x, x_{\max}, x_{\min})$ as $\text{clip}(x, x_{\max})$.
Fourier transform of a function $\varphi: \bbR \rightarrow \bbC$ is either denoted by (whenever it exists) $\calF(\varphi)$ or $\hat{\varphi}$ and inverse Fourier transform is denoted by either $\calF^{-1}(\varphi)$ or $\check{\varphi}$. We denote the variable in Fourier transform by $\omega \in \bbR$. Note that $\omega$ does not represent an element in the sample space. Any function $f,g : \bbR \rightarrow \bbC$ of $\epsilon \in \bbR$ such that $\lim_{\epsilon \rightarrow 0}\frac{|f(\epsilon)|}{|g(\epsilon)|}=0$ is denoted by $f(\epsilon)=o(g(\epsilon))$. We denote the space of infinitely differentiable functions with domain $\bbR$ and range $\bbC$ by $\calC_{\textrm{pol}}^\infty(\bbR)$, where $\text{pol}$ stands for ``polynomial like''. A sequence of real-valued random variables $\{X_n\}$ converging in distribution to $X$ is denoted by $X_n \overset{D}{\rightarrow } X$. We consider several distributions in the paper, and so, we define them here for convenience. We denote a Laplace distribution by $\operatorname{Laplace}(m, \lambda, \kappa)$, where \revcolor{$m \in \bbR$ is a location parameter, $\lambda > 0$ is a \pfix{rate} parameter, and $\kappa > 0$ is an asymmetry parameter.} The \pfix{PDF} $f_{\operatorname{Laplace}}(\cdot)$ of $\operatorname{Laplace}(m, \lambda, \kappa)$ is given by
\begin{align*}
f_{\operatorname{Laplace}}(x; m, \lambda, \kappa) = \frac{\lambda}{\kappa+1/\kappa}\begin{cases}
        e^{\frac{\lambda}{\kappa}(x-m)} &\textit{if } x < m \\
        e^{-\lambda \kappa(x-m)} &\textit{if } x \geq m.
    \end{cases} \numberthis \label{eq: def_laplace}
\end{align*}
When $\kappa=1$, the above simplifies to a symmetric Laplace distribution, in which case, we simply write $\operatorname{Laplace}(m, \lambda)$. Next, we denote a truncated exponential distribution by $\operatorname{TrunExp}\pfix{(\lambda, \Phi)}$, where \revcolor{$\lambda \in \bbR$ is a \pfix{rate} parameter, and $\Phi \subseteq \bbR$ is a Borel measurable set such that $\int_{t \in \Phi} e^{-\lambda t} dt < \infty$.} The \pfix{PDF} $f_{\operatorname{TrunExp}}(\cdot)$ of truncated exponential distribution is given by
\begin{align*}
    f_{\operatorname{TrunExp}}(x;\lambda, \Phi) = \frac{e^{-\lambda x}}{\int_{t \in \Phi} e^{-\lambda t} dt} \mathbbm{1}\{x \in \Phi\}. \numberthis \label{eq: def_trun_expo}
\end{align*}
\revcolor{When $\Phi = \bbR_+$ and $\lambda > 0$}, the above simplifies to an exponential distribution denoted by $\operatorname{Exp}(\lambda)$. Next, we denote a uniform distribution by $\operatorname{Uniform}(\Phi)$, \revcolor{where $\Phi \subseteq \bbR$ is Borel measurable set with finite, non-zero Lebesgue measure.} The \pfix{PDF} $f_{\operatorname{Uniform}}(\cdot)$ of Uniform distribution is given by
\begin{align*}
    f_{\operatorname{Uniform}}(x;\Phi) = \frac{\mathbbm{1}\{x\in\Phi\}}{|\Phi|}, \numberthis \label{eq: def_uniform}
\end{align*}
\revcolor{where $|\Phi|$ is the Lebesgue measure of $\Phi$. When $\Phi$ is a singleton, i.e., $\Phi = \{t\}$ for some $t \in \bbR$, then, $\operatorname{Uniform}(\Phi)$ is a Dirac-delta distribution. Lastly, we denote a Gibbs distribution by $\operatorname{Gibbs}\pfix{(g)}$, where $g: \bbR \rightarrow \bbR$ such that $\int_{-\infty}^\infty e^{-\int_0^x g(y) dy} dx < \infty$.} The \pfix{PDF} $f_{\operatorname{Gibbs}}(\cdot)$ of Gibbs distribution is given by
\begin{align*}
    f_{\operatorname{Gibbs}}(x;g) = \frac{e^{-\int_0^x g(y) dy}}{\int_{-\infty}^\infty e^{-\int_0^x g(y) dy} dx}, \numberthis \label{eq: def_gibbs}
\end{align*}
\revcolor{where $\int_0^x g(y) dy$ for $x < 0$ is interpreted as $-\int^0_x g(y) dy$. Consider the special case of $g(x) = b \mathbbm{1}\{x \geq c\} - b \mathbbm{1}\{x \leq -c\}$ for some $b, c > 0$, where the resultant distribution is Uniform in $[-c, c]$ stitched with Laplace tails and we denote this special case by $\operatorname{Hybrid}(b, c)$.}

\section{Model} \label{sec: model}
We consider a matching queue operating in discrete time with customers and servers both arriving in the system. 
At a given time epoch $k$, let $q^c(k)$ and $q^s(k)$ be the number of customers and servers waiting in the queue respectively.
A waiting customer is matched to a server (and vice versa) as soon as possible, and the pair instantaneously departs from the system. Therefore, both servers and customers cannot be waiting at the same time, and so the main quantity of interest is the imbalance in the queue defined by $z(k)\overset{\Delta}{=}q^c(k)-q^s(k)$. Thus, for any $k \in \bbZ_+$, we have $q^c(k)q^s(k)=0$ with probability 1. Hence, it suffices to consider the imbalance $z(k)$ as the state descriptor for the queue as $q^s(k)=[z(k)]^-$ and $q^c(k)=[z(k)]^+$.

Consider a matching queue with customers and servers both arriving in the system with exogenous arrival rates $\lambda^\star$ and $\mu^\star$ respectively. \revcolor{For brevity, we abuse the terminology and say ``arrival rate'' instead of ``expected number of arrivals in one time slot''.} We take $\lambda^\star=\mu^\star$ to balance the arrival rates, otherwise one of the queue lengths will go to infinity. Unfortunately, $\lambda^\star=\mu^\star$ is not a sufficient condition for stability as the system will be null recurrent in this case (refer to \citet[Section III B]{sushil_blockchain} for more detailed explanation). Thus, we need additional external control to stabilize the system. In general, we consider state-dependent arrival rates. It is often the case with matching platforms that the system operator can use pricing to influence the arrival rate of customers and servers. For example, the system operator can increase the customer price to reduce its arrival rate as fewer customers would be willing to accept the higher price and similarly, increase the server price to increase its arrival rate as more servers would be willing to serve for a higher price offered.

\revcolor{Given the time epoch $k \in \bbZ_+$ and imbalance $z \in \bbZ$, let $a^c(z, k) \in \bbZ_+$ and $a^s(z, k) \in \bbZ_+$ be the random variables denoting the number of customer and server arrivals respectively. We assume that the random variables $\{a^c(z, k): z \in \bbZ, k \in \bbZ_+\} \cup \{a^s(z, k): z \in \bbZ, k \in \bbZ_+\}$ are mutually independent. Let $\lambda_{\max}, \mu_{\max}, \sigma_{\max}, \moment > 0$ be constants and $\sigma^c: \bbR \rightarrow [0, \sigma_{\max}]$, $\sigma^s : \bbR \rightarrow [0, \sigma_{\max}]$ be continuous, measurable functions such that for all $z \in \bbZ$, $\{a^c(z, k): k \in \bbZ_+\}$ are identically distributed with $\max_{z\in \bbZ}|\E{a^c(z, 1)}| \leq \lambda_{\max}$, $\max_{z\in \bbZ}|\Var{a^c(z, 1)}| \leq \sigma_{\max}$, $\max_{z \in \bbZ}\E{|a^c(z,1)|^3} \leq \moment$, and $\Var{a^c(z, 1)} = \sigma^c\left(\E{a^c(z, 1)}\right)$. Similarly, for all $z \in \bbZ$, $\{a^s(z, k): k \in \bbZ_+\}$ are identically distributed with $\max_{z\in \bbZ}|\E{a^s(z, 1)}| \leq \mu_{\max}$, $\max_{z\in \bbZ}|\Var{a^s(z, 1)}| \leq \sigma_{\max}$, $\max_{z \in \bbZ}\E{|a^s(z,1)|^3} \leq \moment$, and $\Var{a^s(z, 1)} = \sigma^s\left(\E{a^s(z, 1)}\right)$. In addition, define functions $\lambda : \bbZ \rightarrow [0, \lambda_{\max}]$, $\mu:\bbZ \rightarrow [0, \mu_{\max}]$ as $\lambda(z) \overset{\Delta}{=}\E{a^c(z, 1)}$, $\mu(z)\overset{\Delta}{=}\E{a^s(z, 1)}$ for all $z \in \bbZ$. We denote $(\lambda(\cdot), \mu(\cdot))$ as the pricing policy.
For notational convenience, for all $z \in \bbZ$, we let $a^c(z)$ and $a^s(z)$ \pfix{be} generic random variables with the same distribution as $a^c(z, 1)$ and $a^s(z, 1)$ respectively. Now, we define imbalance as a discrete-time Markov chain (DTMC) denoted as $\{z(k): k \in \bbZ_+\}$. In particular, given $z(0) \in \bbZ_+$, the evolution of the imbalance is given by
\begin{align} \label{eq: imabalance_evolution}
    z(k+1)=z(k)+a^c(z(k),k)-a^s(z(k),k) \quad \forall k \in \bbZ_+.
\end{align}
We restrict ourselves to arrival distributions $(a^c(z), a^s(z): z \in \bbZ)$} such that the underlying DTMC is irreducible and aperiodic with state-space $\bbZ$. For example, irreducibility is guaranteed if there is a non-zero probability that the DTMC will transition to a higher or a lower value for any $z \in \bbZ$. We later impose a weaker restriction than this (see Condition~\ref{ass: irreducibility}) for our analysis. Given irreducibility, aperiodicity is guaranteed if there exists a $z \in \bbZ$ such that $a^c(z) = a^s(z)$ with non-zero probability. To reiterate, we assume the DTMC governing the imbalance is irreducible (over $\bbZ$) and aperiodic. Therefore, given the pricing policy, if the DTMC is positive recurrent, there exists a unique stationary distribution and we say that the DTMC is stable. We denote the imbalance in steady state with a bar on top, i.e., $\barz$. 

Ideally, one would like to analytically obtain the exact distribution of the imbalance in the steady state. However, this is not possible in general, so we study the matching queue in an asymptotic regime. We are interested in the performance of pricing policies such that the external control vanishes, analogous to the heavy traffic regime in a single server queue as explained in the introduction. \revcolor{In particular, we consider a sequence of pricing policies parametrized by $\eta$ and restrict ourselves to the following family of policies characterized by two parameters, $\epsilon_\eta>0$ and $\tau_\eta>0$ as well as Borel measurable control curves $\phi^c : \bbR \rightarrow \bbR$, $\phi^s : \bbR \rightarrow \bbR$ and maximum and minimum arrival rates $\lambda_{\max} \geq \lambda_{\min} \geq 0$ and $\mu_{\max} \geq \mu_{\min} \geq 0$. For all $z \in \bbZ$ and $\eta > 0$, we have
\begin{align} \label{eq: general_pricing_policy}
    \lambda_\eta(z)=\text{clip}\left(\lambda^\star_\eta+\epsilon_\eta\phi^c\left(\frac{z}{\tau_\eta}\right), \lambda_{\max}, \lambda_{\min}\right), \  \mu_\eta(z)=\text{clip}\left(\mu^\star+\epsilon_\eta\phi^s\left(\frac{z}{\tau_\eta}\right), \mu_{\max}, \mu_{\min}\right).
\end{align}
To ensure the above quantities are well defined, we assume $\lambda_{\min} < \mu^\star < \lambda_{\max}$ and $\mu_{\min} < \mu^\star < \mu_{\max}$. One may view $(\lambda^\star_\eta, \mu^\star)$ to be the exogenous arrival rates and the term involving $(\phi^c, \phi^s)$ to be the fluctuations around these arrival rates, achieved via an external control such as pricing. As the exogenous arrival rates may not be exactly equal in practice, we allow them to be unequal, i.e., we allow the exogenous customer arrival rate to depend on $\eta$ with $\lambda^\star_\eta \neq \mu^\star$ but their difference is vanishingly small: we consider $\lambda^\star_\eta=\mu^\star - \drift\min\{\epsilon_\eta, 1/\tau_\eta\}$ for some $\drift \in \bbR$ introducing an additional lower order drift.

For any $\eta > 0$, the customer and server arrivals given the imbalance $z$ in time-slot $k$ is denoted by $a_\eta^c(z, k)$ and $a_\eta^s(z, k)$ respectively. For completeness, we summarize all the assumptions we impose on the arrival distribution for any $\eta > 0$. As before, $a_\eta^c(z, k), a_\eta^s(z, k)$ are mutually independent for all $z \in \bbZ$ and $k \in \bbZ_+$ with $\{a_\eta^c(z, k) : k \in \bbZ_+\}$ being identically distributed as well as $\{a_\eta^s(z, k) : k \in \bbZ_+\}$ being identically distributed for all $z \in \bbZ$. We let $a_\eta^c(z)$ and $a_\eta^s(z)$ \pfix{be} generic random variables with the same distribution as $a^c_\eta(z, 1)$ and $a_\eta^s(z, 1)$ respectively for all $z\in\bbZ$. For all $z \in \bbZ, \eta > 0$, the means are given by $\E{a_\eta^c(z)} = \lambda_\eta(z)$ and $\E{a_\eta^s(z)} = \mu_\eta(z)$ as defined in \eqref{eq: general_pricing_policy}, the variances by $\Var{a^c_\eta(z)} = \sigma^c(\lambda_\eta(z))$ and $\Var{a^s_\eta(z)} = \sigma^s(\mu_\eta(z))$, and the third moments are bounded, i.e., $\E{|a^c_\eta(z)|^3} \leq \moment$ and $\E{|a^s_\eta(z)|^3} \leq \moment$. Then, for any $\eta > 0$, the imbalance $\{z_\eta(k) \in \bbZ : k \in\bbZ_+\}$ is given by \eqref{eq: imabalance_evolution} with $(a^c, a^s)$ replaced by $(a^c_\eta, a^s_\eta)$ and $z_\eta(0) = 0$. As before, we restrict to arrival distributions $(a^c_\eta(z), a^s_\eta(z) : z \in \bbZ)$ such that $\{z_\eta(k) \in \bbZ : k \in\bbZ_+\}$ is irreducible and aperiodic.

As we do not necessarily assume $(\phi^c, \phi^s)$ to be bounded functions, we additionally introduce the upper and lower bounds $(\lambda_{\max}, \lambda_{\min}, \mu_{\max}, \mu_{\min})$ on $(\lambda_\eta(\cdot), \mu_\eta(\cdot))$. In particular, as we assume the absolute third moment of the arrivals to be bounded, the arrival rates must be upper bounded too. In addition, the arrival rates are trivially lower bounded by $0$. More generally, we lower bound it by $\lambda_{\min}, \mu_{\min} \geq 0$.}


\revcolor{This class of state-dependent controls generalizes the ones} in the literature studied by \citet{kim2017value, SIGMETRICS_strategic_servers, ridehailing_sigmetrics} which are shown to have good performance in terms of delay and profit. Such a class of controls was first introduced by \citet{kim2017value} in the context of a classical single server queue. They presented near-optimal static, two-price, and dynamic pricing policies that are of the form \eqref{eq: general_pricing_policy}. This form of state-dependent control was further shown to be near-optimal in the context of two-sided queues by \citet{SIGMETRICS_strategic_servers, ridehailing_sigmetrics}.

The parameter $\epsilon_\eta$ modulates the magnitude of the control, which we call the magnitude scaling parameter. By picking it such that $\lim_{\eta \uparrow \infty} \epsilon_\eta=0$, we let the control vanish. The influence of the parameter $\tau_\eta$ is more subtle. It lets us tune the scale of the imbalance $z$ at which we apply the control.  In other words, by doubling $\tau_\eta$, we apply the same control only when the imbalance is doubled. With a larger $\tau_\eta$, the rate of change of $z_\eta(k) / \tau_\eta$ decreases in time and so, we call it the time scaling parameter. If we let $\lim_{\eta \uparrow \infty} \tau_\eta =\infty$, we end up applying no state-dependent control, and so this is equivalent to removing the control. Thus, we will study the matching queue when $\lim_{\eta \uparrow \infty} \epsilon_\eta=0$ \emph{and/or} $\lim_{\eta \uparrow \infty} \tau_\eta =\infty$.  
The parameter $\epsilon_\eta$  is similar to the heavy-traffic parameter in a classical single server queue. The parameter $\tau_\eta$ is new in this context and it appears because we use state-dependent control.

Whenever the DTMC $\{z_\eta(k): k \in \bbZ_+\}$ is positive recurrent, let $\barz_\eta$ denote a random variable with distribution same as its stationary distribution. In the asymptotic regime when the control goes to zero, the imbalance $\barz_\eta$ also blows up because we know that the system is null recurrent when there is no external control. Therefore, we need to scale it by the rate at which it blows up to study its limiting behavior. A striking feature of the matching queue is that this rate as well as the limiting behavior crucially depends on the rate at which $\epsilon_\eta$ and $\tau_\eta$ converge to $0$ and $\infty$ respectively. In particular, we define $l \overset{\Delta}{=} \lim_{\eta \uparrow \infty} \epsilon_\eta \tau_\eta$ and consider three cases. 
When $l=0$, we will see that $\barz_\eta=\Theta(1/\epsilon_\eta)$ and 
so we study the limiting behavior of $\epsilon_\eta \barz_\eta$ as $\eta \uparrow \infty$. 
When $l=\infty$, we will see that $\barz_\eta=\Theta(\tau_\eta)$ and so we study the limiting behavior of $\barz_\eta/\tau_\eta$ as $\eta \uparrow \infty$. 
Lastly, in the other case when $l \in (0,\infty)$, we will see that $\barz_\eta=\Theta(1/\epsilon_\eta)=\Theta(\tau_\eta)$, and we study the limiting behavior of  $\epsilon_\eta \barz_\eta$ as $\eta \uparrow \infty$, which is same as that of  $\barz_\eta/\tau_\eta$ up to the multiplicative factor $l$. 
The objective of this paper is to characterize the limiting distribution of appropriately scaled imbalance ($\epsilon_\eta\barz_\eta$ or $\barz_\eta/\tau_\eta$) in the steady state as $\eta \uparrow \infty$ for any given $\{\epsilon_\eta\}$, $\{\tau_\eta\}$, $\phi^c$, and $\phi^s$. We will demonstrate a phase transition in the limiting distribution across the three regimes described above. In the next section, we present a simple example to illustrate this behavior. In the further sections, all the quantities concerned with the $\eta^{th}$ system are sub-scripted by $\eta$. 

\section{An Illustrative Example: Bernoulli Matching Queue} \label{sec: bernoulli}
The goal of this section is to illustrate the phase transition phenomenon exhibited by the limiting distribution of the scaled imbalance by considering a simple system. We consider a matching queue operating under the two-price policy and when the arrivals are Bernoulli. First, we define the two-price policy formally.
\begin{definition} \label{def: two_price_policy}
The two price policy is a special case of \eqref{eq: general_pricing_policy} with $\phi^c(x)=-\mathbbm{1}\{x>1\}$ and $\phi^s(x)=-\mathbbm{1}\{x<-1\}$. Specifically, for $\lambda^\star = \mu^\star \in (0, 1)$, we have
\begin{align} \label{eq: two_price_policy}
    \lambda_\eta(z)=\lambda^\star-\epsilon_\eta \mathbbm{1}\{z>\tau_\eta\}, \ \mu_\eta(z)=\mu^\star-\epsilon_\eta \mathbbm{1}\{z<-\tau_\eta\} \quad \forall z \in \bbZ, \forall \eta>0.
\end{align}
\end{definition}
\begin{figure}[hbt!]
\FIGURE{
\begin{tikzpicture}[scale=1.3]
    \draw[black, thick,<->] (-4,0) -- (4,0);
    \draw[black, thick, ->] (0, -0.25) -- (0,1.25);
    \draw[blue, thick] (-3.5, 0.75) -- (2.5, 0.75);
    \draw[blue, thick] (2.5, 0.75) -- (2.5, 0.5);
    \draw[blue, thick] (2.5, 0.5) -- (3.5, 0.5);
    \draw[orange, dashed, very thick] (3.5, 0.75) -- (-2.5, 0.75);
    \draw[orange, dashed, very thick] (-2.5, 0.75) -- (-2.5, 0.5);
    \draw[orange, dashed, very thick] (-2.5, 0.5) -- (-3.5, 0.5);
    \draw[black, dashed] (2.5, 0.5) -- (2.5, 0);
    \node at (2.5, -0.2) {$\tau_\eta$};
    \draw[black, dashed] (-2.5, 0.5) -- (-2.5, 0);
    \node at (-2.5, -0.2) {$-\tau_\eta$};
    \node at (0.5, 0.94) {$\lambda^\star=\mu^\star$};
    \filldraw[black] (0, 0.75) circle (1.2pt);
    \draw [decorate, thick, decoration = {brace, mirror}] (3.6 ,0.5) -- (3.6, 0.75);
    \node at (3.78, 0.625) {$\epsilon_\eta$};
    \draw[black] (4.2, 0.2) -- (4.2, 0.6) -- (6, 0.6) -- (6, 0.2) -- (4.2, 0.2);
    \draw[black] (4.2, 0.6) -- (4.2, 1) -- (6, 1) -- (6, 0.6);
    \node at (5.5, 0.4) {$\mu_\eta(z)$};
    \draw[orange, dashed, very thick] (4.3, 0.4) -- (4.9, 0.4);
    \node at (5.5, 0.8) {$\lambda_\eta(z)$};
    \draw[blue, thick] (4.3, 0.8) -- (4.9, 0.8);
    \node at (4, -0.2) {$z$};
\end{tikzpicture}
}{Illustration of the two price policy defined in \eqref{eq: two_price_policy} \label{fig: two_price_policy}}{}
\end{figure}
In other words, when there are too many customers in the system $(z>\tau_\eta)$, we increase the price for the customers which leads to the reduction in the arrival rate by $\epsilon_\eta$. Similarly, when there are too many servers $(z<-\tau_\eta)$, we decrease the price offered to the servers which leads to an $\epsilon_\eta$ reduction in the arrival rate. Thus, the two-price policy is a simple intuitive pricing policy, where $\epsilon_\eta$ is the perturbation of the arrival rates when the imbalance is outside a threshold $\tau_\eta$. 
This example also illustrates the need for two different parameters $\epsilon_\eta$ and $\tau_\eta$, and how they can be separately tuned to make the control vanish. We also present an illustration of the policy in Figure \ref{fig: two_price_policy}.
It has been shown by \citet{ridehailing_sigmetrics, sushil_selfish_arxiv} that the two-price policy is near-optimal in terms of the profit earned by the system operator and delay experienced by the customers and servers.

Consider a matching queue operating under the two-price policy given by Definition \ref{def: two_price_policy} such that $\lambda^\star<1$ and $\lambda^\star-\epsilon_\eta>0$. The arrivals are Bernoulli i.e. $a^c_\eta(z)=1$ with probability $\lambda_\eta(z)$ and $0$ otherwise. Similarly, $a^s_\eta(z)=1$ with probability $\mu_\eta(z)$ and $0$ otherwise. 
We will use this example to illustrate that the imbalance exhibits phase transition for $l=0$ and $l=\infty$. In particular, we show that appropriately scaled imbalance converges to a Laplace distribution for $l=0$, a Uniform distribution for $l=\infty$, and a hybrid of Laplace and Uniform distribution for $l \in (0,\infty)$, where the \pfix{PDF} of $\textrm{Hybrid}(b,c)$ is given by \eqref{eq: def_gibbs}.
A Laplace distribution with parameters $0$ and $b$ is a two-sided exponential distribution with \pfix{rate $b$ (equivalently, mean $1/b$)} centered at $0$. The hybrid distribution essentially flattens the parts between $-c$ and $c$ as shown in Figure \ref{fig: cdf_pdf_phase_transition}. Note that, when $c=0$, the hybrid distribution is the same as $\operatorname{Laplace}(0,b)$, and when $b \rightarrow 0$, the hybrid distribution is approximately a uniform distribution with support $[-c,c]$, denoted by $\operatorname{Uniform}([-c,c])$. Now, we present the phase transition formally. 
\begin{proposition} \label{prop: bernoulli} Let $\{\epsilon_\eta\}_{\eta>0}$ and $\{\tau_\eta\}_{\eta>0}$ be such that $\lim_{\eta \uparrow \infty}\epsilon_\eta \tau_\eta=l \in [0, \infty]$.
Consider a matching queue operating under the two-price policy given by Definition \ref{def: two_price_policy} with $\lambda^\star=\mu^\star \in (0, 1)$. In addition, also assume that $a^c_\eta(z) \sim \textrm{Bernoulli}(\lambda_\eta(z))$ and $a^s_\eta(z) \sim \textrm{Bernoulli}(\mu_\eta(z))$ for all $z \in \bbZ, \eta>0$.
Then, as $\eta \uparrow \infty$, we have:
\begin{enumerate}
    \item When $l=0$, we have
    \begin{align*}
        \epsilon_\eta \barz_\eta \overset{D}{\rightarrow} \operatorname{Laplace}\left(0,\pfix{\frac{2}{\lambda^\star(1-\lambda^\star)+\mu^\star(1-\mu^\star)}}\right)
    \end{align*}
    \item When $l \in (0,\infty)$, we have
    \begin{subequations}
    \label{eq: hybrid_distribution}
    \begin{align}
        \epsilon_\eta \barz_\eta &\overset{D}{\rightarrow} \operatorname{Hybrid}\left(\pfix{\frac{2}{\lambda^\star(1-\lambda^\star)+\mu^\star(1-\mu^\star)}},l\right)  \\
     \frac{\barz_\eta}{\tau_\eta} &\overset{D}{\rightarrow} \operatorname{Hybrid}\left(\pfix{\frac{2l}{\lambda^\star(1-\lambda^\star)+\mu^\star(1-\mu^\star)}},1\right).
    \end{align}
    \end{subequations}
    \item When $l=\infty$, we have
    \begin{align*}
        \frac{\barz_\eta}{\tau_\eta} \overset{D}{\rightarrow} \operatorname{Uniform}([-1,1])
    \end{align*}
\end{enumerate}
\end{proposition}
\begin{figure}[t]
\FIGURE{
    \begin{minipage}[b]{0.48\textwidth}
    \FIGURE{
    \includegraphics[width=\linewidth]{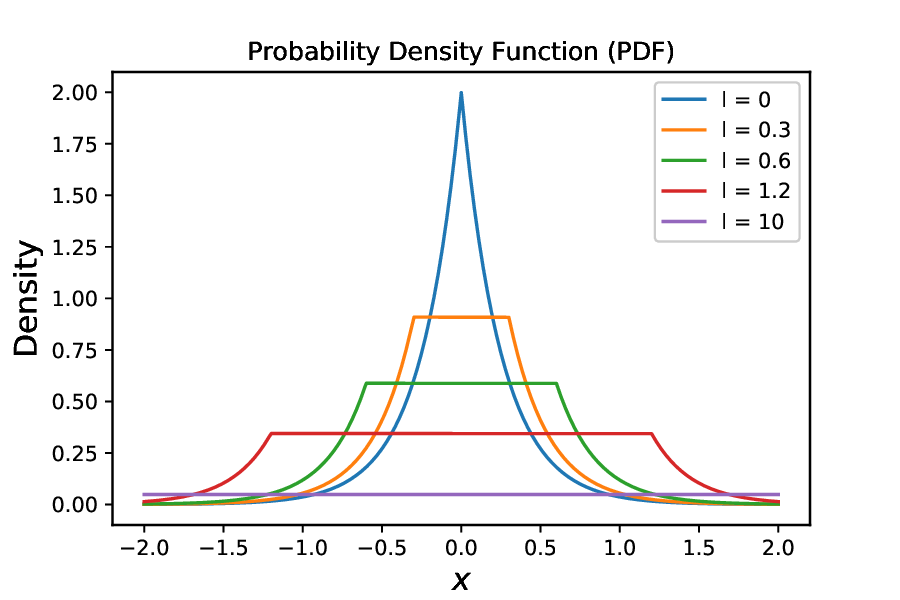}}{Laplace to Hybrid Transition
     \label{fig: lap_to_hyb}}{}
    \end{minipage}
     \begin{minipage}[b]{0.48\textwidth}
     \FIGURE{
    \includegraphics[width=\linewidth]{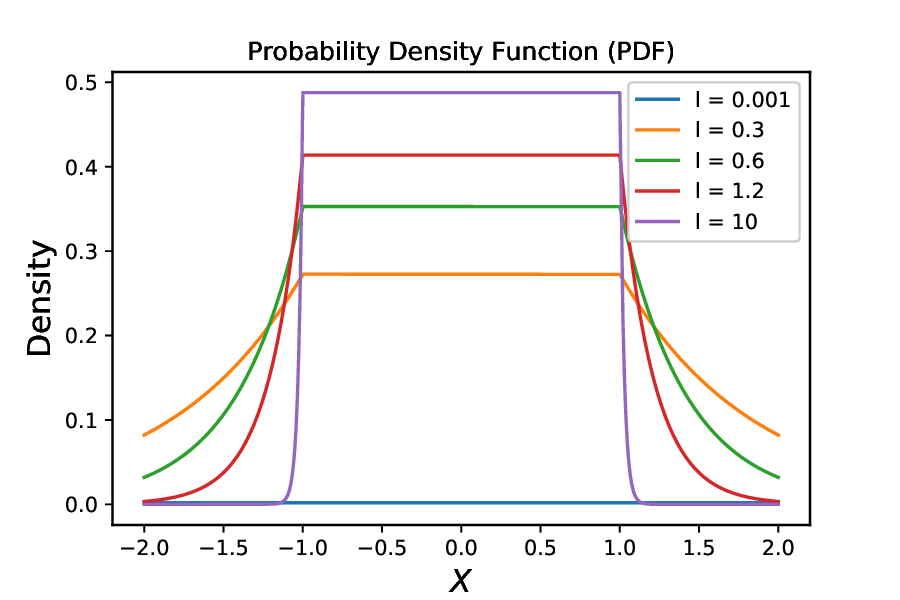}}{Hybrid to Uniform Transition
    \label{fig: hyb_to_uni}}{}
    \end{minipage}}{PDF of the limiting distribution of $\epsilon_\eta\barz_\eta$ and $\frac{\barz_\eta}{\tau_\eta}$ respectively in a single link matching queue with Bernoulli arrivals under the two-price policy
    \label{fig: cdf_pdf_phase_transition}}{}
\end{figure}
The proof of the proposition is presented in Appendix \ref{app: bernoulli}. 
The key idea is that the resulting DTMC governing the imbalance is a simple 
discrete-time birth and death process. This enables us to explicitly evaluate the stationary distribution for every $\eta$. Taking appropriate limits we get the three cases presented in the proposition. 

The above proposition presents a phase transition from a Laplace distribution to a uniform distribution. 
An illustration of the limiting distribution for various values of $l$ is presented in Figure \ref{fig: cdf_pdf_phase_transition}.
For this, we consider $\lambda^\star=\mu^\star=0.5$ and plot the limiting distribution of $\epsilon_\eta \barz_\eta$ and $\frac{\barz_\eta}{\tau_\eta}$ for different values of $l \in [0,\infty)$. When $l=0$, the PDF flattens out for $\frac{\barz_\eta}{\tau_\eta}$ which means that the probability mass escapes to infinity. Thus, to get a meaningful limit when $l=0$, we need to scale $\barz_\eta$ by $\epsilon_\eta$ as $\epsilon_\eta$ decays to zero faster than $1/\tau_\eta$. It can be seen in Figure \ref{fig: lap_to_hyb} that the limiting PDF of $\epsilon_\eta\barz_\eta$ is a Laplace distribution when $l=0$. Similarly, as $l$ becomes very large, the PDF of $\epsilon_\eta \barz_\eta$ flattens out to zero which means that the probability mass escapes to infinity. Thus, to get a meaningful limit when $l=\infty$, we need to scale $\barz_\eta$ by $\tau_\eta$ as $\frac{1}{\tau_\eta}$ decays to zero faster than $\epsilon_\eta$. It can be seen in Figure \ref{fig: hyb_to_uni} that the limiting PDF of $\frac{\barz_\eta}{\tau_\eta}$ is nearly a uniform distribution for $l=10$.

Another way to interpret this phase transition is as follows: The limiting distribution of $\epsilon_\eta \barz_\eta$ in the hybrid regime is $\operatorname{Hybrid}(\pfix{\frac{1}{\lambda^\star(1-\lambda^\star)}},l)$. Now, if we let $l \rightarrow 0$, then the Hybrid distribution converges to the Laplace distribution which is the limiting distribution of $\epsilon_\eta \barz_\eta$ in the delay-driven regime. In addition, in the hybrid regime, the limiting distribution of $\barz_\eta/\tau_\eta$ is $\operatorname{Hybrid}(\pfix{\frac{l}{\lambda^\star(1-\lambda^\star)}},1)$. Now, if we let $l \rightarrow \infty$, then the Hybrid distribution converges to  Uniform distribution which is the limiting distribution of $\barz_\eta/\tau_\eta$ in the cost-driven regime. We will later see that such a phase transition holds in more generality.

To intuitively understand the delay-driven regime, consider the case when $\tau_\eta$ is a constant. Then, $\epsilon_\eta \tau_\eta \rightarrow 0$ as $\epsilon_\eta \rightarrow 0$. In this case, $\epsilon_\eta$ acts similarly to the heavy traffic parameter in a single server queue. The result we obtain is also analogous. In particular, the limiting distribution of scaled queue length in a single server queue is exponential and we obtain a Laplace distribution for two-sided queues which is a two-sided exponential distribution. This is because imbalance is a signed random variable. The proposition says that the limiting distribution is invariant to the growth rate of $\tau_\eta$ as long as $l=0$.

Next, to intuitively understand the cost-driven regime, consider $\epsilon_\eta$ to be a constant. Then, $\epsilon_\eta \tau_\eta \rightarrow \infty$ as $\tau_\eta \rightarrow \infty$. Outside the threshold, we have a drift towards zero which is always bounded away from zero. Due to this, the mass of the limiting distribution of imbalance is concentrated between the two thresholds. Inside the threshold, all the states are identical to each other which leads to a uniform distribution between the thresholds. The proposition proves that even when $\epsilon_\eta \rightarrow 0$, we will observe such a distribution as long as $l=\infty$.  

Lastly, we observe a mixed behavior when $l \in (0,\infty)$. In particular, (corresponding to $\epsilon_\eta\barz_\eta$) when $l \rightarrow 0$, the parameter $c$ in the hybrid distribution converges to 0 which makes it the Laplace distribution and when $l \rightarrow \infty$, it converges to $\infty$ which results in an ill-defined distribution because it appears to be an infinitely spread uniform distribution. 
This is because $\barz_\eta=\Theta(\tau_\eta)$ and so, if the imbalance is scaled by $\tau_\eta$ instead of $\epsilon_\eta$, we obtain a uniform distribution between $\pfix{[-1,1]}$.

The primary reason for considering the limiting regimes is to understand the stationary behavior even when we are unable to explicitly find it. Therefore, in the rest of the paper, we consider the matching queue under general arrivals and a general pricing policy.
\section{\revcolor{Cost of State-Dependent Control: Bernoulli Matching Queue}} \label{sec: motivation}
As seen in the previous section, state-dependent control greatly influences the steady-state imbalance. Thus, by appropriately designing state-dependent control $(\phi^c, \phi^s, \epsilon_\eta, \tau_\eta)$, statistics like the mean, moments, or tail of the imbalance can be optimized. However, it is costly to implement such state-dependent control in practice. Thus, a trade-off between a low cost of control and maintaining a small imbalance arises. We illustrate that the control that optimizes this trade-off can correspond to any of the three regimes, depending on the cost of control.

Let $c: \bbR^2 \rightarrow \bbR$ be a function that maps the amount of customer and server control to its cost. In addition, let $h: \bbZ \rightarrow \bbR$ be a function that captures the cost of waiting. Now, we restrict ourselves to the control of the form \eqref{eq: general_pricing_policy}, and consider the following objective:
\revcolor{\begin{align*}
    \min_{\epsilon_\eta, \tau_\eta, \phi^c, \phi^s} \eta\E{c\left(\epsilon_\eta\phi^c\left(\frac{\barz_\eta}{\tau_\eta}\right), \epsilon_\eta\phi^s\left(\frac{\barz_\eta}{\tau_\eta}\right)\right)} + \E{h(\barz_\eta)}. \numberthis \label{eq:cost_waiting_obj}
\end{align*}}
\revcolor{In practice, one could view the limit $\eta \rightarrow \infty$ as a large-market regime as done by \citet{kim2017value, varma_twosided_or, SIGMETRICS_strategic_servers}, i.e., exogenous arrival rate (size of the market) scales with $\eta$. In a discrete-time system,  this amounts to scaling down the time between consecutive transitions by $\eta$, and so, the cost of waiting is scaled down by $\eta$, or equivalently, the cost of control is scaled up by $\eta$.}

For illustration, we restrict to the special case of two-price policy and Bernoulli arrival distributions with $\lambda^\star=\mu^\star=1/2$ as in Section~\ref{sec: bernoulli} \revcolor{and focus on minimizing the order of $\eta$ in \eqref{eq:cost_waiting_obj}, especially as the constants in \eqref{eq:cost_waiting_obj} could be further optimized by carefully designing the control curves $(\phi^c, \phi^s)$.} Set $c(x, y) = |x+y|^\alpha$ for some $\alpha \in (0, \infty)$ and $x , y \in \bbR$, and $h(z)=|z|$ for $z \in \bbZ$, then the above objective can be simplified to get
\revcolor{\begin{align*}
    \min_{\epsilon_\eta, \tau_\eta} \eta \epsilon_\eta^\alpha \P{|\barz_\eta| > \tau_\eta} + \E{|\barz_\eta|}. 
\end{align*}}
Now, as the arrivals are governed by Bernoulli distribution, the underlying Markov Chain is a birth and death process. By the stationary distribution given in \eqref{eq: stead_state_bernoulli} in Appendix~\ref{app: bernoulli}, we get
\begin{align*}
    \P{|\barz_\eta| > \tau_\eta} &= \frac{1}{4\lfloor \tau_\eta \rfloor \epsilon_\eta + 1+2\epsilon_\eta} \approx \frac{1}{4\tau_\eta \epsilon_\eta + 1}, \\ \E{|\barz_\eta|} &= \frac{2\lfloor \tau_\eta\rfloor \epsilon_\eta \left(1+\lfloor \tau_\eta \rfloor\right)+\lfloor \tau_\eta \rfloor + \frac{1}{4\epsilon_\eta}+\frac{1}{2}}{4\lfloor \tau_\eta \rfloor \epsilon_\eta + 1+2\epsilon_\eta} \approx \frac{2\tau_\eta^2 \epsilon_\eta + \tau_\eta + \frac{1}{4\epsilon_\eta}}{4\tau_\eta \epsilon_\eta + 1}.
\end{align*}
In the approximation step, we drop the floors and approximate $\lfloor \tau_\eta \rfloor + 1$ by $\tau_\eta$ as  $\lfloor \tau_\eta \rfloor + 1 = \Theta(\tau_\eta)$ for any $\tau_\eta \geq 1$. Using the above expressions, the optimization problem is simplified to
\begin{align*}
    \min_{\epsilon_\eta, \tau_\eta} \frac{\eta \epsilon_\eta^\alpha + 2\tau^2_\eta \epsilon_\eta + \tau_\eta + \frac{1}{4\epsilon_\eta}}{4\tau_\eta \epsilon_\eta + 1}. \numberthis \label{eq: simplified_objective}
\end{align*}
 Now, we illustrate that the optimal regime in \eqref{eq: simplified_objective} corresponds to one of the three regimes depending on the value of $\alpha$. In particular, a low cost of control will correspond to the delay-driven regime and a high cost will correspond to the cost-driven regime.

\textbf{Case I [Low Cost of Control] $(\alpha>1)$:} As $\epsilon_\eta < 1$ and $\alpha > 1$, this regime corresponds to a low cost of control ($\eta \epsilon_\eta^\alpha$ is small). In this case, \eqref{eq: simplified_objective} is minimized for $\epsilon_\eta = \Theta(\eta^{-\frac{1}{1+\alpha}})$ and any choice of $\tau_\eta$ such that $\epsilon_\eta\tau_\eta \rightarrow l \in [0, \infty)$. Thus, the optimal control falls in either the \emph{delay-driven or the hybrid regime}. \pfix{Applications to online marketplaces such as ride-hailing systems---where the state-dependent control is implemented via dynamic pricing---fall under this case.}  \citet{varma_twosided_or} showed that a static policy $\lambda(z) = \lambda^\star, \mu(z)=\mu^\star$ upper bounds the achievable profit and an $\epsilon_\eta$ perturbation reduces the profit by $\Theta(\eta \epsilon^2_\eta)$, corresponding to the cost of control with $\alpha=2$.

\textbf{Case II [Medium Cost of Control] $(\alpha=1)$:} In this case, \eqref{eq: simplified_objective} is minimized when $(\epsilon_\eta, \tau_\eta)$ are such that either $\epsilon_\eta = \Theta(1/\sqrt{\eta}), \epsilon_\eta\tau_\eta \rightarrow [0, \infty)$ or $\tau_\eta=\Theta(\sqrt{\eta}), \epsilon_\eta\tau_\eta \rightarrow (0, \infty]$. Thus, the optimal control could correspond to \emph{any of the three regimes}. Application to payment channel networks like lightning network for Bitcoin falls under this case, e.g., see \citet{sushil_blockchain}. Two users form a payment channel by both depositing some funds. The transaction requests can then be processed by adjusting the ownership of funds in the payment channel. However, if not enough capacity is available, the transaction is either queued, lost, or processed via the blockchain. The goal is to ensure high throughput of transactions while minimizing the usage of blockchain as it is an expensive outside option. The state-dependent control is the transactions via blockchain and so the cost of $\epsilon_\eta$ control is $\Theta(\eta \epsilon_\eta)$, proportional to the frequency/amount of transactions, i.e., we have $\alpha=1$.




\textbf{Case III [High Cost of Control] $(\alpha \in (0, 1))$:} In this case, \eqref{eq: simplified_objective} is minimized by $\tau_\eta = \Theta(\sqrt{\eta})$ and $\epsilon_\eta = \pfix{\Theta}(1)$. Thus, the optimal control \emph{corresponds to the cost-driven regime}. A ride-hailing system with an additional customer dissatisfaction cost due to surge pricing falls under this case. \revcolor{In a ride-hailing system, the platform can apply surge pricing (control) to reduce customer arrival rate $\lambda_\eta(z)$, but doing so may upset customers or reduce usage (i.e., high cost of applying the control). A high cost of control is modeled well with $\alpha<1$. Another example with high cost of control is manufacturing systems with machine setup-costs. Suppose in a manufacturing line you have machines that can be reconfigured (control) to shift between product types to match demand, but a reconfiguration has a high cost (time + lost output). The control (reconfiguring) is expensive, so you tolerate some mismatch between supply and demand (imbalance) until it becomes very large before reconfiguring.}

\section{Phase Transition: General Arrivals and Pricing Policy} \label{sec: main_theorem}
\revcolor{We now consider the more challenging setting of general control defined in \eqref{eq: general_pricing_policy} and general arrival distributions discussed in Section~\ref{sec: model}.} First, we show that under suitable conditions on the state-dependent control, the DTMC is positive recurrent. Then, we state the main results of the paper: limiting stationary distribution of the imbalance in the hybrid, delay-driven, and cost-driven regimes.
\subsection{Positive Recurrence}
To ensure positive recurrence, we need a drift that pushes the imbalance towards zero. In particular, if the imbalance is a very large positive value, then there are a lot of customers in the queue and a sensible pricing policy will either reduce the customer arrival rate or increase the server arrival rate. Similarly, if the imbalance is a very large negative value, increasing the customer arrival rate or decreasing the server arrival rate would be sensible. We present the following condition on the control curves which ensures the same. 
\revcolor{\begin{condition}[Negative Drift]
There \pfix{exist} $\delta, K, B>0$ such that $\phi^c(x)-\phi^s(x)<-\delta$ and $\left[\phi^c(x)\right]^+ + \left[\phi^s(x)\right]^{-} \leq B$ for all $x>K$, and $\phi^c(x)-\phi^s(x) >\delta$ and $\left[\phi^c(x)\right]^- + \left[\phi^s(x)\right]^{+} \leq B$ for all $x<-K$.\label{ass: neg_drift}
\end{condition}}
\revcolor{Note that the condition $\phi^c(x)-\phi^s(x)<-\delta$ for $x>K$ is necessary to ensure the customer arrival rate is less than the server arrival rate when the imbalance is a large positive value. However, this condition alone is not sufficient as $\phi^c, \phi^s$ can both grow without bound, which could, for example, result in $\lambda_\eta(z)=\mu_\eta(z)=\lambda_{\max}$ for $z$ large enough. Thus, we additionally impose $\left[\phi^c(x)\right]^+ + \left[\phi^s(x)\right]^{-} \leq B$ for $x > K$.} \pfix{We also require the control curves to grow sub-quadratically, captured by the following condition.}
\begin{condition} \label{ass:poly_growth} \textbf{(sub-Quadratic Growth)} There exists $\beta \in (0, 2)$ such that $|\phi^k(x)| \leq |x|^{\beta}$ for all $x \in \bbR$ and $k \in \{c, s\}$.
\end{condition}
Now, under \pfix{these conditions}, we will show that the underlying DTMC is positive recurrent.
\begin{proposition} \label{prop: pos_rec}
Assume Condition \ref{ass: neg_drift} holds for some $\delta, K, B > 0$ and Condition \ref{ass:poly_growth} holds for some $\beta \in (0, 2)$. Also, 
\begin{align}
    |\lambda^\star_\eta - \mu^\star| \leq \frac{\delta \epsilon_\eta}{2} \quad \forall \eta > 0. \label{eq:small_exogenous_diff}
\end{align}
Then, \revcolor{for all $\eta > \eta_p$ for some $\eta_p>0$}, the DTMC $\{z_{\eta}(k): k \in \bbZ_+\}$ is positive recurrent. Moreover, 
\begin{align}
    -\E{\barz_\eta\left(\lambda_\eta(\barz_\eta)-\mu_\eta(\barz_\eta)\right)} \leq 2\moment. \nonumber
\end{align}
In addition, we have
\revcolor{\begin{align}
    \E{|\barz_\eta|}&\leq \frac{4\moment+2\epsilon_\eta \tau_\eta K (2K^\beta+\delta)}{\epsilon_\eta \delta}, \label{eq: imbalance_expectation} \\
        \E{\barz_\eta^2} &\leq \frac{4\moment}{\epsilon_\eta \delta} + \frac{K^2\tau_\eta^2\left(\delta + 4K^\beta\right)}{\delta} + \frac{8\moment\left(4\moment+2\epsilon_\eta \tau_\eta K  (2K^\beta+\delta)\right)}{ \epsilon_\eta^2 \delta^2}.  \label{eq: imbalance_second_moment}  
\end{align}}
\end{proposition}
The aforementioned bounds will be useful for the analysis later. Note that \eqref{eq:small_exogenous_diff} essentially follows for $\delta \leq 2|\drift|$ from the expression $\lambda^\star_\eta = \mu^\star - \drift \min\{\epsilon_\eta, 1/\tau_\eta\}$ in Section~\ref{sec: model}. The proof employs the Foster-Lyapunov theorem. We analyze the drift of the quadratic test function, $z^2$, and show that it is negative (bounded away from zero) outside a finite set, which immediately implies positive recurrence. The bounds on the imbalance follow from a well-known corollary of Foster-Lyapunov theorem by \citet[Proposition 6.13]{hajekrandomprocbook}. \revcolor{Similarly, the second moment of the imbalance is bounded by studying the drift of the cubic test function: $z^3$.} The details of the proof are deferred to Appendix \ref{app: pos_rec}. We assume \eqref{eq:small_exogenous_diff} for the rest of the paper to ensure positive recurrence.

\subsection{Main Result: Phase Transition} \label{sec: inverse_fourier_transform}
Under \eqref{eq:small_exogenous_diff}, without the loss of generality, we consider $\lambda^\star_\eta = \lambda^\star$ for $\lambda^\star=\mu^\star$ whenever $\epsilon_\eta \tau_\eta \rightarrow l \in [0, \infty)$ as the difference $\lambda^\star_\eta - \mu^\star = \drift\min\{\epsilon_\eta, 1/\tau_\eta\} = \pfix{\Theta(\epsilon_\eta)}$ can be absorbed in the state-dependent control $\epsilon_\eta \phi^i(z/\tau_\eta)$ for $i \in \{c, s\}$. Now, we start by considering the case of $l \in (0,\infty)$. In this regime, it turns out that the limiting distribution explicitly depends on the control curves $\phi^c(\cdot)$ and $\phi^s(\cdot)$. \revcolor{To characterize this distribution, first define $g_{x, y}: \bbR \rightarrow \bbR$ for $x,y \in \bbR$ as follows:
\begin{align}
    g_{x,y}(\cdot)\overset{\Delta}{=} \frac{2x}{\sigma^c(\lambda^\star)+\sigma^s(\mu^\star)}\left(\phi^s\left(\frac{\cdot}{y}\right)-\phi^c\left(\frac{\cdot}{y}\right)\right). \label{eq: gibbs_distribution}
\end{align}
Intuitively, $g_{x, y}(z)$ characterizes the drift towards zero for imbalance equal to $z$}. We show that the limiting stationary distribution of the imbalance in this regime is $\operatorname{Gibbs}(g_{x, y})$ with appropriately defined $x, y \in \bbR$. Recall that Gibbs distribution is defined in \eqref{eq: def_gibbs}. \pfix{Along with the sub-Quadratic Growth condition (Condition~\ref{ass:poly_growth}), we need the following technical condition to state the limiting behavior in this regime.}
\begin{condition}[Smoothness]
 $\phi^c(\cdot), \phi^s(\cdot) \in \calC_{\textrm{pol}}^\infty(\bbR)$, i.e. they are infinitely differentiable. \label{ass: c_pol}
\end{condition}
Now, we present the main theorem in this regime.
\begin{theorem}[Hybrid Regime] \label{theo: case_2_itm} Let $\{\epsilon_\eta\}_{\eta>0}$ and $\{\tau_\eta\}_{\eta>0}$ be such that $\lim_{\eta \uparrow \infty}\epsilon_\eta \tau_\eta=l \in (0, \infty)$ and \revcolor{$\lambda^\star_\eta = \lambda^\star$ for all $\eta > 0$ with $\lambda^\star=\mu^\star$}. Consider the positive recurrent DTMC $\{z_\eta(k): k \in \bbZ_+\}$ \revcolor{for any $\eta >\eta_p$} and for any $\phi^c(\cdot)$ and $\phi^s(\cdot)$ satisfying Conditions \ref{ass: neg_drift} and \ref{ass:poly_growth} and let $\barz_\eta$ denote its steady state random variable. If, in addition, \revcolor{Condition \ref{ass: c_pol} is satisfied}, then as $\eta \uparrow \infty$, we have,
\begin{align}
    \epsilon_\eta \barz_\eta \overset{D}{\rightarrow} \operatorname{Gibbs}\left(g_{1,l}\right), \quad
    \frac{\barz_\eta}{\tau_\eta}\overset{D}{\rightarrow} \operatorname{Gibbs}\left(g_{l,1}\right). \label{eq: gibbs_small_and_large_l}
\end{align}
\end{theorem}
To understand the Gibbs distribution, consider the special case of the two-price policy from Definition \ref{def: two_price_policy}. Here, we get $g_{1,l}(z)=2\frac{\mathbbm{1}\{z>l\}-\mathbbm{1}\{z<-l\}}{\sigma^c(\lambda^\star) + \sigma^s(\mu^\star)}$ for $l \in (0, \infty)$. In addition, as the distribution of the arrivals is Bernoulli, we have $\sigma^c(\lambda^\star) = \lambda^\star(1-\lambda^\star)$ and $\sigma^s(\mu^\star)=\mu^\star(1-\mu^\star)$. Thus, we have
\begin{align*}
    \int_0^x g_{1,l}(y) dy &= \frac{2}{\lambda^\star(1-\lambda^\star)+\mu^\star(1-\mu^\star)}\left(\mathbbm{1}\{x>l\}(x-l)-\mathbbm{1}\{x<-l\}(l+x)\right).
\end{align*}
Therefore, $\operatorname{Gibbs}(g)$ as defined in \eqref{eq: def_gibbs} is the hybrid distribution given by \eqref{eq: hybrid_distribution}. 
While this is consistent with Proposition \ref{prop: bernoulli}, we cannot directly apply Theorem \ref{theo: case_2_itm} because the control curves under the two-price policy have jumps, and so do not satisfy Condition \ref{ass: c_pol}. This suggests that the Condition \ref{ass: c_pol} may not be necessary in Theorem \ref{theo: case_2_itm}, and is possibly an artifact of our proof. 
While relaxing Condition \ref{ass: c_pol} is future work, note that it is 
not too restrictive because 
 all polynomials satisfy it. Any continuous control curve can be then approximated by a polynomial arbitrarily well due to the Stone-Weierstrass theorem. The proof of Theorem~\ref{theo: case_2_itm} is presented in Section~\ref{sec: proof_itm}.

Now, we consider the limiting distribution of $\operatorname{Gibbs}(g_{1,l})$ defined above in the Theorem \ref{theo: case_2_itm} as $l \rightarrow 0$. Note that, we have
\begin{align*}
    e^{-\frac{2}{\sigma^c(\lambda^\star)+\sigma^s(\mu^\star)}\int_0^x \left(\phi^s\left(\frac{t}{l}\right)-\phi^c\left(\frac{t}{l}\right)\right) dt} \rightarrow \begin{cases}
     e^{-x\frac{2}{\sigma^c(\lambda^\star)+\sigma^s(\mu^\star)} \left(\phi^s(\infty)-\phi^c(\infty)\right)} &\textit{if } x>0 \\
     e^{\pfix{-}x\frac{2}{\sigma^c(\lambda^\star)+\sigma^s(\mu^\star)} \left(\phi^s(-\infty)-\phi^c(-\infty)\right)} &\textit{if } x<0,
    \end{cases} \numberthis \label{eq: asym_laplace_calc}
\end{align*}
Observe that the right-hand side is the PDF of a Laplace or an asymmetric Laplace distribution as defined in \eqref{eq: def_laplace}. 
Thus, we expect the limiting distribution of $\epsilon_\eta \barz_\eta$ to be (asymmetric) Laplace distribution in the delay-driven regime. However, note that this is not a formal proof due to unjustified limit interchanges. We now first state the result formally and then rigorously prove it using a different approach. \revcolor{Now, to state the result in the delay-driven regime, we impose the following condition:
\begin{condition}[Finite Limits] \label{cond:limits}
    The limits $\lim_{x \rightarrow \infty} \phi^i(x) < \infty$ and $\lim_{x \rightarrow -\infty} \phi^i(x) < \infty$ exists and denote them by $\phi^i(\infty)$ and $\phi^i(-\infty)$ respectively for $i \in \{c,s\}$.
\end{condition}}
Now, we define the following step function:
\begin{align*}
    \step(x) = \begin{cases}
        \phi^s(\infty) - \phi^c(\infty) &\textit{if } x \geq 0 \\
        \phi^s(-\infty) - \phi^c(-\infty) &\textit{otherwise}.
    \end{cases} \numberthis \label{eq:step_function}
\end{align*} 
Additionally, we require the following technical assumption on the arrival distributions in the rest of the section:
\begin{condition}[$\pm$ Jumps] \label{ass: irreducibility}
Assume Condition~\ref{ass: neg_drift} holds for some $\delta, K, B > 0$. Then, there exists $p_{\min}>0$ such that for all $\eta >0$, \revcolor{$|z| \leq K \max\left\{\tau_\eta, 1/\epsilon_\eta\right\}$}, we have
\begin{align*}
\P{a^c_\eta(z)>a^s_\eta(z)} \geq p_{\min}, \  \P{a^c_\eta(z)<a^s_\eta(z)} \geq p_{\min}.
\end{align*}
\end{condition}
Intuitively, given a time epoch $k$ and a small imbalance \revcolor{$|z| \leq K\max\{\tau_\eta, 1/\epsilon_\eta\}$}, there is a non-zero probability that the DTMC will transition to a higher or a lower value of imbalance. Now, we have the following theorem.
\begin{theorem}[Delay-Driven Regime] \label{theo: case_1} Let $\{\epsilon_\eta\}_{\eta>0}$ and $\{\tau_\eta\}_{\eta>0}$ be such that $\lim_{\eta \uparrow \infty}\epsilon_\eta \tau_\eta=0$ and \revcolor{$\lambda^\star_\eta = \lambda^\star$ for all $\eta > 0$ with $\lambda^\star = \mu^\star$}. Consider the positive recurrent  DTMC $\{z_\eta(k): k \in \bbZ_+\}$ \revcolor{for any $\eta >\eta_p$} and for any $\phi^c(\cdot)$ and $\phi^s(\cdot)$ satisfying Conditions \ref{ass: neg_drift} and \ref{ass:poly_growth} and let $\barz_\eta$ denote its steady state random variable. If, in addition, \revcolor{Conditions \ref{cond:limits} and \ref{ass: irreducibility} are satisfied}, then we have,
\begin{align*}
    \epsilon_\eta \barz_\eta \overset{D}{\rightarrow} \operatorname{Laplace}\left(0,\frac{2\sqrt{|\step(1)\step(-1)|}}{\sigma^c(\lambda^\star)+\sigma^s(\mu^\star)}, \sqrt{\frac{|\step(1)|}{|\step(-1)|}}\right).
\end{align*}
\end{theorem}
Observe that when $|\pfix{\step}(1)| = |\pfix{\step}(-1)|$, then, the above simplifies to a symmetric Laplace distribution as in Proposition~\ref{prop: bernoulli}.

Next, we consider the cost-driven regime. Note that the limiting distribution of $\barz_\eta/\tau_\eta$ is $\operatorname{Gibbs}(g_{l,1})$ in the hybrid regime. Now, if we let $l \rightarrow \infty$, then the PDF of $\operatorname{Gibbs}(g_{l,1})$ will vanish everywhere except where $\int_0^x g_{l,1}(t) dt$ attains its minimum. In particular, \revcolor{let $\Phi^\star$ be the Borel measurable set defined as $\Phi^\star \overset{\Delta}{=} \arg\min_{x \in \bbR}\int_0^x g_{l,1}(t) dt$. Note that $\Phi^\star$ only depends on $(\phi^c, \phi^s)$ and is independent of $l$ due to how $g_{l, 1}$ is defined in \eqref{eq: gibbs_distribution}. We later (Condition~\ref{cond: monotonicity}) impose enough structure on $(\phi^c, \phi^s)$ to ensure that $\Phi^\star$ is non-empty with a finite Lebesgue measure $|\Phi^\star|$. We have}
\begin{align*}
    \frac{e^{\pfix{-\int_0^x g_{l,1}(t) dt}}}{\int_{-\infty}^\infty e^{\pfix{-\int_0^x g_{l,1}(t) dt}} dx} \rightarrow \frac{1}{|\Phi^\star|}\mathbbm{1}\{x \in \Phi^\star\}  \quad \textit{as } l \rightarrow \infty. \numberthis \label{eq: uniform_heuristics}
\end{align*}
\revcolor{This is the same as soft-max (of $-l\int_0^x g_{l, 1}(t) dt$) converging to the exact maximum in the limit as $l \rightarrow \infty$.} Thus, we expect the limiting distribution of $\barz_\eta/\tau_\eta$ to be uniform over the set of arguments that minimizes $\int_0^x (\phi^s(t)-\phi^c(t)) dt$. The case of $\lambda_\eta^\star \neq \mu^\star$ is non-trivial in this regime as the drift for $\barz_\eta/\tau_\eta \notin \Phi^\star$ is $\Theta(\epsilon_\eta)$ and for $\barz_\eta/\tau_\eta \in \Phi^\star$ is $|\lambda_\eta^\star-\mu^\star| = \Theta(1/\tau_\eta)$. As $\epsilon_\eta \gg 1/\tau_\eta$, the probability mass outside $\Phi^\star$ vanishes. Within $\Phi^\star$, we \pfix{obtain either} a Uniform distribution if $\lambda_\eta^\star=\mu^\star$ as in \eqref{eq: uniform_heuristics} \pfix{or a} truncated exponential if $|\lambda_\eta^\star-\mu^\star| = \Theta(1/\tau_\eta)$. Due to technical reasons, we consider a class of control curves satisfying the following conditions, which ensures that the minimum is attained over a unique interval.
\begin{condition}[Monotonicity] \label{cond: monotonicity}
Function $\phi^c(\cdot)$ is non-increasing and $\phi^s(\cdot)$ is non-decreasing. We define $\Phi^\star \overset{\Delta}{=}\left\{x: \phi^c(x)-\phi^s(x)=0\right\}$ and assume $0 \in \Phi^\star$ without loss of generality. We also assume $\phi^c(x)=\phi^s(x) = 0$ for all $x \in \Phi^\star$.
\end{condition}
Intuitively, as the number of customers increases relative to the servers, the arrival rate of customers defined by $\phi^c(\cdot)$ should be non-increasing, and the arrival rate of servers defined by $\phi^s(\cdot)$ should be non-decreasing. This motivates the monotonicity condition which implies the existence of a unique interval for which $(\phi^c(\cdot) - \phi^s(\cdot)) = 0$. \revcolor{We also assume that $\phi^c(x)=\phi^s(x)=0$ for all $x \in \Phi^\star$. In particular, as $\phi^c$ is non-increasing and $\phi^s$ is non-decreasing, there exists $c^\star \in \bbR$ such that $\phi^c(x) = \phi^s(x) = c^\star$ for all $x \in \Phi^\star$. To ensure that the arrival rates defined in \eqref{eq: general_pricing_policy} \pfix{are} not constrained by the maximum and minimum arrival rates $\lambda_{\max}, \mu_{\max}, \lambda_{\min}, \mu_{\min}$, we require $\mu^\star + c^\star \lim_{\eta \rightarrow \infty} \epsilon_\eta \in (\max\{\lambda_{\min}, \mu_{\min}\}, \min\{\lambda_{\max}, \mu_{\max}\})$. We ensure this condition is met by assuming $c^\star = 0$ in Condition~\ref{cond: monotonicity} for simplicity. This assumption is not strong as one could subsume $c^\star \lim_{\eta \rightarrow \infty} \epsilon_\eta$ in the exogenous arrival rate.} Note that $\Phi^\star$ is exactly the set that minimizes $\int_0^x (\phi^s(t)-\phi^c(t)) dt$. By Condition~\ref{ass: neg_drift}, $\Phi^\star \subseteq [-K, K]$ which ensures $|\Phi^\star|< \infty$. Now, we present the result below which extends the Case III of Proposition \ref{prop: bernoulli}.
\begin{theorem}[Cost-Driven Regime] \label{theo: profit_driven_regime_3} Let $\{\epsilon_\eta\}_{\eta>0}$ and $\{\tau_\eta\}_{\eta>0}$ be such that $\lim_{\eta \uparrow \infty}\epsilon_\eta \tau_\eta=\infty$ \revcolor{and $\lambda^\star_\eta = \mu^\star - \drift/\tau_\eta$ for some $\drift \in \bbR$. Also, let $\eta_0 > 0$ be such that $\tau_\eta \epsilon_\eta \geq 2\drift/\delta$ for all $\eta \geq \eta_0$.} Consider the positive recurrent  DTMC $\{z_\eta(k): k \in \bbZ_+\}$ for any \revcolor{$\eta \geq \max\{\eta_0, \eta_p\}$} and for any $\phi^c(\cdot)$ and $\phi^s(\cdot)$ satisfying Conditions \ref{ass: neg_drift} and \ref{ass:poly_growth} and let $\barz_\eta$ denote its steady state random variable. If in addition, Conditions \ref{ass: irreducibility} and \ref{cond: monotonicity} are satisfied, then, we have the following.
\begin{itemize}
    \item If $\Phi^\star \backslash \{0\} \neq \emptyset$ and $\drift = 0$, then as $\eta \uparrow \infty$, we have
\begin{align*}
    \frac{\barz_\eta}{\tau_\eta} \overset{D}{\rightarrow} \operatorname{Uniform}(\Phi^\star).
\end{align*} 
    \item Let $\sigma^\star = \sigma^c\left(\mu^\star\right) + \sigma^s\left(\mu^\star\right)$. If $\Phi^\star \backslash \{0\} \neq \emptyset$ and $\drift \neq 0$, then as $\eta \uparrow \infty$, we have
\begin{align*}
    \frac{\barz_\eta}{\tau_\eta} \overset{D}{\rightarrow} \operatorname{TrunExp}\left(\frac{2d}{\sigma^\star}, \Phi^\star\right). \numberthis \label{eq: trun_expo}
\end{align*} 
\item If $\Phi^\star = \{0\}$, then as $\eta \uparrow \infty$, we have
\begin{align*}
    \frac{\barz_\eta}{\tau_\eta} \overset{D}{\rightarrow} 0. 
\end{align*}
In other words, $\lim_{\eta \uparrow \infty}\P{\frac{|\barz_\eta|}{\tau_\eta} \geq \tilde{\epsilon}} = 0$ for all $\tilde{\epsilon} > 0$.
\end{itemize}
\end{theorem}
When $\Phi^\star$ is an interval, the limiting distribution is Uniform in this interval as in Proposition~\ref{prop: bernoulli} (Case 3). However, if $\Phi^\star$ is a singleton, then, we obtain a Dirac-Delta distribution at this point. Lastly, if $\drift\neq 0$, then, we introduce an additional drift of $\Theta(1/\tau_\eta)$ and so, the distribution within $\Phi^\star$ is not uniform anymore. As the drift is constant within $\Phi^\star$, the limiting distribution is truncated exponential as shown in \eqref{eq: trun_expo}.

With the above result, the limiting distribution of imbalance is characterized for all three regimes - delay-driven, hybrid, and cost-driven for a large class of control curves as summarized in Table~\ref{tab:req_cond}.
\begin{table}[bth!]
\TABLE{Summary of the three regimes and the required conditions on the control curves \label{tab:req_cond}}{
    \centering
    \begin{tabular}{|c|c|} \hline
        Regime & Required Conditions \\ \hline
        Positive Recurrence (all regimes) & C\ref{ass: neg_drift} (Negative drift), C\ref{ass:poly_growth} (sub-Quadratic Growth) \\ \hline 
        Delay-Driven $(\epsilon_\eta \tau_\eta \to 0)$ &  C\ref{cond:limits} (Finite Limits), C\ref{ass: irreducibility} ($\pm$ Jumps) \\ \hline
        Hybrid $(\epsilon_\eta \tau_\eta \to (0,\infty))$ &  C\ref{ass: c_pol} (Smoothness) \\ \hline
        Cost-Driven $(\epsilon_\eta \tau_\eta \to \pfix{\infty})$ &  C\ref{ass: irreducibility} ($\pm$ Jumps), C\ref{cond: monotonicity} (Monotonicity) \\ \hline
    \end{tabular}
}{}
\end{table}
The advantage of considering a heavy traffic limit is apparent now. It allows us to consider general arrival distributions and general control policies while embodying the first-order effect of state-dependent control on the imbalance. In particular, even though the control vanishes in the limit, the limiting distribution depends explicitly on the control curves $(\phi^c, \phi^s)$. Also, the limiting distribution is considerably generalized in all three regimes as compared to the two-price policy and we continue to observe a phase transition. \revcolor{Lastly, we  point out that all our results are valid for settings where it is \pfix{possible to apply the control on only} one of the sides of the market, i.e., either $\phi^c \equiv 0$ or $\pfix{\phi^s} \equiv 0$.} Now, we present the details of the proofs in the next few sections.
\section{Proof of Theorem \ref{theo: case_2_itm}: Inverse Fourier Transform Method} \label{sec: proof_itm}
The proof is based on the novel inverse Fourier transform method. We first outline our proof methodology and then follow it up with the details of the proof. 
\subsection{Overview of the Proof} The proof is divided into three steps as outlined below.
\revcolor{\subsubsection{Step 1: Tightness} \label{sec:tightness}
The first step is to establish the tightness of the family of random variables $\Pi\overset{\Delta}{=}\{\epsilon_\eta \barz_\eta: \eta>0\}$, which can be easily established using the bound on the imbalance in \eqref{eq: imbalance_expectation} of Proposition \ref{prop: pos_rec}.
\begin{lemma} \label{lemma: tightness}
Under Condition \ref{ass: neg_drift}, for the choice of $\epsilon_\eta$ and $\tau_\eta$ such that $\epsilon_\eta \tau_\eta \rightarrow l \in (0, \infty)$ as $\eta \uparrow \infty$, the family of random variables $\Pi=\left\{\epsilon_\eta \barz_\eta \right\}$ is tight. 
\end{lemma}
 The proof of the Lemma is presented in Appendix \ref{app: tightness}. By \citet[Theorem 5.1]{billingsley2013convergence}, $\Pi$ is relatively compact. Thus, for any sequence in the family $\Pi$, there exists a sub-sequence that converges in distribution.}
 Denote by $\zinfty$ the limit of this convergent sub-sequence. We work with this sub-sequence in the further subsections.
\subsubsection{Step 2: Drift Analysis}
The key idea in the proof is to use $e^{j\epsilon_\eta \omega z}$ as the test function for $\omega \in \bbR$ and set its drift to zero in the steady state. \revcolor{Note that $j = \sqrt{-1}$ as defined in Section~\ref{sec:notation}.} While this step is similar to the transform method of \citet{hurtado2020transform}, the key challenge is that when we let $\eta \uparrow \infty$, we do not get an explicit expression for the characteristic function of $\zinfty$. We instead get that the limit of every convergent sub-sequence satisfies the following implicit equation:
\begin{lemma} Under the same setup as in Theorem \ref{theo: case_2_itm}, we have \label{lemma: drift_analysis}
\begin{align*}
    \E{e^{j\omega \zinfty}g_{1,l}(\zinfty)} = j\omega \E{e^{j\omega \zinfty}} \quad \forall \omega \in \bbR. \numberthis \label{eq:IFT_eqn}
\end{align*}
\end{lemma}
The existence of a limit (of the sub-sequence) from Step 1 plays a crucial role in obtaining this equation. 
Suppose we show that there is a \textit{unique} distribution that solves this equation, then
using standard arguments on convergence, it follows that the family $\Pi$ also converges to the same distribution. The details of the proof are presented in Section \ref{sec: drift_gibbs}.
\subsubsection{Step 3: Solving the Functional Equation} \label{sec:ift}
To complete the proof, we show that \eqref{eq:IFT_eqn} has  $\operatorname{Gibbs}(g_{1,l})$ distribution as the unique solution.
\begin{lemma} \label{lemma: distribution_function}
Let $\zinfty$ be a random variable and $g_{1, l} \in \calC_{pol}^\infty(\bbR)$ be such that
\begin{align} \label{eq: functional_equation_mgf}
    \E{e^{j\omega \zinfty}g_{1, l}(\zinfty)} = j\omega\E{e^{j\omega \zinfty}} \quad \forall \omega \in \bbR.
\end{align}
Then $\zinfty$ has $\operatorname{Gibbs}(g_{1,l})$ distribution.
\end{lemma}
The above lemma generalizes Stein's lemma that asserts that $\zinfty$ has a Gaussian distribution when $g_{1,l}(z)=z$. In that case, \eqref{eq: functional_equation_mgf} can be viewed as the generator of the OU process applied to complex exponential functions. In more generality, \eqref{eq: functional_equation_mgf} is the same as applying the generator of Langevin dynamics (w.r.t. $g_{1,l}$) to complex exponential functions. So naturally, the distribution of $\zinfty$ is Gibbs which is same as the stationary distribution of the Langevin dynamics. Now, we provide a short proof sketch below:

\textbf{Step 3a: Inverse Fourier Transform:} Suppose that $\zinfty$ has a continuously differentiable PDF $\rho_{\zinfty}$, whose Fourier transform exists. Then, by substituting $-\omega$ for $\omega$, \eqref{eq: functional_equation_mgf} can be interpreted as 
\begin{align*}
    \calF\left(\rho_{\zinfty}g_{1,l}+ \rho_{\zinfty}'\right)=0,
\end{align*}
since the differentiation theorem of Fourier transform gives $j \omega\calF(\rho_{\zinfty})=\calF\left(\rho_{\zinfty}'\right)$. Applying the inverse Fourier transform, we get the differential equation,
\begin{align*}
    \rho_{\zinfty}g_{1,l}+ \rho_{\zinfty}'=0. \numberthis \label{eq:sketch_diff_eqn_pdf}
\end{align*}
\textbf{Step 3b: Solving the Differential Equation:} It is straightforward to solve the above differential equation to obtain that $\zinfty$ has the Gibbs distribution. However, one cannot assume that $\zinfty$  exhibits a PDF. We use the theory of inverse Fourier transforms based on Schwartz functions to make the above argument formal without assuming the existence of a PDF. We do this by interpreting the distribution of $\zinfty$ as a `generalized' function defined as a linear functional on the Schwartz space. This interpretation leads to a differential equation in terms of the generalized function which can then be solved to obtain the distribution of $\zinfty$, albeit with additional technical overhead. The details of the proof and background on Schwartz functions are provided in Section~\ref{sec:schwartz_lemma}.

\subsection{Proof of Theorem \ref{theo: case_2_itm}}
\proof{Proof of Theorem \ref{theo: case_2_itm}}
\textbf{Step 1:} By Proposition \ref{prop: pos_rec}, the DTMC $\{z_\eta(k) : k \in \bbZ_+\}$ is positive recurrent for all \revcolor{$\eta>\eta_p$}. Consider an arbitrary sequence in the family of random variables $\Pi$ and a sub-sequence that converges to a random variable $\zinfty$ in distribution. The existence of such a sub-sequence is guaranteed by Lemma \ref{lemma: tightness}.

\textbf{Step 2:} By Lemma \ref{lemma: drift_analysis}, we have
\begin{align*}
    \E{e^{j\omega \zinfty}\pfix{g_{1,l}}(\zinfty)}=j\omega \E{e^{j\omega \zinfty}}.
\end{align*}

\textbf{Step 3:} By using Condition \ref{ass: c_pol}, we have $g_{1,l}(\cdot) \in \mathcal{C}_{\textrm{pol}}^{\infty}(\bbR)$. Thus, by Lemma \ref{lemma: distribution_function}, the above equation has $\operatorname{Gibbs}(g_{1,l})$ as the unique solution. Thus, $\zinfty$ has  $\operatorname{Gibbs}(g_{1,l})$ distribution.

Now, consider any sequence in $\Pi$. Any sub-sequence of this sequence has a further sub-sequence that converges to $\zinfty$ in distribution. Thus, by \citet[Theorem 2.6]{billingsley2013convergence}, any sequence in $\Pi$, converges to $\zinfty$ in distribution. Finally, note that, by Slutsky's Theorem, we have
\begin{align*}
    \frac{\barz_\eta}{\tau_\eta}=\frac{\epsilon_\eta \barz_\eta}{\epsilon_\eta \tau_\eta} \overset{D}{\rightarrow} \frac{\zinfty}{l}.
\end{align*}
Thus, as $\zinfty \sim \operatorname{Gibbs}(g_{1,l})$, by simple variable substitution, we have $\zinfty/l \sim \operatorname{Gibbs}(g_{l,1})$ which implies that $\barz_\eta/\tau_\eta \overset{D}{\rightarrow} \operatorname{Gibbs}(g_{l,1})$. This completes the proof. \hfill $\square$
\endproof
The transform method was first introduced by \citet{hurtado2020transform}, and was used to study queues under static arrival rates. Consequently, \citet{hurtado2020transform} directly obtains a closed-form expression for the characteristic function of the limiting distribution which immediately establishes convergence in distribution to an exponential distribution. In contrast, due to the dynamic arrivals, we obtain an \emph{implicit} equation \eqref{eq:IFT_eqn}. A major methodological contribution in this section is the introduction of the use of inverse Fourier transform to establish the uniqueness of the solution of the implicit equation. Moreover, due to the implicit equation, we have to separately establish the guarantee that our family converges in distribution. We believe that our proposed method will enable one to use transform techniques in a large class of stochastic networks beyond the ones studied by \citet{hurtado2020transform}. 
\subsection{Proof of Lemma~\ref{lemma: drift_analysis}} \label{sec: drift_gibbs}
We consider $e^{j \omega \epsilon_\eta z}$ as the test function and set its drift to zero in the steady state. First, we present the evolution equation for the imbalance in the steady state. Denote by $a^c_\eta(\barz_\eta)$ and $a^s_\eta(\barz_\eta)$ \pfix{the} effective arrivals in the steady state. The imbalance after one transition $\barz_\eta^+$ is governed by the following evolution equation:
\begin{align}
    \barz_\eta^+=\barz_\eta+a^c_\eta(\barz_\eta)-a^s_\eta(\barz_\eta). \label{eq: imbalance_update_steady_state}
\end{align}
We formally define the drift of a test function $V(z)$. This is similar to the definition of \citet[Definition 1]{hurtado2020transform} and we present it below for completeness.
\begin{definition}[Drift of a function] Let $V: \bbR \rightarrow \bbC$ be a function. We define the drift of $V$ at $z$ for $\eta>0$ as
\revcolor{\begin{align*}
    \Delta V(z, \eta)=V\left(z + a^c_\eta(z)-a^s_\eta(z)\right)-V(z).
\end{align*}}
If $\E{|V(\barz_\eta)|}<\infty$, then we say that we set the drift of $V$ to zero when we use the property
\begin{align*}
    \E{\Delta V(\barz_\eta, \eta)}=\E{V(\barz^+_\eta)-V(\barz_\eta)}=0.
\end{align*}
\end{definition}
Now we are ready to prove the lemma.
\proof{Proof of Lemma \ref{lemma: drift_analysis}}
For $\omega \in \bbR$, we define the test function $W(z)\overset{\Delta}{=}e^{j\epsilon_\eta \omega z}$. We now analyze its drift in the steady state.
\begin{align*}
   \lefteqn{ \E{\Delta W(\barz_\eta, \eta)}} \\ ={}&\E{e^{j\epsilon_\eta \omega \barz_\eta^+}-e^{j\epsilon_\eta \omega \barz_\eta}} \\
    ={}&\E{e^{j\epsilon_\eta \omega \left(\barz_\eta+a_{\eta}^c(\barz_\eta)-a_{\eta}^s(\barz_\eta)\right)}-e^{j\epsilon_\eta \omega \barz_\eta}} \\
    ={}&\E{e^{j\epsilon_\eta \omega \barz_\eta}\left(e^{j\epsilon_\eta \omega\left(a_{\eta}^c(\barz_\eta)-a_{\eta}^s(\barz_\eta)\right)}-1\right)} \\
    \overset{(a)}{=}{}&\E{e^{j\epsilon_\eta \omega \barz_\eta}\left(j\epsilon_\eta \omega \left(a_{\eta}^c(\barz_\eta)-a_{\eta}^s(\barz_\eta)\right)-\frac{1}{2}\epsilon_\eta^2\omega^2\left(a_{\eta}^c(\barz_\eta)-a_{\eta}^s(\barz_\eta)\right)^2\right)}+o(\epsilon_\eta^2) \\
    \overset{(b)}{=}{}&j\epsilon_\eta^2\omega\E{e^{j\epsilon_\eta \omega \barz_\eta}\left(\phi^c\left(\frac{\barz_\eta}{\tau_\eta}\right)-\phi^s\left(\frac{\barz_\eta}{\tau_\eta}\right)\right)}-\frac{1}{2}\epsilon_\eta^2\omega^2 \E{e^{j\epsilon_\eta \omega \barz_\eta}\E{\left(a_{\eta}^c(\barz_\eta)-a_{\eta}^s(\barz_\eta)\right)^2| \barz_\eta}}\\
    &+o(\epsilon_\eta^2) \\
    \overset{(c)}{=}{}&j\epsilon_\eta^2\omega\E{e^{j\epsilon_\eta \omega \barz_\eta}\left(\phi^c\left(\frac{\barz_\eta}{\tau_\eta}\right)-\phi^s\left(\frac{\barz_\eta}{\tau_\eta}\right)\right)}-\frac{1}{2}\epsilon_\eta^2\omega^2 \E{e^{j\epsilon_\eta \omega \barz_\eta}\left(\sigma^c(\lambda_\eta(\barz_\eta))+\sigma^s(\mu_\eta(\barz_\eta))\right)}\\
    &+o(\epsilon_\eta^2) 
\end{align*}
\revcolor{where $(a)$ follows by Taylor's Theorem,
$\E{|a^c(z)|^3}, \E{|a^s(z)|^3} \leq \moment$, $|e^{j\epsilon_\eta \omega \barz_\eta}| \leq 1$, and Lemma \ref{lemma: taylor_series} in Appendix~\ref{app:taylor_series}. Next, $(b)$ holds by simplifying the first term in the RHS as follows:
\begin{align*}
    \E{e^{j \epsilon_\eta \omega \barz_\eta}a_\eta^c(\barz_\eta)} ={}&  \E{e^{j \epsilon_\eta \omega \barz_\eta}\text{clip}\left(\lambda^\star + \epsilon_\eta \phi^c\left(\frac{\barz_\eta}{\tau_\eta}\right), \lambda_{\max}, \lambda_{\min}\right)} \\
    ={}& \E{e^{j \epsilon_\eta \omega \barz_\eta}\left(\lambda^\star + \epsilon_\eta \phi^c\left(\frac{\barz_\eta}{\tau_\eta}\right)\right)} \\
    &+ \E{e^{j \epsilon_\eta \omega \barz_\eta}\left(\text{clip}\left(\lambda^\star + \epsilon_\eta \phi^c\left(\frac{\barz_\eta}{\tau_\eta}\right), \lambda_{\max}, \lambda_{\min}\right)-\left(\lambda^\star + \epsilon_\eta \phi^c\left(\frac{\barz_\eta}{\tau_\eta}\right)\right)\right)} \\
    ={}& \E{e^{j \epsilon_\eta \omega \barz_\eta}\left(\lambda^\star + \epsilon_\eta \phi^c\left(\frac{\barz_\eta}{\tau_\eta}\right)\right)} + o(\epsilon_\eta) 
\end{align*}
where the last equality holds due to the following:
\begin{align*}
    &\bigg|\E{e^{j \epsilon_\eta \omega \barz_\eta}\left(\text{clip}\left(\lambda^\star + \epsilon_\eta \phi^c\left(\frac{\barz_\eta}{\tau_\eta}\right), \lambda_{\max}, \lambda_{\min}\right)-\left(\lambda^\star + \epsilon_\eta \phi^c\left(\frac{\barz_\eta}{\tau_\eta}\right)\right)\right)}\bigg| \\
    \leq{}& \E{\bigg|\text{clip}\left(\lambda^\star + \epsilon_\eta \phi^c\left(\frac{\barz_\eta}{\tau_\eta}\right), \lambda_{\max}, \lambda_{\min}\right)-\left(\lambda^\star + \epsilon_\eta \phi^c\left(\frac{\barz_\eta}{\tau_\eta}\right)\right)\bigg|} \\
    \leq{}& \epsilon_\eta \E{\phi^c\left(\frac{\barz_\eta}{\tau_\eta}\right)\mathbbm{1}\left\{\phi^c\left(\frac{\barz_\eta}{\tau_\eta}\right) > \frac{\lambda_{\max}-\lambda^\star}{\epsilon_\eta}\right\}} - \epsilon_\eta \E{\phi^c\left(\frac{\barz_\eta}{\tau_\eta}\right)\mathbbm{1}\left\{\phi^c\left(\frac{\barz_\eta}{\tau_\eta}\right) < \frac{\lambda_{\min}-\lambda^\star}{\epsilon_\eta}\right\}} \\
    \overset{(d)}{\leq}{}& 2\epsilon_\eta \E{\left(\frac{|\barz_\eta|}{\tau_\eta}\right)^\beta\mathbbm{1}\left\{\left(\frac{|\barz_\eta|}{\tau_\eta}\right)^\beta > \frac{\min\left\{\lambda_{\max}-\lambda^\star, \lambda^\star-\lambda_{\min}\right\}}{\epsilon_\eta}\right\}} \\
    \overset{(e)}{\leq}{}& 2\epsilon_\eta\E{\left(\frac{|\barz_\eta|}{\tau_\eta}\right)^2}^{\beta/2} \P{\left(\frac{|\barz_\eta|}{\tau_\eta}\right)^\beta > \frac{\min\left\{\lambda_{\max}-\lambda^\star, \lambda^\star-\lambda_{\min}\right\}}{\epsilon_\eta}}^{1-\beta/2} \\
    \overset{(f)}{\leq}{}& \frac{2\epsilon_\eta^{2/\beta}}{\min\left\{\lambda_{\max}-\lambda^\star, \lambda^\star-\lambda_{\min}\right\}^{2/\beta-1}}\E{\left(\frac{|\barz_\eta|}{\tau_\eta}\right)^2} \overset{(g)}{=} o(\epsilon_\eta), \numberthis 
\end{align*}
where $(d)$ follows by Condition~\ref{ass:poly_growth}. Next, $(e)$ holds by H\"older's inequality with $p=2/\beta$ and $q=2/(2-\beta)$. Now, $(f)$ follows by Markov's inequality and $(g)$ follows as $\beta<2$ by Condition~\ref{ass:poly_growth} and $\E{\barz_\eta^2} = O(\tau_\eta^2)$ by \eqref{eq: imbalance_second_moment} in Proposition~\ref{prop: pos_rec}.}

\revcolor{Next, $(c)$ follows by the tower property of expectation and using the definition of arrivals. We have
\begin{align*}
    \E{(a_{\eta}^c(\barz_\eta)-a_{\eta}^s(\barz_\eta))^2 | \barz_\eta}&=\Var{a_{\eta}^c(\barz_\eta) | \barz_\eta}+\Var{a_{\eta}^s(\barz_\eta) | \barz_\eta}+\E{a_{\eta}^c(\barz_\eta)-a_{\eta}^s(\barz_\eta)| \barz_\eta}^2 \\
    &=\sigma^c(\lambda_\eta(\barz_\eta))+\sigma^s(\mu_\eta(\barz_\eta))+  \left(\lambda_\eta(\barz_\eta)-\mu_\eta(\barz_\eta)\right)^2, 
\end{align*}
Thus, by the tower property, we have
\begin{align*}
    &\bigg|\E{e^{j\epsilon_\eta \omega \barz_\eta}\left(a_\eta^c(\barz_\eta)-a^s_\eta(\barz_\eta)\right)^2} - \E{e^{j\epsilon_\eta \omega \barz_\eta}\left(\sigma^c(\lambda_{\pfix{\eta}}(\barz_\eta))+\sigma^s(\mu_{\pfix{\eta}}(\barz_\eta))\right)}\bigg| \\
    \leq{}& \E{\left(\lambda_\eta(\barz_\eta)-\mu_\eta(\barz_\eta)\right)^2} \\
    ={}& \E{\left(\lambda_\eta(\barz_\eta)-\mu_\eta(\barz_\eta)\right)^2\mathbbm{1}\left\{\frac{|\barz_\eta|}{\tau_\eta} > \frac{1}{\epsilon_\eta^{1/4}}\right\}} + \E{\left(\lambda_\eta(\barz_\eta)-\mu_\eta(\barz_\eta)\right)^2\mathbbm{1}\left\{\frac{|\barz_\eta|}{\tau_\eta} \leq \frac{1}{\epsilon_\eta^{1/4}}\right\}} \\
    \overset{*}{\leq}{}& \left(\lambda_{\max}+\mu_{\max}\right)^2 \P{\frac{|\barz_\eta|}{\tau_\eta} > \frac{1}{\epsilon_\eta^{1/4}}} + 4\epsilon_\eta \\
    \overset{**}{\leq}{}& \left(\lambda_{\max}+\mu_{\max}\right)^2 \epsilon_\eta^{1/4} \frac{\E{|\barz_\eta|}}{\tau_\eta} + 4\epsilon_\eta \overset{***}{=} o(1),
\end{align*}
where $(*)$ follows by \eqref{eq: general_pricing_policy} and Condition~\ref{ass:poly_growth}. In particular, the first term is upper bounded by noting that $\lambda_\eta(z) \leq \lambda_{\max}$ and $\mu_\eta(z) \leq \mu_{\max}$ for all $z \in \bbZ$. In addition, the second term is upper bounded by Condition~\ref{ass:poly_growth}. The bound $|z|/\tau_\eta \leq \epsilon_\eta^{-1/4}$ implies $|\phi^c(z/\tau_\eta)|, |\phi^s(z/\tau_\eta)| \leq \epsilon_\eta^{-1/2}$, so, $|\lambda_\eta(z)-\lambda^\star| \leq \sqrt{\epsilon_\eta}$ and $|\mu_\eta(z)-\mu^\star| \leq \sqrt{\epsilon_\eta}$ with $\lambda^\star=\mu^\star$. Next, $(**)$ follows by the Markov's inequality and $(***)$ holds as $\E{|\barz_\eta|} = O(\tau_\eta)$ by \eqref{eq: imbalance_expectation} in Proposition~\ref{prop: pos_rec}.}

Now, as $|e^{j\epsilon_\eta \omega z}|=1$ and so $\E{|W(\barz_\eta)|}<\infty$ is trivially true. By setting $\E{\Delta W(\barz_\eta, \eta)}=0$, dividing by $j\epsilon_\eta^2\omega$, and rearranging the terms, we get
\begin{align}
    &\E{e^{j\epsilon_\eta \omega \barz_\eta}\left(\phi^s\left(\frac{\barz_\eta}{\tau_\eta}\right)-\phi^c\left(\frac{\barz_\eta}{\tau_\eta}\right)\right)}=\frac{j\omega}{2}\left(\sigma^c(\lambda^\star)+\sigma^s(\mu^\star)\right) \E{e^{j\epsilon_\eta \omega \barz_\eta}}+o(1)\nonumber\\
    &+\frac{j\omega}{2} \E{e^{j\epsilon_\eta \omega \barz_\eta}\left(\sigma^c(\lambda_\eta(\barz_\eta))+\sigma^s(\mu_\eta(\barz_\eta))-\sigma^c(\lambda^\star)-\sigma^s(\mu^\star)\right)}. \label{eq: prelimit_functional_equation}
\end{align}
Now, taking the limit as $\eta \uparrow \infty$, the last term on the RHS disappears as $\epsilon_\eta \rightarrow 0$, and so $\lambda_\eta(\barz_\eta) \rightarrow \lambda^\star$ and $\mu_\eta(\barz_\eta) \rightarrow \mu^\star$. Also, as $\epsilon_\eta \tau_\eta \rightarrow l$, $\barz_\eta/\tau_\eta$ converges in distribution to $\zinfty/l$, and so, we expect a functional equation of the following form.
\begin{claim} \label{claim: case2_limiting_functional_equation}
By taking the limit as $\eta \uparrow \infty$ in \eqref{eq: prelimit_functional_equation}, we get
\begin{align}
    \E{e^{j\omega \zinfty}g_{1,l}(\zinfty)}=j\omega \E{e^{j\omega \zinfty}}. \nonumber
\end{align}
\end{claim}
However, this is not straightforward as one needs to show limits of functions as opposed to numbers due to the state-dependent control. The proof of the claim is presented in Appendix \ref{app: drift_analysis}. This completes the proof of the Lemma. \hfill $\square$
\endproof
\subsection{Proof of Lemma~\ref{lemma: distribution_function}} \label{sec:schwartz_lemma}
The proof of this lemma closely follows the steps mentioned in Section~\ref{sec:ift}. To make the proof rigorous without assuming the existence of a PDF, we resort to the theory of Schwartz functions, allowing one to generalize Fourier transforms to a broader class of functions.
\subsubsection{Preliminaries}
Before presenting the proof of Lemma \ref{lemma: distribution_function}, we brief on the theory of Schwartz functions. For more details, we refer the reader to the comprehensive texts of \citet[Chapter 7]{rudin1991functional}, and \citet[Chapter 11]{hunter2001applied}.
\begin{definition} The Schwartz space is a vector space given by
\begin{align*}
    \mathcal{S}(\bbR)=\left\{f: \bbR \rightarrow \bbC : f \in \mathcal{C}^\infty_{pol}(\bbR), ||f||_{\tilde{\alpha}, \tilde{\beta}}<\infty \ \forall \tilde{\alpha}, \tilde{\beta} \in \bbZ_+\right\}
\end{align*}
where $||f||_{\tilde{\alpha}, \tilde{\beta}}=\sup_{x \in \bbR} |x^{\tilde{\alpha}} f^{(\tilde{\beta})}(x)|$ and $f^{(\tilde{\beta})}(\cdot)$ is the $\tilde{\beta}\text{-th}$ derivative of $f(\cdot)$.
\end{definition}
Intuitively, Schwartz functions are infinitely differentiable, and each derivative decays faster than any polynomial. 
The Fourier transform of a Schwartz function is given by
\begin{align}
    \calF(\varphi)(x)=\hat{\varphi}(x)=\frac{1}{\sqrt{2\pi}} \int_\bbR \varphi(y)e^{-jxy} dy. \label{eq: fourier_schwartz}
\end{align}
In addition, the inverse Fourier transform also exists for a Schwartz function, and by Fourier inversion theorem, we have
\begin{align*}
    \calF^{-1}(\varphi)(x)=\check{\varphi}(x)=\frac{1}{\sqrt{2\pi}}\int_\bbR \varphi(y) e^{jxy} dy.
\end{align*}
We will need the following results about Schwartz functions which we present without proof.
\begin{proposition} \label{prop: bijection}
The Fourier transform is a bijection on $\calS(\bbR)$, e.g., see \citet[Theorem 7.7]{rudin1991functional}. 
\end{proposition}
Lastly, we also consider the smaller vector space $\calD(\bbR)$ of functions with bounded support defined as follows:
\begin{definition}
    The space of test functions $\calD(\bbR)$ is a vector space given by
    \begin{align*}
        \calD(\bbR) = \left\{f : \bbR \rightarrow \bbC : f \in \calC_{pol}^\infty(\bbR), \{f \neq 0\}   \textit{ is bounded}\right\}.
    \end{align*}
\end{definition}
\subsubsection{Step 3a: Inverse Fourier Transform}
For any function $\varphi$ such that $\E{|\varphi(\zinfty)|} < \infty$, and in particular, for $\varphi \in \calS(\bbR)$, let the expectation with respect to the probability measure defined by $\zinfty$ be
\begin{align*}
    \calT_{\zinfty}[\varphi]=\int_\bbR \varphi(x) dF_{\zinfty}(x)=\E{\varphi(\zinfty)}, \numberthis \label{eq: tempered_distributions}
\end{align*}
where $F_{\zinfty}$ is the CDF of $\zinfty$. Note that $\calT_{\zinfty}$ is a continuous linear functional on the Schwartz space and such operators are famously known as tempered distributions (e.g., see \citet{gates1956linear} and \citet{gates1952differential}). They were introduced by Laurent Schwartz to generalize the notion of functions which makes them amenable to analysis involving Fourier transforms. Now, we extend \eqref{eq: functional_equation_mgf} to hold for all $\varphi \in \calS(\bbR)$ in the form of a differential equation in $\calT_{\zinfty}$.
\begin{claim} \label{eq: tempered_dist_eqn}
    The functional equation \eqref{eq: functional_equation_mgf} implies $\calT_{\zinfty}[\hat{\varphi}^\prime-g_{1,l}\hat{\varphi}]=0$ for all $\varphi \in \calS(\bbR)$.
\end{claim}
The proof of the claim is deferred to Section~\ref{sec: claim_schwartz_before_fourier}. As Fourier transform is a bijection on $\calS(\bbR)$ (Proposition \ref{prop: bijection}), the above is equivalent to
\begin{align*}
    \calT_{\zinfty}[\varphi^\prime-g_{1,l}\varphi]=0 \quad \forall \varphi \in \calS(\bbR). \numberthis \label{eq: diff_eq_schwartz}
\end{align*}
In spirit, the above step is equivalent to performing inverse Fourier transform of the functional equation \eqref{eq: functional_equation_mgf}, and \eqref{eq: diff_eq_schwartz} is equivalent to the resultant differential equation \eqref{eq:sketch_diff_eqn_pdf}. In particular, we can view \eqref{eq: diff_eq_schwartz} as a differential equation of tempered distribution $\calT_{\zinfty}$. This completes the Step 3a.
\subsubsection{Step 3b: Solving the Differential Equation}
In this section, we solve \eqref{eq: diff_eq_schwartz} to show that $\operatorname{Gibbs}(g_{1,l})$ is its unique solution. Recall that $\calD(\bbR)$ is the space of infinitely differentiable functions with bounded support and define
\begin{align*}
    \tilde{\calD}(\bbR) = \left\{\psi \in \calD(\bbR) : \int_{-\infty}^{\infty} e^{-G(t)} \psi(t) dt = 0\right\},
\end{align*}
where $G(x) = \int_0^x g_{1,l}(t) dt$. Now, for any $\tilde{\psi} \in \tilde{D}(\bbR)$, define $\psi \in \calD(\bbR)$ as
\begin{align*}
    \psi(x) = e^{G(x)} \int_{-\infty}^x e^{-G(t)} \tilde{\psi}(t) dt.
\end{align*}
As $\tilde{\psi} \in \tilde{\calD}(\bbR)$ and $g_{1,l} \in \calC_{pol}^\infty(\bbR)$, it is easy to verify that $\psi \in \calD(\bbR) \subseteq \calS(\bbR)$. Now, by \eqref{eq: diff_eq_schwartz} and noting that $\psi^{\prime} - g_{1, l} \psi = \tilde{\psi}$, we have $\calT_{\zinfty}[\tilde{\psi}] = 0$. As $\tilde{\psi} \in \tilde{\calD}(\bbR)$ is arbitrary, we have
\begin{align*}
    \calT_{\zinfty}[\tilde{\psi}] = 0 \quad \forall \tilde{\psi} \in \tilde{\calD}(\bbR). \numberthis \label{eq: subset_d_space}
\end{align*}
Now, for an arbitrary $\psi \in \calD(\bbR) \backslash \tilde{\calD}(\bbR)$, define $a = \int_{-\infty}^{\infty} e^{-G(t)} \psi(t) dt$. Note that $a \neq 0$ as $\psi \notin \tilde{\calD}(\bbR)$ and $|a| < \infty$ as $\psi \in \calD(\bbR)$. Next, fix a $\psi_0 \in \calD(\bbR) \backslash \tilde{\calD}(\bbR)$ such that $\int_{-\infty}^{\infty} e^{-G(t)} \psi_0(t) dt = 1$ and define 
\begin{align*}
    \tilde{\psi} = \psi - a \psi_0.
\end{align*}
Observe that $\tilde{\psi} \in \tilde{\calD}(\bbR)$. Now, using the above equation and linearity of $\calT_{\zinfty}$, we have
\begin{align*}
    \calT_{\zinfty}[\psi] = \calT_{\zinfty}[\tilde{\psi}] + a\calT_{\zinfty}[\psi_0] \overset{\eqref{eq: subset_d_space}}{=} a \calT_{\zinfty}[\psi_0] = \calT_{\zinfty}[\psi_0] \int_{-\infty}^{\infty} e^{-G(t)} \psi(t) dt.
\end{align*}
As $\psi \in \calD(\bbR) \backslash \tilde{\calD}(\bbR)$ is arbitrary, by the above equation and \eqref{eq: subset_d_space}, we have
\begin{align*}
    \calT_{\zinfty}[\psi] = \calT_{\zinfty}[\psi_0] \int_{-\infty}^{\infty} e^{-G(t)} \psi(t) dt \quad \forall \psi \in \calD(\bbR). \numberthis \label{eq: all_of_d_space_without_const}
\end{align*}
Note that $\calT_{\zinfty}[\psi_0]$ is a constant independent of $\psi$. Next, we use the fact that $\zinfty$ is a random variable to characterize $\calT_{\zinfty}[\psi_0]$. For any $k \geq 1$, let $\chi_k \in \calD(\bbR)$ be such that $\chi_k(x) = 1$ for all $|x| \leq k$, and $|\chi_k(x)| \in [0, 1]$ for all $x \in \bbR$. Then, we have
\begin{align*}
    &\bigg| 1 - \calT_{\zinfty}[\psi_0] \int_{-\infty}^{\infty} e^{-G(t)} dt \bigg| \\
    \overset{(a)}{=}{}& \bigg| \calT_{\zinfty}[1] - \calT_{\zinfty}[\chi_k] + \calT_{\zinfty}[\chi_k] - \calT_{\zinfty}[\psi_0] \int_{-\infty}^{\infty} e^{-G(t)} dt \bigg| \\
    \overset{(b)}{=}{}& \bigg| \calT_{\zinfty}[1 - \chi_k] + \calT_{\zinfty}[\psi_0] \int_{-\infty}^{\infty} e^{-G(t)} \left(1-\chi_k\right) dt \bigg| \\
    \overset{(c)}{\leq}{}& 2\P{|\zinfty| \geq k} + 2\big|\calT_{\zinfty}[\psi_0]\big|\int_{-\infty}^{-k} e^{-G(t)} dt + 2\big|\calT_{\zinfty}[\psi_0]\big|\int_{k}^{\infty} e^{-G(t)} dt \\
    \overset{(d)}{\leq}{}& 2\P{|\zinfty| \geq k} + 4\big|\calT_{\zinfty}[\psi_0]\big| e^{\int_{-K\pfix{l}}^{K\pfix{l}} |g_{1,l}(x)| dx} \int_{k}^\infty e^{-\frac{2\delta[t-K\pfix{l}]^{+}}{\sigma^c(\lambda^\star)+\sigma^s(\mu^\star)}} dt \\
    \leq{}& 2\P{|\zinfty| \geq k} + \frac{2}{\delta}\big|\calT_{\zinfty}[\psi_0]\big| e^{\int_{-K\pfix{l}}^{K\pfix{l}} |g_{1,l}(x)| dx}\left(\sigma^c(\lambda^\star) + \sigma^s(\mu^\star)\right) e^{\pfix{\frac{2\delta Kl}{\sigma^c(\lambda^\star)+\sigma^s(\mu^\star)}}} e^{-\frac{2\delta k}{\sigma^c(\lambda^\star)+\sigma^s(\mu^\star)}} \numberthis \label{eq: characterizing_normalizing_constant},
\end{align*}
where $(a)$ follows as $\calT_{\zinfty}[1]=1$ by the definition given by \eqref{eq: tempered_distributions}. Next, as $\chi_k \in \calD(\bbR)$ by construction, $(b)$ follows by \eqref{eq: all_of_d_space_without_const}. Now, $(c)$ follows by triangle inequality to take the expectation of individual terms followed by Jensen's inequality to get $|\calT_{\zinfty}[\varphi]| \leq \calT_{\zinfty}[|\varphi|]$ and noting that $1-\chi_k(x) = 0$ for all $|x| \leq k$ and $|1-\chi_k(x)| \leq 1+|\chi_k(x)| \leq 2$ for all $x \in \bbR$. Next, $(d)$ follows by bounding $e^{-G}$ using Condition \ref{ass: neg_drift}. In particular, we have
\begin{align*}
    -G(t) = -\int_0^t g_{1,l}(x) dx \leq \int_{-K\pfix{l}}^{K\pfix{l}} |g_{1,l}(x)| dx -\frac{2\delta [|t|-K\pfix{l}]^+}{\sigma^c(\lambda^\star)+\sigma^s(\mu^\star)}.
\end{align*}
Now, taking the limit $k \rightarrow \infty$ on both sides of \eqref{eq: characterizing_normalizing_constant} and noting that the LHS is independent of $k$, we have $\calT_{\zinfty}[\psi_0] = \frac{1}{\int_{-\infty}^\infty e^{-G(t)} dt}$. Thus, by \eqref{eq: all_of_d_space_without_const}, we have
\begin{align*}
    \calT_{\zinfty}[\psi] = \frac{1}{\int_{-\infty}^\infty e^{-G(t)} dt} \int_{-\infty}^{\infty} e^{-G(t)} \psi(t) dt \quad \forall \psi \in \calD(\bbR). \numberthis \label{eq: characterizing_d_space_with_const}
\end{align*}
Now, we extend the above characterization for all $\varphi \in \calS(\bbR)$. We have
\begin{align*}
&\bigg|\calT_{\zinfty}[\varphi] - \frac{1}{\int_{-\infty}^\infty e^{-G(t)} dt} \int_{-\infty}^{\infty} e^{-G(t)} \varphi(t) dt\bigg| \\
={}& \bigg|\calT_{\zinfty}[\varphi(1-\chi_k)] + \calT_{\zinfty}[\varphi\chi_k] - \frac{1}{\int_{-\infty}^\infty e^{-G(t)} dt} \int_{-\infty}^{\infty} e^{-G(t)} \varphi(t) dt\bigg| \\
\overset{(a)}{\leq}{}& \max_{x \in \bbR} |\varphi(x)(1-\chi_k(x))| + \frac{1}{\int_{-\infty}^\infty e^{-G(t)} dt} \int_{-\infty}^{\infty} e^{-G(t)} \varphi(t)\left(\chi_k - 1\right) dt \\
\overset{(b)}{\leq}{}& 2\max_{x \in \bbR} |\varphi(x)(1-\chi_k(x))| 
\overset{(c)}{\leq} 4\max_{|x| \geq k} |\varphi(x)|, \numberthis \label{eq: extending_to_schwartz_space}
\end{align*}
where $(a)$ follows as $\calT_{\zinfty}[\varphi] \leq \max_{x \in \bbR} | \varphi(x)|$ for all $\varphi \in \calS(\bbR)$ by the definition of $\calT_{\zinfty}$ given by \eqref{eq: tempered_distributions}. In addition, we also use \eqref{eq: characterizing_d_space_with_const} to characterize $\calT_{\zinfty} [\varphi \chi_k]$ by noting that $\varphi \chi_k \in \calD(\bbR)$ as $\varphi \in \calC_{pol}^\infty$ and $\chi_k \in \calD(\bbR)$. Next, $(b)$ follows by noting that the second term is the expectation of $\varphi(\chi_k-1)$ with respect to the Gibbs distribution. Lastly, $(c)$ follows by noting that $1-\chi_k(x) = 0$ for all $|x| \leq k$ and $|1-\chi_{\pfix{k}}(x)| \leq 1+|\chi_k(x)| \leq 2$ for all $x \in \bbR$. By taking the limit as $k \rightarrow \infty$ in \eqref{eq: extending_to_schwartz_space} and noting that $\varphi \in \calS(\bbR)$, we have
\begin{align*}
\calT_{\zinfty}[\varphi] = \frac{1}{\int_{-\infty}^\infty e^{-G(t)} dt} \int_{-\infty}^{\infty} e^{-G(t)} \varphi(t) dt \quad \forall \varphi \in \calS(\bbR). \numberthis \label{eq: all_schwartz_func}
\end{align*}
Now, let $\mu_1$ and $\mu_2$ be the probability measures corresponding to $\zinfty$ and Gibbs distribution respectively. Also, let $\hat{\mu}_1$ and $\hat{\mu}_2$ be their characteristic functions. Now, we show that $\hat{\mu}_1 = \hat{\mu}_2$. For any $\varphi \in \calS(\bbR)$ and $k \in \{1, 2\}$, we have
\begin{align*}
    \calT_{\mu_k}[\hat{\varphi}] &= \int_{-\infty}^\infty \hat{\varphi} d\mu_k = \frac{1}{\sqrt{2\pi}}\int_{-\infty}^\infty \int_{-\infty}^\infty \varphi(t) e^{-j t x} dt d\mu_k(x) \\
    &\overset{(a)}{=} \frac{1}{\sqrt{2\pi}}\int_{-\infty}^\infty \varphi(t) \int_{-\infty}^\infty  e^{-jt x}  d\mu_k(x) dt \\
    &= \frac{1}{\sqrt{2\pi}} \int_{-\infty}^\infty \varphi(t) \hat{\mu}_k(t) dt.
\end{align*}
where $(a)$ follows by Fubini's theorem. Thus, we have
\begin{align*}
    0 \overset{\eqref{eq: all_schwartz_func}}{=} \calT_{\mu_1}[\hat{\varphi}] - \calT_{\mu_2}[\hat{\varphi}] = \frac{1}{\sqrt{2\pi}} \int_{-\infty}^\infty \varphi(t) \left(\hat{\mu}_1(t) - \hat{\mu}_2(t)\right) dt \quad \forall \varphi \in \calS(\bbR).
\end{align*}
Now, for any $y \in \bbR$, define $\varphi_{\epsilon_0, y}(x) = \frac{1}{\sqrt{\pi} \epsilon_0} e^{-(x-y)^2/\epsilon_0^2}$ and note that
\begin{align*}
    0 &= \lim_{\epsilon_0 \rightarrow 0} \left(\calT_{\mu_1}[\hat{\varphi}_{\epsilon_0, y}] - \calT_{\mu_2}[\hat{\varphi}_{\epsilon_0, y}]\right) = \lim_{\epsilon_0 \rightarrow 0} \frac{1}{\sqrt{2\pi}} \int_{-\infty}^\infty \varphi_{\epsilon_0, y}(t) \left(\hat{\mu}_1(t) - \hat{\mu}_2(t)\right) dt \\
    &= \frac{1}{\sqrt{2\pi}}\left(\hat{\mu}_1(y)-\hat{\mu}_2(y)\right) 
\end{align*}
where the last equality is followed by the dominated convergence theorem. In particular, let $\hat{\mu} = \hat{\mu}_1 - \hat{\mu}_2$ and note that
\begin{align*}
\bigg|\lim_{\epsilon_0 \rightarrow 0}  \int_{-\infty}^\infty \varphi_{\epsilon_0, y}(t) \hat{\mu}(t) dt - \hat{\mu}(y)\bigg| &= \bigg|\lim_{\epsilon_0 \rightarrow 0}  \int_{-\infty}^\infty \epsilon_0 \varphi_{\epsilon_0, y}(\epsilon_0 t+y) \hat{\mu}(\epsilon_0 t+y) dt - \hat{\mu}(y) \int_{-\infty} ^\infty \frac{e^{-t^2}}{\sqrt{\pi}} dt\bigg| \\
&\leq \lim_{\epsilon_0 \rightarrow 0} \frac{1}{\sqrt{\pi}}\int_{-\infty}^\infty  |\hat{\mu}(\epsilon_0 t+y)-\hat{\mu}(y)| e^{-t^2} dt \\
&= \frac{1}{\sqrt{\pi}}\int_{-\infty}^\infty  \lim_{\epsilon_0 \rightarrow 0} |\hat{\mu}(\epsilon_0 t+y)-\hat{\mu}(y)| e^{-t^2} dt = 0,
\end{align*}
where the last equality follows by the continuity of the characteristic functions.
%
Thus, we have $\hat{\mu}_1(y) = \hat{\mu}_2(y)$ for all $y \in \bbR$ which completes the proof.
\endproof
\section{Preliminary Lemmas for the Proof of Theorems \ref{theo: case_1} and \ref{theo: profit_driven_regime_3}}
The proof of Theorems \ref{theo: case_1} and \ref{theo: profit_driven_regime_3} are based on using complex exponential test functions. To bound certain terms in the process, we need the following lemmas which are obtained by setting the drift of several auxiliary test functions to zero. The proof of all the lemmas is deferred to Appendix \ref{app: preliminary_lemmas}.
\subsection{Auxiliary Bounds on the Imbalance}
Under Condition \ref{ass: neg_drift}, we have the following.
\begin{lemma} \label{lemma: z_sgnz} Let $g_\eta(\cdot)$ be a function such that there exists $G > 0$ with $|g_\eta(x)| \leq G$ for all $x \in \bbR$. Then, for any $\eta >0$, we have
\revcolor{\begin{align}
   \E{\barz_\eta^+(g_\eta(\barz_\eta^+)-g_\eta(\barz_\eta))}=-\E{(\lambda_\eta(\barz_\eta)-\mu_\eta(\barz_\eta))g_\eta(\barz_\eta)} \label{eq: drift_z_sgnz}
\end{align}}
\end{lemma}
The proof of the above lemma follows by setting the drift of the test function $z g(z)$ to zero in the steady state. Now, we prove a high probability bound on imbalance not being in any given state \revcolor{$|x| \leq K \max\left\{\tau_\eta, 1/\epsilon_\eta\right\}$} below.
 \begin{lemma} \label{lemma: prob_bound_on_any_state}
Under Condition~\ref{ass: irreducibility}, for any $\eta > 0$ and \revcolor{$|x| \leq K \max\left\{\tau_\eta, 1/\epsilon_\eta\right\}$}, we have
\revcolor{\begin{align*}
    \P{\barz_\eta=x} \leq \frac{1}{p_{\min}}\E{\big|\lambda_\eta(\barz_\eta)-\mu_\eta(\barz_\eta)\big|} 
\end{align*}}
\end{lemma}
The proof of the above lemma makes use of two test functions. In particular, we set the drift of the test functions $z\mathbbm{1}\{\textrm{sgn}(x)z > \textrm{sgn}(x)x\}$ and $\mathbbm{1}\{\textrm{sgn}(x)z > \textrm{sgn}(x)x\}$ to zero in steady state, where $\textrm{sgn}(x) = 1$ if $x \geq 0$ and $-1$ otherwise.
\subsection{Convergence of State Dependent Random Variables}
Due to state-dependent control, limits of certain quantities in the proof are non-trivial. So, we need the following two lemmas.
\begin{lemma} \label{claim: case1_t3_t5} Assume the same setup as in Theorem \ref{theo: case_1}. \revcolor{Let $g_{\max} > 0$ be a constant. Then, for any Borel measurable function $g: \bbR \rightarrow [-g_{\max}, g_{\max}]$} such that $\lim_{x \rightarrow \infty} g(x)=\lim_{x \rightarrow -\infty} g(x)=c_{\infty}$, we have
\begin{align*}
    \lim_{\eta \uparrow \infty} \E{g\left(\frac{\barz_\eta}{\tau_\eta}\right)}=c_{\infty}.
\end{align*}
\end{lemma}
\begin{lemma} \label{lemma: convergence_technical}
\revcolor{For some $c, g_{\max}>0$, let $g_n: \bbR \rightarrow [-g_{\max}, g_{\max}]$ be a sequence of Borel measurable functions} such that 
\begin{align*}
    \lim_{n \uparrow \infty} g_n(x)=c \quad \forall x \in \bbR.
\end{align*}
Moreover, the convergence is uniform over $x \in \bbR$. Then for any sequence of random variables \revcolor{$\{Y_n \in \bbR\}$}, we have
\begin{align*}
    \lim_{n \uparrow \infty }\E{g_n(Y_n)}= c.
\end{align*}
\end{lemma}
\section{Proof of Theorem \ref{theo: case_1}}
We prove Theorem \ref{theo: case_1} by again building upon the transform method of \citet{hurtado2020transform} and set the drift of complex exponential test function to zero. Similar to Theorem \ref{theo: case_2_itm} in Section \ref{sec: inverse_fourier_transform}, we obtain an implicit equation due to the state-dependent arrivals with $g_{1,l}(\cdot)=c\chi(\cdot)$ for some $c>0$. As $g_{1, l}(\cdot)$ does not satisfy Condition \ref{ass: c_pol}, the inverse Fourier transform method is not directly applicable here. So, we adopt a different technique to resolve this regime which we now outline below.


Positive recurrence under Condition \ref{ass: neg_drift} is already shown in Proposition \ref{prop: pos_rec} \revcolor{for $\eta > \eta_p$.} 
The key idea in proving the convergence result is the following. 
It is known from Levy's continuity theorem (e.g. see: \citet[Chapter 18]{levy_book}) that convergence in distribution is equivalent to convergence of characteristic functions. So, we will focus on finding the characteristic function of the limiting imbalance. \revcolor{Note that
\begin{align}
    \text{Conditions~\ref{ass:poly_growth} and \ref{cond:limits}} \implies \exists \phi_{\max} > 0: \ |\phi^c(x)|, |\phi^s(x)| \leq \phi_{\max}. \label{eq:bounded_control_curves}
\end{align}
Thus, as $\epsilon_\eta \rightarrow 0$, there exists $\eta_1>0$ such that \eqref{eq: general_pricing_policy} can be simplified to the following with $\lambda^\star = \mu^\star$:
\begin{align*}
    \lambda_\eta(z) = \lambda^\star + \epsilon_\eta \phi^c\left(\frac{z}{\tau_\eta}\right), \quad \mu_\eta(z) = \mu^\star + \epsilon_\eta \phi^s\left(\frac{z}{\tau_\eta}\right) \quad \forall z \in \bbZ, \ \eta \geq \eta_{\pfix{1}}. \numberthis \label{eq:simplified_general_pricing_policy}
\end{align*}
For the rest of the proof, we assume $\eta \geq \max\{\eta_1, \eta_p\}$. Now, we proceed with the proof of Theorem~\ref{theo: case_1}.} The proof consists of two key steps. 
\subsection{Overview of the Proof}
\subsubsection{Step 1: Absolute Scaled Imbalance} First, we show that $\epsilon_\eta\barz_\eta \chi(\barz_\eta)$ converges to an exponential distribution. Recall that $\chi$ is a step function defined in \eqref{eq:step_function}.
\begin{lemma} \label{lemma: mod_z}
Under the same conditions as in Theorem \ref{theo: case_1}, we have
\begin{align*}
    \lim_{\eta \uparrow \infty}\E{e^{j\epsilon_\eta \omega \barz_\eta \step(\barz_\eta)}}=\frac{1}{1-j\omega\frac{\sigma^c(\lambda^\star)+\sigma^s(\mu^\star)}{2}}.
\end{align*}
\end{lemma}
To gain intuition, consider the special case of $\chi(1)=-\chi(-1)=1$, then, $z \chi(z) = |z|$. We exploit the symmetry in the limiting distribution by setting the drift of the test function, $e^{j\epsilon_\eta \omega |\barz_\eta|}$ to zero in steady state to obtain $\E{e^{j\epsilon_\eta \omega |\barz_\eta|}}$. Then, taking the limit as $\eta \uparrow \infty$, we show that 
$\lim_{\eta \uparrow \infty} \E{e^{j\epsilon_\eta \omega |\barz_\eta|}}$  converges to the characteristic function of an exponential distribution. More generally, we consider the test function $e^{j \epsilon_\eta \omega \barz_\eta \chi(\barz_\eta)}/\chi(\barz_\eta)^2$. While this step broadly follows the characteristic function method introduced by \citet{hurtado2020transform} (but applied to $z \chi(z)$ instead of $z$),
we overcome several technical challenges that arise due to the state-dependent control in a matching queue. In essence, one needs to work with a sequence of functions as opposed to numbers. The proof details are presented in Section~\ref{sec: lemma_mod_z_proof}.
\subsubsection{Step 2: Symmetry} 
The previous lemma shows that the limiting distribution of $\epsilon_\eta \barz_\eta \step(\barz_\eta)$ is exponential. Now, we establish symmetry in the limiting distribution of scaled imbalance, which will complete Step 2 of proof of the Theorem \ref{theo: case_1}.
\begin{lemma} \label{lemma: symmetry}
Under the same conditions as in Theorem \ref{theo: case_1}, we have
\begin{align*}
    \lim_{\eta \uparrow \infty} \E{\step(\barz_\eta)e^{j\epsilon_\eta \omega \barz_\eta \step(\barz_\eta)}} = 0.
\end{align*}
\end{lemma}
For the special case of $\chi(1)=-\chi(-1)=1$, the above establishes that $\epsilon_\eta |\barz_\eta|$ has a symmetric distribution around the origin. To prove the lemma, we set the drift of the test function, $\text{sgn}(\barz_\eta)e^{j\epsilon_\eta \omega |\barz_\eta|}$ to zero and show that the expectation,  $\E{ \text{sgn}(\barz_\eta)e^{j\epsilon_\eta \omega |\barz_\eta|} }$ converges to zero as $\eta \uparrow \infty$. 
This establishes symmetry of the limiting  $\barz_\eta$. More generally, we consider the test function $e^{j\epsilon_\eta \omega \barz_\eta \step(\barz_\eta)}/\step(\barz_\eta)$. The technical details of the proof are deferred to Appendix \ref{app: lemma_symmetry}.
\subsubsection{Step 3: Proof of Theorem \ref{theo: case_1}}
Putting together the results from Step 1 and Step 2, we show that the limiting distribution of $\epsilon_\eta\barz_\eta$ is the Laplace distribution. To gain intuition, first consider the special case of $\chi(1)=-\chi(-1)=1$. We start with the following functional identity. 
\begin{align*}
    e^{j\omega|x|}+e^{-j\omega|x|}=e^{j\omega x}+e^{-j\omega x}= 2e^{j\omega x}+\chi(x)e^{-j\omega |x|}-\chi(x)e^{j\omega |x|}.
\end{align*}
Now taking expectation of these under the distribution $\epsilon_\eta\barz_\eta$ and taking the limit as $\eta \to \infty$, we get, 
\begin{align*}
   \lim_{\eta \uparrow \infty }\E{e^{j\omega|\epsilon_\eta \barz_\eta|}+e^{-j\omega|\epsilon_\eta \barz_\eta|}}= \lim_{\eta \uparrow \infty} \E{2e^{j\omega \epsilon_\eta \barz_\eta}+\chi(\barz_\eta)e^{-j\omega \epsilon_\eta |\barz_\eta|}-\chi(\barz_\eta)e^{j\omega \epsilon_\eta |\barz_\eta|}}.
\end{align*}
Both the terms on LHS are characterized from Step 1. We know that the last two terms on the RHS are zero from Step 2. Thus, we have an explicit form for $   \lim_{\eta \uparrow \infty }2\E{e^{j\omega \epsilon_\eta \barz_\eta}}$ completing the proof. This argument is suitably generalized in the following proof for a general step function $\chi(x)$.
\proof{Proof of Theorem \ref{theo: case_1}}
For all $x \in \bbR$, we have
\begin{align*}
  e^{j\omega \epsilon_\eta x} = \frac{\step(-1)e^{j \omega \epsilon_\eta x \frac{\step(x)}{\step(1)}}-\step(x)e^{j \omega \epsilon_\eta x  \frac{\step(x)}{\step(1)}}}{\step(-1)-\step(1)} + \frac{\step(1)e^{j \omega \epsilon_\eta x  \frac{\step(x)}{\step(-1)}}-\step(x)e^{j \omega \epsilon_\eta x  \frac{\step(x)}{\step(-1)}}}{\step(1)-\step(-1)}.
\end{align*}
The above equality can be verified by considering cases based on the sign of $x$. For example, when $x \geq 0$, we have $\step(x) = \step(1)$ and so, the first term is equal to $e^{j \omega \epsilon_\eta x}$ and the second term is zero. Now, by substituting $x= \barz_\eta$ and taking expectation on both sides, we get
\begin{align*}
  \E{e^{j\omega \epsilon_\eta \barz_\eta}} ={}& \frac{\step(-1)\E{e^{j \omega \epsilon_\eta \barz_\eta \frac{\step(\barz_\eta)}{\step(1)}}}-\E{\step(\barz_\eta)e^{j \omega \epsilon_\eta \barz_\eta  \frac{\step(\barz_\eta)}{\step(1)}}}}{\step(-1)-\step(1)} \\
  &+ \frac{\step(1)\E{e^{j \omega \epsilon_\eta \barz_\eta  \frac{\step(\barz_\eta)}{\step(-1)}}}-\E{\step(\barz_\eta)e^{j \omega \epsilon_\eta \barz_\eta  \frac{\step(\barz_\eta)}{\step(-1)}}}}{\step(1)-\step(-1)}.  
\end{align*}
 Now, by taking the limit as $\eta \uparrow \infty$ on both sides and using Lemma \ref{lemma: mod_z} and Lemma \ref{lemma: symmetry} for the RHS, we get
\begin{align*}
\lim_{\eta \uparrow \infty} \E{e^{j\epsilon_\eta \omega \barz_\eta}} &= \frac{1}{\step(1)-\step(-1)}\left(\frac{-\step(-1)}{1-\frac{j\omega}{2\step(1)}\left(\sigma^c(\lambda^\star)+\sigma^s(\mu^\star)\right)}+\frac{\step(1)}{1-\frac{j\omega}{2\step(-1)}\left(\sigma^c(\lambda^\star)+\sigma^s(\mu^\star)\right)}\right) \\
   &=\frac{1}{\left(1-\frac{j\omega}{2\step(1)}\left(\sigma^c(\lambda^\star)+\sigma^s(\mu^\star)\right)\right)\left(1-\frac{j\omega}{2\step(-1)}\left(\sigma^c(\lambda^\star)+\sigma^s(\mu^\star)\right)\right)}.
\end{align*}
Note that $\chi(1) > 0$ and $\chi(-1) < 0$ by Condition~\ref{ass: neg_drift} and so the above is the characteristic function of an asymmetric Laplace distribution with the required parameters. Thus, by Levy's continuity theorem (e.g. see \citet[Chapter 18]{levy_book}), $\epsilon_\eta\barz_\eta$ converges to an asymmetric Laplace random variable. This completes the proof. \hfill $\square$
\endproof
\subsection{Proof of Lemma \ref{lemma: mod_z}} \label{sec: lemma_mod_z_proof}
\proof{Proof}
For $\omega \in \bbR$, we define the test function
\begin{align*}
    V(z)\overset{\Delta}{=} \frac{1}{\step(z)^2}e^{j\epsilon_\eta \omega z \step(z)},
\end{align*}
where $\step(\cdot)$ is a step function as defined in \eqref{eq:step_function}. Now, consider the one-step drift of $V(z)$.
\begin{align*}
    \Delta V(\barz_\eta, \eta)&=V(\barz^+_\eta)-V(\barz_\eta) \\
    &=\frac{1}{\step(\barz_\eta^+)^2}e^{j\epsilon_\eta \omega \barz_\eta^+ \step(\barz_\eta^+)}-\frac{1}{\step(\barz_\eta)^2}e^{j\epsilon_\eta \omega \barz_\eta \step(\barz_\eta)}  \\
    &=\underbrace{\frac{1}{\step(\barz_\eta^+)^2}e^{j\epsilon_\eta \omega \barz_\eta^+ \step(\barz_\eta^+)}-\frac{1}{\step(\barz_\eta)^2}e^{j\epsilon_\eta \omega \barz_\eta^+ \step(\barz_\eta)}}_{\calT_1}+\underbrace{\frac{1}{\step(\barz_\eta)^2}e^{j\epsilon_\eta \omega \barz_\eta^+ \step(\barz_\eta)}-\frac{1}{\step(\barz_\eta)^2}e^{j\epsilon_\eta \omega \barz_\eta \step(\barz_\eta)}}_{\calT_2}. \numberthis \label{eq: def_of_t2}
\end{align*}
We will analyze each term separately. To simplify $\calT_1$, directly use Taylor's Theorem and consider the expansion up to the second order term.
\revcolor{\begin{align*}
    \E{\calT_1}&=\E{\frac{1}{\step(\barz_\eta^+)^2}e^{j\epsilon_\eta \omega \barz_\eta^+\step(\barz_\eta^+)}-\frac{1}{\step(\barz_\eta)^2}e^{j\epsilon_\eta \omega \barz_\eta^+ \step(\barz_\eta)}} \\
    &=\E{\left(\frac{1}{\step(\barz_\eta^+)^2}e^{j\epsilon_\eta \omega \barz_\eta^+\step(\barz_\eta^+)}-\frac{1}{\step(\barz_\eta)^2}e^{j\epsilon_\eta \omega \barz_\eta^+ \step(\barz_\eta)}\right)\mathbbm{1}\{\step(\barz_\eta^+) \neq \step(\barz_\eta)\}} \\
     &=\E{\left(\frac{1}{\step(\barz_\eta^+)^2}e^{j\epsilon_\eta \omega \barz_\eta^+\step(\barz_\eta^+)\mathbbm{1}\{\step(\barz_\eta^+) \neq \step(\barz_\eta)\}}-\frac{1}{\step(\barz_\eta)^2}e^{j\epsilon_\eta \omega \barz_\eta^+ \step(\barz_\eta)\mathbbm{1}\{\step(\barz_\eta^+) \neq \step(\barz_\eta)\}}\right)} \\
    &\overset{(*)}{=} \E{\frac{1}{\chi(\barz_\eta^+)^2}-\frac{1}{\chi(\barz_\eta)^2}+j\epsilon_\eta \omega \barz_\eta^+\left(\frac{1}{\step(\barz_\eta^+)}-\frac{1}{\step(\barz_\eta)}\right)\mathbbm{1}\{\step(\barz_\eta^+) \neq \step(\barz_\eta)\}}  +o(\epsilon_\eta^2) \\
    &=\E{\frac{1}{\chi(\barz_\eta^+)^2}-\frac{1}{\chi(\barz_\eta)^2}+j\epsilon_\eta \omega \barz_\eta^+\left(\frac{1}{\step(\barz_\eta^+)}-\frac{1}{\step(\barz_\eta)}\right)}   +o(\epsilon_\eta^2), \numberthis \label{eq:regime1_t1}
\end{align*}
where, $(*)$ follows by Taylor's theorem. Note that the second-order term in the expansion is exactly equal to zero. To make this argument precise, we verify that the third moment of $|\barz_\eta^+\step(\barz_\eta^+)\mathbbm{1}\{\step(\barz_\eta^+) \neq \step(\barz_\eta)\}|$ is bounded and use Lemma \ref{lemma: taylor_series}. Note that if $\mathbbm{1}\{\step(\barz_\eta^+) \neq \step(\barz_\eta)\} = 1$, then, either $\barz_\eta^+ \geq 0$ and $\barz_\eta < 0$ or $\barz_\eta^+ < 0$ and $\barz_\eta \geq 0$. In either case, we obtain $|\barz_\eta^+| \leq |a^c_\eta(\barz_\eta)-a^s_\eta(\barz_\eta)|$. Thus, we get
\begin{align*}
    &\E{|\barz_\eta^+\step(\barz_\eta^+)\mathbbm{1}\{\step(\barz_\eta^+) \neq \step(\barz_\eta)\}|^3} \\
    \leq{}& \E{|\barz_\eta^+\mathbbm{1}\{\step(\barz_\eta^+) \neq \step(\barz_\eta)\}|^3} \max\left\{|\phi^s(\infty)-\phi^c(\infty)|, |\phi^s(-\infty)-\phi^c(-\infty)|\right\} \\
    \leq{}& \E{|(a^c_\eta(\barz_\eta)-a^s_\eta(\barz_\eta))\mathbbm{1}\{\step(\barz_\eta^+) \neq \step(\barz_\eta)\}|^3} \max\left\{|\phi^s(\infty)-\phi^c(\infty)|, |\phi^s(-\infty)-\phi^c(-\infty)|\right\} \\
    \leq{}& 6\moment \max\left\{|\phi^s(\infty)-\phi^c(\infty)|, |\phi^s(-\infty)-\phi^c(-\infty)|\right\}, \numberthis \label{eq:third_moment_sign_change}
\end{align*}
where the last inequality holds by using $(a+b)^3 \leq 3|a|^3+3|b|^3$ and noting that $\E{|a^c_\eta(\barz_\eta)|^3} \leq \moment$ and $\E{|a^s_\eta(\barz_\eta)|^3} \leq \moment$. Similarly, we can also bound $\E{|\barz_\eta^+\step(\barz_\eta)\mathbbm{1}\{\step(\barz_\eta^+) \neq \step(\barz_\eta)\}|^3}$. Now, \eqref{eq:regime1_t1} can be simplified by Lemma \ref{lemma: z_sgnz} $(\lambda^\star_\eta = \mu^\star)$ and \eqref{eq:simplified_general_pricing_policy} to get}
\begin{align*}
    \E{\calT_1}=-j\epsilon^2_\eta \omega\E{\left(\phi^c\left(\frac{\barz_\eta}{\tau_\eta}\right)-\phi^s\left(\frac{\barz_\eta}{\tau_\eta}\right)\right)\frac{1}{\step(\barz_\eta)}}+o(\epsilon_\eta^2).
\end{align*}
Now, to simplify $\calT_2$, we first use the update equation of imbalance given by \eqref{eq: imbalance_update_steady_state}. Then, we use Taylor's Theorem to expand and consider up to the second-order term. After using some basic properties about expectations, we get the following:
\begin{claim} \label{claim: case1_t2}
\begin{align*}
    \E{\calT_2 | \barz_\eta}={}&j\epsilon_\eta^2 \omega e^{j\epsilon_\eta \omega \barz_\eta \step(\barz_\eta)} \left(\phi^c\left(\frac{\barz_\eta}{\tau_\eta}\right)-\phi^s\left(\frac{\barz_\eta}{\tau_\eta}\right)\right)\frac{1}{\step(\barz_\eta)}\\
    &-\frac{1}{2}\epsilon_\eta^2\omega^2 e^{j\epsilon_\eta \omega \barz_\eta \step(\barz_\eta)}\left(\sigma^c(\lambda_\eta(\barz_\eta))+\sigma^s(\mu_\eta(\barz_\eta))\right)+o(\epsilon_\eta^2).
\end{align*}
\end{claim}
The proof of the claim has been deferred to Appendix \ref{app: proof_of_claims} and here we continue with the proof of Lemma \ref{lemma: mod_z}. As $|V(z)|=\frac{1}{\step(z)^2}$, its expectation in steady state is finite. Thus, we set the drift of $V(z)$ to zero in steady state to get $\E{\calT_1+\calT_2}=0$. Now, by substituting $\calT_1$ and $\calT_2$ and dividing by $-j\epsilon^2_\eta\omega$, we get
\begin{align*}
  &-\E{e^{j\epsilon_\eta \omega \barz_\eta \step(\barz_\eta)}\left(\phi^c\left(\frac{\barz_\eta}{\tau_\eta}\right)-\phi^s\left(\frac{\barz_\eta}{\tau_\eta}\right)\right)\frac{1}{\step(\barz_\eta)}} -\frac{j\omega}{2}\E{e^{j\epsilon_\eta \omega \barz_\eta \step(\barz_\eta)}\left(\sigma^c(\lambda_{\pfix{\eta}}(\barz_\eta))+\sigma^s(\mu_{\pfix{\eta}}(\barz_\eta))\right)}\\
  &+\E{\left(\phi^c\left(\frac{\barz_\eta}{\tau_\eta}\right)-\phi^s\left(\frac{\barz_\eta}{\tau_\eta}\right)\right)\frac{1}{\step(\barz_\eta)}}+o(1)=0
\end{align*}
Rearranging the above terms, we get
\begin{align*}
    &\left(-1+\frac{j\omega}{2}\left(\sigma^c(\lambda^\star)+\sigma^s(\mu^\star)\right)\right)\E{e^{j\epsilon_\eta \omega \barz_\eta \step(\barz_\eta)}}\\
    ={}&\mathbb{E}\bigg[\underbrace{\left(\phi^c\left(\frac{\barz_\eta}{\tau_\eta}\right)-\phi^s\left(\frac{\barz_\eta}{\tau_\eta}\right)\right)\frac{1}{\step(\barz_\eta)}}_{\calT_3}\bigg]-\mathbb{E}\bigg[\underbrace{e^{j\epsilon_\eta \omega \barz_\eta \chi(\barz_\eta)}\left(1+\left(\phi^c\left(\frac{\barz_\eta}{\tau_\eta}\right)-\phi^s\left(\frac{\barz_\eta}{\tau_\eta}\right)\right)\frac{1}{\step(\barz_\eta)}\right)}_{\calT_4}\bigg]\\
    &+\frac{j\omega}{2}\mathbb{E}\bigg[\underbrace{e^{j\epsilon_\eta \omega \barz_\eta \chi(\barz_\eta)}\left(\sigma^c(\lambda^\star)+\sigma^s(\mu^\star)-\sigma^c(\lambda_{\pfix{\eta}}(\barz_\eta))-\sigma^s(\mu_{\pfix{\eta}}(\barz_\eta))\right)}_{\calT_5}\bigg] +o(1).
\end{align*}
Now, we show that $\E{\calT_3} \rightarrow -1$ ,$\E{\calT_4} \rightarrow 0$ and $\E{\calT_5} \rightarrow 0$ which give us the result. This is easy to show for state-independent control. Due to state-dependent control, we have a dependence on $\barz_\eta$ in the above terms which makes it non-trivial. We use Lemma \ref{claim: case1_t3_t5} and Lemma \ref{lemma: convergence_technical} to provide a bound on these terms. We start by analyzing $\calT_3$. By the definition of $\step(\cdot)$, we have $\left(\phi^c(x)-\phi^s(x)\right)/\step(x) \rightarrow -1$ as $x \rightarrow \pm \infty$. Also, $|\left(\phi^c(x)-\phi^s(x)\right)/\step(x)| \leq \frac{2\phi_{\max}}{\min\{|\chi(1)|, |\chi(-1)|\}}$ \revcolor{due to \eqref{eq:bounded_control_curves}}. Thus, by Lemma \ref{claim: case1_t3_t5}, we have
\begin{align*}
    \lim_{\eta \uparrow \infty}\E{\calT_3}= \lim_{\eta \uparrow \infty}\E{\left(\phi^c\left(\frac{\barz_\eta}{\tau_\eta}\right)-\phi^s\left(\frac{\barz_\eta}{\tau_\eta}\right)\right)\frac{1}{\step\left(\frac{\barz_\eta}{\tau_\eta}\right)}}&=-1.
\end{align*}
Now, we present the following claim to bound $\calT_4$ and $\calT_5$.
\begin{claim} \label{claim: t4_t5}
\begin{align*}
    \lim_{\eta \uparrow \infty} \E{\bigg|1+\left(\phi^c\left(\frac{\barz_\eta}{\tau_\eta}\right)-\phi^s\left(\frac{\barz_\eta}{\tau_\eta}\right)\right)\frac{1}{\step(\barz_\eta)}\bigg|}&=0 \\
    \lim_{\eta \uparrow \infty} \E{\bigg|\sigma^c(\lambda^\star)+\sigma^s(\mu^\star)-\sigma^c(\lambda_{\pfix{\eta}}(\barz_\eta))-\sigma^s(\mu_{\pfix{\eta}}(\barz_\eta))\bigg|}&=0.
\end{align*}
\end{claim}
The proof of the claim has been deferred to Appendix \ref{app: proof_of_claims} and here we continue with the proof of Lemma \ref{lemma: mod_z}. By the above claim, we have
\begin{align*}
\lim_{\eta \uparrow \infty} |\E{\calT_4}| &\leq \lim_{\eta \uparrow \infty} \E{|\calT_4|} = \lim_{\eta \uparrow \infty} \E{\bigg|1+\left(\phi^c\left(\frac{\barz_\eta}{\tau_\eta}\right)-\phi^s\left(\frac{\barz_\eta}{\tau_\eta}\right)\right)\frac{1}{\step(\barz_\eta)}\bigg|} =0 \\
   \lim_{\eta \uparrow \infty} |\E{\calT_5}| &\leq \lim_{\eta \uparrow \infty}\E{|\calT_5|} =\lim_{\eta \uparrow \infty} \frac{|\omega|}{2}\E{\bigg|\sigma^c(\lambda^\star)+\sigma^s(\mu^\star)-\sigma^c(\lambda_{\pfix{\eta}}(\barz_\eta))-\sigma^s(\mu_{\pfix{\eta}}(\barz_\eta))\bigg|}=0 
\end{align*}
This completes the proof. \hfill $\square$
\endproof
\section{Proof of Theorem \ref{theo: profit_driven_regime_3}}
Similar to the previous section, we prove Theorem \ref{theo: profit_driven_regime_3} by building upon the transform method presented by \citet{hurtado2020transform}. Applying the Transform method as is, does not work for this problem as we get an implicit equation, due to the state-dependent arrival rates. However, unlike in Theorem \ref{theo: case_2_itm}, the implicit equation may not have a unique solution. This suggests that working with $e^{j \omega z/\tau_\eta}$ alone as a test function is insufficient. So, we construct more test functions.

The key idea is to separately analyze the region where the probability mass vanishes and where it does not. We engineer separate test functions for the two regions and then combine the results to obtain the limiting distribution. In particular, we use $e^{\theta z/\tau_\eta}\mathbbm{1}\{z/\tau_\eta \in \Phi^\star\}$ as the test function with $\theta \in \bbR$ and $\theta = j\omega$ for $\omega \in \bbR$ to establish uniform/truncated exponential distribution between the thresholds. Then, we use $z^2$ as the test function to establish vanishing mass outside the threshold.
\subsection{Overview of the Proof}
Similar to the proof of Theorem \ref{theo: case_1}, we focus on finding the characteristic function of the limiting distribution. \revcolor{When $\Phi^\star \backslash \{0\} \neq \emptyset$, there exists $t_\star < t^\star$ such that one of the following four cases arises:  $\Phi^\star = [t_\star, t^\star]$, $\Phi^\star = (t_\star, t^\star)$, $\Phi^\star = (t_\star, t^\star]$, and $\Phi^\star = [t_\star, t^\star)$. Under all the four cases, we focus on characterizing the distribution of $\barz_\eta$ within and outside of $(t_\star, t^\star)$ separately. }
\subsubsection{Step 1: Distribution within the thresholds}
In this section, we show that $\barz_\eta\mathbbm{1}\{\barz_\eta/\tau_\eta \in (t_\star, t^\star)\}/\tau_\eta$ converges to a uniform/truncated exponential distribution depending on the value of $\drift$ as stated in the following lemma.
\begin{lemma} \label{lemma: regime_3_inside}
Consider the same setting as in Theorem \ref{theo: profit_driven_regime_3}. \revcolor{Let $\sigma^\star = \sigma^c\left(\mu^\star\right) + \sigma^s\left(\mu^\star\right)$. Also, there exists $t_\star < 0 < t^\star$ such that $\Phi^\star \in \left\{(t_\star, t^\star), (t_\star, t^\star], [t_\star, t^\star), [t_\star, t^\star]\right\}$}. Then, if $\drift = 0$, then for all $\omega \in \mathbb{R}$,
\begin{align*}
    \E{e^{j\omega \barz_\eta / \tau_\eta} \mathbbm{1}\left\{\frac{\barz_\eta}{\tau_\eta} \in (t_\star, t^\star)\right\}} 
    =
    \frac{e^{j \omega t^\star} - e^{j \omega t_\star}}{j\omega(t^\star-t_\star)}\P{\frac{\barz_\eta}{\tau_\eta} \in (t_\star, t^\star)} +o(1).
\end{align*}
On the other hand, if $\drift \neq 0$, then, for all $\omega \in \mathbb{R}$,
\begin{align*}
    &\E{e^{j\omega \barz_\eta / \tau_\eta} \mathbbm{1}\left\{\frac{\barz_\eta}{\tau_\eta} \in (t_\star, t^\star)\right\}} \\
    ={}& \frac{1}{-1 + \frac{j\omega \sigma^\star}{2\drift}}\left(\P{\frac{\barz_\eta}{\tau_\eta} < t^\star}e^{2\drift t^\star/\sigma^\star}- \P{\frac{\barz_\eta}{\tau_\eta} \leq t_\star}e^{2 \drift t_\star/\sigma^\star}\right)\frac{e^{j \omega t^\star} - e^{j \omega t_\star}}{e^{2\drift t^\star/
    \sigma^\star} - e^{2\drift t_\star/
    \sigma^\star}} \\
    &-\frac{1}{-1 + \frac{j\omega \sigma^\star}{2\drift}}\left(\P{\frac{\barz_\eta}{\tau_\eta} < t^\star}e^{j \omega t^\star}- \P{\frac{\barz_\eta}{\tau_\eta} \leq t_\star}e^{j \omega t_\star}\right) + o(1). 
\end{align*}
\end{lemma}
We prove the above lemma by setting the drift of the test function, $e^{\theta z/\tau_\eta}\mathbbm{1}\{z/\tau_\eta \in (t_\star, t^\star)\}$ to zero in the steady-state for $\theta \in \bbR$ and $\theta = j\omega$ for $\omega \in \bbR$. To gain intuition, consider the extreme case of $\lambda_\eta(z)=\mu_\eta(z)=0$ for all $z/\tau_\eta \notin \Phi^\star$. Then, we obtain a DTMC with a drift of $-d/\tau_\eta$ and two reflections at $t_\star \tau_\eta$ and $t^\star \tau_\eta$ respectively, contrary to a classical $\text{G/G/1}$ queue with a single reflection. Building upon this simpler setting, we consider a $\text{G/G/1}$ queue with a finite buffer in Section \ref{sec: finite_buffer}. We observe that two reflections introduced due to a finite buffer hinder us from yielding a closed-form expression of the characteristic function by a direct application of the transform method of \citet{hurtado2020transform}. We overcome this difficulty by additionally using the test function $e^{\theta z /\tau_\eta}\mathbbm{1}\{z /\tau_\eta \in (t_\star, t^\star)\}$ for $\theta \in \bbR$ and exploiting that $|\barz_\eta| \mathbbm{1}\{\barz_\eta/\tau_\eta \in (t_\star, t^\star)\}/\tau_\eta$ is bounded with probability 1. The proof details of Lemma \ref{lemma: regime_3_inside} \pfix{are} presented in Sections \ref{sec: lemma_regime_3_inside_proof} and \ref{sec: lemma_regime_3_inside_proof_d_0}.
\subsubsection{Step 2: No Mass Outside the Thresholds}
In this step, we show that $\barz_\eta / \tau_\eta$ vanishes outside the set $(t_\star, t^\star)$, thus completing the two main pieces of the proof. First, we assume $t_\star \neq t^\star$ and handle the case of $t_\star = t^\star$ separately. 
\begin{lemma} \label{lemma: regime_3_outside}
Under the same conditions as in Theorem \ref{theo: profit_driven_regime_3}, \revcolor{if there exists $t_\star < 0 < t^\star$ such that $\Phi^\star \in \left\{(t_\star, t^\star), (t_\star, t^\star], [t_\star, t^\star), [t_\star, t^\star]\right\}$,} then for all $\drift \in \bbR$, we have
\begin{align*}
    \lim_{\eta \uparrow \infty} \P{\frac{\barz_\eta}{\tau_\eta} \notin (t_\star, t^\star)} = 0.
\end{align*}
\end{lemma}
We prove the above lemma by first bounding the required probability \revcolor{$\P{\barz_\eta / \tau_\eta \notin (t_\star, t^\star)}$} by the expectation of the absolute imbalance using Markov's inequality. Then, we set the drift of the test function, $\barz_\eta^2$ to zero in steady state to obtain a useful upper bound on $\E{|\barz_\eta|}$. The proof details are presented in Section \ref{sec: lemma_regime_3_outside_proof}. 

\subsubsection{Step 3: Proof of Theorem \ref{theo: profit_driven_regime_3}}
Putting together the results from Step 1 and Step 2, we obtain the limiting distribution of $\barz_\eta / \tau_\eta$. \revcolor{Here, we assume $t_\star \neq t^\star$ and handle the case of $t_\star = t^\star$ in Appendix~\ref{sec:theorem3_singleton}.}
\proof{Proof of Theorem \ref{theo: profit_driven_regime_3} (Case 1 and 2)}
We divide the characteristic function into two regions: i.e. when $\phi^c(\cdot) - \phi^s(\cdot)$ is equal to zero and non-zero respectively. \revcolor{In particular, let $t_\star < 0 < t^\star$ be such that $\Phi^\star \in \left\{(t_\star, t^\star), (t_\star, t^\star], [t_\star, t^\star), [t_\star, t^\star]\right\}$. Then, we have
\begin{align*}
    \E{e^{j\omega\barz_\eta/\tau_\eta}} = \underbrace{\E{e^{j\omega \barz_\eta / \tau_\eta} \mathbbm{1}\left\{\frac{\barz_\eta}{\tau_\eta} \in (t_\star, t^\star)\right\}}}_{\calT_5} + \underbrace{\E{e^{j\omega \barz_\eta / \tau_\eta} \mathbbm{1}\left\{\frac{\barz_\eta}{\tau_\eta} \notin (t_\star, t^\star)\right\}}}_{\calT_6}.
\end{align*}
We characterize the limit of $\calT_6$ using Lemma \ref{lemma: regime_3_outside} as follows:
\begin{align}
0 \leq \liminf_{\eta \uparrow \infty} |\calT_6| \leq \limsup_{\eta \uparrow \infty} |\calT_6| &=
    \limsup_{\eta \uparrow \infty} \bigg|\E{e^{j\omega \barz_\eta / \tau_\eta} \mathbbm{1}\left\{\frac{\barz_\eta}{\tau_\eta} \notin (t_\star, t^\star)\right\}}\bigg| \nonumber \\
    &\leq  \limsup_{\eta \uparrow \infty} \P{\frac{\barz_\eta}{\tau_\eta} \notin (t_\star, t^\star)}  = 0. \label{eq: regime_3_t2}
\end{align}
Thus, we have $\lim_{\eta\uparrow \infty}\calT_6 = 0$. Next, we analyze the limit of $\calT_5$ by using Lemma \ref{lemma: regime_3_inside} as follows.}

\revcolor{\textbf{Case I ($\drift = 0$):} By Lemma~\ref{lemma: regime_3_inside}, we have
\begin{align}
    \calT_5 &= \frac{e^{j \omega t^\star} - e^{j \omega t_\star}}{j\omega(t^\star-t_\star)}\P{\frac{\barz_\eta}{\tau_\eta} \in (t_\star, t^\star)}  + o_{\eta}(1) \nonumber \\
    &= \frac{e^{j \omega t^\star} - e^{j \omega t_\star}}{j\omega(t^\star-t_\star)} -  \frac{e^{j \omega t^\star} - e^{j \omega t_\star}}{j\omega(t^\star-t_\star)}\P{\frac{\barz_\eta}{\tau_\eta} \notin (t_\star, t^\star)} + o_{\eta}(1). \nonumber
\end{align}
Now, by Lemma \ref{lemma: regime_3_outside}, the limit of the RHS exists, thus, we have
\begin{align*}
    \lim_{\eta \uparrow \infty} \calT_5 = \frac{e^{j \omega t^\star} - e^{j \omega t_\star}}{j\omega(t^\star-t_\star)}. \numberthis \label{eq: regime_3_t1}
\end{align*}
Using \eqref{eq: regime_3_t2} and \eqref{eq: regime_3_t1}, we get
\begin{align*}
    \lim_{\eta \uparrow \infty} \E{e^{j \omega \barz_\eta / \tau_\eta}} = \frac{e^{j \omega t^\star} - e^{j \omega t_\star}}{j\omega(t^\star-t_\star)}.
\end{align*}
Note that the RHS is the characteristic function of $\operatorname{Uniform}(\Phi^\star)$. Thus, by Levy's continuity theorem (e.g. see \citet[Chapter 18]{levy_book}), $\barz_\eta / \tau_\eta$ converges to a uniform random variable. This completes one part of the proof.}

\revcolor{\textbf{Case II ($\drift \neq 0$):} By Lemma \ref{lemma: regime_3_inside}, we have
\begin{align*}
    \calT_5 ={}& \frac{1}{-1 + \frac{j\omega \sigma^\star}{2\drift}}\left(\P{\frac{\barz_\eta}{\tau_\eta} < t^\star}e^{2\drift t^\star/\sigma^\star}- \P{\frac{\barz_\eta}{\tau_\eta} \leq t_\star}e^{2 \drift t_\star/\sigma^\star}\right)\frac{e^{j \omega t^\star} - e^{j \omega t_\star}}{e^{2\drift t^\star/
    \sigma^\star} - e^{2\drift t_\star/
    \sigma^\star}} \\
    &-\frac{1}{-1 + \frac{j\omega \sigma^\star}{2\drift}}\left(\P{\frac{\barz_\eta}{\tau_\eta} < t^\star}e^{j \omega t^\star}- \P{\frac{\barz_\eta}{\tau_\eta} \leq t_\star}e^{j \omega t_\star}\right) + o_{\eta}(1). 
\end{align*}
By Lemma \ref{lemma: regime_3_outside}, we have 
\begin{align*}
    1 &\geq \limsup_{\eta \uparrow \infty} \P{\frac{\barz_\eta}{\tau_\eta} < t^\star} \geq \liminf_{\eta \uparrow \infty} \P{\frac{\barz_\eta}{\tau_\eta} < t^\star} = 1 - \limsup_{\eta \uparrow \infty} \P{\frac{\barz_\eta}{\tau_\eta} \geq t^\star} \\
    &\geq 1 - \limsup_{\eta \uparrow \infty} \P{\frac{\barz_\eta}{\tau_\eta} \notin (t_\star, t^\star)} = 1.
\end{align*}
Similarly, we have $\lim_{\eta \uparrow \infty} \P{\frac{\barz_\eta}{\tau_\eta} \leq t_\star} = 0$. Thus, we have
\begin{align*}
    \lim_{\eta \uparrow \infty}\calT_5 &= \frac{1}{-1 + \frac{j\omega \sigma^\star}{2\drift}} \left( e^{2\drift t^\star/\sigma^\star}\frac{e^{j \omega t^\star} - e^{j \omega t_\star}}{e^{2\drift t^\star/
    \sigma^\star} - e^{2\drift t_\star/
    \sigma^\star}} - e^{j \omega t^\star}\right) \\
    &= \frac{1}{-1 + \frac{j\omega \sigma^\star}{2\drift}} \frac{e^{2\drift t_\star/\sigma^\star} e^{j \omega t^\star} - e^{2\drift t^\star/\sigma^\star} e^{j \omega t_\star}}{e^{2\drift t^\star/
    \sigma^\star} - e^{2\drift t_\star/\sigma^\star}}\numberthis \label{eq: t5_positive_drift}
\end{align*}
Using \eqref{eq: regime_3_t2} and \eqref{eq: t5_positive_drift}, we get
\begin{align*}
    \lim_{\eta \uparrow \infty} \E{e^{j \omega \barz_\eta / \tau_\eta}} = \frac{1}{-1 + \frac{j\omega \sigma^\star}{2\drift}} \frac{e^{2\drift t_\star/\sigma^\star} e^{j \omega t^\star} - e^{2\drift t^\star/\sigma^\star} e^{j \omega t_\star}}{e^{2\drift t^\star/
    \sigma^\star} - e^{2\drift t_\star/
    \sigma^\star}}.
\end{align*}
Note that the RHS is the characteristic function of $\operatorname{TrunExp}\left(\frac{2\drift}{\sigma^\star}, \Phi^\star\right)$. Thus, by Levy's continuity theorem (e.g. see \citet[Chapter 18]{levy_book}), $\barz_\eta / \tau_\eta$ converges to a truncated exponential random variable. This completes the proof.}
\hfill $\square$
\endproof
\subsubsection{Discussion}
To analyze the drift of complex exponential test functions, we analyze the drift of several auxiliary test functions in the appendix. The resultant bounds allow us to complete Steps 1 and 2 in the proof of Theorems \ref{theo: case_1} and \ref{theo: profit_driven_regime_3}.

In the proof of Theorem \ref{theo: profit_driven_regime_3}, it is noteworthy that Step 2 which shows bounded support is a form of state space collapse (SSC). In particular, we say that the imbalance stays within the thresholds with a high probability for a finite, large enough $\eta$. Proving such an SSC is a crucial step in obtaining the complete distribution of imbalance. In addition, showing that the imbalance has a symmetrical distribution in the proof of Theorem \ref{theo: case_1} is technically not an SSC but reminiscent of it.
\section{Classical Single Server Queue} \label{sec: discussion}
The heavy-traffic limiting behavior of a matching queue studied in previous sections exhibits a much richer phase transition behavior than that of a classical single server queue studied in the literature. This is primarily because most of the literature focuses on a constant arrival rate for a single server queue, whereas we studied a matching queue under state-dependent control. 
Single server queue with state-dependent control also exhibits phase transition as studied in the previous sections. 

\subsection{Infinite Waiting Area}
Consider a sequence of single server queue $\{q_\eta(k): k \in \bbZ_+\}$ for $\eta>0$ with infinite waiting area for the customers. We consider state-dependent arrival $a_\eta(q_\eta)$ with expectation given by $\E{a_\eta(q_\eta)}=\lambda_\eta(q_\eta)=\lambda^\star+\phi^c\left(\frac{q_\eta}{\tau_\eta}\right)\epsilon_\eta$, and variance given by $\Var{a_\eta(q_\eta)}=\sigma^c(\lambda_\eta(q_\eta))$. Similarly, the state-dependent potential service $s_\eta(q_\eta)$ has expectation  $\E{s_\eta(q_\eta)}=\mu_\eta(q_\eta)=\mu^\star+\phi^s\left(\frac{q_\eta}{\tau_\eta}\right)\epsilon_\eta$, and variance $\Var{s_\eta(q_\eta)}=\sigma^s(\mu_\eta(q_\eta))$. Without loss of generality, we assume that $\lambda^\star=\mu^\star$ and $\phi^s(\infty)-\phi^c(\infty)=1$. Next, we assume that there exists $A_{\max}>0$ and $\sigma_{\max}>0$ such that $|a_\eta(q_\eta)|\leq A_{\max} $, $|s_\eta(q_\eta)|\leq A_{\max}$, and $|\sigma^c(\lambda_\eta(q_\eta))| \leq \sigma_{\max}$, $|\sigma^s(\mu_\eta(q_\eta))| \leq \sigma_{\max}$ with probability 1. \revcolor{We also assume that there exists $\phi_{\max} > 0$ such that $\max_{x \in \bbR}|\phi^c(x)| \leq \phi_{\max}$ and $\max_{x \in \bbR}|\phi^s(x)| \leq \phi_{\max}$.} These boundedness assumptions are made for convenience and can be relaxed by imposing assumptions such as finite third moment for the arrivals and sub-quadratic growth for the control curve. Now we analyze a single server queue with infinite waiting area $(B_\eta = \infty)$ followed by the case of finite waiting area $(B_\eta < \infty)$.

We write the queue evolution equation as follows:
\begin{align*}
    q_\eta(k+1)=q_\eta(k)+a_\eta(q_\eta(k))-s_\eta(q_\eta(k))+u_\eta(q_\eta(k)),
\end{align*}
where $u_\eta(q_\eta(k))$ is the unused service if there are not enough customers waiting in the queue to be served and a fraction of the potential service $s_\eta(q_\eta(k))$ is not utilized. This implies that
\begin{align*}
    q_\eta(k+1)u_\eta(q_\eta(k))=0.
\end{align*}

Before presenting the general result, to illustrate the phase transition, we consider a single server queue operating in discrete time under Bernoulli arrivals, Bernoulli service, and a two-price policy. In particular, assume that $\lambda^\star-\epsilon_\eta>0$, $\lambda^\star<1$, and the control curves are such that $\lambda_\eta(q)=\lambda^\star-\epsilon_\eta\pfix{\mathbbm{1}}\{q>\tau_\eta\}$, and $\mu_\eta(q)=\mu^\star$ for all $q \in \bbZ_+$. Lastly, the arrival and service distribution are given by $a_\eta(q) \sim \operatorname{Bernoulli}(\lambda_\eta(q))$, and $s_\eta(q) \sim \operatorname{Bernoulli}(\mu_\eta(q))$. As the arrival and service distribution is Bernoulli, the single server queue is a birth-death process as shown in Figure \ref{fig: single_server} with $m\overset{\Delta}{=} \lambda^\star(1-\mu^\star)$. Now, we present the phase transition result below:
\begin{figure}[b]
\FIGURE{
      \begin{tikzpicture}[->, >=stealth', auto, semithick, node distance=2.5cm,scale=0.65,transform shape]
\tikzstyle{every state}=[fill=white,draw=black,thick,text=black,scale=1]
\node[state]    (0)                     {$1$};
\node[state]    (1)[right of=0]   {$2$};
\node[state]    (nminusone)[right of=1]   {$\tau_\eta$};
\node[state]    (n)[ right of=nminusone]   {$\tau_\eta+1$};
\node[state,white]    (nplusone)[ right of=n] {};
\node [state](-1)[left of=0] {$0$};
\path
(0) edge[bend right,below]     node{$m$}         (1)
(1) edge[bend right,below,dashed] node{$m$} (nminusone)
(nminusone) edge[bend right,below]     node{$m$}         (n)
(n) edge[bend right,below]     node{$m-\epsilon_\eta(1-\mu^\star)$}         (nplusone)
(nplusone) edge[bend right,above]     node{$m+\epsilon_\eta\mu^\star$}         (n)
(n) edge[bend right,above]     node{$m+\epsilon_\eta\mu^\star$}         (nminusone)
(nminusone) edge[bend right,above,dashed]     node{$m$}         (1)
(1) edge[bend right,above]     node{$m$}         (0)
(0) edge[bend right,above] node{$m$} (-1)
(-1) edge[bend right,below] node{$m$} (0)
(-1) edge[loop, above] node{$1-m$} (-1)
(0) edge[loop, above] node{$1-2m$} (0)
(1) edge[loop, above] node{$1-2m$} (1)
(nminusone) edge[loop, above] node{$1-2m$} (nminusone)
(n) edge[loop, above] node{$1-2m+\epsilon_\eta(1-2\mu^\star)$} (n);
\end{tikzpicture}}{Single server queue with Bernoulli arrivals and two-price policy with $m\overset{\Delta}{=}\lambda^\star(1-\mu^\star)$
\label{fig: single_server}}{}
\end{figure}
\begin{proposition} \label{prop: single_server_queue} Let $\{\epsilon_\eta\}_{\eta>0}$ and $\{\tau_\eta\}_{\eta>0}$ be such that $\lim_{\eta \uparrow \infty}\epsilon_\eta \tau_\eta=l$.
Consider a single server queue operating under the two-price policy for customer arrival. Then, we have the following results. 
\begin{enumerate}
    \item When $l=0$, we have
    \begin{align*}
        \epsilon_\eta \barq_\eta \overset{D}{\rightarrow} \operatorname{Exp}\left(\pfix{\frac{2}{\lambda^\star(1-\lambda^\star)+\mu^\star(1-\mu^\star)}}\right)
    \end{align*}
    \item When $l \in (0,\infty)$, we have $\epsilon_\eta \barq_\eta \overset{D}{\rightarrow} \qinfty$ such that
     \begin{align*}
   \P{\qinfty \leq q}=\begin{cases}
    1-\frac{m}{(l+m)}e^{-\frac{q-l}{m}} &\textit{if } q \geq l \\
    \frac{q}{(l+m)} &\textit{if } q \in [0,l).
    \end{cases}
\end{align*}
The distribution mentioned above is a one-sided hybrid distribution.
    \item When $l=\infty$, we have
    \begin{align*}
        \frac{\bar{\pfix{q}}_\eta}{\tau_\eta} \overset{D}{\rightarrow} \operatorname{Uniform}([0,1]).
    \end{align*}
\end{enumerate}
\end{proposition}
Thus, a single server queue also exhibits phase transition behavior under state-dependent control. 
For general arrivals, service, and pricing policies, one can obtain results similar to the ones in Section \ref{sec: main_theorem}, by essentially following the same methods. We present the result below.  
First, we introduce a negative drift condition analogous to Condition \ref{ass: neg_drift} to ensure positive recurrence for a general pricing policy.
\begin{condition}[Negative Drift]
There exists $\delta>0$ and $K>0$ such that for all $x>K$, we have $\phi^c(x)-\phi^s(x)<-\delta$. \label{ass: classical_neg_drift}
\end{condition}
Next, we show that $\epsilon_\eta \barq_\eta$ converges in distribution to an exponential distribution in the delay-driven regime, uniform distribution in the cost-driven regime, and to Gibbs distribution on the nonnegative axis, denoted by $\operatorname{Gibbs}_{+}(\cdot)$ in the hybrid regime. In particular, the PDF of $\operatorname{Gibbs}_{+}(g)$ is given by 
\begin{align*}
    f_{\operatorname{Gibbs}_+}(x;g) = \frac{1}{\int_{0}^\infty e^{-\int_0^x g(t) dt} dx} e^{-\int_0^x g(t) dt}  \mathbbm{1}\{x \geq 0\} \quad \forall x \in \bbR.
\end{align*}
Now, we present the result below.
\begin{theorem} \label{theo: classical_general}
Consider the positive recurrent  DTMC $\{q_\eta(k): k \in \bbZ_+\}$ for any $\eta >0$ and for any $\phi^c(\cdot)$ and $\phi^s(\cdot)$ satisfying Condition \ref{ass: classical_neg_drift} and let $\barq_\eta$ denote its steady state random variable. 
\begin{enumerate}
    \item If $\epsilon_\eta \tau_\eta \rightarrow 0$, then we have,
\begin{align*}
    \epsilon_\eta \barq_\eta \overset{D}{\rightarrow} \operatorname{Exp}\left(\pfix{\frac{2}{\sigma^c(\lambda^\star)+\sigma^s(\mu^\star)}}\right).
\end{align*}
\item If $\epsilon_\eta \tau_\eta \rightarrow l$ for $l \in (0,\infty)$ and Condition \ref{ass: c_pol} is satisfied, then we have
\begin{align*}
    \epsilon_\eta \barq_\eta \overset{D}{\rightarrow} \operatorname{Gibbs}_+(g_{1,l})
\end{align*}
where $g_{1,l}$ is given by \eqref{eq: gibbs_distribution}.
\item If $\epsilon_\eta\tau_\eta \rightarrow \infty$ and Condition \ref{cond: monotonicity} is satisfied with $\Phi^\star  = [0, t^\star]$ for some $t^\star > 0$, then we have
\begin{align*}
    \frac{\barq_\eta}{\tau_\eta} \overset{D}{\rightarrow} \operatorname{Uniform}([0, t^\star]).
\end{align*}
\end{enumerate}
\end{theorem}
The transform method introduced by \citet{hurtado2020transform} provided the stationary distribution of a single server queue with static arrival and service rate. In particular, they obtain a closed-form expression for the characteristic function of the limiting distribution which immediately establishes convergence in distribution. In contrast, to prove part 2 of the theorem, we obtain an \emph{implicit} equation (\eqref{eq: classical_functional_eq} in Appendix), to solve which, we use the inverse Fourier transform method. Therefore, the method of \citet{hurtado2020transform} is inadequate to obtain the limiting distribution for a single server queue with dynamic arrival rates. The proof of Theorem \ref{theo: classical_general} is analogous to the two-sided analogs and thus, we defer all the proof details to the Appendix \ref{app: classical_queue}.

\subsection{Finite Waiting Area} \label{sec: finite_buffer}
In this section, we consider a single server queue with a \emph{state-independent} arrival and service rate and a finite waiting area for the customers. This setting relates to the cost-driven regime, and in particular, Lemma \ref{lemma: regime_3_inside}. Specifically, it is a simpler setup while still preserving the technical difficulties posed by two reflections as in Lemma \ref{lemma: regime_3_inside}. The reader can read the proof of Lemma \ref{lemma: ssq_finite_buffer} below before dwelling into the proof of Lemma \ref{lemma: regime_3_inside} to gain intuition.
\begin{lemma} \label{lemma: ssq_finite_buffer}
Consider a discrete-time single server queue $q_\eta(k)$ with a space for $\tau_\eta$ customers. That is, we have $q_\eta(k+1) = \min\{[q_\eta(k)+a(k)-s_\eta(k)]^+, \tau_\eta\}$, where $a(k)$ and $s_\eta(k)$ are potential arrivals and service respectively with $\E{a(k)} = \lambda^\star$, $\E{s_\eta(k)} = \lambda^\star +\drift/\tau_\eta$, and $\max\{a(k), s_\eta(k)\} \leq A_{\max}$ w.p. 1. In addition, let $\sigma^\star = \operatorname{Var}[a(k)] + \lim_{\eta \uparrow \infty} \operatorname{Var}[s_\eta(k)]$. Then, the steady state of the queue length $\barq_\eta$ satisfies
    \begin{align*}
        \frac{\barq_\eta}{\tau_\eta} \overset{D}{\rightarrow} \begin{cases}
            \operatorname{Uniform}[0, 1]  &\textit{if } \drift = 0 \\
            \operatorname{TrunExp}\left(\frac{2\drift}{\sigma^\star}, [0, 1]\right) &\textit{if } \drift \neq 0
        \end{cases}
    \end{align*}
\end{lemma}
\proof{Proof of Lemma \ref{lemma: ssq_finite_buffer}}
For a single server queue with finite capacity, let $\bar{a}$ and $\bar{s}_\eta$ be the potential arrivals and services in the steady state. In addition, let $\baru_{a, \eta}$ and $\baru_{s, \eta}$ be the dropped arrival due to finite capacity and unused service respectively. We consider $\baru_{a, \eta}\baru_{s, \eta} = 0$ w.p. 1 without loss of generality. Then, we have $\barq_\eta^+ = \barq_\eta + \bara - \bars_\eta + \baru_{s, \eta} - \baru_{a, \eta}$, where $\barq_\eta^+$ is the queue-length one step after $\barq_\eta$ in the steady state. In addition, we have $\barq_\eta^+ \baru_{s, \eta} = 0$ and $(\tau_\eta - \barq_\eta^+)\baru_{a, \eta} = 0$ w.p. 1. We start by obtaining useful bounds on $\E{\baru_{a, \eta}}$ and $\E{\baru_{s, \eta}}$. By setting the drift of $V(q)=q$ to zero, we get $\E{\baru_{a, \eta}} = \E{\baru_{s, \eta}}-\drift/\tau_\eta$. Next, we analyze the drift of $V(q)=q^2$. We have 
\begin{align*}
    \left(\barq^+_\eta\right)^2 &= \left(\barq_\eta^+ - \baru_{s, \eta} + \baru_{s, \eta}\right)^2 \\
    &= \left(\barq_\eta^+ - \baru_{s, \eta}\right)^2 - \baru_{s, \eta}^2 \\
    &= \left(\barq_\eta^+ - \baru_{s, \eta} + \baru_{a, \eta} - \baru_{a, \eta}\right)^2 - \baru_{s, \eta}^2 \\
    &= \left(\barq_\eta^+ - \baru_{s,\eta}+\baru_{a, \eta}\right)^2 - \baru_{a, \eta}^2 - 2\tau_\eta \baru_{a, \eta} - \baru_{s, \eta}^2 \\
    &= \left(\barq_\eta + \bar{a} - \bar{s}_\eta\right)^2 - \baru_{a, \eta}^2 - 2\tau_\eta \baru_{a, \eta} - \baru_{s, \eta}^2 \\
    &= \barq_\eta^2 + \left(\bar{a}-\bar{s}_\eta\right)^2 + 2\barq_\eta \left(\pfix{\bar{a}}-\bar{s}_\eta\right)-\baru_{a, \eta}^2 - 2\tau_\eta \baru_{a, \eta} - \baru_{s, \eta}^2.
\end{align*}
Now, by taking the expectation on both sides, letting $\sigma_\eta = \operatorname{Var}[\bara] + \operatorname{Var}[\bars_\eta]$ and noting that $\E{\left(\barq_\eta^+\right)^2} = \E{\barq_\eta^2}$, we have
\begin{align*}
    \sigma_\eta + \frac{\drift^2}{\tau_\eta^2} - \frac{2\drift}{\tau_\eta}\E{\barq_\eta} - \E{\baru_{a, \eta}^2}-2\tau_\eta \E{\baru_{a, \eta}}-\E{\baru_{s, \eta}^2} = 0.
\end{align*}
Now, we consider two cases depending on the value of $\drift \in \bbR$.

\textbf{Case I $(d = 0)$:} Noting that $0 \leq \baru_{a, \eta} \leq A_{\max}$ and $0 \leq \baru_{s, \eta} \leq A_{\max}$, we have
\begin{align*}
    \frac{\sigma_\eta}{2\tau_\eta + 2A_{\max}} \leq \E{\baru_{a, \eta}} = \E{\baru_{s, \eta}} \leq \frac{\sigma_\eta}{2\tau_\eta} \implies \E{\baru_{a, \eta}} = \E{\baru_{s, \eta}} = \frac{\sigma^\star}{2\tau_\eta} + o\left(\frac{1}{\tau_\eta}\right). \numberthis \label{eq: no_drift_unused_service}
\end{align*}
\textbf{Case II $(\drift \neq 0)$:} In this case, we have two equations but three unknowns $(\E{\barq_\eta}, \E{\baru_{a, \eta}}, \E{\baru_{s, \eta}})$ and so, we do not get a tight characterization of $\E{\baru_{a, \eta}}$ and $\E{\baru_{s, \eta}}$. However, we can obtain the following weaker bound that will be useful later. As $\barq_\eta \leq \tau_\eta$, we have
\begin{align*}
    \E{\baru_{a, \eta}} \leq \frac{\sigma_\eta+2|d|}{2\tau_\eta} + \frac{\drift^2}{2\tau_\eta^3} = o(1), \quad \E{\baru_{s, \eta}} \leq \frac{\sigma_\eta+4|\drift|}{2\tau_\eta}+ \frac{\drift^2}{2\tau_\eta^3} = o(1). \numberthis \label{eq: drift_unused_service}
\end{align*}
With these preliminary bounds, we are now ready to characterize the distribution of $\barq_\eta/\tau_\eta$. We do this by analyzing the drift of $e^{\theta \barq_\eta/\tau_\eta}$ for $\theta \in \bbR$. We use the following algebraic trick (similar to \citet{hurtado2020transform}) to simplify the drift of $e^{\theta \barq_\eta/\tau_\eta}$. As $\barq_\eta^+ \baru_{s, \eta} = 0$ with probability 1, we have
\begin{align*}
    &\left(e^{\theta \barq^+_\eta/\tau_\eta}-1\right)\left(e^{-\theta \baru_{s, \eta}/\tau_\eta}-1\right)=0 \\
    \implies{}& e^{\theta (\barq^+_\eta-\baru_{s, \eta})/\tau_\eta} - e^{- \theta \baru_{s, \eta}/\tau_\eta} - e^{\theta \barq^+_\eta/\tau_\eta}+1 = 0
\end{align*}
In addition, as $(\barq_\eta^+-\tau_\eta)\baru_{a, \eta}= 0$ with probability 1, we have
\begin{align*}
&\left(e^{\theta\left(\barq^+_\eta-\baru_{s, \eta}-\tau_\eta\right)/\tau_\eta}-1\right)\left(e^{\theta \baru_{a, \eta}/\tau_\eta}-1\right) = 0 \\
    \implies{}& e^{\theta\left(\barq^+_\eta-\baru_{s, \eta}+\baru_{a, \eta}\right)/\tau_\eta} e^{-\theta} - e^{\theta \baru_{a, \eta}/\tau_\eta} - e^{\theta \left(\barq^+_\eta-\baru_{s, \eta}-\tau_\eta\right)/\tau_\eta}+1 = 0 \\
    \implies{}& e^{\theta\left(\barq^+_\eta-\baru_{s, \eta}+\baru_{a, \eta}\right)/\tau_\eta} e^{-\theta} - e^{\theta \baru_{a, \eta}/\tau_\eta} -e^{-\theta \baru_{s, \eta}/\tau_\eta} e^{-\theta} - e^{\theta \barq^+_\eta/\tau_\eta} e^{-\theta} + e^{-\theta}+1 = 0.
\end{align*}
Noting that $0 \leq \barq_\eta \leq \tau_\eta$, and taking the expectation on both sides, we get
\begin{align*}
    \left(\E{e^{\theta \baru_{a, \eta}/\tau_\eta}}-1\right) e^\theta + \left(\E{e^{-\theta \baru_{s, \eta}/\tau_\eta}}-1\right) &= \E{e^{\theta \left(\barq^+_\eta - \baru_{s, \eta} + \baru_{a, \eta}\right)/\tau_\eta}} - \E{e^{\theta \barq_\eta/\tau_\eta}} \\\
    &= \E{e^{\theta \barq_\eta/\tau_\eta}} \left(\E{e^{\theta (\bar{a}-\bar{s}_\eta)/\tau_\eta}}-1\right) \numberthis \label{eq: ssq_just_before_taylor_series}
\end{align*}
Now, by using Taylor's theorem (Lemma \ref{lemma: taylor_series}), we have
\begin{align*}
    \E{e^{\theta (\bar{a}-\bar{s}_\eta)/\tau_\eta}}-1 &= -\frac{\drift\theta}{\tau_\eta^2} + \frac{\theta^2 \sigma^\star}{2\tau_\eta^2} + o\left(\frac{1}{\tau_\eta^2}\right), \\
    \E{e^{\theta \baru_{a, \eta}/\tau_\eta}}-1 &= \frac{\theta}{\tau_\eta}\E{\baru_{a, \eta}} + \frac{\theta^2}{2\tau_\eta^2}\E{\baru_{a, \eta}^2} + o\left(\frac{1}{\tau_\eta^2}\right) \overset{(a)}{=} \frac{\theta}{\tau_\eta}\E{\baru_{a, \eta}} + o\left(\frac{1}{\tau_\eta^2}\right), \\
    \E{e^{-\theta \baru_{s, \eta}/\tau_\eta}}-1 &= -\frac{\theta}{\tau_\eta}\E{\baru_{s, \eta}} + \frac{\theta^2}{2\tau_\eta^2} \E{\baru_{s, \eta}^2} + o\left(\frac{1}{\tau_\eta^2}\right) \overset{(b)}{=} -\frac{\drift\theta}{\tau_\eta^2} - \frac{\theta}{\tau_\eta}\E{\baru_{a, \eta}} + o\left(\frac{1}{\tau_\eta^2}\right),
\end{align*}
where $(a)$ follows as $\E{\baru_{a, \eta}^2} \leq A_{\max}\E{\baru_{a, \eta}} = o(1)$ by \eqref{eq: no_drift_unused_service} and \eqref{eq: drift_unused_service}. Similarly, we also get $\E{\baru_{s, \eta}^2} = o(1)$ which implies $(b)$. Now, by substituting the above equations back in \eqref{eq: ssq_just_before_taylor_series}, we get
\begin{align*}
     \left(-\frac{\drift\theta}{\tau_\eta^2} + \frac{\theta^2 \sigma^\star}{2\tau_\eta^2} + o\left(\frac{1}{\tau_\eta^2}\right)\right) \E{e^{\theta \barq_\eta/\tau_\eta}} = -\frac{\drift\theta}{\tau_\eta^2} + \frac{\theta \E{\baru_{a, \eta}}}{\tau_\eta} \left(e^\theta-1\right) + o\left(\frac{1}{\tau_\eta^2}\right). \numberthis \label{eq: pre_limit_mgf}
\end{align*}
\textbf{Case I $(d = 0)$:} Using \eqref{eq: no_drift_unused_service}, we have
\begin{align*}
    \E{e^{\theta \barq_\eta / \tau_\eta}} = \frac{e^{\theta}-1 + o(1)}{\theta +o(1)} \implies \lim_{\eta \uparrow \infty} \E{e^{\theta \barq_\eta/\tau_\eta}} = 
    \frac{e^\theta-1}{\theta}.
\end{align*}
This completes the first part of the proof.

\textbf{Case II $(d \neq 0)$:} Set $\theta  = 2\drift/\sigma^\star$ and observe that $|\E{e^{\theta \barq_\eta/\tau_\eta}}| \leq e^{\theta}$ to obtain 
\begin{align*}
    \E{\baru_{a, \eta}} = \frac{\drift}{\tau_\eta \left(e^{2\drift/\sigma^\star}-1\right)} + o\left(\frac{1}{\tau_\eta}\right).
\end{align*}
Note that the above step is the main difference from the proof of \citet{hurtado2020transform} for a single server queue with infinite buffer. Although, the term $\E{\baru_{a, \eta}}$ appears in \eqref{eq: pre_limit_mgf} due to a finite buffer and hinders us from obtaining the closed-form expression of the MGF; the finite buffer ensures $\barq_\eta / \tau_\eta \leq 1$ w.p. 1 which in turn allows us to obtain $\E{\baru_{a, \eta}}$ as above. The above step is also the reason why we are working with the MGF and not the characteristic function. Now, substituting this back again in \eqref{eq: pre_limit_mgf}, we get
\begin{align*}
    \E{e^{\theta \barq_\eta/\tau_\eta}} = \frac{\frac{e^{\theta}-e^{2\drift/\sigma^\star}}{e^{2\drift/\sigma^\star}-1} + o(1)}{-1+\frac{\theta \sigma^\star}{2\drift} +o(1)}.
\end{align*}
Now, by taking the limit as $\eta \uparrow \infty$, we get
\begin{align*}
    \lim_{\eta \uparrow \infty}\E{e^{\theta \barq_\eta/\tau_\eta}} = \frac{1}{-1+\frac{\theta \sigma^\star}{2\drift}}\frac{e^{\theta}-e^{2\drift/\sigma^\star}}{e^{2\drift/\sigma^\star}-1} = \frac{1}{1-\frac{\theta \sigma^\star}{2\drift}}\frac{1-e^{\theta-2\drift/\sigma^\star}}{1-e^{-2\drift/\sigma^\star}}.
\end{align*}
This completes the proof. \hfill $\square$
\endproof

\section{Conclusion and Future Work}
In this paper, we develop a heavy traffic theory of a matching or a two-sided queue. We consider general, state-dependent arrivals and define heavy traffic as the limit such that the external control goes to zero. There are two different ways in which the control can be sent to zero which is modeled by the magnitude parameter, $\epsilon_\eta$ that goes to zero, and the position parameter, $\tau_\eta$ that goes to infinity. Based on the relative speeds at which these parameters converge to their asymptotes, we observe a phase transition in the distribution of the imbalance. A similar phase transition is also observed in a single server queue under state-dependent controls. To obtain these results, we develop non-trivial generalizations of the transform method of  \citet{hurtado2020transform} in each of the three regimes. The hybrid regime employed the inverse Fourier transforms and the other two regimes engineer multiple complex exponential Lyapunov functions.


This paper analyzes the simplest matching queue, a building block of matching networks that arise in applications such as blockchain (\citet{sushil_blockchain}), ride-hailing (\citet{SIGMETRICS_strategic_servers, varma_twosided_or}), quantum switches (\citet{quantum_towsley_sigmetrics, quantum2-towsley}), and assemble-to-order systems (\citet{plambeck2006optimal, reiman2015asymptotically}). Future work includes studying the heavy traffic behavior of these systems. The key ingredient in such an endeavor is to establish an appropriate state space collapse. Another line of future work building upon our result on state-dependent control in a classical single server queue is to study stochastic processing networks (e.g., load balancing studied by \citet{atilla}) under state-dependent control.

\bibliographystyle{informs2014} 
\bibliography{references.bib} 

@article{atilla,
author={Eryilmaz, A. and Srikant, R.},
title={Asymptotically tight steady-state queue length bounds implied by drift conditions},
year={2012},
issn={0257-0130},
journal={Queueing Systems},
publisher={Springer US},
volume={72},
number={3-4},
pages={311--359},
}

@article{kingman,
  title={Some inequalities for the queue {GI/G/1}},
  author={J Kingman},
  journal={Biometrika},
  pages={315-324},
  year={1962}
}

@article{stolyar2004maxweight,
  title={Max{W}eight scheduling in a generalized switch: State space collapse and workload minimization in heavy traffic},
  author={Stolyar, A},
  journal={Annals of Applied Probability},
  pages={1--53},
  year={2004},
  publisher={JSTOR}
}

@article{MagSri_SSY16_Switch,
  title={Heavy traffic queue length behavior in a switch under the {M}ax{W}eight algorithm},
  author={Maguluri, Siva Theja and Srikant, R},
  journal={Stochastic Systems},
  volume={6},
  number={1},
  pages={211--250},
  year={2016},
  publisher={INFORMS},
  url = {http://dx.doi.org/10.1214/15-SSY193}
}

@article{braverman_BAR,
  title={Heavy traffic approximation for the stationary distribution of a {G}eneralized {J}ackson {N}etwork: The {BAR} approach},
  author={Braverman, Anton and Dai, JG and Miyazawa, Masakiyo},
  journal={Stochastic Systems},
  volume={7},
  number={1},
  pages={143--196},
  year={2017},
  publisher={INFORMS}
}

@article{hurtado2020transform,
  title={Transform methods for heavy-traffic analysis},
  author={Hurtado-Lange, Daniela and Maguluri, Siva Theja},
  journal={Stochastic Systems},
  volume={10},
  number={4},
  pages={275--309},
  year={2020},
  publisher={INFORMS}
}

@book{van2000asymptotic,
	title={Asymptotic statistics},
	author={Van der Vaart, Aad W},
	volume={3},
	year={2000},
	publisher={Cambridge university press}
}

@article{zhou2018flexible,
	title={Flexible Load Balancing with Multi-dimensional State-space Collapse: Throughput and Heavy-traffic Delay Optimality},
	author={Zhou, Xingyu and Tan, Jian and Shroff, Ness},
	journal={Performance Evaluation},
	volume={127},
	pages={176--193},
	year={2018},
	publisher={Elsevier}
}

@book{hajekrandomprocbook,
  title={Random Processes for Engineers},
  author={Hajek, B.},
  year={2015},
  publisher={Cambridge university press}
}

@article{gamarnik2006validity,
  title={Validity of Heavy Traffic Steady-State Approximations in {G}eneralized {J}ackson {N}etworks},
  author={Gamarnik, D. and Zeevi, A.},
  journal={The Annals of Applied Probability},
  pages={56--90},
  year={2006},
  publisher={JSTOR}
}

@incollection{harrison1988brownian,
  title={Brownian Models of Queueing Networks with Heterogeneous Customer Populations},
  author={Harrison, J},
  booktitle={Stochastic Differential Systems, Stochastic Control Theory and Applications},
  pages={147--186},
  year={1988},
  publisher={Springer}
}

@article{gurvich2014diffusion,
  title={Diffusion models and steady-state approximations for exponentially ergodic {M}arkovian queues},
  author={Gurvich, Itai},
  journal={The Annals of Applied Probability},
  volume={24},
  number={6},
  pages={2527--2559},
  year={2014},
  publisher={Institute of Mathematical Statistics}
}

@article{mandelbaum_stolyar2004cmu,
  title={Scheduling flexible servers with convex delay costs: {H}eavy-traffic optimality of the generalized c$\mu$-rule},
  author={Mandelbaum, Avishai and Stolyar, Alexander},
  journal={Operations Research},
  volume={52},
  number={6},
  pages={836--855},
  year={2004},
  publisher={INFORMS}
}

@article{Walton_SteinHT,
  title={Stein's Method for the Single Server Queue in Heavy Traffic},
  author={Gaunt, Robert and Walton, Neil},
  journal={Statistics \& Probability Letters},
  volume={156},
  pages={108566},
  year={2020},
  publisher={Elsevier}
}

@article{harrison1998heavy,
  title={Heavy traffic analysis of a system with parallel servers: asymptotic optimality of discrete-review policies},
  author={Harrison, J Michael},
  journal={Annals of applied probability},
  pages={822--848},
  year={1998},
  publisher={JSTOR}
}

@article{LiuLei_JSQ_UniversalScaling,
  title={On Universal Scaling of Distributed Queues under Load Balancing},
  author={Liu, Xin and Ying, Lei},
  journal={Preprint arXiv:1912.11904},
  year={2019}
}

@article{HalfinWhitt_Regime,
  title={Heavy-traffic limits for queues with many exponential servers},
  author={Halfin, Shlomo and Whitt, Ward},
  journal={Operations research},
  volume={29},
  number={3},
  pages={567--588},
  year={1981},
  publisher={Informs}
}

@article{Hurtado_JSQ_alpha_discrete,
  title={Load balancing system under join the shortest queue: Many-server-heavy-traffic asymptotics},
  author={Hurtado-Lange, Daniela and Maguluri, Siva Theja},
  journal={Preprint arXiv:2004.04826v2},
  year={2020}
}

@article{adan2012exact,
  title={Exact {FCFS} matching rates for two infinite multitype sequences},
  author={Adan, Ivo and Weiss, Gideon},
  journal={Operations Research},
  volume={60},
  number={2},
  pages={475--489},
  year={2012},
  publisher={INFORMS}
}

@article{akbarpour2017thickness,
  title={Thickness and information in dynamic matching markets},
  author={Akbarpour, Mohammad and Li, Shengwu and Oveis Gharan, Shayan},
  journal = {Journal of Political Economy},
  volume = {(forthcoming)},
  number = {},
  pages = {},
  year={2019}
}

@article{caldentey2009fcfs,
  title={{FCFS} infinite bipartite matching of servers and customers},
  author={Caldentey, Ren{\'e} and Kaplan, Edward H and Weiss, Gideon},
  journal={Advances in Applied Probability},
  volume={41},
  number={3},
  pages={695--730},
  year={2009},
  publisher={Cambridge University Press}
}

@article{matchingqueues,
  title={On the dynamic control of matching queues},
  author={Gurvich, Itai and Ward, Amy},
  journal={Stochastic Systems},
  volume={4},
  number={2},
  pages={479--523},
  year={2014},
  publisher={INFORMS}
}

@article{song1998order,
  title={On the order fill rate in a multi-item, base-stock inventory system},
  author={Song, Jing-Sheng},
  journal={Operations Research},
  volume={46},
  number={6},
  pages={831--845},
  year={1998},
  publisher={INFORMS}
}

@article{song2002performance,
  title={Performance analysis and optimization of assemble-to-order systems with random lead times},
  author={Song, Jing-Sheng and Yao, David D},
  journal={Operations Research},
  volume={50},
  number={5},
  pages={889--903},
  year={2002},
  publisher={INFORMS}
}

@article{blanchet2021asymptotically,
      title={Asymptotically Optimal Control of a Centralized Dynamic Matching Market with General Utilities}, 
      author={Jose H. Blanchet and Martin I. Reiman and Viragh Shah and Lawrence M. Wein and Linjia Wu},
      year={2020},
      journal={Preprint arXiv:2002.03205},
}

@article{aveklouris2021matching,
      title={Matching Impatient and Heterogeneous Demand and Supply}, 
      author={Angelos Aveklouris and Levi DeValve and Amy R. Ward and Xiaofan Wu},
      year={2021},
      journal={Preprint arXiv:2102.02710},
}

@article{ondemandservers,
  title={A queueing system with on-demand servers: local stability of fluid limits},
  author={Nguyen, Lam M and Stolyar, Alexander L},
  journal={Queueing Systems},
  volume={89},
  number={3-4},
  pages={243--268},
  year={2018},
  publisher={Springer}
}

@article{yashkanoriabarter,
  title={Efficient dynamic barter exchange},
  author={Anderson, Ross and Ashlagi, Itai and Gamarnik, David and Kanoria, Yash},
  journal={Operations Research},
  volume={65},
  number={6},
  pages={1446--1459},
  year={2017},
  publisher={INFORMS}
}

@inproceedings{banerjeeridehailing,
 author = {Banerjee, Siddhartha and Freund, Daniel and Lykouris, Thodoris},
 title = {Pricing and Optimization in Shared Vehicle Systems: An Approximation Framework},
 booktitle = {Proceedings of the 2017 ACM Conference on Economics and Computation},
 year = {2017},
 location = {Cambridge, Massachusetts, USA},
 pages = {517--517}
}

@article{banerjee2018state,
  title={State dependent control of closed queueing networks},
  author={Banerjee, Siddhartha and Kanoria, Yash and Qian, Pengyu},
  journal={ACM SIGMETRICS Performance Evaluation Review},
  volume={46},
  number={1},
  pages={2--4},
  year={2018},
  publisher={ACM New York, NY, USA}
}

@inproceedings{spider_nsdi,
  title={High throughput cryptocurrency routing in payment channel networks},
  author={Sivaraman, Vibhaalakshmi and Venkatakrishnan, Shaileshh Bojja and Ruan, Kathleen and Negi, Parimarjan and Yang, Lei and Mittal, Radhika and Fanti, Giulia and Alizadeh, Mohammad},
  booktitle={17th $\{$USENIX$\}$ Symposium on Networked Systems Design and Implementation ($\{$NSDI$\}$ 20)},
  pages={777--796},
  year={2020}
}

@article{quantum2-towsley,
    title={On the Capacity Region of Bipartite and Tripartite Entanglement Switching},
    author={Gayane Vardoyan and Saikat Guha and Philippe Nain and Don Towsley},
    year={2019},
    journal={Preprint arXiv:1901.06786}
}

@article{quantum_towsley_sigmetrics,
  title={On the capacity region of bipartite and tripartite entanglement switching},
  author={Vardoyan, Gayane and Guha, Saikat and Nain, Philippe and Towsley, Don},
  journal={ACM SIGMETRICS Performance Evaluation Review},
  volume={48},
  number={3},
  pages={45--50},
  year={2021},
  publisher={ACM New York, NY, USA}
}

@article{dougru2010stochastic,
  title={A stochastic programming based inventory policy for assemble-to-order systems with application to the W model},
  author={Do{\u{g}}ru, Mustafa K and Reiman, Martin I and Wang, Qiong},
  journal={Operations research},
  volume={58},
  number={4-part-1},
  pages={849--864},
  year={2010},
  publisher={INFORMS}
}

@article{reiman2015asymptotically,
  title={Asymptotically optimal inventory control for assemble-to-order systems with identical lead times},
  author={Reiman, Martin I and Wang, Qiong},
  journal={Operations Research},
  volume={63},
  number={3},
  pages={716--732},
  year={2015},
  publisher={INFORMS}
}

@article{plambeck2006optimal,
  title={Optimal control of a high-volume assemble-to-order system},
  author={Plambeck, Erica L and Ward, Amy R},
  journal={Mathematics of Operations Research},
  volume={31},
  number={3},
  pages={453--477},
  year={2006},
  publisher={INFORMS}
}

@unpublished{dynamictypematchinghu,
  title={Dynamic type matching},
  author={Hu, Ming and Zhou, Yun},
  year={2018},
  note={{R}otman School of Management Working Paper {No.} 2592622}
}

@article{kim2017value,
  title={The value of dynamic pricing in large queueing systems},
  author={Kim, Jeunghyun and Randhawa, Ramandeep S},
  journal={Operations Research},
  volume={66},
  number={2},
  pages={409--425},
  year={2017},
  publisher={INFORMS}
}

@book{talluri2006theory,
  title={The theory and practice of revenue management},
  author={Talluri, Kalyan T and Van Ryzin, Garrett J},
  year={2006},
  publisher={Springer Science \& Business Media},
  address={Boston, MA, USA}
}

@article{low1974optimal,
  title={Optimal pricing for an unbounded queue},
  author={Low, David W.},
  journal={IBM Journal of research and Development},
  volume={18},
  number={4},
  pages={290--302},
  year={1974},
  publisher={IBM}
}

@inproceedings{ks2020optimal,
  title={Optimal Pricing in Finite Server Systems},
  author={KS, Ashok Krishnan and Singh, Chandramani and Maguluri, Siva Theja and Parag, Parimal},
  booktitle={2020 18th International Symposium on Modeling and Optimization in Mobile, Ad Hoc, and Wireless Networks (WiOPT)},
  pages={1--8},
  year={2020},
  organization={IEEE}
}

@article{Tsitsiklis2000congestion,
  title={Congestion-dependent pricing of network services},
  author={Paschalidis, I Ch and Tsitsiklis, John N},
  journal={IEEE/ACM Transactions on networking},
  volume={8},
  number={2},
  pages={171--184},
  year={2000},
  publisher={IEEE}
}

@article{pricinginqueuing2001,
  title={State dependent pricing with a queue},
  author={Chen, Hong and Frank, Murray Z},
  journal={IIE Transactions},
  volume={33},
  number={10},
  pages={847--860},
  year={2001},
  publisher={Springer}
}

@article{amywardpricingmatching,
  title={Dynamic Matching for Real-Time Ride Sharing},
  author={{\"O}zkan, Erhun and Ward, Amy R},
  journal={Stochastic Systems},
  year={2020},
  publisher={INFORMS}
}

@article{Hurtado-gen-switch-SIGMETRICS,
author = {Hurtado-Lange, Daniela and Maguluri, Siva Theja},
title = {Heavy-Traffic Analysis of the Generalized Switch under Multidimensional State Space Collapse},
year = {2019},
issue_date = {September 2019},
publisher = {Association for Computing Machinery},
address = {New York, NY, USA},
volume = {47},
number = {2},
issn = {0163-5999},
doi = {10.1145/3374888.3374902},
journal = {SIGMETRICS Perform. Eval. Rev.},
month = dec,
pages = {36–38},
numpages = {3},
}

@article{sushil_blockchain,
  title={Throughput optimal routing in blockchain based payment systems},
  author={Varma, Sushil Mahavir and Maguluri, Siva Theja},
  journal={IEEE Transactions on Control of Network Systems},
  year={2021},
  publisher={IEEE}
}

@inproceedings{ridehailing_sigmetrics,
author = {Varma, Sushil Mahavir and Bumpensanti, Pornpawee and Maguluri, Siva Theja and Wang, He},
title = {Dynamic Pricing and Matching for Two-Sided Queues},
year = {2020},
isbn = {9781450379854},
publisher = {Association for Computing Machinery},
doi = {10.1145/3393691.3394183},
booktitle = {Abstracts of the 2020 SIGMETRICS/Performance Joint International Conference on Measurement and Modeling of Computer Systems},
pages = {105–106},
series = {SIGMETRICS ’20}
}

@article{sushil_selfish_arxiv,
      title={Near Optimal Control in Ride Hailing Platforms with Strategic Servers}, 
      author={Sushil Mahavir Varma and Francisco Castro and Siva Theja Maguluri},
      year={2021},
      journal={Preprint arXiv:2008.03762}
}

@article{adan2018reversibility,
  title={Reversibility and further properties of FCFS infinite bipartite matching},
  author={Adan, Ivo and Bu{\v{s}}i{\'c}, Ana and Mairesse, Jean and Weiss, Gideon},
  journal={Mathematics of Operations Research},
  volume={43},
  number={2},
  pages={598--621},
  year={2018},
  publisher={INFORMS}
}

@article{cadas2020flexibility,
Author = {Arnaud Cadas and Josu Doncel and Jean-Michel Fourneau and Ana Bušić},
Title = {Flexibility can hurt dynamic matching system performance},
Year = {2020},
journal={Preprint arXiv:2009.10009},
}

@article{ozkan2020joint,
  title={Joint pricing and matching in ride-sharing systems},
  author={{\"O}zkan, Erhun},
  journal={European Journal of Operational Research},
  year={2020},
  publisher={Elsevier}
}

@article{besbes2021surge,
  title={Surge pricing and its spatial supply response},
  author={Besbes, Omar and Castro, Francisco and Lobel, Ilan},
  journal={Management Science},
  volume={67},
  number={3},
  pages={1350--1367},
  year={2021},
  publisher={INFORMS}
}

@article{hosseini2021dynamic,
  title={Dynamic Relocations in Car-Sharing Networks},
  author={Hosseini, Mahsa and Milner, Joseph and Romero, Gonzalo},
  journal={Available at SSRN},
  year={2021}
}

@article{bimpikis2019spatial,
  title={Spatial pricing in ride-sharing networks},
  author={Bimpikis, Kostas and Candogan, Ozan and Saban, Daniela},
  journal={Operations Research},
  volume={67},
  number={3},
  pages={744--769},
  year={2019},
  publisher={INFORMS}
}

@inproceedings{castillo2017surge,
  title={Surge pricing solves the wild goose chase},
  author={Castillo, Juan Camilo and Knoepfle, Dan and Weyl, Glen},
  booktitle={Proceedings of the 2017 ACM Conference on Economics and Computation},
  pages={241--242},
  year={2017}
}

@article{guda2019your,
  title={Your uber is arriving: Managing on-demand workers through surge pricing, forecast communication, and worker incentives},
  author={Guda, Harish and Subramanian, Upender},
  journal={Management Science},
  volume={65},
  number={5},
  pages={1995--2014},
  year={2019},
  publisher={INFORMS}
}

@article{cachon2017role,
  title={The role of surge pricing on a service platform with self-scheduling capacity},
  author={Cachon, Gerard P and Daniels, Kaitlin M and Lobel, Ruben},
  journal={Manufacturing \& Service Operations Management},
  volume={19},
  number={3},
  pages={368--384},
  year={2017},
  publisher={INFORMS}
}

@article{braverman2019empty,
  title={Empty-car routing in ridesharing systems},
  author={Braverman, Anton and Dai, Jim G and Liu, Xin and Ying, Lei},
  journal={Operations Research},
  volume={67},
  number={5},
  pages={1437--1452},
  year={2019},
  publisher={INFORMS}
}

@article{erlang1909theory,
  title={The theory of probabilities and telephone conversations},
  author={Erlang, Agner Krarup},
  journal={Nyt. Tidsskr. Mat. Ser. B},
  volume={20},
  pages={33--39},
  year={1909}
}

@inproceedings{kanoria2019backpressure, author = {Kanoria, Yash and Qian, Pengyu}, title = {Blind Dynamic Resource Allocation in Closed Networks via Mirror Backpressure}, year = {2020}, publisher = {Association for Computing Machinery}, address = {New York, NY, USA}, booktitle = {Proceedings of the 21st ACM Conference on Economics and Computation}, pages = {503}, numpages = {1}}

@article{varma_twosided_or,
  title={Dynamic pricing and matching for two-sided queues},
  author={Varma, Sushil Mahavir and Bumpensanti, Pornpawee and Maguluri, Siva Theja and Wang, He},
  journal={Operations Research},
  year={2022},
  publisher={INFORMS}
}

@inproceedings{SIGMETRICS_strategic_servers,
  title={Dynamic Pricing and Matching for Two-Sided Markets with Strategic Servers},
  author={Varma, Sushil Mahavir and Castro, Francisco and Maguluri, Siva Theja},
  booktitle={Abstract Proceedings of the 2021 ACM SIGMETRICS/International Conference on Measurement and Modeling of Computer Systems},
  pages={61--62},
  year={2021}
}

@article{ward2005diffusion,
  title={A diffusion approximation for a {GI/GI/1} queue with balking or reneging},
  author={Ward, Amy R and Glynn, Peter W},
  journal={Queueing Systems},
  volume={50},
  number={4},
  pages={371--400},
  year={2005},
  publisher={Springer}
}

@book{levy_book,
  title={Probability with martingales},
  author={Williams, David},
  year={1991},
  publisher={Cambridge university press},
  Address={Cambridge}
}

@book{billingsley2013convergence,
  title={Convergence of probability measures},
  author={Billingsley, Patrick},
  year={2013},
  publisher={John Wiley \& Sons}
}

@article{rudin1991functional,
  title={Functional analysis, mcgrawhill},
  author={Rudin, Walter},
  journal={Inc, New York},
  volume={45},
  pages={46},
  year={1991}
}

@book{hunter2001applied,
  title={Applied analysis},
  author={Hunter, John K and Nachtergaele, Bruno},
  year={2001},
  publisher={World Scientific Publishing Company},
  Address={Singapore}
}

@article{gates1956linear,
  title={Linear differential equations in distributions},
  author={Gates, Leslie D},
  journal={Proceedings of the American Mathematical Society},
  volume={7},
  number={5},
  pages={933--939},
  year={1956},
  publisher={JSTOR}
}

@phdthesis{gates1952differential,
  title={Differential equations in the distributions of Schwartz},
  school={Iowa State University},
  author={Gates Jr, Leslie D},
  year={1952},
  publisher={Digital Repository@ Iowa State University, http://lib. dr. iastate. edu/}
}

@article{baccelli1984single,
  title={Single-server queues with impatient customers},
  author={Baccelli, F and Boyer, P and Hebuterne, G},
  journal={Advances in Applied Probability},
  volume={16},
  number={4},
  pages={887--905},
  year={1984},
  publisher={Cambridge University Press}
}

@article{lee2020stationary,
  title={Stationary distribution convergence of the offered waiting processes for {GI/GI/1+ GI} queues in heavy traffic},
  author={Lee, Chihoon and Ward, Amy R and Ye, Heng-Qing},
  journal={Queueing Systems},
  volume={94},
  number={1},
  pages={147--173},
  year={2020},
  publisher={Springer}
}

@article{he2013one,
  title={A one-dimensional diffusion model for overloaded queues with customer abandonment},
  author={He, Shuangchi},
  journal={arXiv preprint arXiv:1312.4244},
  year={2013}
}

@article{huang2018beyond,
  title={Beyond heavy-traffic regimes: Universal bounds and controls for the single-server queue},
  author={Huang, Junfei and Gurvich, Itai},
  journal={Operations Research},
  volume={66},
  number={4},
  pages={1168--1188},
  year={2018},
  publisher={INFORMS}
}

@article{castro2020matching,
  title={Matching queues with reneging: a product form solution},
  author={Castro, Francisco and Nazerzadeh, Hamid and Yan, Chiwei},
  journal={Queueing Systems},
  volume={96},
  number={3},
  pages={359--385},
  year={2020},
  publisher={Springer}
}

@article{weiss2020directed,
  title={Directed FCFS infinite bipartite matching},
  author={Weiss, Gideon},
  journal={Queueing Systems},
  volume={96},
  number={3},
  pages={387--418},
  year={2020},
  publisher={Springer}
}

@book{halperin1952introduction,
  title={Introduction to the Theory of Distributions},
  author={Halperin, Israel and Schwartz, Laurent},
  year={1952},
  publisher={University of Toronto Press}
}
\newpage
\renewcommand{\theHsection}{A\arabic{section}}
\begin{APPENDICES}
\section{Proof of Proposition \ref{prop: bernoulli}} \label{app: bernoulli}
\proof{Proof}
In this case, the underlying DTMC governing the imbalance is just a birth-and-death process. Thus, we can evaluate its stationary distribution exactly. In particular, the imbalance increase by 1 if a customer arrives and no server arrives. That occurs with probability $\lambda_\eta(z)(1-\mu_\eta(z))$ when the imbalance is $z \in \bbZ$. Similarly, the imbalance decrease by 1 if a server arrives and no customer arrives. That occurs with probability $(1-\lambda_\eta(z))\mu_\eta(z)$ when the imbalance is $z \in \bbZ$. If both customer and server arrive or if no one arrives, the imbalance will stay the same. Define $m\overset{\Delta}{=}\mu^\star(1-\lambda^\star)=\lambda^\star(1-\mu^\star)$ and the transition probability matrix is given by
\begin{align*}
    P_{i,j}= \begin{cases} 
    m &\textit{if } j=i \pm 1, |i| \leq \lfloor \tau_\eta \rfloor \\
    m+\epsilon_\eta \mu^\star &\textit{if } j=i-\textrm{sgn}(i), |i|>\lfloor \tau_\eta \rfloor \\
    m-\epsilon_\eta(1-\mu^\star) &\textit{if } j=i+\textrm{sgn}(i), |i|>\lfloor \tau_\eta \rfloor \\
    1-2m &\textit{if } j=i, |i| \leq \lfloor \tau_\eta \rfloor \\
    1-2m+\epsilon_\eta(1-2\mu^\star) &\textit{if } j=i, |i|> \lfloor \tau_\eta \rfloor \\
    0 &\textit{otherwise}.
    \end{cases}
\end{align*}
As this is a birth and death process, it is a reversible DTMC. Thus, by solving the detailed balance equations, we get
\begin{align*}
    \pi_i=\begin{cases}\frac{1}{2\lfloor\tau_\eta\rfloor+1+2m/\epsilon_\eta} &\textit{if } |i| \leq \lfloor\tau_\eta\rfloor \\
    \frac{1}{2\lfloor\tau_\eta\rfloor+1+2m/\epsilon_\eta}\frac{m}{m+\epsilon_\eta\mu^\star}\left(1-\frac{\epsilon_\eta}{m+\epsilon_\eta\mu^\star}\right)^{|i|-\lfloor\tau_\eta\rfloor-1} &\textit{otherwise}.
    \end{cases} \numberthis \label{eq: stead_state_bernoulli}
\end{align*}
\begin{figure}[b]
\FIGURE{
      \begin{tikzpicture}[->, >=stealth', auto, semithick, node distance=2.5cm,scale=0.65,transform shape]
\tikzstyle{every state}=[fill=white,draw=black,thick,text=black,scale=1]
\node[state]    (0)                     {$0$};
\node[state]    (1)[right of=0]   {$1$};
\node[state]    (nminusone)[right of=1]   {{$\tau_\eta$}};
\node[state]    (n)[ right of=nminusone]   {$\tau_\eta+1$};
\node[state,white]    (nplusone)[ right of=n] {};
\node [state](-1)[left of=0] {$-1$};
\node[state](-n+1)[left of=-1]{$-\tau_\eta$};
\node[state](-n)[left of =-n+1]{\tiny{$-\tau_\eta-1$}};
\node[state,white](-n-1)[left of=-n]{};
\path
(0) edge[bend right,below]     node{$m$}         (1)
(1) edge[bend right,below,dashed] node{$m$} (nminusone)
(nminusone) edge[bend right,below]     node{$m$}         (n)
(n) edge[bend right,below]     node{$m-\epsilon_\eta(1-\mu^\star)$}         (nplusone)
(nplusone) edge[bend right,above]     node{$m+\epsilon_\eta\mu^\star$}         (n)
(n) edge[bend right,above]     node{$m+\epsilon_\eta\mu^\star$}         (nminusone)
(nminusone) edge[bend right,above,dashed]     node{$m$}         (1)
(1) edge[bend right,above]     node{$m$}         (0)
(0) edge[bend right,above] node{$m$} (-1)
(-1) edge[bend right,above,dashed] node{$m$} (-n+1)
(-n+1) edge[bend right,above] node{$m$} (-n)
(-n) edge[bend right,above] node{$m-\epsilon_\eta(1-\mu^\star)$} (-n-1)
(-n-1) edge[bend right,below] node{$m+\epsilon_\eta\mu^\star$} (-n)
(-n) edge[bend right,below] node{$m+\epsilon_\eta\mu^\star$} (-n+1)
(-n+1) edge[bend right,below,dashed] node{$m$} (-1)
(-1) edge[bend right,below] node{$m$} (0)
(-n) edge[loop, above] node{$1-2m+\epsilon_\eta(1-2\mu^\star)$} (-n)
(-n+1) edge[loop, above] node{$1-2m$} (-n+1)
(-1) edge[loop, above] node{$1-2m$} (-1)
(0) edge[loop, above] node{$1-2m$} (0)
(1) edge[loop, above] node{$1-2m$} (1)
(nminusone) edge[loop, above] node{$1-2m$} (nminusone)
(n) edge[loop, above] node{$1-2m+\epsilon_\eta(1-2\mu^\star)$} (n);
\end{tikzpicture}}{Single link matching queue with Bernoulli arrivals
\label{fig: intuition_singlelinkbirthanddeath}}{}
\end{figure}
Now, we calculate the moment generating function of $\epsilon_\eta \barz_\eta$ for Case 1 and 2 and $\barz_\eta/\tau_\eta$ for Case 3. We fix an arbitrary $t$ such that $t<1/m$. Then, there exists $\epsilon_0(t) > 0$ such that for all $\epsilon_\eta < \epsilon_0$, we have $(1-\epsilon_\eta/(m+\epsilon_\eta \mu^\star))e^{\epsilon_\eta t}<1$. We have
\begin{align*}
    \E{e^{\epsilon_\eta t \barz_\eta}}={}&\sum_{i=-\infty}^\infty \pi_i e^{\epsilon_\eta t i} \\
    ={}&\sum_{i=-\infty}^{-\lfloor\tau_\eta\rfloor-1} \pi_i e^{\epsilon_\eta t i} + \sum_{i=-\lfloor\tau_\eta\rfloor}^{\lfloor\tau_\eta\rfloor} \pi_i e^{\epsilon_\eta t i} + \sum_{i=\lfloor\tau_\eta\rfloor+1}^\infty \pi_i e^{\epsilon_\eta t i} \\
    ={}&\frac{1}{2\lfloor\tau_\eta\rfloor+1+2m/\epsilon_\eta}\left(\frac{m}{m+\epsilon_\eta \mu^\star}\sum_{i=-\infty}^{-\lfloor\tau_\eta\rfloor-1}\left(1-\frac{\epsilon_\eta}{m+\epsilon_\eta \mu^\star}\right)^{-i-\lfloor\tau_\eta\rfloor-1}e^{\epsilon_\eta t i}+\sum_{i=-\lfloor\tau_\eta\rfloor}^{\lfloor\tau_\eta\rfloor} e^{\epsilon_\eta t i}\right.\\
    &\left.+\frac{m}{m+\epsilon_\eta \mu^\star}\sum_{i=\lfloor\tau_\eta\rfloor+1}^\infty \left(1-\frac{\epsilon_\eta}{m+\epsilon_\eta \mu^\star}\right)^{i-\lfloor\tau_\eta\rfloor-1}e^{\epsilon_\eta t i} \right) \\
    ={}&\frac{\epsilon_\eta}{2\lfloor\tau_\eta\rfloor\epsilon_\eta+\epsilon_\eta+2m}\left(\frac{m}{m+\epsilon_\eta\mu^\star}\frac{e^{-\epsilon_\eta t (\lfloor\tau_\eta\rfloor+1)}}{1-e^{-\epsilon_\eta t}+\frac{\epsilon_\eta e^{-\epsilon_\eta t}}{m+\epsilon_\eta \mu^\star}}+\frac{e^{-\epsilon_\eta t\lfloor\tau_\eta\rfloor}-e^{\epsilon_\eta t(\lfloor\tau_\eta\rfloor+1)}}{1-e^{\epsilon_\eta t}}\right.\\
    &\left. +\frac{m}{m+\epsilon_\eta\mu^\star}\frac{e^{\epsilon_\eta t (\lfloor\tau_\eta\rfloor+1)}}{1-e^{\epsilon_\eta t}+\frac{\epsilon_\eta e^{\epsilon_\eta t}}{m+\epsilon_\eta \mu^\star}}\right) \numberthis \label{eq: mgf_bernoulli}
\end{align*}
Now, we take the limit as $\eta \uparrow \infty$ for the above equation for Case 1 and 2 separately.

\textbf{Case 1:} In this case, we have $\epsilon_\eta \tau_\eta \rightarrow 0$ and $\epsilon_\eta \rightarrow 0$. In addition, note that $|\epsilon_\eta \lfloor\tau_\eta\rfloor -\epsilon_\eta \tau_\eta| \leq \epsilon_\eta$, which gives us $\epsilon_\eta \lfloor\tau_\eta\rfloor \rightarrow 0$. Now, we take the limit for each of the three terms in \eqref{eq: mgf_bernoulli} separately. In particular, we resort to Taylor series expansion of the exponential terms and then take the limit as $\eta \uparrow \infty$.
\begin{subequations}
\label{eq: case1_bernoulli}
\begin{align}
    \epsilon_\eta\frac{e^{-\epsilon_\eta t\lfloor\tau_\eta\rfloor}-e^{\epsilon_\eta t(\lfloor\tau_\eta\rfloor+1)}}{1-e^{\epsilon_\eta t}}&=\epsilon_\eta\frac{
    -\epsilon_\eta t \lfloor\tau_\eta\rfloor -\epsilon_\eta t(\lfloor\tau_\eta\rfloor+1)  + o(\epsilon_\eta \tau_\eta)}{-\epsilon_\eta t+o(\epsilon_\eta)}\nonumber\\
    &= \frac{-\epsilon_\eta t \lfloor\tau_\eta\rfloor -\epsilon_\eta t(\lfloor\tau_\eta\rfloor+1)  + o(\epsilon_\eta \tau_\eta)}{-t+o(1)} \rightarrow 0  \\
    \epsilon_\eta\frac{e^{\epsilon_\eta t (\lfloor\tau_\eta\rfloor+1)}}{1-e^{\epsilon_\eta t}+\frac{\epsilon_\eta e^{\epsilon_\eta t}}{m+\epsilon_\eta \mu^\star}}&=\epsilon_\eta \frac{e^{\epsilon_\eta t (\lfloor\tau_\eta\rfloor+1)}}{-\epsilon_\eta t +\frac{\epsilon_\eta}{m+\epsilon_\eta \mu^\star}+o(\epsilon_\eta)}= \frac{e^{\epsilon_\eta t (\lfloor\tau_\eta\rfloor+1)}}{- t +\frac{1}{m+\epsilon_\eta \mu^\star}+o(1)} \rightarrow \frac{1}{1/m-t} \\
    \epsilon_\eta\frac{e^{-\epsilon_\eta t (\lfloor\tau_\eta\rfloor+1)}}{1-e^{-\epsilon_\eta t}+\frac{\epsilon_\eta e^{-\epsilon_\eta t}}{m+\epsilon_\eta \mu^\star}}&=\epsilon_\eta \frac{e^{-\epsilon_\eta t (\lfloor\tau_\eta\rfloor+1)}}{\epsilon_\eta t +\frac{\epsilon_\eta}{m+\epsilon_\eta \mu^\star}+o(\epsilon_\eta)}= \frac{e^{-\epsilon_\eta t (\lfloor\tau_\eta\rfloor+1)}}{t +\frac{1}{m+\epsilon_\eta \mu^\star}+o(1)} \rightarrow \frac{1}{1/m+t}.
\end{align}
\end{subequations}
Now, by using \eqref{eq: case1_bernoulli} to simplify \eqref{eq: mgf_bernoulli} we get
\begin{align*}
    \E{e^{\epsilon_\eta t \barz_\eta}} \rightarrow \frac{1}{2m}\left(\frac{1}{1/m+t}+0+\frac{1}{1/m-t}\right)=\frac{1}{1-m^2t^2} \quad \textit{as } \eta \uparrow \infty.
\end{align*}
The above is the MGF of a Laplace distribution with parameters $(0,\pfix{1/m})$. Thus, by Levy's continuity theorem (\citet[Chapter 18]{levy_book}), the proof of Case 1 is complete.

\textbf{Case 2:} In this case, we have $\epsilon_\eta \tau_\eta \rightarrow l \in (0,\infty)$, $\epsilon_\eta \rightarrow 0$ and $\tau_\eta \rightarrow \infty$. In addition, note that $|\epsilon_\eta \lfloor\tau_\eta\rfloor-\epsilon_\eta \tau_\eta| \leq \epsilon_\eta$ which gives us $\epsilon_\eta \lfloor\tau_\eta\rfloor \rightarrow l$. Now, we evaluate the limit of \eqref{eq: mgf_bernoulli} by using Taylor series expansion.
\begin{subequations}
\label{eq: case2_bernoulli}
\begin{align}
    \epsilon_\eta\frac{e^{-\epsilon_\eta t\lfloor\tau_\eta\rfloor}-e^{\epsilon_\eta t(\lfloor\tau_\eta\rfloor+1)}}{1-e^{\epsilon_\eta t}}&= \epsilon_\eta\frac{e^{-\epsilon_\eta t\lfloor\tau_\eta\rfloor}-e^{\epsilon_\eta t(\lfloor\tau_\eta\rfloor+1)}}{-\epsilon_\eta t+o(\epsilon_\eta)}=\frac{e^{-\epsilon_\eta t\lfloor\tau_\eta\rfloor}-e^{\epsilon_\eta t(\lfloor\tau_\eta\rfloor+1)}}{- t+o(1)} \rightarrow \frac{e^{l t}-e^{-l t}}{t} \\
    \epsilon_\eta\frac{e^{\epsilon_\eta t (\lfloor\tau_\eta\rfloor+1)}}{1-e^{\epsilon_\eta t}+\frac{\epsilon_\eta e^{\epsilon_\eta t}}{m+\epsilon_\eta \mu^\star}}&=\epsilon_\eta \frac{e^{\epsilon_\eta t (\lfloor\tau_\eta\rfloor+1)}}{-\epsilon_\eta t +\frac{\epsilon_\eta}{m+\epsilon_\eta \mu^\star}+o(\epsilon_\eta)}= \frac{e^{\epsilon_\eta t (\lfloor\tau_\eta\rfloor+1)}}{- t +\frac{1}{m+\epsilon_\eta \mu^\star}+o(1)} \rightarrow \frac{e^{l t}}{1/m-t}  \\
    \epsilon_\eta\frac{e^{-\epsilon_\eta t (\lfloor\tau_\eta\rfloor+1)}}{1-e^{-\epsilon_\eta t}+\frac{\epsilon_\eta e^{-\epsilon_\eta t}}{m+\epsilon_\eta \mu^\star}}&=\epsilon_\eta \frac{e^{-\epsilon_\eta t (\lfloor\tau_\eta\rfloor+1)}}{\epsilon_\eta t +\frac{\epsilon_\eta}{m+\epsilon_\eta \mu^\star}+o(\epsilon_\eta)}= \frac{e^{-\epsilon_\eta t (\lfloor\tau_\eta\rfloor+1)}}{t +\frac{1}{m+\epsilon_\eta \mu^\star}+o(1)} \rightarrow \frac{e^{-lt}}{1/m+t} 
\end{align}
\end{subequations}
Now, by using \eqref{eq: case2_bernoulli} to simplify \eqref{eq: mgf_bernoulli} we get
\begin{align}
    \E{e^{\epsilon_\eta t \barz_\eta}} \rightarrow \frac{1}{2l+2m}\left(\frac{e^{-lt}}{1/m+t}+\frac{e^{l t}-e^{-lt}}{t}+\frac{e^{lt}}{1/m-t}\right). \label{eq: hybrind_with_epsilon}
\end{align}
The above is the MGF of the hybrid distribution with parameters $(\pfix{1/m},l)$. We verify it for $X \sim \textrm{Hybrid}(b,c)$ as follows:
\begin{align*}
    \E{e^{tX}}&=\pfix{\int_{-\infty}^\infty \rho_{\textrm{Hybrid}}(x) e^{t x} dx = \frac{b}{2(1+bc)}\left(\int_{-\infty}^{-c} e^{b(x+c)}e^{tx} dx+\int_{-c}^c e^{tx} dx +\int_c^\infty e^{-b(x-c)}e^{tx} dx \right)} \\
    &=\pfix{\frac{b}{2(1+bc)}\left(e^{bc}\frac{e^{x(b+t)}}{b+t}\bigg|_{-\infty}^{-c}+\frac{e^{tx}}{t}\bigg|_{-c}^c+e^{bc}\frac{e^{x(t-b)}}{t-b}\bigg|_c^\infty\right)} \\
    &=\pfix{\frac{b}{2(1+bc)}\left(\frac{e^{-ct}}{b+t}+\frac{e^{ct}-e^{-ct}}{t}+\frac{e^{ct}}{b-t}\right)}.
\end{align*}
Note that, the limiting value of  $\E{e^{t\frac{\barz_\eta}{\tau_\eta}}}$ can be similarly calculated. We omit the details here as they are repetitive. Intuitively, by substituting $t \rightarrow t/l$ in \eqref{eq: hybrind_with_epsilon}, we get
\begin{align*}
     \E{e^{ t \frac{\barz_\eta}{\tau_\eta}}} \rightarrow \frac{1}{2+2m/l}\left(\frac{e^{\pfix{-t}}}{l/m+t}+\frac{e^{\pfix{t}}-e^{\pfix{-t}}}{t}+\frac{e^{\pfix{t}}}{l/m-t}\right)
\end{align*}
The above is the MGF of the Hybrid distribution with parameters $(\pfix{l/m},1)$. Thus, by Levy's continuity theorem (\citet[Chapter 18]{levy_book}), the proof of Case 2 is complete.

\textbf{Case 3:} In this case, we consider the MGF of $\barz_\eta/\tau_\eta$ and carry out similar calculations as in  \eqref{eq: mgf_bernoulli}. 
Fix an arbitrary $t \in \bbR$ and let $\eta>\eta_0$ such that $\epsilon_\eta\tau_\eta>2t(m+\epsilon_\eta \mu^\star)$ which implies that $(1-\epsilon_\eta/(m+\epsilon_\eta\mu^\star))e^{t/\tau_\eta} <1$. We skip some steps as they are repetitive and directly write the MGF below.
\begin{align*}
    \E{e^{t\frac{\barz_\eta}{\tau_\eta}}}={}&\frac{\epsilon_\eta}{2\lfloor\tau_\eta\rfloor\epsilon_\eta+\epsilon_\eta+2m}\left(\frac{m}{m+\epsilon_\eta\mu^\star}\frac{e^{- \frac{t}{\tau_\eta} (\lfloor\tau_\eta\rfloor+1)}}{1-e^{- t/\tau_\eta}+\frac{\epsilon_\eta e^{-t/\tau_\eta}}{m+\epsilon_\eta \mu^\star}}+\frac{e^{- t\lfloor\tau_\eta\rfloor/\tau_\eta}-e^{ t(\lfloor\tau_\eta\rfloor+1)/\tau_\eta}}{1-e^{t/\tau_\eta}}\right.\\
    &\left. +\frac{m}{m+\epsilon_\eta\mu^\star}\frac{e^{t (\lfloor\tau_\eta\rfloor+1)/\tau_\eta}}{1-e^{t/\tau_\eta}+\frac{\epsilon_\eta e^{t/\tau_\eta}}{m+\epsilon_\eta \mu^\star}}\right) \numberthis \label{eq: case3_mgf_bernoulli}
\end{align*}
In this case, we have $\epsilon_\eta \tau_\eta \rightarrow \infty$ and $\tau_\eta \rightarrow \infty$. In addition, note that $|\epsilon_\eta \lfloor\tau_\eta\rfloor-\epsilon_\eta \tau_\eta| \leq \epsilon_\eta$ which gives us $\epsilon_\eta \lfloor\tau_\eta\rfloor \rightarrow \infty$. In addition, as $|\lfloor \tau_\eta \rfloor/\tau_\eta-1| \leq 1/\tau_\eta$, we have $\lfloor \tau_\eta \rfloor/\tau_\eta \rightarrow 1$. Now, we evaluate the limit of \eqref{eq: case3_mgf_bernoulli} by using Taylor series expansion.
\begin{subequations}
\begin{align}
    \frac{\epsilon_\eta}{2\lfloor \tau_\eta \rfloor \epsilon_\eta+\epsilon_\eta+2m}\frac{e^{ \frac{t}{\tau_\eta} (\lfloor \tau_\eta \rfloor+1)}}{1-e^{ t/\tau_\eta}+\frac{\epsilon_\eta e^{ t/\tau_\eta}}{m+\epsilon_\eta \mu^\star}}&= \frac{\epsilon_\eta}{2\lfloor \tau_\eta \rfloor \epsilon_\eta+\epsilon_\eta+2m}\frac{e^{ \frac{t}{\tau_\eta} (\lfloor \tau_\eta \rfloor+1)}}{ -t/\tau_\eta+\frac{\epsilon_\eta }{m+\epsilon_\eta \mu^\star}+o(1/\tau_\eta)}\nonumber\\
    &=\frac{ e^{ \frac{t}{\tau_\eta} (\lfloor \tau_\eta \rfloor+1)} }{-2t \lfloor \tau_\eta \rfloor/\tau_\eta+\frac{2 \lfloor \tau_\eta \rfloor \epsilon_\eta+2m+\epsilon_\eta}{m+\epsilon_\eta \mu^\star}+o(1)} \rightarrow 0 \\
    \frac{\epsilon_\eta}{2\lfloor \tau_\eta \rfloor \epsilon_\eta+\epsilon_\eta+2m}\frac{e^{ -\frac{t}{\tau_\eta} (\lfloor \tau_\eta \rfloor+1)}}{1-e^{ -t/\tau_\eta}+\frac{\epsilon_\eta e^{ -t/\tau_\eta}}{m+\epsilon_\eta \mu^\star}}&= \frac{\epsilon_\eta}{2\lfloor \tau_\eta \rfloor \epsilon_\eta+\epsilon_\eta+2m}\frac{e^{ -\frac{t}{\tau_\eta} (\lfloor \tau_\eta \rfloor+1)}}{ t/\tau_\eta+\frac{\epsilon_\eta }{m+\epsilon_\eta \mu^\star}+o(1/\tau_\eta)}\nonumber\\
    &=\frac{ e^{ -\frac{t}{\tau_\eta} (\lfloor \tau_\eta \rfloor+1)} }{2t\lfloor \tau_\eta \rfloor/\tau_\eta+\frac{2 \lfloor \tau_\eta \rfloor \epsilon_\eta+2m+\epsilon_\eta}{m+\epsilon_\eta \mu^\star}+o(1)} \rightarrow 0  \\
    \frac{\epsilon_\eta}{2\lfloor \tau_\eta \rfloor \epsilon_\eta+\epsilon_\eta+2m} \frac{e^{-t\lfloor \tau_\eta \rfloor/\tau_\eta}-e^{t(\lfloor \tau_\eta \rfloor+1)/\tau_\eta}}{1-e^{t/\tau_\eta}}&=\frac{\epsilon_\eta}{2\lfloor \tau_\eta \rfloor\epsilon_\eta+\epsilon_\eta+2m} \frac{e^{-t\lfloor \tau_\eta \rfloor/\tau_\eta}-e^{t(\lfloor \tau_\eta \rfloor+1)/\tau_\eta}}{-t/\tau_\eta+o(1/\tau_\eta)} \nonumber\\
    &=\frac{e^{-t\lfloor \tau_\eta \rfloor/\tau_\eta}-e^{t(\lfloor \tau_\eta \rfloor+1)/\tau_\eta}}{-2t\lfloor \tau_\eta \rfloor/\tau_\eta+o(1)} \rightarrow \frac{e^t-e^{-t}}{2t}
\end{align}
\label{eq: case3_bernoulli}
\end{subequations}
Now, by using \eqref{eq: case3_bernoulli} in \eqref{eq: case3_mgf_bernoulli}, we get
\begin{align*}
    \E{e^{t\frac{\barz_\eta}{\tau_\eta}}} \rightarrow \frac{e^t-e^{-t}}{2t} \text{ and so } \frac{\barz_\eta}{\tau_\eta} \overset{D}{\rightarrow} \calU([-1,1]),
\end{align*}
where the last assertion follows Levy's continuity theorem (\citet[Chapter 18]{levy_book}). This completes the proof of Case 3. \hfill $\square$
\section{Positive Recurrence} \label{app: pos_rec}
\proof{Proof of Proposition \ref{prop: pos_rec}}
We will analyze the drift of the test function $z^2$ and show that it is negative outside a finite set. 
\revcolor{\begin{align*}
    \lefteqn{\E{z^2_\eta(k+1)-z^2_\eta(k) | z_\eta(k)=z}}\\
    &\overset{(a)}{=}\E{(z+a_{\eta}^c(z)-a_{\eta}^s(z))^2-z^2} \\
    &= \E{(a_{\eta}^c(z)-a_{\eta}^s(z))^2+2z(a_{\eta}^c(z)-a_{\eta}^s(z))} \\
    &=\E{(a_{\eta}^c(z)-a_{\eta}^s(z))^2}+2z\E{a_{\eta}^c(z)-a_{\eta}^s(z)} \\
    &\overset{(b)}{\leq} 4\moment+2z\E{a_{\eta}^c(z)-a_{\eta}^s(z)} \\
    &= 4\moment+2z\left(\lambda_\eta(z)-\mu_\eta(z)\right) \numberthis \label{eq: bound_with_phi} \\
    &= 4\moment+2z\left(\lambda_\eta(z)-\mu_\eta(z)\right)\left(\mathbbm{1}\{|z| \leq K\tau_\eta\}+\mathbbm{1}\{|z|>K \tau_\eta\}\right) \\
    &\overset{(c)}{\leq} 4\moment+\epsilon_\eta \tau_\eta K\left(\delta + 4 K^\beta\right)-\epsilon_\eta |z| \delta \mathbbm{1}\{|z|>K\tau_\eta\} \numberthis \label{eq: bound_without_phi} \\
    &\leq -4\moment-\epsilon_\eta \tau_\eta K\left(\delta + 4 K^\beta\right) \quad \forall |z| > \max\left\{K\tau_\eta,\frac{8\moment+2\epsilon_\eta \tau_\eta K\left(\delta + 4 K^\beta\right)}{\epsilon_\eta \delta}\right\},
\end{align*}
where $(a)$ follows by the update equation of imbalance given by \eqref{eq: imabalance_evolution}. Next,  $(b)$ follows as 
\begin{align*}
    \E{(a^c_\eta(z)-a^s_\eta(z))^2} \leq 2\E{(a^c_\eta(z))^2+(a^s_\eta(z))^2} \leq 4\moment
\end{align*}
for all $z \in \bbZ$ and $k \in \bbZ_+$ as second moment is upper bounded by the third moment for a non-negative integer-valued random variable.
Finally, $(c)$ follows for $\eta>0$ large enough by Condition~\ref{ass:poly_growth} and Condition~\ref{ass: neg_drift}. In particular, $|\phi^c(x)| \leq K^\beta$ and $|\phi^s(x)| \leq K^\beta$ for all $x \in [-K, K]$. Thus, by \eqref{eq: general_pricing_policy}, we have 
\begin{align}
    |\lambda_\eta(z) - \mu_\eta(z)| \leq |\lambda_\eta^\star - \mu^\star| + \epsilon_\eta \bigg|\phi^c\left(\frac{z}{\tau_\eta}\right)\bigg|+\epsilon_\eta\bigg|\phi^s\left(\frac{z}{\tau_\eta}\right)\bigg| \leq \frac{\delta \epsilon_\eta}{2} + 2\epsilon_\eta K^\beta \quad \forall |z| \leq K\tau_\eta. \label{eq:diff_arr_rate_small_z}
\end{align}
Moreover, by Condition~\ref{ass: neg_drift}, we have 
\begin{align*}
    \lambda_\eta(z)-\mu_\eta(z) &\overset{*}{=} \max\left\{\lambda^\star_\eta + \epsilon_\eta\phi^c\left(\frac{z}{\tau_\eta}\right), \lambda_{\min}\right\} - \min\left\{\mu^\star + \epsilon_\eta\phi^s\left(\frac{z}{\tau_\eta}\right), \mu_{\max}\right\} \\
    &\overset{**}{\leq} \lambda_\eta^\star -\mu^\star+\epsilon_\eta\left(\phi^c\left(\frac{z}{\tau_\eta}\right) - \phi^s\left(\frac{z}{\tau_\eta}\right)\right) \overset{***}{\leq} -\frac{\delta \epsilon_\eta}{2} \quad \forall z \geq K \tau_\eta,
\end{align*}
where $(*)$ holds for all $\eta>\eta_{p1}$ for some $\eta_{p1}>0$. In particular, by Condition~\ref{ass: neg_drift}, $[\phi^c(x)]^+ \leq B$ for $x > K$. Thus, for $\eta>0$ large enough, we have $\lambda_\eta^\star + \epsilon_\eta \phi^c(z/\tau_\eta) \leq \lambda_{\max}$ for $z > K \tau_\eta$. Similarly, $\mu + \epsilon_\eta \phi^s(z/\tau_\eta) \geq \mu_{\min}$ for $z > K\tau_\eta$ as $[\phi^s(x)]^- \leq B$ for $x > K$. Next, $(**)$ holds for all $\eta > \eta_{p2}$ for some $\eta_{p2}>0$ by noting that $\lambda_{\min} < \mu^\star < \mu_{\max}$ and $|\lambda^\star_\eta - \mu^\star| \leq \delta \epsilon_\eta/2$. Lastly, $(**)$ holds by Condition~\ref{ass: neg_drift} as $\phi^c(x)-\phi^s(x) \leq -\delta$ for $x > K$. Similarly, $\lambda_\eta(z)-\mu_\eta(z) > \delta \epsilon_\eta/2$ for all $z < -K\tau_\eta$.}

\revcolor{So, we have negative drift outside a finite set for all $\eta > \eta_p \overset{\Delta}{=} \max\{\eta_{p1}, \eta_{p2}\}$. Thus, by the Foster-Lyapunov theorem, $\{z_\eta(k): k \in \bbZ_+\}$ is positive recurrent. By Moment bound theorem (\citet[Proposition 6.14]{hajekrandomprocbook}) and \eqref{eq: bound_without_phi}, we have
\begin{align*}
    \epsilon_\eta \delta \E{|\barz_\eta| \mathbbm{1}\{|\barz_\eta|>K\tau_\eta\}} &\leq 4\moment+\epsilon_\eta \tau_\eta K\left(\delta + 4 K^\beta\right) \\
    \Rightarrow \E{|\barz_\eta|} &\leq  \frac{4\moment+\epsilon_\eta \tau_\eta K\left(2\delta + 4 K^\beta\right)}{\epsilon_\eta \delta}. \numberthis \label{eq:first_moment_in_proof}
\end{align*}
Lastly, by Moment bound theorem (\citet[Proposition 6.14]{hajekrandomprocbook}) and \eqref{eq: bound_with_phi}, we get
\begin{align*}
    -\E{\barz_\eta\left(\lambda_\eta(\barz_\eta)-\mu_\eta(\barz_\eta)\right)} \leq 2\moment.
\end{align*}
Now, to prove \eqref{eq: imbalance_second_moment}, consider $V(z) = z^3$ as the Lyapunov function.
 \begin{align*}
    &\E{\barz_\eta(k+1)^3-\barz_\eta(k)^3 \big| \barz_\eta(k)=z} \\
    ={}& \E{(a^c_\eta(z)-a^s_\eta(z))^3} + 3\E{(a^c_\eta(z)-a^s_\eta(z))^2} z + 3z^2\E{(a^c_\eta(z)-a^s_\eta(z))} \\
    \overset{(a)}{\leq}{}& 6\moment + 12\moment |z| + 3z^2\left(\lambda_\eta(z)-\mu_\eta(z)\right) \\
    \leq{}& 6\moment + 12\moment |z| + 3z^2\left(\lambda_\eta(z)-\mu_\eta(z)\right)\mathbbm{1}\left\{|z| \leq K \tau_\eta\right\} + 3z^2\left(\lambda_\eta(z)-\mu_\eta(z)\right)\mathbbm{1}\left\{|z| > K \tau_\eta\right\} \\
    \overset{(b)}{\leq}{}& 6\moment + 12\moment |z| + 3K^2\tau_\eta^2\epsilon_\eta \left(\frac{\delta}{2} + 2K^\beta\right) - \frac{3\epsilon_\eta \delta}{2} z^2 \numberthis \label{eq:second_moment_drift} \\
    \leq{}& -\moment \quad \forall |z| \geq \sqrt{\frac{14\moment + 3K^2\tau_\eta^2\epsilon_\eta(\delta + 4K^\beta)}{3\epsilon_\eta \delta}} + \frac{8\moment}{\epsilon_\eta \delta},
\end{align*}
where $(a)$ follows as $\E{(a_\eta^c(z)-a_\eta^s(z))^3} \leq 3\E{|a_\eta^c(z)|^3 + |a_\eta^s(z)|^3} \leq 6\moment$. Next, $(b)$ holds by \eqref{eq:diff_arr_rate_small_z} similar to how \eqref{eq: bound_without_phi} follows. Now, by the moment bound theorem (\citet[Proposition 6.14]{hajekrandomprocbook}) and \eqref{eq:second_moment_drift}, we have
\begin{align*}
    \E{\barz_\eta^2} &\leq \frac{4\moment +8\moment \E{|\barz_\eta|} + K^2\tau_\eta^2\epsilon_\eta\left(\delta + 4K^\beta\right)}{\epsilon_\eta \delta} \\
    &\leq \frac{4\moment}{\epsilon_\eta \delta} + \frac{K^2\tau_\eta^2\left(\delta + 4K^\beta\right)}{\delta} + \frac{8\moment\left(4\moment+2\epsilon_\eta \tau_\eta K  (2K^\beta+\delta)\right)}{ \epsilon_\eta^2 \delta^2},
\end{align*}
where the last inequality holds by \eqref{eq:first_moment_in_proof}.} \hfill $\square$
\endproof
\section{Technical Details for the Proof of Theorem \ref{theo: case_2_itm}}
\subsection{Tightness} \label{app: tightness}
\proof{Proof of Lemma \ref{lemma: tightness}}
We show the tightness of the sequence of random variables $\epsilon_\eta \barz_\eta$. As $\epsilon_\eta \tau_\eta \rightarrow l$, there exists $\tau_{\max}>0$ such that $\epsilon_\eta \tau_\eta \leq \tau_{\max}$ for all $\eta>0$. Now, by Proposition \ref{prop: pos_rec}, we have
\begin{align*}
    \E{\epsilon_\eta|\barz_\eta|}&\leq \frac{4\moment+2\epsilon_\eta \tau_\eta K (2K^\beta+\delta)}{ \delta} \leq  \frac{4\moment+2\tau_{\max}K (2K^\beta+\delta)}{ \delta} 
\end{align*}
Now, for any $\omega>0$, we pick $N=\frac{4\moment+
2\tau_{\max}K (2K^\beta+\delta)}{\delta \omega}$ independent of $\eta$ to get
\begin{align*}
    \P{\epsilon_\eta |\barz_\eta| > N} \leq \frac{1}{N}\E{\epsilon_\eta |\barz_\eta|}\leq\omega \quad \forall \eta >0.
\end{align*}
Thus, the family of random variables $\Pi=\left\{\epsilon_\eta \barz_\eta \right\}$ is tight, which completes the proof. 
\endproof
\subsection{Drift Analysis} \label{app: drift_analysis}
Before proving Claim \ref{claim: case2_limiting_functional_equation}, we need the following technical lemma which we present with proof below.
\begin{lemma} \label{lemma: convergence_uoc}
For some $c \in \bbR$, let $g_n : \bbR \rightarrow \bbR$ be a Borel measurable sequence of functions such that $\max_{x \in \bbR}|g_n(x)| \leq g_{\max}$ for some constant $g_{\max} > 0$ and
\begin{align*}
    \lim_{n \uparrow \infty} g_n(x)=c \quad \forall x \in \bbR.
\end{align*}
\revcolor{Moreover, the convergence is uniform
on compact sets.} Then for any tight sequence of random variables $\{Y_n \in \bbR\}$, we have
\begin{align*}
    \lim_{n \uparrow \infty }\E{g_n(Y_n)}= c.
\end{align*}
\end{lemma}
Note that the above lemma is similar to Lemma \ref{lemma: convergence_technical} with key differences. Lemma \ref{lemma: convergence_technical} assumed uniform convergence of $g_n(\cdot)$ which is much stronger than uniform convergence over compact sets. This allows one to relax the tightness condition on the sequence of random variables $\{Y_n\}$ in Lemma \ref{lemma: convergence_technical}. Now, we present the proof of Lemma \ref{lemma: convergence_uoc}.
\proof{Proof of Lemma \ref{lemma: convergence_uoc}}
For any $\delta,K>0$, there exists $n_0$ such that for all $n>n_0$
\begin{align*}
    |g_n(x)-c| &\leq \delta \quad \forall |x| \leq K ,
\end{align*}
As the sequence of random variables $\{Y_n\}$ is tight, for any $\delta'>0$, there exists $K>0$ such that $\P{|Y_n| > K} \leq \delta'$ for all $n$. Pick such a $K$ for the rest of the proof. Thus for all $n \geq n_0$
\begin{align*}
    \P{|g_n(Y_n)-c|>\delta}\leq \P{|Y_n| > K} \leq \delta'
\end{align*}
Thus for all $n \geq n_0$, by using Jensen's inequality, we have
\revcolor{\begin{align*}
  \big| \E{g_n(Y_n)-c}\big| \leq \E{\big|g_n(Y_n)-c|} \leq \delta(1-\delta')+(g_{\max}+c)\delta'
\end{align*}}
Now, by taking $n \uparrow \infty$, we get
\revcolor{\begin{align*}
    \lim \sup_{n \uparrow \infty}\big|\E{g_n(Y_n)}-c\big| \leq \delta(1-\delta')+(g_{\max}+c)\delta'
\end{align*}}
As the above is true for all $\delta,\delta'>0$, we get
\begin{align*}
   \lim_{n \uparrow \infty}\E{g_n(Y_n)} = c. 
\end{align*}
This completes the proof of the Lemma. \hfill $\square$
\endproof
Now, we will use the above lemma to prove Claim \ref{claim: case2_limiting_functional_equation}.
\proof{Proof of Claim \ref{claim: case2_limiting_functional_equation}}
Divide both sides of the pre-limit equation \eqref{eq: prelimit_functional_equation} by $2/(\sigma^c(\lambda^\star)+\sigma^s(\mu^\star))$ to get
\begin{align*}
    &\frac{2}{\sigma^c(\lambda^\star)+\sigma^s(\mu^\star)}\E{e^{j\epsilon_\eta \omega \barz_\eta}\left(\phi^s\left(\frac{\barz_\eta}{\tau_\eta}\right)-\phi^c\left(\frac{\barz_\eta}{\tau_\eta}\right)\right)}+o(1)\\
    ={}&j\omega \E{e^{j\epsilon_\eta \omega \barz_\eta}}+\frac{j\omega}{\sigma^c(\lambda^\star)+\sigma^s(\mu^\star)} \E{e^{j\epsilon_\eta \omega \barz_\eta}\left(\sigma^c(\lambda_{\pfix{\eta}}(\barz_\eta))+\sigma^s(\mu_{\pfix{\eta}}(\barz_\eta))-\sigma^c(\lambda^\star)-\sigma^s(\mu^\star)\right)}.
\end{align*}
Now, the first term in the RHS converges by Levy's continuity theorem (e.g. see: \citet[Chapter 18]{levy_book}). In particular, if a sequence of random variables converges in distribution, the corresponding characteristic functions also converge.
\begin{align*}
    \lim_{\eta \uparrow \infty} \E{e^{j\epsilon_\eta \omega \barz_\eta}}=\E{e^{j\omega \zinfty}} \quad \textit{(as $\epsilon_\eta \barz_\eta \overset{D}{\rightarrow } \zinfty$)}.
\end{align*}
Next, the second term on the RHS can be bounded as follows:
\begin{align*}
   & \bigg|\frac{j\omega}{\sigma^c(\lambda^\star)+\sigma^s(\mu^\star)} \E{e^{j\epsilon_\eta \omega \barz_\eta}\left(\sigma^c(\lambda_{\pfix{\eta}}(\barz_\eta))+\sigma^s(\mu_{\pfix{\eta}}(\barz_\eta))-\sigma^c(\lambda^\star)-\sigma^s(\mu^\star)\right)}\bigg| \\
    \leq{}& \frac{|\omega|}{\sigma^c(\lambda^\star)+\sigma^s(\mu^\star)} \E{\bigg|\sigma^c(\lambda_{\pfix{\eta}}(\barz_\eta))+\sigma^s(\mu_{\pfix{\eta}}(\barz_\eta))-\sigma^c(\lambda^\star)-\sigma^s(\mu^\star)\bigg|}.
\end{align*}
Now, by Lemma \ref{lemma: convergence_uoc} the upper bound converges to zero, since by the continuity of $\sigma^c(\cdot)$ and $\sigma^s(\cdot)$, we have for all $x \in \bbR$ that
\begin{align*}
    &\lim_{\eta \uparrow \infty}\big(\sigma^c\left(\text{clip}\left(\lambda^\star+\epsilon_\eta \phi^c\left(x\right), \lambda_{\max}, \lambda_{\min}\right)\right)+\sigma^s\left(\text{clip}\left(\mu^\star+\epsilon_\eta \phi^s\left(x\right), \mu_{\max}, \mu_{\min}\right)\right)\\
    &-\sigma^c(\lambda^\star)-\sigma^s(\mu^\star)\big)=0.
\end{align*}
\revcolor{Moreover, the convergence is uniformly over compact sets due to Condition~\ref{ass:poly_growth}. In addition, we have $\max_{x\in \bbR} |\sigma^c(x)| \leq \sigma_{\max},  \max_{x \in \bbR}|\sigma^s(x)| \leq \sigma_{\max}$ and $\barz_\eta/\tau_\eta$ is a tight sequence of random variables by Lemma~\ref{lemma: tightness} and $\epsilon_\eta \tau_\eta \rightarrow l$.} So, the result of Lemma \ref{lemma: convergence_uoc} is valid. Now, we simplify the LHS. \revcolor{First, we verify that the random variable on the LHS is uniformly integrable. We have
\begin{align*}
    &\E{\bigg|e^{j\omega \epsilon_\eta \barz_\eta}\left(\phi^c\left(\frac{\barz_\eta}{\tau_\eta}\right)-\phi^s\left(\frac{\barz_\eta}{\tau_\eta}\right)\right)\bigg|^{2/\beta}} \overset{(a)}{\leq} 4\E{\left(\frac{\barz_\eta}{\tau_\eta}\right)^{2}} \\
    \overset{(b)}{\leq}{}& \frac{16\moment}{\epsilon_\eta \tau_\eta^2 \delta} + \frac{4K^2\left(\delta + 4K^\beta\right)}{\delta} + \frac{32\moment\left(4\moment+2\epsilon_\eta \tau_\eta K  (2K^\beta+\delta)\right)}{ \epsilon_\eta^2\tau_\eta^2 \delta^2} \\
    \overset{(c)}{\leq}{}& \frac{32\moment}{l \delta} + \frac{4K^2\left(\delta + 4K^\beta\right)}{\delta} + \frac{128\moment\left(4\moment+2K  (2K^\beta+\delta)\right)}{ l^2 \delta^2}
\end{align*}
where $(a)$ follows by Condition~\ref{ass:poly_growth}, $(b)$ due to \eqref{eq: imbalance_second_moment} in Proposition~\ref{prop: pos_rec}, and $(c)$ holds for $\eta$ large enough. In particular, as $\epsilon_\eta \tau_\eta \rightarrow l$, there exists $\eta_0 > 0$ such that for all $\eta \geq \eta_0$, we have $\epsilon_\eta \tau_\eta \geq l/2$. As the upper bound is independent of $\eta$ and $\beta<2$ by Condition~\ref{ass:poly_growth}, $e^{j\omega \epsilon_\eta \barz_\eta}\left(\phi^c\left(\frac{\barz_\eta}{\tau_\eta}\right)-\phi^s\left(\frac{\barz_\eta}{\tau_\eta}\right)\right)$ is a uniformly integrable sequence of random variables. Next, by \citet[Theorem 2.7]{van2000asymptotic}, we have $(\epsilon_\eta\barz_\eta, \epsilon_\eta\tau_\eta) \overset{D}{\rightarrow} (\zinfty, l)$. Thus, by the continuous mapping theorem and continuity of $(\phi^c, \phi^s)$, we have
\begin{align*}
    e^{j\omega \epsilon_\eta \barz_\eta}\left(\phi^c\left(\frac{\epsilon_\eta\barz_\eta}{\epsilon_\eta\tau_\eta}\right)-\phi^s\left(\frac{\epsilon_\eta\barz_\eta}{\epsilon_\eta\tau_\eta}\right)\right) \overset{D}{\rightarrow} e^{j\omega \zinfty}\left(\phi^c\left(\frac{\zinfty}{l}\right)-\phi^s\left(\frac{\zinfty}{l}\right)\right). 
\end{align*}
Thus, by \citet[Theorem 3.5]{billingsley2013convergence}, the limit of the LHS of \eqref{eq: prelimit_functional_equation} is as follows:
\begin{align*}
    \lim_{\eta \uparrow \infty}\E{e^{j \omega \epsilon_\eta \barz_\eta}\left(\phi^c\left(\frac{\epsilon_\eta\barz_\eta}{\epsilon_\eta\tau_\eta}\right)-\phi^s\left(\frac{\epsilon_\eta\barz_\eta}{\epsilon_\eta\tau_\eta}\right)\right)} &= \E{e^{j \omega \zinfty}\left(\phi^c\left(\frac{\zinfty}{l}\right)-\phi^s\left(\frac{\zinfty}{l}\right)\right)}.
\end{align*}
This completes the proof.} \hfill $\square$
\endproof
\subsection{Solving the Functional Equation} \label{sec: claim_schwartz_before_fourier}
\proof{Proof of Claim \ref{eq: tempered_dist_eqn}}
Replace $\omega$ with $-\omega$ in \eqref{eq: functional_equation_mgf} and note that for any $\varphi \in \calS$, we have
\begin{align*}
    \int_{-\infty}^\infty\frac{-j\omega}{\sqrt{2\pi}}\E{e^{-j \omega \zinfty}} \varphi(\omega) d\omega &= \int_{-\infty}^\infty \int_{-\infty}^\infty \frac{-j\omega}{\sqrt{2\pi}}e^{-j\omega z}\varphi(\omega) dF_{\zinfty}(z) d\omega \\
    &\overset{(a)}{=}\int_{-\infty}^\infty \int_{-\infty}^\infty \frac{-j\omega}{\sqrt{2\pi}}e^{-j\omega z}\varphi(\omega)  d\omega dF_{\zinfty}(z) \\
    &=\int_{-\infty}^\infty \int_{-\infty}^\infty \frac{\partial}{\partial z} \left(\frac{1}{\sqrt{2\pi}}e^{-j\omega z}\varphi(\omega)\right)  d\omega dF_{\zinfty}(z) \\
    &\overset{(b)}{=}\int_{-\infty}^\infty  \frac{d \hat{\varphi}(z)}{dz}  dF_{\zinfty}(z) \\
    &\overset{\eqref{eq: tempered_distributions}}{=}\calT_{\zinfty}[\hat{\varphi}^\prime]. \numberthis \label{eq: lhs_mgf_equation}
\end{align*}
where $(a)$ follows by Fubini's Theorem as $|\varphi(\omega)|$ is integrable with respect to $dF_{\zinfty} d\omega$. Next, $(b)$ follows by the definition of Fourier transform given by \eqref{eq: fourier_schwartz}, where we rely on the Leibniz integral rule to interchange the derivative and integral. In addition, we also have
\begin{align*}
    \int_{-\infty}^\infty \frac{1}{\sqrt{2\pi}} \E{g_{1,l}(\zinfty)e^{-j\omega \zinfty}} \varphi(\omega) d\omega &= \int_{-\infty}^\infty \int_{-\infty}^\infty \frac{1}{\sqrt{2\pi}} g_{1,l}(z) e^{-j \omega z} \varphi(\omega) dF_{\zinfty}(z) d\omega \\
    &\overset{(a)}{=}\int_{-\infty}^\infty g_{1,l}(z) \int_{-\infty}^\infty \frac{1}{\sqrt{2\pi}}  e^{-j \omega z} \varphi(\omega)  d\omega dF_{\zinfty}(z) \\
    &\overset{\eqref{eq: fourier_schwartz}}{=}\int_{-\infty}^\infty g_{1,l}(z)\hat{\varphi}(z) dF_{\zinfty}(z) \\
    &\overset{\eqref{eq: tempered_distributions}}{=}\calT_{\zinfty}[g_{1,l}\hat{\varphi}] \numberthis \label{eq: rhs_mgf_equation}
\end{align*}
\revcolor{where $(a)$ follows by Fubini's Theorem as $|\varphi(\omega)g_{1,l}(z)| $ is integrable with respect to $dF_{\zinfty}(z) d\omega$ for any sub-Quadratically growing $g_{1, l}$ (as $\phi^c, \phi^s$ are sub-Quadratic by Condition~\ref{ass:poly_growth}, $g_{1, l}$ is also sub-Quadratic).  Thus, 
\begin{align*}
    \calT_{\zinfty}[|g_{1, l}\hat{\varphi}|] &\overset{(b)}{\leq} \varphi_{\max}\calT_{\zinfty}[|g_{1, l}|] \leq \pfix{\frac{2\varphi_{\max}}{\sigma^c(\lambda^\star) + \sigma^s(\mu^\star)}}\left(\E{\bigg|\phi^c\left(\frac{\zinfty}{l}\right)\bigg|} + \E{\bigg|\phi^s\left(\frac{\zinfty}{l}\right)\bigg|}\right) \\
    &\leq \pfix{\frac{4\varphi_{\max}}{\left(\sigma^c(\lambda^\star) + \sigma^s(\mu^\star)\right)l^\beta}}\E{\pfix{|}\zinfty\pfix{|}^\beta} \overset{(c)}{=} \pfix{\frac{4\varphi_{\max}}{\left(\sigma^c(\lambda^\star) + \sigma^s(\mu^\star)\right)l^\beta}}\lim_{\eta \rightarrow \infty}\E{\pfix{|}\epsilon_\eta\barz_\eta\pfix{|}^\beta} \overset{(d)}{<} \infty.
\end{align*}
As Fourier transform is a bijection on Schwartz space by Proposition~\ref{prop: bijection}, $\hat{\varphi} \in \calS(\bbR)$. Thus, there exists a constant $\varphi_{\max} > 0$ such that $\max_{x \in \bbR} |\hat{\varphi}(x)| \leq \varphi_{\max}$. So, $(b)$ follows. Next, $(c)$ holds by \citet[Theorem 3.5]{billingsley2013convergence} as $\left(\epsilon_\eta\barz_\eta\right)^\beta$ is absolutely continuous and $\epsilon_\eta \barz_\eta \overset{D}{\rightarrow} \zinfty$. In particular, absolute continuity of $\left(\epsilon_\eta\barz_\eta\right)^\beta$ holds as $\beta<2$ by Condition~\ref{ass:poly_growth} and $\E{(\epsilon_\eta\barz_\eta)^2} = O(1)$ by \eqref{eq: imbalance_second_moment} in Proposition~\ref{prop: pos_rec}. The bound $\E{(\epsilon_\eta\barz_\eta)^2} = O(1)$ also implies $(d)$.}
Now, multiplying \eqref{eq: functional_equation_mgf} on both sides by $\varphi(\cdot)$ and integrating and using \eqref{eq: lhs_mgf_equation} and \eqref{eq: rhs_mgf_equation} imply
\begin{align*}
    \calT_{\zinfty}[\hat{\varphi}^\prime-g_{1,l}\hat{\varphi}]=0 \quad \forall \varphi \in \calS(\bbR).
\end{align*}
This completes the proof of the claim. \hfill $\square$
\endproof
\section{Proof of Preliminary Lemmas Required for Theorem \ref{theo: case_1} and \ref{theo: profit_driven_regime_3}}
\label{app: preliminary_lemmas}
\proof{Proof of Lemma \ref{lemma: z_sgnz}}
We consider $z g_\eta(z)$ as the test function. Note that $$\E{|\barz_\eta g_\eta(\barz_\eta)|} \leq G \E{|\barz_\eta|}<\infty$$ by Proposition \ref{prop: pos_rec}. Thus, we can set the drift of this function to zero in the steady state. We have
\revcolor{\begin{align*}
    0 ={}& \E{\barz_\eta^+g_\eta(\barz_\eta^+)-\barz_\eta g_\eta(\barz_\eta)} \\
    ={}&\E{\barz_\eta^+(g_\eta(\barz_\eta^+)-g_\eta(\barz_\eta))}+\E{(\barz_\eta^+ -\barz_\eta)g_\eta(\barz_\eta)} \\
    ={}&\E{\barz_\eta^+(g_\eta(\barz_\eta^+)-g_\eta(\barz_\eta))}+\E{(\bara^c_\eta(\barz_\eta) -\bara_\eta^s(\barz_\eta))g_\eta(\barz_\eta)} \\
    ={}&\E{\barz_\eta^+(g_\eta(\barz_\eta^+)-g_\eta(\barz_\eta))}+\E{(\lambda_\eta(\barz_\eta)-\mu_\eta(\barz_\eta))g_\eta(\barz_\eta)} 
\end{align*}}
where the last \pfix{equality follows} by the tower property of expectation and Eq. \eqref{eq: general_pricing_policy}.
This completes the proof. \hfill $\square$
\endproof
\proof{Proof of Lemma \ref{lemma: prob_bound_on_any_state}}
For any given \revcolor{$|x| \leq K \max\left\{\tau_\eta, 1/\epsilon_\eta\right\}$}, consider the test function
\begin{align*}
    V_x(z)=z\mathbbm{1}\{\textrm{sgn}(x)z > \textrm{sgn}(x)x\}.
\end{align*}
We will set the drift of above defined test function to zero in steady state to get a bound on $\P{\barz_\eta=x}$. Alternatively, one can analyze the drift of $z\mathbbm{1}\{z >x\}$ for $x \geq 0$ and $z\mathbbm{1}\{z <x\}$ for $x <0$. To make the presentation compact, we consider $V_x(z)$ which combines both the cases mentioned above. For a better understanding of the proof, the reader can consider the case when $x \geq 0$ and $x<0$ separately.

Note that, by Proposition \ref{prop: pos_rec}, we have $\E{|V_x(\barz_\eta)|} \leq \E{|\barz_\eta|}<\infty$. Thus, setting the drift to zero of $V_x(\cdot)$ in steady state, we get
\revcolor{\begin{align*}
    0={}&\E{\barz_\eta^+ \mathbbm{1}\{\textrm{sgn}(x)\barz_\eta^+>\textrm{sgn}(x)x\}-\barz_\eta \mathbbm{1}\{\textrm{sgn}(x)\barz_\eta>\textrm{sgn}(x)x\}} \\
    \overset{(a)}{=}{}&\E{\barz_\eta^+ \mathbbm{1}\{\textrm{sgn}(x)\barz_\eta^+>\textrm{sgn}(x)x\}-\barz_\eta^+ \mathbbm{1}\{\textrm{sgn}(x)\barz_\eta>\textrm{sgn}(x)x\}}\\
    &+\E{\barz_\eta^+ \mathbbm{1}\{\textrm{sgn}(x)\barz_\eta>\textrm{sgn}(x)x\}-\barz_\eta \mathbbm{1}\{\textrm{sgn}(x)\barz_\eta>\textrm{sgn}(x)x\}} \\\overset{(b)}{=}{}&\E{\barz_\eta^+ \mathbbm{1}\{\textrm{sgn}(x)\barz_\eta^+>\textrm{sgn}(x)x\}-\barz_\eta^+ \mathbbm{1}\{\textrm{sgn}(x)\barz_\eta>\textrm{sgn}(x)x\}}\\
    &+ \E{\left(\lambda_\eta(\barz_\eta)-\mu_\eta(\barz_\eta)\right)\mathbbm{1}\{\textrm{sgn}(x)\barz_\eta>\textrm{sgn}(x)x\}}.
\end{align*}}
In step $(a)$, we add and subtract $\barz_\eta^+\mathbbm{1}\{\textrm{sgn}(x)\barz_\eta > \textrm{sgn}(x)x\}$. Next, $(b)$ follows by the evolution equation given by \eqref{eq: imbalance_update_steady_state} and the tower property of expectation. Now, by sending the first term to LHS and taking absolute value on both sides, we get
\begin{align*}
\E{\big|\lambda_\eta(\barz_\eta)-\mu_\eta(\barz_\eta)\big|} 
    \geq{}& \bigg|\E{\barz_\eta^+\left( \mathbbm{1}\{\textrm{sgn}(x)\barz_\eta^+>\textrm{sgn}(x)x\}- \mathbbm{1}\{\textrm{sgn}(x)\barz_\eta>\textrm{sgn}(x)x\}\right)}\bigg| \\
    \overset{(c)}{=}{}&\bigg|\E{(\barz_\eta^+-x)\left( \mathbbm{1}\{\textrm{sgn}(x)\barz_\eta^+>\textrm{sgn}(x)x\}- \mathbbm{1}\{\textrm{sgn}(x)\barz_\eta>\textrm{sgn}(x)x\}\right)}\\
    &+x\E{\mathbbm{1}\{\textrm{sgn}(x)\barz_\eta^+>\textrm{sgn}(x)x\}- \mathbbm{1}\{\textrm{sgn}(x)\barz_\eta>\textrm{sgn}(x)x\}}\bigg| \\
    \overset{(d)}{=}{}&\bigg|\E{(\barz_\eta^+-x)\left( \mathbbm{1}\{\textrm{sgn}(x)\barz_\eta^+>\textrm{sgn}(x)x\}- \mathbbm{1}\{\textrm{sgn}(x)\barz_\eta>\textrm{sgn}(x)x\}\right)}\bigg| \\
    \overset{(e)}{=}{}&\E{\bigg|(\barz_\eta^+-x)\left( \mathbbm{1}\{\textrm{sgn}(x)\barz_\eta^+>\textrm{sgn}(x)x\}- \mathbbm{1}\{\textrm{sgn}(x)\barz_\eta>\textrm{sgn}(x)x\}\right)\bigg|} \\
    \overset{(f)}{\geq}{}& \E{ \mathbbm{1}\{\barz_\eta=x\}\mathbbm{1}\{\textrm{sgn}(x)a^c_\eta(x)>\textrm{sgn}(x)a^s_\eta(x)\}} \\
    ={}&\P{\barz_\eta=x}\P{\textrm{sgn}(x)a^c_\eta(x)>\textrm{sgn}(x)a^s_\eta(x)}.
\end{align*}
 In $(c)$, we simply add and subtract $x\E{\mathbbm{1}\{\textrm{sgn}(x)\barz_\eta^+>\textrm{sgn}(x)x\}- \mathbbm{1}\{\textrm{sgn}(x)\barz_\eta>\textrm{sgn}(x)x\}}$ and note that it is just the drift of $\mathbbm{1}\{\textrm{sgn}(x)\barz_\eta>\textrm{sgn}(x)x\}$ in steady state which is equal to zero. This gives us $(d)$. Next, $(e)$ follows by noting that the term whose expectation we are calculating is non-negative with probability 1 when $x\geq 0$ and non-positive with probability 1 when $x<0$. Lastly, $(f)$ follows by expanding the expectation and only preserving the term corresponding to $\barz_\eta=x$. Now, using the above equation, for all \revcolor{$|x| \leq K \max\left\{\tau_\eta, 1/\epsilon_\eta\right\}$}, by Condition~\ref{ass: irreducibility}, we have
\revcolor{\begin{align*}
    \P{\barz_\eta=x} \leq \frac{1}{p_{\min}}\E{\big|\lambda_\eta(\barz_\eta)-\mu_\eta(\barz_\eta)\big|} 
\end{align*}}
This completes the proof. \hfill $\square$
\endproof
\proof{Proof of Lemma \ref{claim: case1_t3_t5}}
For any given $\delta>0$,  there exists $\tilde{K}>0$ such that $|g(x)-c_{\infty}| \leq \delta$ and $|g(-x)-c_{\infty}| \leq \delta$ for all $x \geq \tilde{K}$. Thus, we have
\begin{align*}
    \P{\bigg|g\left(\frac{\barz_\eta}{\tau_\eta}\right)-c_{\infty}\bigg|>\delta}={}& \P{\bigg|g\left(\frac{\barz_\eta}{\tau_\eta}\right)-c_{\infty}\bigg|>\delta \bigg| |\barz_\eta| \leq \tilde{K} \tau_\eta}\P{|\barz_\eta| \leq \tilde{K} \tau_\eta}\\
    &+\P{\bigg|g\left(\frac{\barz_\eta}{\tau_\eta}\right)-c_{\infty}\bigg|>\delta \bigg| |\barz_\eta| > \tilde{K} \tau_\eta}\P{|\barz_\eta| > \tilde{K} \tau_\eta} \\
    ={}&\P{\bigg|g\left(\frac{\barz_\eta}{\tau_\eta}\right)-c_{\infty}\bigg|>\delta \bigg| |\barz_\eta| \leq \tilde{K} \tau_\eta}\P{|\barz_\eta| \leq \tilde{K} \tau_\eta} \\
    \leq{}&\P{|\barz_\eta| \leq \tilde{K} \tau_\eta} \\
    \overset{*}{\leq}{}& \frac{(2\tilde{K}\tau_\eta+1)\epsilon_\eta}{p_{\min}}\E{\bigg|\phi^c\left(\frac{\barz_\eta}{\tau_\eta}\right) - \phi^s\left(\frac{\barz_\eta}{\tau_\eta}\right)\bigg|} \\
    \overset{**}{\leq}{}&  \frac{2(2\tilde{K}\tau_\eta+1)\epsilon_\eta\phi_{\max}}{p_{\min}},
\end{align*}
\revcolor{where $(*)$ follows for $\eta\geq\eta_0$ for some $\eta_0 > 0$ (depending on $\tilde{K}$) by Lemma \ref{lemma: prob_bound_on_any_state}, Condition~\ref{ass: irreducibility}, and \eqref{eq:simplified_general_pricing_policy}. In particular, for $\eta > 0$ large enough, as $\epsilon_\eta \tau_\eta \rightarrow 0$, we have $K/\epsilon_\eta \geq \tilde{K}\tau_\eta$, ensuring that Lemma \ref{lemma: prob_bound_on_any_state} holds. Next, $(**)$ holds due to \eqref{eq:bounded_control_curves}.} Thus, we have
\begin{align*}
   \E{\bigg|g\left(\frac{\barz_\eta}{\tau_\eta}\right)-c_{\infty}\bigg|} &\leq \delta\P{\bigg|g\left(\frac{\barz_\eta}{\tau_\eta}\right)-c_{\infty}\bigg| \leq \delta} + (g_{\max}+c_\infty)\P{\bigg|g\left(\frac{\barz_\eta}{\tau_\eta}\right)-c_{\infty}\bigg|>\delta} \\
   &\leq \delta + (g_{\max}+c_\infty) \frac{2(2\tilde{K}\tau_\eta+1)\epsilon_\eta\phi_{\max}}{p_{\min}}
\end{align*}
Now, by taking the limit as $\eta \uparrow \infty$, and noting that $\epsilon_\eta \tau_\eta \rightarrow 0$ and $\epsilon_\eta \rightarrow 0$, we get
\begin{align*}
    \lim_{\eta \uparrow \infty}\bigg|\E{g\left(\frac{\barz_\eta}{\tau_\eta}\right)}-c_{\infty}\bigg| \leq \delta.
\end{align*}
As $\delta>0$ is arbitrary, the proof is complete. \hfill $\square$
\endproof
\proof{Proof of Lemma \ref{lemma: convergence_technical}}
For any $\delta>0$, there exists $n_0 > 0$ such that for all $n>n_0$
\begin{align*}
    |g_n(x)-c| &\leq \delta \quad \forall x \in \bbR.
\end{align*}
By substituting $Y_n$ for $x$, taking expectations on both sides and using Jensen's inequality, we get
\begin{align*}
    |\E{g_n(Y_n)-c}| \leq \E{|g_n(Y_n)-c|} \leq \delta
\end{align*}
Now, by taking $n \uparrow \infty$, we get
\begin{align*}
    \lim_{n \uparrow \infty}\big|\E{g_n(Y_n)}-c\big| \leq \delta
\end{align*}
As the above is true for all $\delta>0$, the proof is complete.
\endproof
\section{Technical Details for the Proof of Theorem \ref{theo: case_1}}
\subsection{Proof of Lemma \ref{lemma: symmetry}} \label{app: lemma_symmetry}
\proof{Proof of Lemma \ref{lemma: symmetry}}
For $\omega \in \bbR$, we define the test function 
\begin{align*}
    U(z)\overset{\Delta}{=} \frac{1}{\step(z)}e^{j\epsilon_\eta \omega z \step(z)}
\end{align*}
Now, consider the one-step drift of the above test function.
\begin{align*}
    \Delta U(\barz_\eta, \eta) ={}& \frac{1}{\step(\barz_\eta^+)}e^{j\epsilon_\eta \omega \barz_\eta^+ \chi(\barz_\eta^+)}-\frac{1}{\step(\barz_\eta)}e^{j\epsilon_\eta \omega \barz_\eta \step(\barz_\eta)} \\
    ={}&\underbrace{\frac{1}{\step(\barz_\eta^+)}e^{j\epsilon_\eta \omega \barz_\eta^+ \step(\barz_\eta^+)}-\frac{1}{\step(\barz_\eta)}e^{j\epsilon_\eta \omega \barz_\eta^+ \step(\barz_\eta)}}_{\calT_6}+\underbrace{\frac{1}{\step(\barz_\eta)}e^{j\epsilon_\eta \omega \barz_\eta^+ \step(\barz_\eta)}-\frac{1}{\step(\barz_\eta)}e^{j\epsilon_\eta \omega \barz_\eta \step(\barz_\eta)}}_{\calT_7}.
\end{align*}
We analyze each of the above terms separately. We simplify $\calT_6$ by using Taylor's Theorem to get 
 \begin{claim} \label{claim: case1_t6}
 \begin{align*}
    \E{\calT_6}=o(\epsilon_\eta^2).
\end{align*}
 \end{claim}
 The proof of the Claim \ref{claim: case1_t6} has been deferred to Appendix \ref{app: proof_of_claims} and here we continue with the proof of Lemma \ref{lemma: symmetry}. Now, to simplify $\calT_7$, note that $\calT_7=\calT_2\step(\barz_\eta)$, where $\calT_2$ is defined in \eqref{eq: def_of_t2}. Thus, by Claim \ref{claim: case1_t2}, we have
\begin{align*}
   \E{ \calT_7 | \barz_\eta}={}&j\epsilon_\eta^2 \omega e^{j\epsilon_\eta \omega \barz_\eta \step(\barz_\eta)} \left(\phi^c\left(\frac{\barz_\eta}{\tau_\eta}\right)-\phi^s\left(\frac{\barz_\eta}{\tau_\eta}\right)\right)\\
    &-\frac{1}{2}\epsilon_\eta^2\omega^2 \step(\barz_\eta)e^{j\epsilon_\eta \omega \barz_\eta \step(\barz_\eta)}\left(\sigma^c(\lambda_\eta(\barz_\eta))+\sigma^s(\mu_\eta(\barz_\eta))\right)+o(\epsilon_\eta^2).
\end{align*}
Note that, $|\E{U(\barz_\eta)}|\leq \E{|U(\barz_\eta)|} \leq 1/\min\{|\step(1)|, |\step(-1)|\}$. Thus, by setting the drift of $U(\cdot)$ to zero in steady state, we get $\E{\calT_6}+\E{\calT_7}=0$. Now, by substituting $\calT_6$ and $\calT_7$ and dividing by $j\epsilon_\eta^2\omega$, we get
\begin{align*}
&\E{e^{j\epsilon_\eta \omega \barz_\eta \step(\barz_\eta)}\left(\phi^c\left(\frac{\barz_\eta}{\tau_\eta}\right)-\phi^s\left(\frac{\barz_\eta}{\tau_\eta}\right)\right)}\\
&+\frac{j\omega}{2}\E{\step(\barz_\eta)e^{j\epsilon_\eta \omega \barz_\eta \step(\barz_\eta)}\left(\sigma^c(\lambda_{\pfix{\eta}}(\barz_\eta))+\sigma^s(\mu_{\pfix{\eta}}(\barz_\eta))\right)}+o(1) =0.
\end{align*}
Now, by adding and subtracting certain terms, we get
\begin{align*}
    &\left(-1+\frac{j\omega}{2}(\sigma^c(\lambda^\star)+\sigma^s(\mu^\star))\right)\E{\step(\barz_\eta)e^{j\epsilon_\eta \omega \barz_\eta \step(\barz_\eta)}} \\
    ={}& -\E{e^{j\epsilon_\eta \omega \barz_\eta \step(\barz_\eta)}\left(1+\left(\phi^c\left(\frac{\barz_\eta}{\tau_\eta}\right)-\phi^s\left(\frac{\barz_\eta}{\tau_\eta}\right)\right)\frac{1}{\step(\barz_\eta)}\right)\step(\barz_\eta)}\\
    &+\frac{j\omega}{2}\E{\step(\barz_\eta)e^{j\epsilon_\eta \omega \barz_\eta \step(\barz_\eta)}\left(\sigma^c(\lambda^\star)+\sigma^s(\mu^\star)-\sigma^c(\lambda_{\pfix{\eta}}(\barz_\eta))-\sigma^s(\mu_{\pfix{\eta}}(\barz_\eta))\right)}+o(1).
\end{align*}
Now, by taking the absolute value on both sides and upper bounding the RHS by using triangle inequality and Jensen's inequality, we get
\begin{align*}
    &\left(1+\omega^2\frac{(\sigma^c(\lambda^\star)+\sigma^s(\mu^\star))^2}{4}\right)\bigg|\E{\step(\barz_\eta)e^{j\epsilon_\eta \omega \barz_\eta \step(\barz_\eta)}}\bigg|\\
    \leq{}& \frac{|\omega|\max\{|\step(1)|, |\step(-1)|\}}{2}\E{\big|\sigma^c(\lambda^\star)+\sigma^s(\mu^\star)-\sigma^c(\lambda_{\pfix{\eta}}(\barz_\eta))-\sigma^s(\mu_{\pfix{\eta}}(\barz_\eta))\big|} \\
    &+\max\{|\step(1)|, |\step(-1)|\}\E{\bigg|\left(1+\left(\phi^c\left(\frac{\barz_\eta}{\tau_\eta}\right)-\phi^s\left(\frac{\barz_\eta}{\tau_\eta}\right)\right)\frac{1}{\step(\barz_\eta)}\right)\bigg|}+o(1).
\end{align*}
Lastly, by taking the limit $\eta \uparrow \infty$ and using Claim \ref{claim: t4_t5}, we get the result. \hfill $\square$
\endproof
\subsection{Proof of Claims for Lemma \ref{lemma: mod_z} and \ref{lemma: symmetry}} \label{app: proof_of_claims}
\proof{Proof of Claim \ref{claim: case1_t2}}
$\calT_2$ can be simplified as follows:
\begin{align*}
    &\step(\barz_\eta)^2\E{\calT_2 | \barz_\eta } \\
    ={}& \E{e^{j\epsilon_\eta \omega \barz_\eta^+ \step(\barz_\eta)}-e^{j\epsilon_\eta \omega \barz_\eta \step(\barz_\eta)} \big| \barz_\eta} \\
    ={}& \E{e^{j\epsilon_\eta \omega \left(\barz_\eta+a_\eta^c(\barz_\eta)-a^s_\eta(\barz_\eta)\right) \step(\barz_\eta)}-e^{j\epsilon_\eta \omega \barz_\eta \step(\barz_\eta)} \big| \barz_\eta} \\
    ={}& \E{e^{j \epsilon_\eta \omega \barz_\eta \step(\barz_\eta)}\left(e^{j\epsilon_\eta \omega \left(a_\eta^c(\barz_\eta)-a^s_\eta(\barz_\eta)\right) \step(\barz_\eta)}-1\right)\big| \barz_\eta} \\
    \overset{(a)}{=}{}& \E{e^{j \epsilon_\eta \omega \barz_\eta \step(\barz_\eta)}\left(j\epsilon_\eta \omega \left(a_\eta^c(\barz_\eta)-a^s_\eta(\barz_\eta)\right) \step(\barz_\eta) - \frac{1}{2}\epsilon_\eta^2\omega^2\left(a_\eta^c(\barz_\eta)-a^s_\eta(\barz_\eta)\right)^2 \step(\barz_\eta)^2\right)\hspace{-1pt}\big| \barz_\eta} +o(\epsilon_\eta^2) \\
    \overset{(b)}{=}{}& e^{j \epsilon_\eta \omega \barz_\eta \step(\barz_\eta)}\left(j\epsilon_\eta^2 \omega \left(\phi^c\left(\frac{\barz_\eta}{\tau_\eta}\right)-\phi^s\left(\frac{\barz_\eta}{\tau_\eta}\right)\right) \step(\barz_\eta) - \frac{1}{2}\epsilon_\eta^2\omega^2 \E{\left(a_\eta^c(\barz_\eta)-a^s_\eta(\barz_\eta)\right)^2 \big| \barz_\eta} \step(\barz_\eta)^2\right) \\
    &+o(\epsilon_\eta^2)
\end{align*}
where $(a)$ follows by Lemma \ref{lemma: taylor_series} \revcolor{and $\E{|a^c_\eta(\barz_\eta)|^3}, \E{|a^s_\eta(\barz_\eta)|^3} \leq \moment$.} Next, $(b)$ follows by using the tower property of expectation \revcolor{and \eqref{eq:simplified_general_pricing_policy}.} Now, we simplify the second term above by calculating the second moment of the arrivals.
\begin{align*}
    \E{(a_{\eta}^c(\barz_\eta)-a_{\eta}^s(\barz_\eta))^2 | \barz_\eta}&=\Var{a_{\eta}^c(\barz_\eta) | \barz_\eta}+\Var{a_{\eta}^s(\barz_\eta) | \barz_\eta}+\E{a_{\eta}^c(\barz_\eta)-a_{\eta}^s(\barz_\eta)| \barz_\eta}^2 \\
    &=\sigma^c(\lambda_\eta(\barz_\eta))+\sigma^s(\mu_\eta(\barz_\eta))+ \epsilon_\eta^2 \left(\phi^c\left(\frac{\barz_\eta}{\tau_\eta}\right)-\phi^s\left(\frac{\barz_\eta}{\tau_\eta}\right)\right)^2 \\
    &=\sigma^c(\lambda_\eta(\barz_\eta))+\sigma^s(\mu_\eta(\barz_\eta))+o(\epsilon_\eta), \quad \numberthis \label{eq: second_moment}
\end{align*}
\revcolor{where the last equality holds due to \eqref{eq:bounded_control_curves}.} This completes the proof. \hfill $\square$
\endproof
\proof{Proof of Claim \ref{claim: t4_t5}}
By the definition of $\step(\cdot)$, we have $|1+(\phi^c(x)-\phi^s(x))/\step(x)| \rightarrow 0$ as $x \rightarrow \pm \infty$. Also, we have $|1+(\phi^c(x)-\phi^s(x))/\step(x)| \leq 1+\frac{2\phi_{\max}}{\min\{|\step(1)|, |\step(-1)|\}}$ \revcolor{by \eqref{eq:bounded_control_curves}}. Thus, by Lemma \ref{claim: case1_t3_t5}, we have 
\begin{align*}
    \lim_{\eta \uparrow \infty} \E{\bigg|1+\left(\phi^c\left(\frac{\barz_\eta}{\tau_\eta}\right)-\phi^s\left(\frac{\barz_\eta}{\tau_\eta}\right)\right)\frac{1}{\step\left(\frac{\barz_\eta}{\tau_\eta}\right)}\bigg|}=0.
\end{align*}
This completes the first part of the claim. Now, to prove the second part of the claim, note that
\begin{align*}
    \big|\sigma^c(\lambda^\star)+\sigma^s(\mu^\star)-\sigma^c\left(\lambda^\star+\epsilon_\eta \phi^c(x)\right)-\sigma^s\left(\mu^\star+\epsilon_\eta \phi^s(x)\right)\big| \rightarrow 0 \quad \forall x \in \bbR,
\end{align*}
\revcolor{and the convergence is uniformly over $x \in \bbR$ as $|\phi^c(x)| \leq \phi_{\max}$, $|\phi^s(x)| \leq \phi_{\max}$ by \eqref{eq:bounded_control_curves}.} Also note that the above function is absolutely bounded by $4\sigma_{\max}$. Thus, by setting $Y_n=\barz_\eta/\tau_\eta$ in Lemma \ref{lemma: convergence_technical} and using \eqref{eq:simplified_general_pricing_policy}, we get
\begin{align*}
    \lim_{\eta \uparrow \infty} \E{\bigg|\sigma^c(\lambda^\star)+\sigma^s(\mu^\star)-\sigma^c\left(\lambda^\star + \epsilon_\eta \phi^c\left(\frac{\barz_\eta}{\tau_\eta}\right)\right)-\sigma^s\left(\mu^\star + \epsilon_\eta \phi^s\left(\frac{\barz_\eta}{\tau_\eta}\right)\right)\bigg|}&=0.
\end{align*}
This completes the proof of the claim. \hfill $\square$
\endproof
\proof{Proof of Claim \ref{claim: case1_t6}}
\revcolor{\begin{align*}
\E{\calT_6}&=\E{\frac{1}{\step(\barz_\eta^+)}e^{j\epsilon_\eta \omega \barz_\eta^+ \step(\barz_\eta^+)}-\frac{1}{\step(\barz_\eta)}e^{j\epsilon_\eta \omega \barz_\eta^+ \step(\barz_\eta)}} \\
&=\E{\left(\frac{1}{\step(\barz_\eta^+)}e^{j\epsilon_\eta \omega \barz_\eta^+ \step(\barz_\eta^+)}-\frac{1}{\step(\barz_\eta)}e^{j\epsilon_\eta \omega \barz_\eta^+ \step(\barz_\eta)}\right)\mathbbm{1}\{\step(\barz_\eta^+) \neq \step(\barz_\eta)\}} \\
&=\E{\frac{1}{\step(\barz_\eta^+)}e^{j\epsilon_\eta \omega \barz_\eta^+ \step(\barz_\eta^+)\mathbbm{1}\{\step(\barz_\eta^+) \neq \step(\barz_\eta)\}}-\frac{1}{\step(\barz_\eta)}e^{j\epsilon_\eta \omega \barz_\eta^+ \step(\barz_\eta)\mathbbm{1}\{\step(\barz_\eta^+) \neq \step(\barz_\eta)\}}} \\
&\overset{(a)}{=}\E{\frac{1}{\step(\barz_\eta^+)}-\frac{1}{\step(\barz_\eta)}}-\E{\frac{1}{2}\epsilon_\eta^2\omega^2(\barz_\eta^+)^2\left(\step(\barz_\eta^+)-\step(\barz_\eta)\right)\mathbbm{1}\{\step(\barz_\eta^+) \neq \step(\barz_\eta)\}}+o(\epsilon_\eta^2) \\
&=\E{\frac{1}{\step(\barz_\eta^+)}-\frac{1}{\step(\barz_\eta)}}-\E{\frac{1}{2}\epsilon_\eta^2\omega^2(\barz_\eta^+)^2\left(\step(\barz_\eta^+)-\step(\barz_\eta)\right)}+o(\epsilon_\eta^2) \\
    &\overset{(b)}{=}-\frac{1}{2}\epsilon_\eta^2\omega^2\E{(\barz_\eta^+)^2\left(\step(\barz_\eta^+)-\step(\barz_\eta)\right)}+o(\epsilon_\eta^2) 
\end{align*}
where $(a)$ follows by Lemma \ref{lemma: taylor_series} and noting that the third moment of $|\barz_\eta^+ \chi(\barz_\eta^+)\mathbbm{1}\{\step(\barz_\eta^+) \neq \step(\barz_\eta)\}$ is bounded by \eqref{eq:third_moment_sign_change}.} Further, $(b)$ follows as $1/|\step(z)| \leq 1/\min\{|\step(1)|, |\step(-1)|\}$ and thus, $\E{\Delta (1/\step)(\barz_\eta, \eta)}=0$. Now, by taking the absolute value on both sides, and using Jensen's inequality, we get
\revcolor{\begin{align*}
     |\E{\calT_6}|\leq{}& \frac{1}{2}\epsilon_\eta^2\omega^2 \E{\big|(\barz_\eta^+)^2\left(\step(\barz_\eta^+)-\step(\barz_\eta)\right)\big|}+o(\epsilon_\eta^2) \\ 
     ={}& \frac{1}{2}\epsilon_\eta^2\omega^2 \E{\big|(\barz_\eta^+)^2\left(\step(\barz_\eta^+)-\step(\barz_\eta)\right)\big| \mathbbm{1}\left\{|\barz_\eta^+| \leq \frac{1}{\sqrt{\epsilon_\eta}}\right\}} \\
     &+ \frac{1}{2}\epsilon_\eta^2\omega^2\E{\big|(\barz_\eta^+)^2\left(\step(\barz_\eta^+)-\step(\barz_\eta)\right)\big|\mathbbm{1}\left\{|\barz_\eta^+| \geq \frac{1}{\sqrt{\epsilon_\eta}}\right\}} +o(\epsilon_\eta^2).
\end{align*}
Now, we bound the above two terms separately. We have
\begin{align*}
    \E{\big|(\barz_\eta^+)^2\left(\step(\barz_\eta^+)-\step(\barz_\eta)\right)\big| \mathbbm{1}\left\{|\barz_\eta^+| \leq \frac{1}{\sqrt{\epsilon_\eta}}\right\}}
    &\leq  \frac{1}{\sqrt{\epsilon_\eta}}\E{\big|\barz_\eta^+\left(\step(\barz_\eta^+)-\step(\barz_\eta)\right)\big|} \\
    &\overset{(a)}{=} \frac{1}{\sqrt{\epsilon_\eta}}\big|\E{\barz_\eta^+\left(\step(\barz_\eta^+)-\step(\barz_\eta)\right)}\big| \\
    &\overset{(b)}{=}\sqrt{\epsilon_\eta}\bigg|\E{\left(\phi^s\left(\frac{\barz_\eta}{\tau_\eta}\right)-\phi^c\left(\frac{\barz_\eta}{\tau_\eta}\right)\right)\step(\barz_\eta)}\bigg| \\
    &\overset{(c)}{\leq} \sqrt{\epsilon_\eta}  \phi_{\max}\max\{|\step(1)|, |\step(-1)|\}=o(1),
\end{align*}
where, $(a)$ follows as $\barz_\eta^+(\step(\barz_\eta^+)-\step(\barz_\eta)) \geq 0$ by noting that $\step(x) = \phi^s(\infty)-\phi^c(\infty) > 0$ for all $x \geq 0$ and $\step(x) = \phi^s(-\infty)-\phi^c(-\infty) < 0$ for all $x < 0$. Also, $(b)$ follows by Lemma \ref{lemma: z_sgnz} and \eqref{eq:simplified_general_pricing_policy} and $(c)$ follows by \eqref{eq:bounded_control_curves}. Next, we have
\begin{align*}
    &\E{\big|(\barz_\eta^+)^2\left(\step(\barz_\eta^+)-\step(\barz_\eta)\right)\big|\mathbbm{1}\left\{|\barz_\eta^+| \geq \frac{1}{\sqrt{\epsilon_\eta}}\right\}} \\
    \overset{(a)}{\leq}{}& \E{(a^c_\eta(\barz_\eta)-a^s_\eta(\barz_\eta))^2\big|\step(\barz_\eta^+)-\step(\barz_\eta)\big|\mathbbm{1}\left\{|a^c_\eta(\barz_\eta)-a^s_\eta(\barz_\eta)| \geq \frac{1}{\sqrt{\epsilon_\eta}}\right\}} \\
    \leq{}& \left(\step(1)-\step(-1)\right)\E{(a^c_\eta(\barz_\eta)-a^s_\eta(\barz_\eta))^2\mathbbm{1}\left\{|a^c_\eta(\barz_\eta)-a^s_\eta(\barz_\eta)| \geq \frac{1}{\sqrt{\epsilon_\eta}}\right\}} \\
    \overset{(b)}{\leq}{}& \left(\step(1)-\step(-1)\right)\E{\pfix{|}a^c_\eta(\barz_\eta)-a^s_\eta(\barz_\eta)\pfix{|}^3}^{2/3}\P{|a^c_\eta(\barz_\eta)-a^s_\eta(\barz_\eta)| \geq \frac{1}{\sqrt{\epsilon_\eta}}}^{1/3} \\
    \overset{(c)}{\leq}{}& \epsilon_\eta^{1/2}\left(\step(1)-\step(-1)\right)\E{\pfix{|}a^c_\eta(\barz_\eta)-a^s_\eta(\barz_\eta)\pfix{|}^3} \overset{(d)}{=} o(1),
\end{align*}
where $(a)$ holds by considering the following two cases: if $\barz_\eta^+$ and $\barz_\eta$ have the same sign, then $\step(\barz_\eta^+)-\step(\barz_\eta) = 0$. Otherwise, if $\barz_\eta^+$ and $\barz_\eta$ have opposite signs, then, we have $|\barz_\eta^+| \leq |a^c_\eta(\barz_\eta)-a^s_\eta(\barz_\eta)|$. Next, $(b)$ follows by H\"older's inequality with $p=3/2$ and $q=3$. Further, $(c)$ follows by the Markov's inequality. Lastly, $(d)$ follows as $\E{\pfix{|}a^c_\eta(z)\pfix{|}^3} \leq \moment$ and $\E{\pfix{|}a^s_\eta(z)\pfix{|}^3} \leq \moment$ for all $\eta > 0, z \in \bbZ$. This completes the proof.} \hfill $\square$
\endproof

\section{Technical Details for the Proof of Theorem \ref{theo: profit_driven_regime_3}} 
\subsection{Proof of Theorem \ref{theo: profit_driven_regime_3} with \texorpdfstring{$\Phi^\star = \{0\}$}{}} \label{sec:theorem3_singleton}
\proof{Proof of Theorem \ref{theo: profit_driven_regime_3} (Case 3)}
For the ease of notation, we define $\phi(\cdot) = \phi^c(\cdot) - \phi^s(\cdot)$. Let $\gamma_{\eta}^\star = \inf\left\{x \geq 0: |x\phi(x)| \geq \frac{1}{\sqrt{\epsilon_\eta \tau_\eta}}\right\}$ and $\gamma_{\star, \eta} = \inf\left\{x \leq 0: |x\phi(x)| \geq \frac{1}{\sqrt{\epsilon_\eta \tau_\eta}}\right\}$. As $\phi(\cdot)$ is non-increasing and $\phi(0)=0$, $|x\phi(x)|$ is non-decreasing for $x \geq 0$ and non-increasing for $x \leq 0$. Consider an arbitrary convergent sub-sequence of $\gamma_{\eta}^\star$ and denote its limit by $\gamma_{\infty}^\star$. Now we show that $\gamma_{\infty}^\star = 0$ by contradiction. Assume $\gamma_{\infty}^\star > 0$. Then, there exists $\eta_0 > 0$ such that for all $\eta \geq \eta_0$, we have $\gamma_{\eta}^\star \geq \gamma_{\infty}^\star/2 > 0$. Now, by monotonicity of $|x\phi(x)|$ and the definition of infimum, we have for all $\eta \geq \eta_0$,
\begin{align*}
\bigg|\frac{\gamma_\infty^\star}{4} \phi\left(\frac{\gamma_\infty^\star}{4}\right)\bigg| \leq \bigg|\frac{\gamma_\eta^\star}{2} \phi\left(\frac{\gamma_\eta^\star}{2}\right)\bigg| \leq \frac{1}{\sqrt{\epsilon_\eta \tau_\eta}} \implies \bigg|\frac{\gamma_\infty^\star}{4} \phi\left(\frac{\gamma_\infty^\star}{4}\right)\bigg| = 0 \implies \gamma_{\infty}^\star = 0,
\end{align*}
where the last assertion follows as $\phi(x) \neq 0$ for all $x \neq 0$. Thus, we have $\lim_{\eta \uparrow \infty} \gamma_{\eta}^\star = 0$. Similarly, we also have $\lim_{\eta \uparrow \infty} \gamma_{\star, \eta} = 0$. Now we are ready to provide a tail bound on $\barz_\eta/\tau_\eta$. \revcolor{For simplicity, first define 
\begin{align*}
    \Delta = \min\{\mu^\star - \lambda_{\min}, \lambda_{\max}-\mu^\star, \mu^\star-\mu_{\min},\mu_{\max}-\mu^\star\} > 0.
\end{align*}
Then, as $\phi^c(x) \leq 0$ and $\phi^s(x) \geq 0$ for all $x \geq 0$ by Condition~\ref{cond: monotonicity}, we have
\begin{align*}
    \lambda_\eta(z)-\mu_\eta(z) + \frac{\drift}{\tau_\eta} &= \max\left\{\mu^\star + \epsilon_\eta \phi^c\left(\frac{z}{\tau_\eta}\right), \lambda_{\min}+ \frac{\drift}{\tau_\eta}\right\} - \min\left\{\mu^\star + \epsilon_\eta\phi^s\left(\frac{z}{\tau_\eta}\right), \mu_{\max}\right\} \\
    &\leq \max\left\{\epsilon_\eta\phi\left(\frac{z}{\tau_\eta}\right), -\Delta + \frac{\drift}{\tau_\eta}\right\} \leq \max\left\{\epsilon_\eta\phi\left(\frac{z}{\tau_\eta}\right), -\frac{\Delta}{2}\right\} \leq 0 \quad \forall z \geq 0, \numberthis \label{eq:diff_arrivals_pos}
\end{align*}
where the last inequality holds for $\eta>0$ large enough as $\tau_\eta \rightarrow \infty$. Similarly, we have
\begin{align*}
    \lambda_\eta(z)-\mu_\eta(z) + \frac{\drift}{\tau_\eta} \geq \min\left\{\epsilon_\eta\phi\left(\frac{z}{\tau_\eta}\right), \frac{\Delta}{2}\right\} \geq 0 \quad \forall z \leq 0 \numberthis \label{eq:diff_arrivals_neg}
\end{align*}
Using the above two inequalities, we get
\begin{align*}
\P{\bigg|\frac{\barz_\eta}{\tau_\eta}\bigg| \geq 2\max\left\{\gamma_{\eta}^\star, \big|\gamma_{\star, \eta}\big|, \frac{\sqrt{\epsilon_\eta}}{\Delta\sqrt{\tau_\eta}}\right\}} 
&\overset{(a)}{\leq} \P{\bigg|\frac{\barz_\eta}{\epsilon_\eta\tau_\eta}\left(\lambda_\eta(\barz_\eta)-\mu_\eta(\barz_\eta)+\frac{\drift}{\tau_\eta}\right)\bigg| \geq \frac{1}{\sqrt{\epsilon_\eta \tau_\eta}}}\\
&\overset{(b)}{\leq} \frac{1}{\sqrt{\epsilon_\eta\tau_\eta}}\E{\bigg|\barz_\eta\left(\lambda_\eta(\barz_\eta)-\mu_\eta(\barz_\eta)+\frac{\drift}{\tau_\eta}\right)\bigg|} \\
&\overset{(c)}{=} -\frac{1}{\sqrt{\epsilon_\eta\tau_\eta}}\E{\barz_\eta\left(\lambda_\eta(\barz_\eta)-\mu_\eta(\barz_\eta)+\frac{\drift}{\tau_\eta}\right)}\\
&\overset{(d)}{\leq} \frac{2\moment}{\sqrt{\epsilon_\eta\tau_\eta}} + \frac{|\drift|}{\epsilon_\eta^{0.5} \tau_\eta^{1.5}} \E{|\barz_\eta|} \\
&\overset{(e)}{\leq} \frac{2\moment}{\sqrt{\epsilon_\eta\tau_\eta}} + \frac{|\drift|}{\epsilon_\eta^{1.5} \tau_\eta^{1.5}} \frac{4\moment + 2\epsilon_\eta \tau_\eta K(2K^\beta + \delta)}{\delta} = o(1)\\
\end{align*}
where $(a)$ follows by the definition of $\gamma^\star_\eta$ and $\gamma_{\star, \eta}$ and by \eqref{eq:diff_arrivals_pos} and \eqref{eq:diff_arrivals_neg}. For example, if $z \geq 0$, then
\begin{align*}
    \frac{z}{\epsilon_\eta\tau_\eta} \left(\lambda_\eta(z)-\mu_\eta(z) + \frac{\drift}{\tau_\eta}\right) \leq \max\left\{\frac{z}{\tau_\eta} \phi\left(\frac{z}{\tau_\eta}\right), -\frac{z\Delta}{2\epsilon_\eta \tau_\eta}\right\} \leq -\frac{1}{\sqrt{\epsilon_\eta \tau_\eta}}
\end{align*}
where the first term in the max is at most $-\frac{1}{\sqrt{\epsilon_\eta \tau_\eta}}$ as $z/\tau_\eta > \gamma^\star_\eta$ and the second term is at most $-\frac{1}{\sqrt{\epsilon_\eta \tau_\eta}}$ as $z \geq 2\sqrt{\epsilon_\eta \tau_\eta}/\Delta$.
Next, $(b)$ follows by the Markov's inequality. Further, $(c)$ follows by \eqref{eq:diff_arrivals_pos} and \eqref{eq:diff_arrivals_neg}. Now, $(d)$ follows by the first equation of Proposition \ref{prop: pos_rec} and $(e)$ follows by the second equation of Proposition \ref{prop: pos_rec}. In particular, $|\lambda^\star_\eta-\mu^\star_\eta| = |d|/\tau_\eta \leq \epsilon_\eta/2$ for $\eta$ large enough as $\epsilon_\eta \tau_\eta \rightarrow \infty$. Now, by taking the limit as $\eta \uparrow \infty$, we get
\begin{align*}
\lim_{\eta \uparrow \infty}\P{\bigg|\frac{\barz_\eta}{\tau_\eta}\bigg| \geq 2\max\left\{\gamma_{\eta}^\star, \big|\gamma_{\star, \eta}\big|, \frac{2\sqrt{\epsilon_\eta}}{\Delta\sqrt{\tau_\eta}}\right\}} = 0.
\end{align*}
This completes the proof by observing that $2\max\left\{\gamma_{\eta}^\star, \big|\gamma_{\star, \eta}\big|, \frac{1}{\tau_\eta}\right\} \rightarrow 0$ as $\eta \uparrow \infty$ as $\lim_{\eta \uparrow \infty} \gamma_{\eta}^\star = 0$, $\lim_{\eta \uparrow \infty} \gamma_{\star, \eta} = 0$, and $\epsilon_\eta \tau_\eta \rightarrow \infty$.} \hfill $\square$
\endproof

\subsection{Proof of Lemma \ref{lemma: regime_3_inside} with \texorpdfstring{$\drift \neq 0$}{}}
\proof{Proof of Lemma \ref{lemma: regime_3_inside} $[d \neq 0]$} 
\label{sec: lemma_regime_3_inside_proof}
To prove the lemma, we analyze the drift of the Lyapunov function defined as follows:
\begin{align*}
    V(z) = e^{\theta z / \tau_\eta}\mathbbm{1}\left\{\frac{z}{\tau_\eta} \in (t_\star, t^\star)\right\}
\end{align*}
for $\theta =d/2\sigma^\star$ or $\theta = j\omega$ with $\omega \in \bbR$. As $|V(z)|\leq e^{|\theta|\max\{-t_\star, t^\star\}}$, its expectation in the steady state is finite. Thus, we set the drift of the above-defined test function to zero in the steady state. Define $B_\eta = \tau_\eta^{3/4}$ so that $1/B_\eta^2 = o(1/\tau_\eta)$ and $1/B_\eta^3 = o(1/\tau_\eta^2)$. Then, we have
\begin{align*}
    0 ={}& \E{\Delta V(\barz_\eta, \eta)} \\
    ={}& \E{e^{\theta \barz_\eta^+ / \tau_\eta} \mathbbm{1}\left\{\frac{\barz_\eta^+}{\tau_\eta} \in (t_\star, t^\star)\right\} - e^{\theta \barz_\eta / \tau_\eta} \mathbbm{1}\left\{\frac{\barz_\eta}{\tau_\eta} \in (t_\star, t^\star)\right\}} \\
    ={}& \E{e^{\theta  \barz_\eta^+ / \tau_\eta} \mathbbm{1}\left\{\frac{\barz_\eta^+}{\tau_\eta} \in (t_\star, t^\star)\right\} - e^{\theta (\barz_\eta + \text{clip}(a^c_\eta(\barz_\eta)-a^s_\eta(\barz_\eta), B_\eta)) / \tau_\eta} \mathbbm{1}\left\{\frac{\barz_\eta}{\tau_\eta} \in (t_\star, t^\star)\right\}} \\
    & + \E{e^{\theta (\barz_\eta + \text{clip}(a^c_\eta(\barz_\eta)-a^s_\eta(\barz_\eta), B_\eta)) / \tau_\eta} \mathbbm{1}\left\{\frac{\barz_\eta}{\tau_\eta} \in (t_\star, t^\star)\right\} - e^{\theta \barz_\eta / \tau_\eta} \mathbbm{1}\left\{\frac{\barz_\eta}{\tau_\eta} \in (t_\star, t^\star)\right\}} \\
    ={}& \E{e^{\theta  \barz_\eta^+ / \tau_\eta} \left(\mathbbm{1}\left\{\frac{\barz_\eta^+}{\tau_\eta} \in (t_\star, t^\star)\right\} - \mathbbm{1}\left\{\frac{\barz_\eta}{\tau_\eta} \in (t_\star, t^\star)\right\}\right)\mathbbm{1}\left\{|a^c_\eta(\barz_\eta)-a^s_\eta(\barz_\eta)| \leq B_\eta\right\}} \\
    & + \E{e^{\theta  \barz_\eta^+ / \tau_\eta} \mathbbm{1}\left\{\frac{\barz_\eta^+}{\tau_\eta} \in (t_\star, t^\star)\right\}\mathbbm{1}\{|a^c_\eta(\barz_\eta)-a^s_\eta(\barz_\eta)| > B_\eta\}} \\
    &- \E{e^{\theta (\barz_\eta + \text{clip}(a^c_\eta(\barz_\eta)-a^s_\eta(\barz_\eta), B_\eta)) / \tau_\eta} \mathbbm{1}\left\{\frac{\barz_\eta}{\tau_\eta} \in (t_\star, t^\star)\right\}\mathbbm{1}\{|a^c_\eta(\barz_\eta)-a^s_\eta(\barz_\eta)| > B_\eta\}} \\
    & + \E{e^{\theta (\barz_\eta + \text{clip}(a^c_\eta(\barz_\eta)-a^s_\eta(\barz_\eta), B_\eta)) / \tau_\eta} \mathbbm{1}\left\{\frac{\barz_\eta}{\tau_\eta} \in (t_\star, t^\star)\right\} - e^{\theta \barz_\eta / \tau_\eta} \mathbbm{1}\left\{\frac{\barz_\eta}{\tau_\eta} \in (t_\star, t^\star)\right\}} \\
    ={}& \E{e^{\theta  \barz_\eta^+ / \tau_\eta} \left(\mathbbm{1}\left\{\frac{\barz_\eta^+}{\tau_\eta} \in (t_\star, t^\star)\right\} - \mathbbm{1}\left\{\frac{\barz_\eta}{\tau_\eta} \in (t_\star, t^\star)\right\}\right)\mathbbm{1}\left\{|a^c_\eta(\barz_\eta)-a^s_\eta(\barz_\eta)| \leq B_\eta\right\}} \\
    & + \E{e^{\theta (\barz_\eta + \text{clip}(a^c_\eta(\barz_\eta)-a^s_\eta(\barz_\eta), B_\eta)) / \tau_\eta} \mathbbm{1}\left\{\frac{\barz_\eta}{\tau_\eta} \in (t_\star, t^\star)\right\} - e^{\theta \barz_\eta / \tau_\eta} \mathbbm{1}\left\{\frac{\barz_\eta}{\tau_\eta} \in (t_\star, t^\star)\right\}} + o\left(\frac{1}{\tau_\eta^2}\right),
    \numberthis \label{eq: drift_case3}
\end{align*}
where the last inequality follows by noting the following bounds:
\begin{align*}
    &\bigg|\E{e^{\theta (\barz_\eta + \text{clip}(a^c_\eta(\barz_\eta)-a^s_\eta(\barz_\eta), B_\eta)) / \tau_\eta} \mathbbm{1}\left\{\frac{\barz_\eta}{\tau_\eta} \in (t_\star, t^\star)\right\}\mathbbm{1}\{|a^c_\eta(\barz_\eta)-a^s_\eta(\barz_\eta)| > B_\eta\}}\bigg| \\
    \leq{}& e^{|\text{Re}(\theta)| (t^\star-t_\star + B_\eta/\tau_\eta)}\P{|a^c_\eta(\barz_\eta)-a^s_\eta(\barz_\eta)| > B_\eta} \leq e^{|\text{Re}(\theta)| (t^\star-t_\star + B_\eta/\tau_\eta)}\frac{\E{|a^c_\eta(\barz_\eta)-a^s_\eta(\barz_\eta)|^3}}{B_\eta^3} \\
    \leq{}& 6\moment e^{|\text{Re}(\theta)| (t^\star-t_\star + 1)} \frac{1}{B_\eta^3} = o \left(\frac{1}{\tau_\eta^2}\right).
\end{align*}
Similarly, we can show that
\begin{align*}
    \bigg|\E{e^{\theta  \barz_\eta^+ / \tau_\eta} \mathbbm{1}\left\{\frac{\barz_\eta^+}{\tau_\eta} \in (t_\star, t^\star)\right\}\mathbbm{1}\{|a^c_\eta(\barz_\eta)-a^s_\eta(\barz_\eta)| > B_\eta\}}\bigg| = o \left(\frac{1}{\tau_\eta^2}\right). 
\end{align*}
Now, denote the two terms in \eqref{eq: drift_case3} by $\calT_1$ and $\calT_2$. First, we consider $\calT_1$. We have
\begin{align*}
   \calT_1  
    ={}& \E{e^{\theta \barz_\eta^{+} / \tau_\eta}\left( \mathbbm{1}\left\{\frac{\barz_\eta^+}{\tau_\eta} < t^\star\right\} - \mathbbm{1}\left\{\frac{\barz_\eta}{\tau_\eta} <  t^\star\right\}\right)\mathbbm{1}\left\{|a^c_\eta(\barz_\eta)-a^s_\eta(\barz_\eta)| \leq B_\eta\right\}} \\
    &- \E{e^{\theta \barz_\eta^{+} / \tau_\eta}\left( \mathbbm{1}\left\{\frac{\barz_\eta^+}{\tau_\eta} \leq t_\star\right\}  - \mathbbm{1}\left\{\frac{\barz_\eta}{\tau_\eta} \leq  t_\star\right\}\right)\mathbbm{1}\left\{|a^c_\eta(\barz_\eta)-a^s_\eta(\barz_\eta)| \leq B_\eta\right\} } \\
     ={}& e^{\theta t^\star} \E{e^{\theta (\barz_\eta^{+} - \tau_\eta t^\star) / \tau_\eta}\left( \mathbbm{1}\left\{\frac{\barz_\eta^+}{\tau_\eta} < t^\star\right\} - \mathbbm{1}\left\{\frac{\barz_\eta}{\tau_\eta} < t^\star\right\}\right)\mathbbm{1}\left\{|a^c_\eta(\barz_\eta)-a^s_\eta(\barz_\eta)| \leq B_\eta\right\}} \\
     &- e^{\theta t_\star} \E{e^{\theta (\barz_\eta^{+} - \tau_\eta t_\star) / \tau_\eta}\left( \mathbbm{1}\left\{\frac{\barz_\eta^+}{\tau_\eta} \leq t_\star\right\} - \mathbbm{1}\left\{\frac{\barz_\eta}{\tau_\eta} \leq t_\star\right\}\right)\mathbbm{1}\left\{|a^c_\eta(\barz_\eta)-a^s_\eta(\barz_\eta)| \leq B_\eta\right\}}. \numberthis \label{eq:two_terms}
\end{align*}
Now, to analyze the two terms in \eqref{eq:two_terms}, we first use Taylor's series expansion and then show that the first-order term is dominating. We present the following claim:
\begin{claim} \label{claim: t1_regime_3_d_neq_0} For any $t \in \bbR$, let $g_t(\cdot) \in \left\{\mathbbm{1}\left\{\cdot \leq t\right\}, \mathbbm{1}\left\{\cdot < t\right\}\right\}$. Then, we have
\begin{align*}
    & \E{e^{\theta (\barz_\eta^{+} - \tau_\eta t) / \tau_\eta}\left( g_t\left(\frac{\barz_\eta^+}{\tau_\eta}\right)- g_t\left(\frac{\barz_\eta}{\tau_\eta}\right)\right)\mathbbm{1}\left\{|a^c_\eta(\barz_\eta)-a^s_\eta(\barz_\eta)| \leq B_\eta\right\}} \\
    ={}& -\frac{\theta}{\tau_\eta}\E{g_t\left(\frac{\barz_\eta}{\tau_\eta}\right) \left(\lambda_\eta(\barz_\eta)-\mu_\eta(\barz_\eta)\right)} + o\left(\frac{1}{\tau_\eta^2}\right).
\end{align*}
\end{claim}
We defer the proof of the above claim to the Appendix \ref{appendix: claim_third_regime} and continue with the proof of Lemma \ref{lemma: regime_3_inside} here. Now, we work with $\calT_2$ which is defined as follows:
\begin{align*}
    \calT_2 = \E{e^{\theta (\barz_\eta + \text{clip}(a^c_\eta(\barz_\eta)-a^s_\eta(\barz_\eta), B_\eta)) / \tau_\eta} \mathbbm{1}\left\{\frac{\barz_\eta}{\tau_\eta} \in (t_\star, t^\star)\right\} - e^{\theta \barz_\eta / \tau_\eta} \mathbbm{1}\left\{\frac{\barz_\eta}{\tau_\eta} \in (t_\star, t^\star)\right\}}. \numberthis \label{eq:definition_t2d0}
\end{align*}
We simplify $\calT_2$ in the following claim:
\begin{claim} We have \label{claim: t_2_regime_3_inside_d_neq_0}
\begin{align*}
    \calT_2 ={}& -\frac{\theta}{\tau_\eta^2}\left(\drift - \frac{\theta \sigma^\star}{2}\right)\E{ \mathbbm{1}\left\{\frac{\barz_\eta}{\tau_\eta} \in (t_\star, t^\star)\right\} e^{\theta \barz_\eta / \tau_\eta}} +  
     o\left(\frac{1}{\tau_\eta^2}\right).
\end{align*}
\end{claim}
We defer the details of the proof of the claim to Appendix \ref{appendix: claim_third_regime}. Now, we use the above two claims to substitute the expressions for $\calT_1, \calT_2$ in \eqref{eq: drift_case3}. First, note that
\begin{align*}
    &\E{\mathbbm{1}\left\{\frac{\barz_\eta}{\tau_\eta} < t^\star\right\} \left(\lambda_\eta(\barz_\eta)-\mu_\eta(\barz_\eta)\right)} \\
    ={}& \E{\mathbbm{1}\left\{\frac{\barz_\eta}{\tau_\eta} \leq t_\star\right\} \left(\lambda_\eta(\barz_\eta)-\mu_\eta(\barz_\eta)\right)} - \frac{\drift}{\tau_\eta}\P{t_\star<\frac{\barz_\eta}{\tau_\eta} < t^\star},
\end{align*}
as $\phi^c(x) = \phi^s(x) = 0$ for $x \in (t_\star, t^\star)$. Now, dividing both sides of \eqref{eq: drift_case3} by $\frac{\theta}{\tau_\eta^2}$, we get
\begin{align*}
    &\left(-\drift + \frac{\theta \sigma^\star}{2}\right)\E{ \mathbbm{1}\left\{\frac{\barz_\eta}{\tau_\eta} \in (t_\star, t^\star)\right\} e^{\theta \barz_\eta / \tau_\eta}}+  o_{\eta}\left(1\right)  \\
     ={}& \tau_\eta \E{\mathbbm{1}\left\{\frac{\barz_\eta}{\tau_\eta} \leq t_\star\right\} \left(\lambda_\eta(\barz_\eta)-\mu_\eta(\barz_\eta)\right)}\left(e^{\theta t^\star} - e^{\theta t_\star}\right)-\drift\P{t_\star < \frac{\barz_\eta}{\tau_\eta} < t^\star}e^{\theta t^\star}. \numberthis \label{eq: master_eq_d_neq_0}
\end{align*}
Now, substitute $\theta = \frac{2\drift}{\sigma^\star}$ in \eqref{eq: master_eq_d_neq_0}, we get
\begin{align*}
    &\tau_\eta \E{\mathbbm{1}\left\{\frac{\barz_\eta}{\tau_\eta} \leq t_\star\right\} \left(\lambda_\eta(\barz_\eta)-\mu_\eta(\barz_\eta)\right)}\left(e^{2\drift t^\star/
    \sigma^\star} - e^{2\drift t_\star/
    \sigma^\star}\right)= \drift\P{t_\star < \frac{\barz_\eta}{\tau_\eta} < t^\star}e^{2\drift t^\star/\sigma^\star} + o_{\eta}(1)
\end{align*}
Now, substituting the above back in \eqref{eq: master_eq_d_neq_0} along with $\theta=j\omega$ for $\omega \in \bbR$, we get
\begin{align*}
    &\left(-\drift + \frac{j\omega \sigma^\star}{2}\right)\E{ \mathbbm{1}\left\{\frac{\barz_\eta}{\tau_\eta} \in (t_\star, t^\star)\right\} e^{j\omega \barz_\eta / \tau_\eta}} + o_{\eta}\left(1\right) \\
    ={}& \drift\P{t_\star < \frac{\barz_\eta}{\tau_\eta} < t^\star}\frac{e^{j \omega t^\star} - e^{j \omega t_\star}}{e^{2\drift t^\star/
    \sigma^\star} - e^{2\drift t_\star/
    \sigma^\star}}e^{2\drift t^\star/\sigma^\star} - \drift\P{t_\star < \frac{\barz_\eta}{\tau_\eta} < t^\star} e^{j \omega t^\star} + o_{\eta}(1) \\
     ={}& \drift\left(\P{\frac{\barz_\eta}{\tau_\eta} < t^\star}e^{2\drift t^\star/\sigma^\star}- \P{\frac{\barz_\eta}{\tau_\eta} \leq t_\star}e^{2 \drift t_\star/\sigma^\star}\right)\frac{e^{j \omega t^\star} - e^{j \omega t_\star}}{e^{2\drift t^\star/
    \sigma^\star} - e^{2\drift t_\star/
    \sigma^\star}} \\
    &-\drift\P{\frac{\barz_\eta}{\tau_\eta} < t^\star}e^{j \omega t^\star}+\drift \P{\frac{\barz_\eta}{\tau_\eta} \leq t_\star}e^{j \omega t_\star} + o_{\eta}(1)
\end{align*}
This completes the proof. \hfill $\square$
\subsection{Proof of Lemma \ref{lemma: regime_3_inside} with \texorpdfstring{$\drift = 0$}{}}
\proof{Proof of Lemma \ref{lemma: regime_3_inside} $[d=0]$} \label{sec: lemma_regime_3_inside_proof_d_0}
\revcolor{To prove the lemma, we will analyze the drift of the Lyapunov function defined as follows:
\begin{align*}
    V(z, \eta) = e^{j \omega z / \tau_\eta}\mathbbm{1}\left\{\frac{z}{\tau_\eta} \in (t_\star, t^\star)\right\}
\end{align*}
for $\omega \in \bbR$. As $|V(z)|\leq 1$, its expectation in the steady state is finite. Thus, we set the drift of the above-defined test function to zero in the steady state.
\begin{align*}
    0 ={}& \E{\Delta V(\barz_\eta, \eta)} \\
    ={}& \E{e^{j \omega \barz_\eta^+ / \tau_\eta} \mathbbm{1}\left\{\frac{\barz_\eta^+}{\tau_\eta} \in (t_\star, t^\star)\right\} - e^{j \omega \barz_\eta / \tau_\eta} \mathbbm{1}\left\{\frac{\barz_\eta}{\tau_\eta} \in (t_\star, t^\star)\right\}} \\
    ={}& \underbrace{\E{e^{j \omega  \barz_\eta^+ / \tau_\eta} \mathbbm{1}\left\{\frac{\barz_\eta^+}{\tau_\eta} \in (t_\star, t^\star)\right\} - e^{j \omega \barz_\eta^+ / \tau_\eta} \mathbbm{1}\left\{\frac{\barz_\eta}{\tau_\eta} \in (t_\star, t^\star)\right\}}}_{\calT_1} \\
    & + \underbrace{\E{e^{j \omega \barz_\eta^+ / \tau_\eta} \mathbbm{1}\left\{\frac{\barz_\eta}{\tau_\eta} \in (t_\star, t^\star)\right\} - e^{\pfix{j\omega} \barz_\eta / \tau_\eta} \mathbbm{1}\left\{\frac{\barz_\eta}{\tau_\eta} \in (t_\star, t^\star)\right\}}}_{\calT_2}. \numberthis \label{eq: drift_case3_d_0}
\end{align*}
Now, we analyze the two terms - $\calT_1$ and $\calT_2$ separately. First, we consider $\calT_1$. 
\begin{align*}
   \calT_1 ={}& \E{e^{j \omega \barz_\eta^{+} / \tau_\eta}\left( \mathbbm{1}\left\{\frac{\barz_\eta^+}{\tau_\eta} < t^\star\right\} - \mathbbm{1}\left\{\frac{\barz_\eta^+}{\tau_\eta} \leq t_\star\right\} - \mathbbm{1}\left\{\frac{\barz_\eta}{\tau_\eta} <  t^\star\right\} + \mathbbm{1}\left\{\frac{\barz_\eta}{\tau_\eta} \leq  t_\star\right\}\right)} \\
    ={}& \E{e^{j \omega \barz_\eta^{+} / \tau_\eta}\left( \mathbbm{1}\left\{\frac{\barz_\eta^+}{\tau_\eta} < t^\star\right\} - \mathbbm{1}\left\{\frac{\barz_\eta}{\tau_\eta} \leq  t^\star\right\}\right)} \\
    &- \E{e^{j \omega \barz_\eta^{+} / \tau_\eta}\left( \mathbbm{1}\left\{\frac{\barz_\eta^+}{\tau_\eta} \leq t_\star\right\}  - \mathbbm{1}\left\{\frac{\barz_\eta}{\tau_\eta} \leq  t_\star\right\}\right)} \\
     ={}& e^{j \omega t^\star} \E{e^{j \omega (\barz_\eta^{+} - \tau_\eta t^\star) / \tau_\eta}\left( \mathbbm{1}\left\{\frac{\barz_\eta^+}{\tau_\eta} < t^\star\right\} - \mathbbm{1}\left\{\frac{\barz_\eta}{\tau_\eta} < t^\star\right\}\right)} \\
     &- e^{j \omega t_\star} \E{e^{j \omega (\barz_\eta^{+} - \tau_\eta t_\star) / \tau_\eta}\left( \mathbbm{1}\left\{\frac{\barz_\eta^+}{\tau_\eta} \leq t_\star\right\} - \mathbbm{1}\left\{\frac{\barz_\eta}{\tau_\eta} \leq t_\star\right\}\right)}.
\end{align*}
To analyze the above two terms, we first use Taylor's series expansion and then show that the first-order term is dominating. We present the following claim:
\begin{claim} \label{claim: t1_regime_3} For any $t \in \bbR$, let $g_t(\cdot) \in \left\{\mathbbm{1}\left\{\cdot \leq t\right\}, \mathbbm{1}\left\{\cdot < t\right\}\right\}$. Then, we have
\begin{align*}
    & \E{e^{j \omega (\barz_\eta^{+} - \tau_\eta t) / \tau_\eta}\left( g_t\left(\frac{\barz_\eta^+}{\tau_\eta}\right) - g_t\left(\frac{\barz_\eta}{\tau_\eta}\right)\right)} \\
    ={}& -\frac{j \omega}{\tau_\eta} \E{g_t\left(\frac{\barz_\eta}{\tau_\eta}\right) \left(\lambda_\eta(\barz_\eta)-\mu_\eta(\barz_\eta)\right)} + |\omega|^2 o\left(\frac{1}{\tau_\eta^2}\right) + o\left(\frac{1}{\tau_\eta^2}\right)o\left(|\omega|^2\right).
\end{align*}
\end{claim}
We defer the proof of the above claim to the Appendix \ref{appendix: claim_third_regime} and continue with the proof of Lemma \ref{lemma: regime_3_inside} here. Now, we present the following claim that characterizes $\calT_2$.
\begin{claim} We have \label{claim: t_2_regime_3_inside}
\begin{align*}
    \calT_2 ={}& -\frac{\omega^2 \sigma^\star}{2\tau_\eta^2}\E{ \mathbbm{1}\left\{\frac{\barz_\eta}{\tau_\eta} \in (t_\star, t^\star)\right\} e^{j \omega \barz_\eta / \tau_\eta}} + |\omega|^2 o\left(\frac{1}{\tau_\eta^2}\right) + 
     o\left(\frac{1}{\tau_\eta^2}\right)o\left(|\omega|^2\right).
\end{align*}
\end{claim}
We defer the details of the proof of the claim to Appendix \ref{appendix: claim_third_regime}. Now, we use the above two claims to substitute the expressions for $\calT_1, \calT_2$ in \eqref{eq: drift_case3_d_0}. First, note that $\E{\mathbbm{1}\left\{\frac{\barz_\eta}{\tau_\eta} < t^\star\right\} \left(\lambda_\eta(\barz_\eta)-\mu_\eta(\barz_\eta)\right)} = \E{\mathbbm{1}\left\{\frac{\barz_\eta}{\tau_\eta} \leq t_\star\right\} \left(\lambda_\eta(\barz_\eta)-\mu_\eta(\barz_\eta)\right)}$ as $\lambda_\eta(x)-\mu_\eta(x) = 0$ for $x \in (t_\star \tau_\eta, t^\star \tau_\eta)$ by Condition~\ref{cond: monotonicity} and $\drift=0$. Now, dividing both sides of \eqref{eq: drift_case3_d_0} by $\frac{-\omega^2\sigma^\star}{2\tau_\eta^2}$, we get
\begin{align*}
    &\E{ \mathbbm{1}\left\{\frac{\barz_\eta}{\tau_\eta} \in (t_\star, t^\star)\right\} e^{j \omega \barz_\eta / \tau_\eta}}+  o_{\eta}\left(1\right) + 
     o_{\eta}\left(1\right)o_\omega\left(1\right) \\
     ={}& -\frac{2j \tau_\eta}{\omega \sigma^\star}\E{\mathbbm{1}\left\{\frac{\barz_\eta}{\tau_\eta} \leq t_\star\right\} \left(\lambda_\eta(\barz_\eta)-\mu_\eta(\barz_\eta)\right)}\left(e^{j \omega t^\star} - e^{j \omega t_\star}\right). \numberthis \label{eq: lemma_regime_3_uniform}
\end{align*}
Note that, as we are keeping track of the order in terms of $\tau_\eta$ and $\omega$, we make it explicit by adding a sub-script whenever necessary. Now, we take the limit as $\omega \rightarrow 0^+$ on both sides. As $\big|\mathbbm{1}\left\{\frac{\barz_\eta}{\tau_\eta} \in (t_\star, t^\star)\right\} e^{j\omega \barz_\eta / \tau_\eta}\big| \leq 1$, we use dominated convergence theorem to interchange the limit and expectation to get
\begin{align}
    \P{\frac{\barz_\eta}{\tau_\eta} \in (t_\star, t^\star)} \nonumber 
    ={}& - \frac{2j\tau_\eta}{\sigma^\star} \E{\mathbbm{1}\left\{\frac{\barz_\eta }{\tau_\eta} \leq t_\star\right\} \left(\lambda_\eta(\barz_\eta)-\mu_\eta(\barz_\eta)\right)}\lim_{\omega \rightarrow 0^+}\frac{e^{j \omega t^\star} - e^{j \omega t_\star}}{\omega} +o_{\eta}(1) \nonumber \\
    ={}& \frac{2\tau_\eta}{\sigma^\star} \E{\mathbbm{1}\left\{\frac{\barz_\eta }{\tau_\eta} \leq t_\star\right\} \left(\lambda_\eta(\barz_\eta)-\mu_\eta(\barz_\eta)\right)} (t^\star - t_\star) +o_{\eta}(1). \label{eq: lemma_regime_3_omega_to_zero}
\end{align}
Substituting \eqref{eq: lemma_regime_3_omega_to_zero} in \eqref{eq: lemma_regime_3_uniform}, we get
\begin{align*}
    &\E{ \mathbbm{1}\left\{\frac{\barz_\eta}{\tau_\eta} \in (t_\star, t^\star)\right\} e^{j\omega \barz_\eta / \tau_\eta}} \\
    ={}& -\frac{j}{\omega(t^\star-t_\star)}\P{\frac{\barz_\eta}{\tau_\eta} \in (t_\star, t^\star)} \left(e^{j \omega t^\star} - e^{j \omega t_\star}\right) + o_{\eta}(1)o_{\omega}(1)+o_{\eta}(1) \\
    ={}& \frac{1}{j\omega(t^\star-t_\star)}\P{\frac{\barz_\eta}{\tau_\eta} \in (t_\star, t^\star)} \left(e^{j \omega t^\star} - e^{j \omega t_\star}\right) + o_{\eta}(1)o_{\omega}(1)+o_{\eta}(1) \\
    ={}& \frac{1}{j\omega(t^\star-t_\star)}\P{\frac{\barz_\eta}{\tau_\eta} \in (t_\star, t^\star)} \left(e^{j \omega t^\star} - e^{j \omega t_\star}\right) + o_{\eta}(1).
\end{align*}
This completes the proof for $\drift=0$.} \hfill $\square$
\subsection{Proof of Lemma \ref{lemma: regime_3_outside}}
\proof{Proof of Lemma \ref{lemma: regime_3_outside}} \label{sec: lemma_regime_3_outside_proof}
For the ease of notation, define $\phi^c(\cdot) - \phi^s(\cdot) = \phi(\cdot)$. 
Now, define $h_{\star, \eta} =  t_\star -  \sup\{x : \phi(x) \geq \frac{1}{\sqrt{\epsilon_\eta\tau_\eta}}\}$ and $h^\star_\eta = \inf\{x : \phi(x) \leq -\frac{1}{\sqrt{\epsilon_\eta\tau_\eta}}\} - t^\star$. Using Condition \ref{cond: monotonicity}, note that $h_{\star, \eta} \geq 0$ and $h^\star_\eta \geq 0$. Moreover, noting that $\epsilon_\eta \tau_\eta \rightarrow \infty$ as $\eta \uparrow \infty$, we have $\lim_{\eta \uparrow \infty} h_{\star, \eta}=0$ and $\lim_{\eta \uparrow \infty} h_{\eta}^\star=0$. Now, we proceed by dividing $\P{\frac{\barz_\eta}{\tau_\eta} \notin (t_\star, t^\star)}$ into the following two terms.
\begin{align*} 
\calT_3 &= \P{\frac{\barz_\eta}{\tau_\eta} \in  \left[t_\star - h_{\star, \eta}, t_\star\right] \cup \left[t^\star,t^\star + h^\star_\eta\right]} \\ 
\calT_4 &= \P{\frac{\barz_\eta}{\tau_\eta} \in \left(-\infty, t_\star - h_{\star, \eta}\right) \cup \left(t^\star + h^\star_\eta, \infty\right)} 
\end{align*}
Note that, $\calT_3$ comprises the region close to the thresholds. We will upper bound it using Lemma \ref{lemma: prob_bound_on_any_state}. We get
\revcolor{\begin{align*}
    \calT_3 ={}& \sum_{k = \lceil(t_\star - h_{\star, \eta})\tau_\eta\rceil}^{\lfloor t_\star \tau_\eta \rfloor} \P{\barz_\eta = k} + \sum_{k = \lceil t^\star\tau_\eta \rceil}^{\lfloor (t^\star + h^\star_\eta)\tau_\eta \rfloor} \P{\barz_\eta = k} \\
    \overset{(a)}{\leq}{}& \frac{(h_{\star, \eta} + h^\star_\eta)\tau_\eta + 2}{p_{\min}}\E{\big|\lambda_\eta(\barz_\eta)-\mu_\eta(\barz_\eta)\big|} \\
    \leq{}& \frac{(h_{\star, \eta} + h^\star_\eta)\tau_\eta + 2}{p_{\min}}\E{\bigg|\lambda_\eta(\barz_\eta)-\mu_\eta(\barz_\eta)+\frac{\drift}{\tau_\eta}\bigg|} + \frac{(h_{\star, \eta}+h^\star_\eta)|\drift|\tau_\eta+2|\drift|}{p_{\min}\tau_\eta} \\
    \overset{(b)}{\leq}{}& -\frac{(h_{\star, \eta} + h^\star_\eta)\tau_\eta + 2}{p_{\min}\tau_\eta\min\left\{-t_\star, t^\star\right\}}\E{\barz_\eta\left(\lambda_\eta(\barz_\eta)-\mu_\eta(\barz_\eta)+\frac{\drift}{\tau_\eta}\right)} + \frac{(h_{\star, \eta}+h^\star_\eta)|\drift|\tau_\eta+2|\drift|}{p_{\min}\tau_\eta} \\
    \overset{(c)}{\leq}{}& 2\moment\frac{(h_{\star, \eta} + h^\star_\eta)\tau_\eta + 2}{p_{\min}\tau_\eta\min\left\{-t_\star, t^\star\right\}} + |\drift|\frac{(h_{\star, \eta} + h^\star_\eta)\tau_\eta + 2}{p_{\min}\tau_\eta^2\min\left\{-t_\star, t^\star\right\}}\E{|\barz_\eta|} + \frac{(h_{\star, \eta}+h^\star_\eta)|\drift|\tau_\eta+2|\drift|}{p_{\min}\tau_\eta} \\
    \overset{(c)}{\leq}{}& 2\moment\frac{(h_{\star, \eta} + h^\star_\eta)\tau_\eta + 2}{p_{\min}\tau_\eta\min\left\{-t_\star, t^\star\right\}} + |\drift|\frac{(h_{\star, \eta} + h^\star_\eta)\tau_\eta + 2}{p_{\min}\tau_\eta^2\min\left\{-t_\star, t^\star\right\}}\frac{4\moment+2\epsilon_\eta \tau_\eta K (2K^\beta+\delta)}{\epsilon_\eta \delta} \\
    &+ \frac{(h_{\star, \eta}+h^\star_\eta)|\drift|\tau_\eta+2|\drift|}{p_{\min}\tau_\eta} \numberthis \label{eq:small_prob} 
\end{align*}
where $(a)$ follows by Lemma \ref{lemma: prob_bound_on_any_state} and Condition~\ref{ass: irreducibility}. In particular, the $K>0$ in Condition~\ref{ass: irreducibility} is such that $K \geq \max\{t^\star+h^\star_\eta, |t_\star-h_{\star, \eta}|\}$ for $\eta>0$ large enough. Next, $(b)$ follows as $\lambda_\eta(z)-\mu_\eta(z)+d/\tau_\eta=0$ for all $z \in (-t_\star,t^\star]$ and $\text{sgn}(z)\left(\lambda_\eta(z)-\mu_\eta(z)+d/\tau_\eta\right) \leq 0$. Lastly, $(c)$ follows by Proposition \ref{prop: pos_rec}. As $\lim_{\eta \uparrow \infty} h_{\star, \eta}=0$, $\lim_{\eta \uparrow \infty} h_{\eta}^\star=0$, and $\lim_{\eta \uparrow \infty} \epsilon_\eta \tau_\eta = \infty$, we get $\limsup_{\eta \uparrow \infty}\calT_3 \leq 0$. Next, to upper bound $\calT_4$, note that, as $\phi(\cdot)$ is monotonic, by the definition of $h^\star_\eta$ and $h_{\star, \eta}$, we have
\begin{align*}
    \phi(x) \leq -\frac{1}{\sqrt{\epsilon_\eta \tau_\eta}} \ \forall x > t^\star + h^\star_\eta, \quad
    \phi(x) \geq \frac{1}{\sqrt{\epsilon_\eta \tau_\eta}} \ \forall x < t_\star - h_{\star, \eta}.
\end{align*}
Thus, as $\phi^c(x) \leq 0$ for $x \geq 0$ and $\phi^s(x) \geq 0$ for $x \geq 0$, we have for $\eta > 0$ large enough
\begin{align*}
    \bigg|\lambda_\eta(\barz_\eta)-\mu_\eta(\barz_\eta)+\frac{\drift}{\tau_\eta}\bigg| \geq \sqrt{\frac{\epsilon_\eta}{\tau_\eta}} \quad \forall \ \frac{\barz_\eta}{\tau_\eta} \in (-\infty, t_\star-h_{\star, \eta}) \cup (t^\star+h^\star_\eta, \infty)
\end{align*}
Thus, we have
\begin{align*}
    &\P{\frac{\barz_\eta}{\tau_\eta} \in  \left(-\infty, t_\star - h_{\star, \eta}\right) \cup \left(t^\star + h^\star_\eta, \infty\right)} \\
    \leq{}& \P{\bigg|\lambda_\eta(\barz_\eta)-\mu_\eta(\barz_\eta)+\frac{\drift}{\tau_\eta}\bigg| \geq \sqrt{\frac{\epsilon_\eta}{\tau_\eta}}} 
    \overset{(a)}{\leq} \sqrt{\frac{\tau_\eta}{\epsilon_\eta}}\E{\bigg|\lambda_\eta(\barz_\eta)-\mu_\eta(\barz_\eta)+\frac{\drift}{\tau_\eta}\bigg|} \\
    \overset{(b)}{\leq}{}& -\frac{1}{\min\left\{-t_\star, t^\star\right\}\sqrt{\epsilon_\eta\tau_\eta}}\E{\barz_\eta\left(\lambda_\eta(\barz_\eta)-\mu_\eta(\barz_\eta)+\frac{\drift}{\tau_\eta}\right)} \\
    \overset{(c)}{\leq}{}& \frac{2\moment}{\min\{-t_\star, t^\star\}\sqrt{\epsilon_\eta \tau_\eta}} + \frac{|\drift|}{\min\{-t_\star, t^\star\}\sqrt{\epsilon_\eta}\tau_\eta^{3/2}} \E{|\barz_\eta|} \\
    \overset{(c)}{\leq}{}& \frac{2\moment}{\min\{-t_\star, t^\star\}\sqrt{\epsilon_\eta \tau_\eta} } +  \frac{|\drift|}{\min\{-t_\star, t^\star\}\sqrt{\epsilon_\eta\tau_\eta}} \frac{4\moment+2\epsilon_\eta \tau_\eta K (2K^\beta+\delta)}{\epsilon_\eta \tau_\eta \delta}
\end{align*}
where $(a)$ is followed by Markov's inequality. Next, $(b)$ follows as  $(\phi^c\left(x\right) - \phi^s\left(x\right))\text{sgn}(x) \leq 0$ and $\phi^c\left(x\right) - \phi^s\left(x\right) = 0$ for $t_\star < 0 \leq t^\star$. Lastly, $(c)$ follows by Proposition \ref{prop: pos_rec}.} Now, taking the limit supremum as $\eta \uparrow \infty$, we get $\limsup_{\eta \uparrow \infty} \calT_4 \leq 0$. Combining everything, we get
\begin{align*}
0 &\leq \liminf_{\eta \uparrow \infty}\P{\phi\left(\frac{\barz_\eta}{\tau_\eta}\right) \neq 0}\leq \limsup_{\eta \uparrow \infty}\P{\phi\left(\frac{\barz_\eta}{\tau_\eta}\right) \neq 0} \leq \limsup_{\eta \uparrow \infty} \calT_3 + \limsup_{\eta \uparrow \infty} \calT_4 \leq 0.
\end{align*}
This completes the proof. \hfill $\square$
\endproof
\subsection{Proof of Claims for Lemma \ref{lemma: regime_3_inside}} \label{appendix: claim_third_regime}
\proof{Proof of Claim \ref{claim: t1_regime_3_d_neq_0}}
\revcolor{Note the functional inequality: $|e^{\theta x} - 1 - \theta x - \theta^2 x^2/2| \leq |\theta x|^3 e^{|\theta x|}/6$. Then, we have
\begin{align*}
    &\bigg|\E{e^{\theta (\barz_\eta^{+} - \tau_\eta t) / \tau_\eta}\left( g_t\left(\frac{\barz_\eta^+}{\tau_\eta}\right) - g_t\left(\frac{\barz_\eta}{\tau_\eta}\right)\right)\mathbbm{1}\left\{|a^c_\eta(\barz_\eta)-a^s_\eta(\barz_\eta)| \leq B_\eta\right\}} \\
    &-\E{\left( g_t\left(\frac{\barz_\eta^+}{\tau_\eta}\right) - g_t\left(\frac{\barz_\eta}{\tau_\eta}\right)\right)\mathbbm{1}\left\{|a^c_\eta(\barz_\eta)-a^s_\eta(\barz_\eta)| \leq B_\eta\right\}} \\
    &- \frac{\theta}{\tau_\eta} \E{(\barz_\eta^{+} - \tau_\eta t)\left(g_t\left(\frac{\barz_\eta^+}{\tau_\eta}\right) - g_t\left(\frac{\barz_\eta}{\tau_\eta}\right)\right)\mathbbm{1}\left\{|a^c_\eta(\barz_\eta)-a^s_\eta(\barz_\eta)| \leq B_\eta\right\}} \\
    &- \frac{\theta^2}{2\tau_\eta^2} \E{(\barz_\eta^{+} - \tau_\eta t)^2\left( g_t\left(\frac{\barz_\eta^+}{\tau_\eta}\right) - g_t\left(\frac{\barz_\eta}{\tau_\eta}\right)\right)\mathbbm{1}\left\{|a^c_\eta(\barz_\eta)-a^s_\eta(\barz_\eta)| \leq B_\eta\right\}}\bigg| \\
    \leq{}& \frac{|\theta|^3}{6\tau_\eta^3} \E{e^{|\theta| |\barz_\eta^+-\tau_\eta t|/\tau_\eta}|\barz_\eta^{+} - \tau_\eta t|^3\bigg|g_t\left(\frac{\barz_\eta^+}{\tau_\eta}\right) - g_t\left(\frac{\barz_\eta}{\tau_\eta}\right)\bigg|\mathbbm{1}\left\{|a^c_\eta(\barz_\eta)-a^s_\eta(\barz_\eta)| \leq B_\eta\right\}} \\
    \overset{(a)}{\leq}{}& \frac{|\theta|^3}{6\tau_\eta^3} \E{e^{|\theta||a^c_\eta(\barz_\eta)-a^s_\eta(\barz_\eta)|/\tau_\eta}|a^c_\eta(\barz_\eta)-a^s_\eta(\barz_\eta)|^3\bigg|g_t\left(\frac{\barz_\eta^+}{\tau_\eta}\right) - g_t\left(\frac{\barz_\eta}{\tau_\eta}\right)\bigg|\mathbbm{1}\left\{|a^c_\eta(\barz_\eta)-a^s_\eta(\barz_\eta)| \leq B_\eta\right\}} \\
    \leq{}& \frac{|\theta|^3 e^{|\theta| B_\eta/\tau_\eta}}{6\tau_\eta^3} \E{|a^c_\eta(\barz_\eta)-a^s_\eta(\barz_\eta)|^3} \overset{(b)}{\leq} \frac{\moment |\theta|^3 e^{|\theta|}}{\tau_\eta^3} = o\left(\frac{1}{\tau_\eta^2}\right),
\end{align*}
where $(a)$ holds by noting the following: by the definition of $g_t(\cdot)$, if $\barz_\eta^+, \barz_\eta > \tau_\eta t$ or $\barz_\eta^+, \barz_\eta < \tau_\eta t$, the RHS is zero. On the other hand, if $\barz_\eta^+ \geq \tau_\eta t$ and $\barz_\eta \leq \tau_\eta t$, then, we must have $|\barz_\eta^+-\tau_\eta t| \leq |a^c(\barz_\eta)-a^s(\barz_\eta)|$. The same bound continues to hold when $\barz_\eta^+ \leq \tau_\eta t$ and $\barz_\eta \geq \tau_\eta t$. Next, $(b)$ holds as $B_\eta \leq \tau_\eta$. Now, we simplify the three terms in the Taylor's expansion below. We have
\begin{align*}
    &\E{\left( g_t\left(\frac{\barz_\eta^+}{\tau_\eta}\right) - g_t\left(\frac{\barz_\eta}{\tau_\eta}\right)\right)\mathbbm{1}\left\{|a^c_\eta(\barz_\eta)-a^s_\eta(\barz_\eta)| \leq B_\eta\right\}} \\
    \overset{(a)}{=}{}& \E{\left( g_t\left(\frac{\barz_\eta^+}{\tau_\eta}\right) - g_t\left(\frac{\barz_\eta}{\tau_\eta}\right)\right)\mathbbm{1}\left\{|a^c_\eta(\barz_\eta)-a^s_\eta(\barz_\eta)| > B_\eta\right\}} \\
    \leq{}& \P{|a^c_\eta(\barz_\eta)-a^s_\eta(\barz_\eta)| > B_\eta} \overset{(b)}{\leq} \frac{\E{|a^c_\eta(\barz_\eta)-a^s_\eta(\barz_\eta)|^3}}{B_\eta^3} \overset{(c)}{\leq} \frac{8\moment}{B_\eta^3} = o\left(\frac{1}{\tau_\eta^2}\right),
\end{align*}
where $(a)$ holds as $\E{g_t\left(\frac{\barz_\eta^+}{\tau_\eta}\right)} = \E{g_t\left(\frac{\barz_\eta}{\tau_\eta}\right)}$. Next, $(b)$ holds due to the Markov's inequality and $(c)$ by noting that $(a+b)^3 \leq 4(a^3+b^3)$ for $a, b \geq 0$ and $\E{|a^c_\eta(z)|^3}, \E{|a^s_\eta(z)|^{\pfix{3}}} \leq \moment$ for all $\eta > 0, z \in \bbZ$. Next, we have
\begin{align*}
 &\bigg|\E{(\barz_\eta^{+} - \tau_\eta t)\left( g_t\left(\frac{\barz_\eta^+}{\tau_\eta}\right) - g_t\left(\frac{\barz_\eta}{\tau_\eta}\right)\right)\mathbbm{1}\left\{|a^c_\eta(\barz_\eta)-a^s_\eta(\barz_\eta)| > B_\eta\right\}}\bigg| \\
 \overset{(a)}{\leq}{}& \E{|a^c_\eta(\barz_\eta)-a^s_\eta(\barz_\eta)|\mathbbm{1}\left\{|a^c_\eta(\barz_\eta)-a^s_\eta(\barz_\eta)| > B_\eta\right\}} \\
 \overset{(b)}{\leq}{}& \E{|a^c_\eta(\barz_\eta)-a^s_\eta(\barz_\eta)|^3}^{1/3} \P{|a^c_\eta(\barz_\eta)-a^s_\eta(\barz_\eta)| > B_\eta}^{2/3} \\
 \overset{(c)}{\leq}{}& \frac{\E{|a^c_\eta(\barz_\eta)-a^s_\eta(\barz_\eta)|^3}}{B_\eta^2} \overset{(d)}{\leq} \frac{8\moment}{B_\eta^2} = o\left(\frac{1}{\tau_\eta}\right), \numberthis \label{eq:arrival_tail}
\end{align*}
where $(a)$ holds by noting the following: if $\barz_\eta^+, \barz_\eta > \tau_\eta t$ or $\barz_\eta^+, \barz_\eta < \tau_\eta t$, then, the RHS is zero. On the other hand, if $\barz_\eta^+ \geq \tau_\eta t$ and $\barz_\eta \leq \tau_\eta t$, then, we must have $|\barz_\eta^+-\tau_\eta t| \leq |a^c(\barz_\eta)-a^s(\barz_\eta)|$. The same bound continues to hold when $\barz_\eta^+ \leq \tau_\eta t$ and $\barz_\eta \geq \tau_\eta t$. Next, $(b)$ follows by H\"older's inequality with $p=3$ and $q=3/2$. Further, $(c)$ follows by the Markov's inequality. Lastly, $(d)$ holds by noting that $(a+b)^3 \leq 4(a^3+b^3)$ for $a, b \geq 0$ and $\E{|a^c_\eta(z)|^3}, \E{|a^s_\eta(z)|^{\pfix{3}}} \leq \moment$ for all $\eta > 0, z \in \bbZ$. Thus, we get
\begin{align*}
    &\E{(\barz_\eta^{+} - \tau_\eta t)\left( g_t\left(\frac{\barz_\eta^+}{\tau_\eta}\right) - g_t\left(\frac{\barz_\eta}{\tau_\eta}\right)\right)\mathbbm{1}\left\{|a^c_\eta(\barz_\eta)-a^s_\eta(\barz_\eta)| \leq B_\eta\right\}} \\
    \overset{\eqref{eq:arrival_tail}}{=}{}& \E{(\barz_\eta^{+} - \tau_\eta t)\left( g_t\left(\frac{\barz_\eta^+}{\tau_\eta}\right) - g_t\left(\frac{\barz_\eta}{\tau_\eta}\right)\right)} + o\left(\frac{1}{\tau_\eta}\right) \\
    ={}& \E{\barz_\eta^{+} \left( g_t\left(\frac{\barz_\eta^+}{\tau_\eta}\right) - g_t\left(\frac{\barz_\eta}{\tau_\eta}\right)\right)} + o\left(\frac{1}{\tau_\eta}\right) \\
    \overset{*}{=}{}& -\E{g_t\left(\frac{\barz_\eta}{\tau_\eta}\right) \left(\lambda_\eta(\barz_\eta)-\mu_\eta(\barz_\eta)\right)}  + o\left(\frac{1}{\tau_\eta}\right),
\end{align*}
where $(*)$ follows by Lemma~\ref{lemma: z_sgnz}. Now, the third order term in the Taylor's expansion is shown to be small using the following claim:
\begin{claim} For any $t \in \bbR$, let $g_t(\cdot) \in \left\{\mathbbm{1}\left\{\cdot \leq t\right\}, \mathbbm{1}\left\{\cdot < t\right\}\right\}$. Then, we have \label{claim: t_11_regime_3_inside}
\begin{align*}
    \E{(\barz_\eta^{+} - \tau_\eta t)^2\bigg| g_t\left(\frac{\barz_\eta^+}{\tau_\eta}\right) - g_t\left(\frac{\barz_\eta}{\tau_\eta}\right)\bigg|} = o_\eta(1).
\end{align*}
\end{claim}
We defer the proof of the claim to the end of this section. With the above result, the proof is complete.} \hfill $\square$
\endproof
\proof{Proof of Claim \ref{claim: t_2_regime_3_inside_d_neq_0}} Note the following functional inequality: $$|e^{\theta x}-1-\theta x - \theta^2 x^2/2| \leq |\pfix{\theta} x|^3 e^{|\theta x|}/6.$$ Also, recall that $\calT_2$ is defined in \eqref{eq:definition_t2d0}. Thus, we have
\begin{align*}
    &\bigg|\calT_2 - \frac{\theta}{\tau_\eta}\E{ \mathbbm{1}\left\{\frac{\barz_\eta}{\tau_\eta} \in (t_\star, t^\star)\right\} e^{\theta \barz_\eta / \tau_\eta}\text{clip}(a_\eta^c(\barz_\eta) - a_\eta^s(\barz_\eta), B_\eta)} \\
    &-\frac{\theta^2}{2\tau_\eta^2}\E{ \mathbbm{1}\left\{\frac{\barz_\eta}{\tau_\eta} \in (t_\star, t^\star)\right\} e^{\theta \barz_\eta / \tau_\eta}\text{clip}(a_\eta^c(\barz_\eta) - a_\eta^s(\barz_\eta), B_\eta)^2}\bigg| \numberthis \label{eq:taylor_series_clip} \\
    \leq{}& \frac{|\theta|^3}{6\tau_\eta^3}\E{\mathbbm{1}\left\{\frac{\barz_\eta}{\tau_\eta} \in (t_\star, t^\star)\right\} e^{|\theta \barz_\eta| / \tau_\eta}|a_\eta^c(\barz_\eta) - a_\eta^s(\barz_\eta)|^3 e^{|\theta \text{ clip}(a_\eta^c(\barz_\eta) - a_\eta^s(\barz_\eta), B_\eta)|/\tau_\eta}} \\
    \leq{}& \frac{|\theta|^3 e^{|\theta| B_\eta/\tau_\eta} e^{|\theta| \max\{|t^\star|, |t_\star|\}}}{6\tau_\eta^3}\E{|a_\eta^c(\barz_\eta) - a_\eta^s(\barz_\eta)|^3} \overset{*}{\leq} \frac{4\moment |\theta|^3 e^{|\theta|} e^{|\theta| \max\{|t^\star|, |t_\star|\}}}{3\tau_\eta^3} = o\left(\frac{1}{\tau_\eta^2}\right),
\end{align*}
where $(*)$ holds as $B_\eta \leq \tau_\eta$, $(a+b)^3 \leq 4(|a|^3+|b|^3)$, and noting that $\E{|a^c_\eta(z)|^3}, \E{|a^s_\eta(z)|^3} \leq \moment$ for all $\eta > 0, z \in \bbZ$. Now we simplify the two terms on the LHS above. First, we have
\begin{align*}
    &\E{\mathbbm{1}\left\{\frac{\barz_\eta}{\tau_\eta} \in (t_\star, t^\star)\right\} e^{\theta \barz_\eta / \tau_\eta}\text{clip}(a_\eta^c(\barz_\eta) - a_\eta^s(\barz_\eta), B_\eta)} \\
    \overset{*}{=}{}&\E{\mathbbm{1}\left\{\frac{\barz_\eta}{\tau_\eta} \in (t_\star, t^\star)\right\} e^{\theta \barz_\eta / \tau_\eta}(a_\eta^c(\barz_\eta) - a_\eta^s(\barz_\eta))} +o\left(\frac{1}{\tau_\eta}\right) \\
    \overset{**}{=}{}&\E{\mathbbm{1}\left\{\frac{\barz_\eta}{\tau_\eta} \in (t_\star, t^\star)\right\} e^{\theta \barz_\eta / \tau_\eta}(\lambda_\eta(\barz_\eta) - \mu_\eta(\barz_\eta))} +o\left(\frac{1}{\tau_\eta}\right) \\
    \overset{***}{=}{}& - \frac{\drift}{\tau_\eta}\E{ \mathbbm{1}\left\{\frac{\barz_\eta}{\tau_\eta} \in (t_\star, t^\star)\right\} e^{\theta \barz_\eta / \tau_\eta}}+o\left(\frac{1}{\tau_\eta}\right), \numberthis \label{eq:taylor_series_first_term_clip}
\end{align*}
where $(*)$ follows as 
\begin{align*}
    &\bigg|\E{\mathbbm{1}\left\{\frac{\barz_\eta}{\tau_\eta} \in (t_\star, t^\star)\right\} e^{\theta \barz_\eta / \tau_\eta}\left(\text{clip}(a_\eta^c(\barz_\eta) - a_\eta^s(\barz_\eta), B_\eta)-(a_\eta^c(\barz_\eta) - a_\eta^s(\barz_\eta))\right)}\bigg| \\
    \leq{}& e^{|\theta| \max\{|t^\star|, |t_\star|\}}\E{\big|\text{clip}(a_\eta^c(\barz_\eta) - a_\eta^s(\barz_\eta), B_\eta)-(a_\eta^c(\barz_\eta) - a_\eta^s(\barz_\eta))\big|} \\
    \leq{}& e^{|\theta| \max\{|t^\star|, |t_\star|\}}\E{\big|a_\eta^c(\barz_\eta) - a_\eta^s(\barz_\eta)\big|\mathbbm{1}\{\big|a_\eta^c(\barz_\eta) - a_\eta^s(\barz_\eta)\big| > B_\eta\}} \\
    \overset{(a)}{\leq}{}&e^{|\theta| \max\{|t^\star|, |t_\star|\}} \E{\big|a_\eta^c(\barz_\eta) - a_\eta^s(\barz_\eta)\big|^3}^{1/3}\P{\big|a_\eta^c(\barz_\eta) - a_\eta^s(\barz_\eta)\big| > B_\eta}^{2/3} \\
    \overset{(b)}{\leq}{}& \frac{\E{\big|a_\eta^c(\barz_\eta) - a_\eta^s(\barz_\eta)\big|^3}}{B_\eta^2} \leq \frac{6\moment}{B_\eta^2} = o\left(\frac{1}{\tau_\eta}\right),
\end{align*}
where $(a)$ holds due to the H\"older's inequality with $p=3$ and $q=3/2$ and $(b)$ follows by the Markov's inequality. Next, $(**)$ holds using the tower property of expectation. Lastly, $(***)$ holds as $\phi^c(x)=\phi^s(x)=0$ for all $x \in (t_\star, t^\star)$ by Condition~\ref{cond: monotonicity}.
Now, we focus on simplifying the second term in the LHS of \eqref{eq:taylor_series_clip}.
\begin{align*}
    &\E{ \mathbbm{1}\left\{\frac{\barz_\eta}{\tau_\eta} \in (t_\star, t^\star)\right\} e^{\theta \barz_\eta / \tau_\eta}\text{clip}(a_\eta^c(\barz_\eta) - a_\eta^s(\barz_\eta), B_\eta)^2} \\
    \overset{(a)}{=}{}& \E{\mathbbm{1}\left\{\frac{\barz_\eta}{\tau_\eta} \in (t_\star, t^\star)\right\} e^{\theta \barz_\eta / \tau_\eta}(a_\eta^c(\barz_\eta) - a_\eta^s(\barz_\eta))^2} + o_{\eta}(1) \\
    ={}& \E{\mathbbm{1}\left\{\frac{\barz_\eta}{\tau_\eta} \in (t_\star, t^\star)\right\} e^{\theta \barz_\eta / \tau_\eta}\E{(a_\eta^c(\barz_\eta) - a_\eta^s(\barz_\eta))^2 | \barz_\eta}} + o_{\eta}(1) \\
    \overset{(b)}{=}{}& \E{\mathbbm{1}\left\{\frac{\barz_\eta}{\tau_\eta} \in (t_\star, t^\star)\right\} e^{\theta \barz_\eta / \tau_\eta}(\sigma^c(\lambda_\eta(\barz_\eta)) + \sigma^s(\mu_\eta(\barz_\eta)))} + o_{\eta}(1) \\
    \overset{(c)}{=}{}& \sigma^\star \E{\mathbbm{1}\left\{\frac{\barz_\eta}{\tau_\eta} \in (t_\star, t^\star)\right\} e^{\theta \barz_\eta / \tau_\eta}} + o_{\eta}(1) \numberthis \label{eq:taylor_series_second_term_clip}
\end{align*}
where $(a)$ holds due to the following:
\begin{align*}
    &\bigg|\E{\mathbbm{1}\left\{\frac{\barz_\eta}{\tau_\eta} \in (t_\star, t^\star)\right\} e^{\theta \barz_\eta / \tau_\eta}\left(\text{clip}(a_\eta^c(\barz_\eta) - a_\eta^s(\barz_\eta), B_\eta)^2-(a_\eta^c(\barz_\eta) - a_\eta^s(\barz_\eta))^2\right)}\bigg| \\
    \leq{}& e^{|\theta| \max\{|t^\star|, |t_\star|\}}\E{\big|\text{clip}(a_\eta^c(\barz_\eta) - a_\eta^s(\barz_\eta), B_\eta)^2-(a_\eta^c(\barz_\eta) - a_\eta^s(\barz_\eta))^2\big|} \\
    \leq{}& e^{|\theta| \max\{|t^\star|, |t_\star|\}}\E{(a_\eta^c(\barz_\eta) - a_\eta^s(\barz_\eta))^2\mathbbm{1}\{\big|a_\eta^c(\barz_\eta) - a_\eta^s(\barz_\eta)\big| > B_\eta\}} \\
    \overset{*}{\leq}{}&e^{|\theta| \max\{|t^\star|, |t_\star|\}} \E{\big|a_\eta^c(\barz_\eta) - a_\eta^s(\barz_\eta)\big|^3}^{2/3}\P{\big|a_\eta^c(\barz_\eta) - a_\eta^s(\barz_\eta)\big| > B_\eta}^{1/3} \\
    \overset{**}{\leq}{}& \frac{\E{\big|a_\eta^c(\barz_\eta) - a_\eta^s(\barz_\eta)\big|^3}}{B_\eta} \leq \frac{6\moment}{B_\eta} = o_{\eta}\left(1\right),
\end{align*}
where $(*)$ holds due to the H\"older's inequality with $p=3/2$ and $q=3$ and $(**)$ follows by the Markov's inequality.
Next, $(b)$ is obtained by computing the second moment of the difference of the arrivals as follows:
\begin{align*}
    &\E{(a_\eta^c(\barz_\eta) - a_\eta^s(\barz_\eta))^2 | \barz_\eta}\mathbbm{1}\left\{\frac{\barz_\eta}{\tau_\eta} \in (t_\star, t^\star)\right\} \\ ={}& \left(\sigma^c(\lambda_\eta(\barz_\eta)) + \sigma^s(\mu_\eta(\barz_\eta)) + \E{a_\eta^c(\barz_\eta) - a_\eta^s(\barz_\eta) | \barz_\eta}^2\right)\mathbbm{1}\left\{\frac{\barz_\eta}{\tau_\eta} \in (t_\star, t^\star)\right\} \\
    ={}& \left(\sigma^c(\lambda_\eta(\barz_\eta)) + \sigma^s(\mu_\eta(\barz_\eta)) + \frac{d^2}{\tau_\eta^2} \right)\mathbbm{1}\left\{\frac{\barz_\eta}{\tau_\eta} \in (t_\star, t^\star)\right\} \\
    ={}& \left(\sigma^c(\lambda_\eta(\barz_\eta)) + \sigma^s(\mu_\eta(\barz_\eta)) \right)\mathbbm{1}\left\{\frac{\barz_\eta}{\tau_\eta} \in (t_\star, t^\star)\right\} + o_{\eta}(1) \numberthis \label{eq:second_moment_ac_as}
\end{align*}
Lastly, $(c)$ follows by Condition \ref{cond: monotonicity}. We have $\phi^c(z)=\phi^s(z)=0$ for all $z \in (t_\star, t^\star)$, and so 
\begin{align*}
    &\bigg|\E{ \mathbbm{1}\left\{\frac{\barz_\eta}{\tau_\eta} \in (t_\star, t^\star)\right\} e^{\theta \barz_\eta / \tau_\eta}(\sigma^c(\mu^\star) + \sigma^s(\mu^\star) - \sigma^c(\lambda_\eta(\barz_\eta)) - \sigma^s(\mu_\eta(\barz_\eta)))}\bigg| \\
    \leq{}&  e^{|\theta|(t^\star-t_\star)}\E{ \bigg|\sigma^c(\mu^\star) - \sigma^c\left(\mu^\star - \frac{\drift}{\tau_\eta}\right)\bigg|\mathbbm{1}\left\{\frac{\barz_\eta}{\tau_\eta} \in (t_\star, t^\star)\right\}}  = o_{\eta}(1), \numberthis \label{eq:second_moment_convergence}
\end{align*}
where the last equality holds as $\sigma^c(\mu^\star) - \sigma^c(\mu^\star-\drift/\tau_\eta) \rightarrow 0$ as $\sigma^c(\cdot)$ is a continuous functions. This completes the proof by combining \eqref{eq:taylor_series_clip},\eqref{eq:taylor_series_first_term_clip}, and \eqref{eq:taylor_series_second_term_clip}. \hfill $\square$
\endproof

\proof{Proof of Claim \ref{claim: t1_regime_3}}
\revcolor{\begin{align*}
    &\E{e^{j \omega (\barz_\eta^{+} - \tau_\eta t) / \tau_\eta}\left( g_t\left(\frac{\barz_\eta^+}{\tau_\eta}\right) - g_t\left(\frac{\barz_\eta}{\tau_\eta}\right)\right)}  \\
    ={}&  \E{e^{\frac{j \omega}{\tau_\eta} (\barz_\eta^{+} - \tau_\eta t)\mathbbm{1}\left\{ g_t\left(\frac{\barz_\eta^+}{\tau_\eta}\right) \neq g_t\left(\frac{\barz_\eta}{\tau_\eta}\right)\right\}}\left( g_t\left(\frac{\barz_\eta^+}{\tau_\eta}\right) - g_t\left(\frac{\barz_\eta}{\tau_\eta}\right)\right)} \\
    \overset{(a)}{=}{}&  \E{ g_t\left(\frac{\barz_\eta^+}{\tau_\eta}\right) - g_t\left(\frac{\barz_\eta}{\tau_\eta}\right)} + \frac{j \omega}{\tau_\eta} \E{(\barz_\eta^{+} - \tau_\eta t)\left( g_t\left(\frac{\barz_\eta^+}{\tau_\eta}\right) - g_t\left(\frac{\barz_\eta}{\tau_\eta}\right)\right)} \\
    &- \frac{\omega^2}{2\tau_\eta^2} \E{(\barz_\eta^{+} - \tau_\eta t)^2\left( g_t\left(\frac{\barz_\eta^+}{\tau_\eta}\right) - g_t\left(\frac{\barz_\eta}{\tau_\eta}\right)\right)} + o\left(\frac{1}{\tau_\eta^2}\right)o\left(|\omega|^2\right) \\
    \overset{(b)}{=}{}&  \frac{j \omega}{\tau_\eta} \E{\barz_\eta^{+} \left( g_t\left(\frac{\barz_\eta^+}{\tau_\eta}\right) - g_t\left(\frac{\barz_\eta}{\tau_\eta}\right)\right)} - \frac{\omega^2}{2\tau_\eta^2} \E{(\barz_\eta^{+} - \tau_\eta t)^2\left( g_t\left(\frac{\barz_\eta^+}{\tau_\eta}\right) - g_t\left(\frac{\barz_\eta}{\tau_\eta}\right)\right)} \\
    &+ o\left(\frac{1}{\tau_\eta^2}\right)o\left(|\omega|^2\right)\\
     \overset{(c)}{=}{}&  \frac{j \omega}{\tau_\eta} \E{\barz_\eta^{+} \left( g_t\left(\frac{\barz_\eta^+}{\tau_\eta}\right) - g_t\left(\frac{\barz_\eta}{\tau_\eta}\right)\right)} + |\omega|^2 o\left(\frac{1}{\tau_\eta^2}\right) + o\left(\frac{1}{\tau_\eta^2}\right)o\left(|\omega|^2\right) \\
     \overset{(d)}{=}{}& -\frac{j \omega}{\tau_\eta} \E{g_t\left(\frac{\barz_\eta}{\tau_\eta}\right) \left(\lambda_\eta(\barz_\eta)-\mu_\eta(\barz_\eta)\right)} + |\omega|^2 o\left(\frac{1}{\tau_\eta^2}\right) + o\left(\frac{1}{\tau_\eta^2}\right)o\left(|\omega|^2\right),
\end{align*}
where $(a)$ follows by Taylor's series expansion up to the second order. In particular, we apply Lemma~\ref{lemma: taylor_series} by noting that $|\barz_\eta^+-\tau_\eta t|\mathbbm{1}\left\{ g_t\left(\frac{\barz_\eta^+}{\tau_\eta}\right) \neq g_t\left(\frac{\barz_\eta}{\tau_\eta}\right)\right\} \leq |a^c(\barz_\eta)-a^s(\barz_\eta)|$ which has a finite third moment. Next, $(b)$ follows by noting that $\E{ g_t\left(\frac{\barz_\eta^+}{\tau_\eta}\right) - g_t\left(\frac{\barz_\eta}{\tau_\eta}\right)} = 0$ as it is the drift of $g_t\left(\frac{\barz_\eta}{\tau_\eta}\right)$ in steady state. Now, $(c)$ follows by Claim~\ref{claim: t_11_regime_3_inside}. Lastly, $(d)$ follows by Lemma \ref{lemma: z_sgnz}. This completes the proof.} \hfill $\square$
\endproof
\proof{Proof of Claim \ref{claim: t_2_regime_3_inside}} \revcolor{We simplify $\calT_2$ as follows:
\begin{align*}
    \calT_2 ={}&\E{ \mathbbm{1}\left\{\frac{\barz_\eta}{\tau_\eta} \in (t_\star, t^\star)\right\}\left(e^{j \omega \barz_\eta^{+} / \tau_\eta} -e^{j \omega \barz_\eta / \tau_\eta}\right)} \\
    \overset{(a)}{=}{}& \E{ \mathbbm{1}\left\{\frac{\barz_\eta}{\tau_\eta} \in (t_\star, t^\star)\right\} e^{j \omega \barz_\eta / \tau_\eta}\left(e^{j \omega(a_\eta^c(\barz_\eta) - a_\eta^s(\barz_\eta)) / \tau_\eta}-1\right)} \\
    \overset{(b)}{=}{}& \frac{j \omega}{\tau_\eta}\E{ \mathbbm{1}\left\{\frac{\barz_\eta}{\tau_\eta} \in (t_\star, t^\star)\right\} e^{j \omega \barz_\eta / \tau_\eta}(a_\eta^c(\barz_\eta) - a_\eta^s(\barz_\eta))} \\
    &- \frac{\omega^2}{2\tau_\eta^2}\E{ \mathbbm{1}\left\{\frac{\barz_\eta}{\tau_\eta} \in (t_\star, t^\star)\right\} e^{j \omega \barz_\eta / \tau_\eta}(a_\eta^c(\barz_\eta) - a_\eta^s(\barz_\eta))^2} + o(|\omega|^2)o\left(\frac{1}{\tau_\eta^2}\right) \\ 
    \overset{(c)}{=}{}& \frac{j \omega}{\tau_\eta}\E{ \mathbbm{1}\left\{\frac{\barz_\eta}{\tau_\eta} \in (t_\star, t^\star)\right\} e^{j \omega \barz_\eta / \tau_\eta}\left(\lambda_\eta(\barz_\eta)-\mu_\eta(\barz_\eta)\right)} \\
     &-\frac{\omega^2}{2\tau_\eta^2}\E{ \mathbbm{1}\left\{\frac{\barz_\eta}{\tau_\eta} \in (t_\star, t^\star)\right\} e^{j \omega \barz_\eta / \tau_\eta}\E{(a_\eta^c(\barz_\eta) - a_\eta^s(\barz_\eta))^2 | \barz_\eta}} + o(|\omega|^2)o\left(\frac{1}{\tau_\eta^2}\right) \\
     \overset{(d)}{=}{}& -\frac{\omega^2}{2\tau_\eta^2}\E{ \mathbbm{1}\left\{\frac{\barz_\eta}{\tau_\eta} \in (t_\star, t^\star)\right\} e^{j \omega \barz_\eta / \tau_\eta}\E{(a_\eta^c(\barz_\eta) - a_\eta^s(\barz_\eta))^2 | \barz_\eta}} + o(|\omega|^2)o\left(\frac{1}{\tau_\eta^2}\right) \\
     \overset{(e)}{=}{}& -\frac{\omega^2}{2\tau_\eta^2}\E{ \mathbbm{1}\left\{\frac{\barz_\eta}{\tau_\eta} \in (t_\star, t^\star)\right\} e^{j \omega \barz_\eta / \tau_\eta}(\sigma^c(\lambda_\eta(\barz_\eta)) + \sigma^s(\mu_\eta(\barz_\eta)))} + o(|\omega|^2)o\left(\frac{1}{\tau_\eta^2}\right) \\
     \overset{(f)}{=}{}& -\frac{\omega^2\sigma^\star}{2\tau_\eta^2}\E{ \mathbbm{1}\left\{\frac{\barz_\eta}{\tau_\eta} \in (t_\star, t^\star)\right\} e^{j \omega \barz_\eta / \tau_\eta}}+ |\omega|^2 o\left(\frac{1}{\tau_\eta^2}\right)+
     o(|\omega|^2)o\left(\frac{1}{\tau_\eta^2}\right),
\end{align*}
where $(a)$ follows by the queue evolution equation given by \eqref{eq: imabalance_evolution}. Next, $(b)$ follows by Taylor series expansion and noting that $\bigg|\mathbbm{1}\left\{\frac{\barz_\eta}{\tau_\eta} \in (t_\star, t^\star)\right\} e^{j \omega \barz_\eta / \tau_\eta}\bigg| \leq 1$ with probability 1 and $a_\eta^c(\barz_\eta)-a_\eta^s(\barz_\eta)$ has a finite third moment. The argument is made precise by Lemma \ref{lemma: taylor_series}.  Next, $(c)$ follows by using the law of total expectation (tower property). Note that $(d)$ follows by dropping the first term as it is equal to zero by Condition~\ref{cond: monotonicity} and $\drift=0$. 
Next, $(e)$ is obtained by computing the second moment of the difference of the arrivals as in \eqref{eq:second_moment_ac_as}.
Lastly, $(f)$ follows exactly as in \eqref{eq:second_moment_convergence}.}
\hfill $\square$
\endproof
\proof{Proof of Claim \ref{claim: t_11_regime_3_inside}} \revcolor{Let $t \in \bbR$. Now, we prove the claim as follows:
\begin{align*}
    & \E{(\barz_\eta^{+} - \tau_\eta t)^2\bigg| g_t\left(\frac{\barz_\eta^+}{\tau_\eta}\right) - g_t\left(\frac{\barz_\eta}{\tau_\eta}\right)\bigg|}\\
    \leq{}& \E{(\barz_\eta^{+} - \tau_\eta t)^2\bigg| g_t\left(\frac{\barz_\eta^+}{\tau_\eta}\right) - g_t\left(\frac{\barz_\eta}{\tau_\eta}\right)\bigg|\mathbbm{1}\left\{|a^c_\eta(\barz_\eta)-a^s_\eta(\barz_\eta)| \leq \sqrt{\tau_\eta}\right\}} \\
    &+ \E{(\barz_\eta^{+} - \tau_\eta t)^2\bigg| g_t\left(\frac{\barz_\eta^+}{\tau_\eta}\right) - g_t\left(\frac{\barz_\eta}{\tau_\eta}\right)\bigg|\mathbbm{1}\left\{|a^c_\eta(\barz_\eta)-a^s_\eta(\barz_\eta)| > \sqrt{\tau_\eta}\right\}} 
\end{align*}
Now, we simplify the two terms separately below. We have
\begin{align*}
    &\E{(\barz_\eta^{+} - \tau_\eta t)^2\bigg| g_t\left(\frac{\barz_\eta^+}{\tau_\eta}\right) - g_t\left(\frac{\barz_\eta}{\tau_\eta}\right)\bigg|\mathbbm{1}\left\{|a^c_\eta(\barz_\eta)-a^s_\eta(\barz_\eta)| > \sqrt{\tau_\eta}\right\}} \\
    \overset{(a)}{\leq}{}& \E{|a^c_\eta(\barz_\eta)-a^s_\eta(\barz_\eta)|^2\mathbbm{1}\left\{|a^c_\eta(\barz_\eta)-a^s_\eta(\barz_\eta)| > \sqrt{\tau_\eta}\right\}} \\
    \overset{(b)}{\leq}{}& \E{|a^c_\eta(\barz_\eta)-a^s_\eta(\barz_\eta)|^3}^{2/3}\P{|a^c_\eta(\barz_\eta)-a^s_\eta(\barz_\eta)| > \sqrt{\tau_\eta}}^{1/3} \\
    \overset{(c)}{\leq}{}& \frac{\E{|a^c_\eta(\barz_\eta)-a^s_\eta(\barz_\eta)|^3}}{\sqrt{\tau_\eta}} \overset{(d)}{\leq} \frac{8\moment}{\sqrt{\tau_\eta}},
\end{align*}
where $(a)$ holds by noting the following: if $\barz_\eta^+, \barz_\eta > \tau_\eta t$ or $\barz_\eta^+, \barz_\eta < \tau_\eta t$, then, the RHS is zero. On the other hand, if $\barz_\eta^+ \geq \tau_\eta t$ and $\barz_\eta \leq \tau_\eta t$, then, we must have $|\barz_\eta^+-\tau_\eta t| \leq |a^c_\eta(\barz_\eta)-a^s_\eta(\barz_\eta)|$. The same bound continues to hold when $\barz_\eta^+ \leq \tau_\eta t$ and $\barz_\eta \geq \tau_\eta t$. Next, $(b)$ follows by H\"older's inequality with $p=3/2$ and $q=3$. Further, $(c)$ follows by the Markov's inequality. Lastly, $(d)$ holds by noting that $(a+b)^3 \leq 4(a^3+b^3)$ for $a, b \geq 0$ and $\E{|a^c_\eta(z)|^3}, \E{|a^s_\eta(z)|^3} \leq \moment$ for all $\eta > 0, z \in \bbZ$. Next, we have
\begin{align*}
     &\E{\bigg|(\barz_\eta^{+} - \tau_\eta t)^2\left( g_t\left(\frac{\barz_\eta^+}{\tau_\eta}\right) - g_t\left(\frac{\barz_\eta}{\tau_\eta}\right)\right) \mathbbm{1}\left\{|a^c_\eta(\barz_\eta)-a^s_\eta(\barz_\eta)| \leq \sqrt{\tau_\eta}\right\}\bigg|} \\
     \leq{}&  \sqrt{\tau_\eta} \E{\bigg|(\barz_\eta^{+} - \tau_\eta t)\left( g_t\left(\frac{\barz_\eta^+}{\tau_\eta}\right) - g_t\left(\frac{\barz_\eta}{\tau_\eta}\right)\right)\bigg|} \\
    \overset{(a)}{=}{}& -\sqrt{\tau_\eta} \E{(\barz_\eta^{+} - \tau_\eta t)\left( g_t\left(\frac{\barz_\eta^+}{\tau_\eta}\right) - g_t\left(\frac{\barz_\eta}{\tau_\eta}\right)\right)} \\
    \overset{(b)}{=}{}& -\sqrt{\tau_\eta} \E{\barz_\eta^{+} \left( g_t\left(\frac{\barz_\eta^+}{\tau_\eta}\right) - g_t\left(\frac{\barz_\eta}{\tau_\eta}\right)\right)} \\
    \overset{(c)}{=}{} & \sqrt{\tau_\eta} \E{g_t\left(\frac{\barz_\eta}{\tau_\eta}\right)\left(\lambda_\eta(\barz_\eta)-\mu_\eta(\barz_\eta)\right)}  \\
    \leq{}&  \sqrt{\tau_\eta} \E{\big|\lambda_\eta(\barz_\eta)-\mu_\eta(\barz_\eta)\big|} \\
    ={}& \sqrt{\tau_\eta} \E{\big|\lambda_\eta(\barz_\eta)-\mu_\eta(\barz_\eta)\big|\mathbbm{1}\{|\barz_\eta| < \min\{-t_\star,t^\star\}\tau_\eta\}} \\
    &+ \sqrt{\tau_\eta}\E{\big|\lambda_\eta(\barz_\eta)-\mu_\eta(\barz_\eta)\big|\mathbbm{1}\{|\barz_\eta| \geq \min\{-t_\star,t^\star\}\tau_\eta\}} \\
    \overset{(d)}{\leq}{}& \frac{|d|}{\sqrt{\tau_\eta}} + \sqrt{\tau_\eta}\E{\big|\lambda_\eta(\barz_\eta)-\mu_\eta(\barz_\eta)\big|\mathbbm{1}\{|\barz_\eta| \geq \min\{-t_\star,t^\star\}\tau_\eta\}} \\
    \overset{(e)}{\leq}{}& \frac{|d|}{\sqrt{\tau_\eta}} - \frac{1}{\min\{-t_\star,t^\star\}\sqrt{\tau_\eta}}\E{\barz_\eta\left(\lambda_\eta(\barz_\eta)-\mu_\eta(\barz_\eta)\right)} + \frac{2|d|\E{|\barz_\eta|}}{\min\{-t_\star,t^\star\}\tau_\eta^{3/2}} \\
    \overset{(f)}{\leq}{}& \frac{|d|}{\sqrt{\tau_\eta}} + \frac{2\moment}{\min\{-t_\star,t^\star\}\sqrt{\tau_\eta}} + \frac{4\moment+2\epsilon_\eta \tau_\eta K (2K^\beta+\delta)}{\epsilon_\eta \delta} \frac{2|d|}{\min\{-t_\star,t^\star\}\tau_\eta^{3/2}} = o_{\eta}(1), 
\end{align*}
where $(a)$ follows by noting that $(z^{+} - \tau_\eta \pfix{t})\left( g_t\left(\frac{\barz_\eta^+}{\tau_\eta}\right) - g_t\left(\frac{\barz_\eta}{\tau_\eta}\right)\right) \leq 0$ for all $z, z^+ \in \bbZ_+$. Next, $(b)$ follows as $\E{g_t\left(\frac{\barz_\eta^+}{\tau_\eta}\right) - g_t\left(\frac{\barz_\eta}{\tau_\eta}\right)} = 0$ in steady state. Further, $(c)$ follows by Lemma \ref{lemma: z_sgnz} and noting that $\lambda^\star_\eta - \mu^\star = -\drift/\tau_\eta$. Next, $(d)$ follows as $\phi^c(x)=\phi^s(x)=0$ for all $x \in (t_{\pfix{\star}}, t^\star)$ by Condition~\ref{cond: monotonicity}. Thus, $\lambda_\eta(z)-\mu_\eta(z) = \lambda^\star_\eta-\mu^\star = -\drift/\tau_\eta$ for all $z \in (t_\star \tau_\eta, t^\star \tau_\eta)$. Further, $(e)$ holds as $\lambda_\eta(z)-\mu_\eta(z) \geq -\frac{|d|}{\tau_\eta}$ for all $z<0$ and $\lambda_\eta(z)-\mu_\eta(z) \leq \frac{|d|}{\tau_\eta}$ for all $z\geq 0$, so $|\lambda_\eta(z)-\mu_\eta(z)||z| \leq -(\lambda_\eta(z)-\mu_\eta(z))z + 2|z\drift|/\tau_\eta$ as $|x| = -x + 2[x]^+$. Lastly, $(f)$ follows by Proposition~\ref{prop: pos_rec}.}
\endproof
\section{Precise Argument for Taylor's Theorem} \label{app:taylor_series}
\revcolor{\begin{lemma} Consider the set of random variables $\{Y_\eta \in \bbZ : \eta >0\}$ and $\{X_\eta(y) \in \bbZ: \eta>0,y \in \bbZ\}$. Also, there exists $\moment$ such that $\E{|X_\eta(Y_\eta)|^3} \leq \moment$ for all $\eta \in \bbZ$. Now, define functions $f,h : \bbR \rightarrow \bbC$ given by $f=f_1+jf_2$ and $h=h_1+jh_2$. Also, $f_1,f_2,h_1,h_2$ are thrice continuously differentiable and $\max_{x \in \bbR}|h_1'''(x)| \leq h_{\max}$, $\max_{x \in \bbR}|h_2'''(x)| \leq h_{\max}$, $\max_{x \in \bbR}|f_1(x)| \leq f_{\max}$, $\max_{x \in \bbR}|f_2(x)| \leq f_{\max}$. Then, for any $\omega, \epsilon \in \bbR$, we have \label{lemma: taylor_series}
\begin{align*}
  &\frac{1}{\epsilon^2}\E{\bigg|f(Y_\eta)h(\epsilon \omega X_\eta(Y_\eta))-f(Y_\eta)h(0)-j\epsilon \omega f(Y_\eta)h'(0)X_\eta(Y_\eta)-\frac{1}{2}\epsilon^2\omega^2f(Y_\eta)h''(0)X_\eta^2(Y_\eta)\bigg|}\nonumber \\
  \leq{}& \frac{h_{\max}f_{\max}\moment}{3}\epsilon|\omega|^3 \quad \textit{w.p. } 1. \numberthis \label{eq: taylor_series}
\end{align*}
With a slight abuse of notation, we write
\begin{align*}
    f(Y_\eta)h(\epsilon \omega X_\eta(Y_\eta))=f(Y_\eta)h(0)+j\epsilon \omega f(Y_\eta)h'(0)X_\eta(Y_\eta)+\frac{1}{2}\epsilon^2\omega^2f(Y_\eta)h''(0)X_\eta^2(Y_\eta)+o(\epsilon^2).
\end{align*}
\end{lemma}
\proof{Proof}
We analyze the real and imaginary parts separately. By Taylor's Theorem, for all $x \in \bbR$ we have
\begin{align*}
    h_i(x)=h_i(0)+h_i'(0)x+\frac{1}{2}h_i''(0)x^2+\frac{1}{6}h_i'''(\tilde{x})x^3 \quad \textit{for some } \tilde{x} \in (-|x|, |x|) \ \forall i \in \{1,2\}.
\end{align*}
Thus
\begin{align*}
    \bigg|h_i(x)-h_i(0)-h_i'(0)x-\frac{1}{2}h_i''(0)x^2\bigg|\leq \frac{h_{\max}}{6}|x|^3 \quad \forall x \in \bbR, \forall i \in \{1,2\}.
\end{align*}
Now, as $|f_i(y)| \leq f_{\max}$ for $i \in \{1, 2\}$, we get
\begin{align*}
    \bigg|f_i(y)\left(h_i(x)-h_i(0)-h_i'(0)x-\frac{1}{2}h_i''(0)x^2\right)\bigg|\leq \frac{h_{\max}f_{\max}}{6}|x|^3 \quad \forall x, y \in \bbR, \forall i \in \{1,2\}.
\end{align*}
Now, substitute $x$ by $\epsilon \omega X_\eta(Y_\eta)$ and $y$ by $Y_\eta$ and taking expectation on both sides, we get
\begin{align*}
   &\E{\bigg|f_i(Y_\eta)\left(h_i(\epsilon \omega X_\eta(Y_\eta))-h_i(0)-h_i'(0)\epsilon \omega X_\eta(Y_\eta)-\frac{1}{2}h_i''(0)\left(\epsilon \omega X_\eta(Y_\eta)\right)^2\right)\bigg|}\\
   \leq{}& \frac{h_{\max}f_{\max} \epsilon^3 |\omega|^3}{6}\E{|X_\eta(Y_\eta)|^3} \leq \frac{h_{\max}f_{\max} \moment \epsilon^3 |\omega|^3}{6} \quad \forall i \in \{1,2\}
\end{align*}
Now, by adding $(i,k)=(1,1)$ and $(i,k)=(2,2)$ and using triangle inequality will give bound on the real part of \eqref{eq: taylor_series} and by adding $(i,k)=(1,2)$ and $(i,k)=(2,1)$ and using triangle inequality will give bound on the imaginary part of \eqref{eq: taylor_series}. Now, by using the fact that for any vector $x$, $\|x\|_1 \geq \|x\|_2$, we get the lemma.} \hfill $\square$
\endproof
\section{Classical Single Server Queue} \label{app: classical_queue}
\subsection{Proof of Proposition \ref{prop: single_server_queue}}
\proof{Proof}
The transition probabilities for the underlying DTMC governing the single server queue operating under the two-price policy are given by
\begin{align*}
    P_{ij}=\begin{cases}
    m &\textit{if } j = i \pm 1, 0<i \leq \lfloor \tau_\eta \rfloor \textit{ or } (i,j)=(0,1) \\
    1-2m &\textit{if } j=i, 0<i \leq \lfloor \tau_\eta \rfloor \\
    1-m &\textit{if } j=i=0 \\
    m+\epsilon_\eta\mu^\star &\textit{if } j=i-1, i>\lfloor \tau_\eta \rfloor \\
    m-\epsilon_\eta(1-\mu^\star) &\textit{if } j=i+1, i>\lfloor \tau_\eta \rfloor \\
    1-2m+\epsilon_\eta(1-2\mu^\star) &\textit{if } j=i, i >\lfloor \tau_\eta \rfloor \\
    0 &\textit{otherwise}.
    \end{cases}
\end{align*}
As the DTMC is a birth and death process, it is reversible. Thus, we can solve the local balance equations to get:
\begin{align*}
    \pi_i=\begin{cases}
    \frac{1}{\lfloor\tau_\eta\rfloor+1+m/\epsilon_\eta} &\textit{if } i \leq \lfloor\tau_\eta\rfloor \\
    \frac{1}{\lfloor\tau_\eta\rfloor+1+m/\epsilon_\eta}\frac{m}{m+\epsilon\mu^\star}\left(1-\frac{\epsilon_\eta}{m+\epsilon_\eta \mu^\star}\right)^{i-\lfloor\tau_\eta\rfloor-1} &\textit{if } i>\lfloor\tau_\eta\rfloor.
    \end{cases}
\end{align*}
Note that for large $\eta$, the stationary distribution of $\bar{q}_\eta$ is almost the same as the stationary distribution of $|\barz_\eta|$. Indeed, the proof follows exactly as the proof of Proposition \ref{prop: bernoulli}. We present it below for completeness. 

We calculate the moment generating function of $\epsilon_\eta \barq_\eta$ for Case 1 and 2 and $\barq_\eta/\tau_\eta$ for Case 3. We fix a $t$ such that $t<1/m$. Then, there exists an $\epsilon_0>0$ such that for all $\epsilon_\eta < \epsilon_0$, we have $(1-\epsilon_\eta/(m+\epsilon \mu^\star))e^{\epsilon_\eta t}<1$. Now, we calculate the MGF as follows:
\begin{align*}
    \E{e^{\epsilon_\eta t \barq_\eta}}={}&\sum_{i=0}^\infty \pi_i e^{\epsilon_\eta t i} \\
    ={}& \sum_{i=0}^{\lfloor\tau_\eta\rfloor} \pi_i e^{\epsilon_\eta t i} + \sum_{i=\lfloor\tau_\eta\rfloor+1}^\infty \pi_i e^{\epsilon_\eta t i} \\
    ={}&\frac{1}{\lfloor\tau_\eta\rfloor+1+m/\epsilon_\eta}\left(\sum_{i=0}^{\lfloor\tau_\eta\rfloor} e^{\epsilon_\eta t i}+\frac{m}{m+\epsilon_\eta \mu^\star}\sum_{i=\lfloor\tau_\eta\rfloor+1}^\infty \left(1-\frac{\epsilon_\eta}{m+\epsilon_\eta \mu^\star}\right)^{i-\lfloor\tau_\eta\rfloor-1}e^{\epsilon_\eta t i} \right) \\
    ={}&\frac{\epsilon_\eta}{\lfloor\tau_\eta\rfloor\epsilon_\eta+\epsilon_\eta+m}\left(\frac{1-e^{\epsilon_\eta t(\lfloor\tau_\eta\rfloor+1)}}{1-e^{\epsilon_\eta t}} +\frac{m}{m+\epsilon_\eta\mu^\star}\frac{e^{\epsilon_\eta t (\lfloor\tau_\eta\rfloor+1)}}{1-e^{\epsilon_\eta t}+\frac{\epsilon_\eta e^{\epsilon_\eta t}}{m+\epsilon_\eta \mu^\star}}\right) \numberthis \label{eq: mgf_bernoulli_single_server_queue}
\end{align*}
Now, we take the limit as $\eta \uparrow \infty$ for the above equation for Case 1 and 2 separately.

\textbf{Case 1:} Similar to the proof of Proposition \ref{prop: bernoulli}, we have
\begin{subequations}
\label{eq: case1_bernoulli_single_server_queue}
\begin{align}
    \epsilon_\eta\frac{1-e^{\epsilon_\eta t(\lfloor\tau_\eta\rfloor+1)}}{1-e^{\epsilon_\eta t}}&=\epsilon_\eta\frac{
     -\epsilon_\eta t(\lfloor\tau_\eta\rfloor+1)  + o(\epsilon_\eta \tau_\eta)}{-\epsilon_\eta t+o(\epsilon_\eta)}= \frac{ -\epsilon_\eta t(\lfloor\tau_\eta\rfloor+1)  + o(\epsilon_\eta \tau_\eta)}{-t+o(1)} \rightarrow 0  \\
    \epsilon_\eta\frac{e^{\epsilon_\eta t (\lfloor\tau_\eta\rfloor+1)}}{1-e^{\epsilon_\eta t}+\frac{\epsilon_\eta e^{\epsilon_\eta t}}{m+\epsilon_\eta \mu^\star}}&=\epsilon_\eta \frac{e^{\epsilon_\eta t (\lfloor\tau_\eta\rfloor+1)}}{-\epsilon_\eta t +\frac{\epsilon_\eta}{m+\epsilon_\eta \mu^\star}+o(\epsilon_\eta)}= \frac{e^{\epsilon_\eta t (\lfloor\tau_\eta\rfloor+1)}}{- t +\frac{1}{m+\epsilon_\eta \mu^\star}+o(\epsilon_\eta)} \rightarrow \frac{1}{1/m-t} 
\end{align}
\end{subequations}
Now, by using \eqref{eq: case1_bernoulli_single_server_queue} to simplify \eqref{eq: mgf_bernoulli_single_server_queue} we get
\begin{align*}
    \E{e^{\epsilon_\eta t \barq_\eta}} \rightarrow \frac{1}{1-tm} \quad \textit{as } \eta \uparrow \infty.
\end{align*}
This completes Case 1 as the above is the MGF of an exponential distribution with parameter $(\pfix{1/m})$.

\textbf{Case 2:} Similar to the proof of Proposition \ref{prop: bernoulli}, we have
\begin{subequations}
\label{eq: case2_bernoulli_single_server_queue}
\begin{align}
    \epsilon_\eta\frac{1-e^{\epsilon_\eta t(\lfloor\tau_\eta\rfloor+1)}}{1-e^{\epsilon_\eta t}}&= \epsilon_\eta\frac{1-e^{\epsilon_\eta t(\lfloor\tau_\eta\rfloor+1)}}{-\epsilon_\eta t+o(\epsilon_\eta)}=\frac{1-e^{\epsilon_\eta t(\lfloor\tau_\eta\rfloor+1)}}{- t+o(1)} \rightarrow \frac{e^{l t}-1}{t} \\
    \epsilon_\eta\frac{e^{\epsilon_\eta t (\lfloor\tau_\eta\rfloor+1)}}{1-e^{\epsilon_\eta t}+\frac{\epsilon_\eta e^{\epsilon_\eta t}}{m+\epsilon_\eta \mu^\star}}&=\epsilon_\eta \frac{e^{\epsilon_\eta t (\lfloor\tau_\eta\rfloor+1)}}{-\epsilon_\eta t +\frac{\epsilon_\eta}{m+\epsilon_\eta \mu^\star}+o(\epsilon_\eta)}= \frac{e^{\epsilon_\eta t (\lfloor\tau_\eta\rfloor+1)}}{- t +\frac{1}{m+\epsilon_\eta \mu^\star}+o(\epsilon_\eta)} \rightarrow \frac{e^{l t}}{1/m-t} 
\end{align}
\end{subequations}
Now, by using \eqref{eq: case2_bernoulli_single_server_queue} to simplify \eqref{eq: mgf_bernoulli_single_server_queue} we get
\begin{align*}
    \E{e^{\epsilon_\eta t \barq_\eta}} \rightarrow \frac{1}{l+m}\left(\frac{e^{l t}-1}{t}+\frac{e^{lt}}{1/m-t}\right)
\end{align*}
The above is the MGF of the non-negative hybrid distribution with parameters $(\pfix{1/m},l)$. This can be verified similarly to the proof of Proposition \ref{prop: bernoulli}. We skip the steps here as they are repetitive.
This completes the proof for Case 2.

\textbf{Case 3:} For large enough $\eta$, the MGF of $\barq/\tau_\eta$ can be calculated similar to \eqref{eq: mgf_bernoulli_single_server_queue} to get
\begin{align*}
    \E{e^{t\frac{\barq_\eta}{\tau_\eta}}}={}&\frac{\epsilon_\eta}{\lfloor\tau_\eta\rfloor\epsilon_\eta+\epsilon_\eta+m}\left(\frac{1-e^{ t(\lfloor\tau_\eta\rfloor+1)/\tau_\eta}}{1-e^{t/\tau_\eta}} +\frac{m}{m+\epsilon_\eta\mu^\star}\frac{e^{t (\lfloor\tau_\eta\rfloor+1)/\tau_\eta}}{1-e^{t/\tau_\eta}+\frac{\epsilon_\eta e^{t/\tau_\eta}}{m+\epsilon_\eta \mu^\star}}\right) \numberthis \label{eq: case3_mgf_bernoulli_single_server_queue}
\end{align*}
Similar to the proof of Proposition \ref{prop: bernoulli}, we have
\begin{subequations}
\begin{align}
    \frac{\epsilon_\eta}{\lfloor \tau_\eta \rfloor \epsilon_\eta+\epsilon_\eta+m}\frac{e^{ \frac{t}{\tau_\eta} (\lfloor \tau_\eta \rfloor+1)}}{1-e^{ t/\tau_\eta}+\frac{\epsilon_\eta e^{ t/\tau_\eta}}{m+\epsilon_\eta \mu^\star}} &\rightarrow 0 \\
    \frac{\epsilon_\eta}{\lfloor \tau_\eta \rfloor \epsilon_\eta+\epsilon_\eta+m} \frac{1-e^{t(\lfloor \tau_\eta \rfloor+1)/\tau_\eta}}{1-e^{t/\tau_\eta}}&=\frac{\epsilon_\eta}{\lfloor \tau_\eta \rfloor\epsilon_\eta+\epsilon_\eta+m} \frac{1-e^{t(\lfloor \tau_\eta \rfloor+1)/\tau_\eta}}{-t/\tau_\eta+o(1/\tau_\eta)}\\
    &=\frac{1-e^{t(\lfloor \tau_\eta \rfloor+1)/\tau_\eta}}{-t\lfloor \tau_\eta \rfloor/\tau_\eta+o(1)} \rightarrow \frac{e^t-1}{t}
\end{align}
\label{eq: case3_bernoulli_single_server_queue}
\end{subequations}
Now, by using \eqref{eq: case3_bernoulli_single_server_queue} in \eqref{eq: case3_mgf_bernoulli_single_server_queue}, we get
\begin{align*}
    \E{e^{\pfix{t}\frac{\barq_\eta}{\tau_\eta}}} \rightarrow \frac{e^t-1}{t} \Rightarrow \frac{\barq_\eta}{\tau_\eta} \overset{D}{\rightarrow} \calU([0,1]).
\end{align*}
This completes the proof of the theorem. \hfill $\square$
\endproof
\subsection{Preliminaries: Single Server Queue}
Similar to Proposition \ref{prop: pos_rec}, we can show that if the functions $\phi^c(\cdot), \phi^s(\cdot)$ satisfy Condition \ref{ass: classical_neg_drift}, then the single server queue is positive recurrent. We omit the details of this proof as it is intuitive and the steps are analogous to the proof of Proposition \ref{prop: pos_rec}. 
To prove the theorem, we will first introduce the following notation: $\barq_\eta$ is the random variable that has the same distribution as the stationary distribution of $\{q_\eta(k): k \in \bbZ_+\}$ and $\barq_\eta^+$ is the queue length one-time epoch after $\barq_\eta$. In addition, given the steady-state queue length $\barq_\eta$, let the arrival, potential service, and unused service be denoted by $\bara_\eta, \bars_\eta$, and $\baru_\eta$ respectively. We omit the dependence of $\barq_\eta$ on $\bara_\eta,\bars_\eta,\baru_\eta$ for simplicity of notation. Thus, we have
\begin{align}
    \barq_\eta^+=\barq_\eta+\bara_\eta-\bars_\eta+\baru_\eta. \label{eq: queue_evolution}
\end{align}
\subsection{Proof of Theorem \ref{theo: classical_general}: Case I}
\proof{Proof of Theorem \ref{theo: classical_general}: $\epsilon_\eta\tau_\eta \rightarrow 0$}
We will use $e^{j \omega \epsilon_\eta\barq_\eta}$ as the test function to prove the theorem. Firstly, note that 
\begin{align*}
    \left(e^{j \omega \epsilon_\eta \barq_\eta^+}-1\right)\left(e^{-j \omega \epsilon_\eta \baru_\eta}-1\right)=0
\end{align*}
as $\barq_\eta^+>0$ implies that $\baru_\eta=0$ and vice versa. Expanding the above expression and using the queue evolution equation results in
\begin{align}
    e^{j \omega \epsilon_\eta \barq_\eta^+}-e^{j \omega \epsilon_\eta (\barq_\eta+\bara_\eta-\bars_\eta)}=1-e^{-j\omega \epsilon_\eta \baru_\eta}.\label{eq: functional_identity_single_server_queue}
\end{align}
Now, by taking the expectation with respect to the stationary distribution of $\{q_\eta(k): k \in \bbZ_+\}$ and noting that $\E{e^{j\omega \epsilon_\eta\barq_\eta}}=\E{e^{j\omega\epsilon_\eta\barq_\eta^+}}$ by stationarity, we get
\begin{align}
    \E{e^{j \omega \epsilon_\eta \barq_\eta}-e^{j \omega \epsilon_\eta (\barq_\eta+\bara_\eta-\bars_\eta)}}=1-\E{e^{-j\omega \epsilon_\eta \baru_\eta}}. \label{eq: classical_drift}
\end{align}
The LHS of the above equation can be simplified to get
\begin{align*}
    \lefteqn{\E{e^{j \omega \epsilon_\eta \barq_\eta}-e^{j \omega \epsilon_\eta (\barq_\eta+\bara_\eta-\bars_\eta)}}}\\
    ={}&\E{ e^{j \omega \epsilon_\eta \barq_\eta}\left(1-e^{j \omega \epsilon_\eta (\bara_\eta-\bars_\eta)}\right)} \\
    \overset{(a)}{=}{}&\E{ e^{j \omega \epsilon_\eta \barq_\eta}\left(-j\omega \epsilon_\eta (\bara_\eta-\bars_\eta)+\frac{\omega^2\epsilon_\eta^2}{2}(\bara_\eta-\bars_\eta)^2)\right)}+o(\epsilon_\eta^2) \\
    \overset{(b)}{=}{}& \E{ e^{j \omega \epsilon_\eta \barq_\eta}\E{-j\omega \epsilon_\eta (\bara_\eta-\bars_\eta)+\frac{\omega^2\epsilon_\eta^2}{2}(\bara_\eta-\bars_\eta)^2)\bigg| \barq_\eta}}+o(\epsilon_\eta^2) \\
    \overset{(c)}{=}{}& \E{ e^{j \omega \epsilon_\eta \barq_\eta}\left(j\omega \epsilon_\eta^2 \left(\phi^s\left(\frac{
    \barq_\eta}{\tau_\eta}\right)-\phi^c\left(\frac{
    \barq_\eta}{\tau_\eta}\right)\right)+\frac{\omega^2\epsilon_\eta^2}{2}\left(\pfix{\sigma^c}(\lambda_\eta(\barq_\eta))+\pfix{\sigma^s}(\mu_\eta(\barq_\eta))\right)\right)}+o(\epsilon_\eta^2), \numberthis \label{eq: classical_LHS}
\end{align*}
where $(a)$ follows from Taylor's theorem and the reader can refer to Lemma \ref{lemma: taylor_series} for precise argument. Next, $(b)$ follows by the tower property of expectation, and lastly, $(c)$ follows by the definition of queue length dependent expectation and variance of $\bara_\eta$ and $\bars_\eta$.

Now, to simplify the RHS of \eqref{eq: classical_drift}, first consider $\barq_\eta$ as a test function and using $\E{\barq_\eta^+}=\E{\barq_\eta}$ will give us
\begin{align}
    \E{\baru_\eta}=\E{\bars_\eta-\bara_\eta}=\E{\E{\bars_\eta-\bara_\eta \big | \barq_\eta}}=\epsilon_\eta\E{\phi^s\left(\frac{
    \barq_\eta}{\tau_\eta}\right)-\phi^c\left(\frac{
    \barq_\eta}{\tau_\eta}\right)} \label{eq: expectation_unused_service}
\end{align}
Now, we can simplify the RHS using Taylor's theorem (Lemma \ref{lemma: taylor_series}) to get
\begin{align}
    1-\E{e^{-j\omega \epsilon_\eta \baru_\eta}}=j\omega \epsilon_\eta \E{\baru_\eta}+o(\epsilon_\eta^2)=j\omega \epsilon_\eta^2 \E{\phi^s\left(\frac{
    \barq_\eta}{\tau_\eta}\right)-\phi^c\left(\frac{
    \barq_\eta}{\tau_\eta}\right)}+o(\epsilon_\eta^2). \label{eq: classical_RHS}
\end{align}
Now, using \eqref{eq: classical_LHS} and \eqref{eq: classical_RHS}, we get
\begin{align}
    &\E{ e^{j \omega \epsilon_\eta \barq_\eta}\left(j\omega \epsilon_\eta^2 \left(\phi^s\left(\frac{
    \barq_\eta}{\tau_\eta}\right)-\phi^c\left(\frac{
    \barq_\eta}{\tau_\eta}\right)\right)+\frac{\omega^2\epsilon_\eta^2}{2}\left(\pfix{\sigma^c}(\lambda_\eta(\barq_\eta))+\pfix{\sigma^s}(\mu_\eta(\barq_\eta))\right)\right)}\nonumber\\
    ={}&j\omega \epsilon_\eta^2 \E{\phi^s\left(\frac{
    \barq_\eta}{\tau_\eta}\right)-\phi^c\left(\frac{
    \barq_\eta}{\tau_\eta}\right)}+o(\epsilon_\eta^2). \label{eq: classical_drift_zero}
\end{align}
Dividing on both sides by $j \omega \epsilon_\eta^2$ and rearranging the terms, we get
\begin{align*}
    &\left(1-\frac{j \omega}{2}\left(\sigma^c(\lambda^\star)+\sigma^s(\mu^\star)\right)\right)\E{e^{\pfix{j}\epsilon_\eta \omega \barq_\eta}} \\
    ={}&\underbrace{\left(\E{\phi^s\left(\frac{
    \barq_\eta}{\tau_\eta}\right)-\phi^c\left(\frac{
    \barq_\eta}{\tau_\eta}\right)}\right)}_{\calT_8}-\underbrace{\E{e^{\pfix{j}\epsilon_\eta \pfix{\omega} \barq_\eta}\left(\phi^s\left(\frac{
    \barq_\eta}{\tau_\eta}\right)-\phi^c\left(\frac{
    \barq_\eta}{\tau_\eta}\right)-1\right)}}_{\calT_9}\\
    &+\underbrace{\frac{j \omega}{2}\E{e^{j \omega \epsilon_\eta \barq_\eta} \left(\pfix{\sigma^c}(\lambda_\eta(\barq_\eta))+\pfix{\sigma^s}(\mu_\eta(\barq_\eta))-\sigma^c(\lambda^\star)-\sigma^s(\mu^\star)\right)}}_{\calT_{10}}+o(1).
\end{align*}
The terms $\calT_8, \calT_9, \calT_{10}$ are analogous to the terms $\calT_3, \calT_4, \calT_5$ in the proof of Theorem \ref{theo: case_1}. It can be similarly shown using Lemma \ref{claim: case1_t3_t5} and the fact that $\phi^s(\infty)-\phi^c(\infty)=1$ and $\barq_\eta \geq 0$, we get $\calT_8 \rightarrow 1$, $\calT_9 \rightarrow 0$. In addition, by using Lemma \ref{lemma: convergence_technical}, we get $\calT_{10} \rightarrow 0$. Thus, by taking the limit as $\eta  \uparrow \infty$, we get 
\begin{align*}
    \lim_{\eta \rightarrow \infty}\E{e^{j \omega \epsilon_\eta \barq_\eta}}=\frac{1}{1-\frac{j \omega}{2}\left(\sigma^c(\lambda^\star)+\sigma^s(\mu^\star)\right)}.
\end{align*}
This implies that $\epsilon_\eta \barq_\eta$ converges to a random variable in distribution having an exponential distribution with mean $0.5(\sigma^c(\lambda^\star)+\sigma^s(\mu^\star))$. \hfill $\square$
\endproof
\subsection{Proof of Theorem \ref{theo: classical_general}: Case II}
\proof{Proof of Theorem \ref{theo: classical_general}: $\epsilon_\eta\tau_\eta \rightarrow l$}
The first step is to show the tightness of the family of random variables $\{\epsilon_\eta \barq_\eta\}$ which can be proved similar to the proof of Lemma \ref{lemma: tightness}. We omit the details here for brevity. Now, the next step is to analyze the drift of the characteristic function of $\barq_\eta$ and then obtain a functional equation by taking the limit as $\eta \uparrow \infty$. Recall that by \eqref{eq: classical_drift_zero}, we have
\begin{align*}
    &\E{ e^{j \omega \epsilon_\eta \barq_\eta}\left( \left(\phi^s\left(\frac{
    \barq_\eta}{\tau_\eta}\right)-\phi^c\left(\frac{
    \barq_\eta}{\tau_\eta}\right)\right)-\frac{j\omega}{2}\left(\pfix{\sigma^c}(\lambda_\eta(\barq_\eta))+\pfix{\sigma^s}(\mu_\eta(\barq_\eta))\right)\right)}\\
    ={}& \E{\phi^s\left(\frac{
    \barq_\eta}{\tau_\eta}\right)-\phi^c\left(\frac{
    \barq_\eta}{\tau_\eta}\right)}+o(1). 
\end{align*}
By taking the limit as $\eta \uparrow \infty$ and following the steps same as in the proof of Claim \ref{claim: case2_limiting_functional_equation}, we get
\begin{align}
    \E{e^{j \omega \qinfty} g_{1,l}(\qinfty)}-j\omega \E{e^{j \omega \qinfty}}=\E{g_{1,l}(\qinfty)} \overset{\Delta}{=} C_1. \label{eq: classical_functional_eq}
\end{align}
Note that, this is the same as the functional equation \eqref{eq:IFT_eqn} if the RHS is zero. In particular, we get a non-zero RHS because of the unused service $\baru_\eta$ which is responsible for ensuring that the queue length is nonnegative. As $\qinfty \geq 0$ with probability 1, we can define $g_{1,l}(x)$ for $x \leq 0$ arbitrarily. In particular, we consider an extension of $g_{1,l}$ such that it belongs to $\calC_{pol}^\infty$. With a little abuse of notation, we denote the extension by $g_{1,l}$ as well. Now, we write \eqref{eq: classical_functional_eq} as a differential equation in the space of generalized functions (tempered distributions) which is analogous to taking the inverse Fourier transform. Define the following tempered distribution
\begin{align*}
    \calT_{\qinfty}[\varphi]=\int_{\bbR} \varphi(x) dF_{\qinfty}(x)=\E{\varphi} \quad \forall \varphi \in \calS(\bbR).
\end{align*}
Now, by using \eqref{eq: lhs_mgf_equation} and \eqref{eq: rhs_mgf_equation}, the LHS of \eqref{eq: classical_functional_eq} can be written as follows:
\begin{align*}
   \E{e^{j \omega \qinfty} g_{1,l}(\qinfty)}-j\omega \E{e^{j \omega \qinfty}}=\sqrt{2 \pi} \calT_{\qinfty}[g\hat{\varphi} - \hat{\varphi}^\prime].
\end{align*}
The RHS can be written in terms of the impulse distribution $\delta(\cdot)$. In particular, denote the constant tempered distribution by $\calT_c$, i.e. $\calT_c[\varphi]=\int_{-\infty}^\infty \varphi(x) dx$. Then, we have
\begin{align*}
    \int_{-\infty}^\infty \varphi(x) dx=\calT_c[\varphi]\overset{(a)}{=}\hat{\calT_c}[\hat{\varphi}]\overset{(b)}{=}\delta[\hat{\varphi}],
\end{align*}
where $(a)$ follows by \citet[Definition 7.14]{rudin1991functional} and $(b)$ follows by \citet[Example 7.16]{rudin1991functional}. In particular, the Fourier transform of the constant tempered distribution is the impulse distribution.
Thus, we can write \eqref{eq: classical_functional_eq} in terms of distribution functions as follows:
\begin{align*}
    \calT_{\qinfty}[\hat{\varphi}^\prime-g_{1,l}\hat{\varphi}]=-\frac{C_1}{\sqrt{2\pi}} \delta[\hat{\varphi}] \quad \forall \varphi \in \calS(\bbR).
\end{align*}
As Fourier transform is a bijection on $\calS(\bbR)$ (Proposition \ref{prop: bijection}), the above is equivalent to
\begin{align*}
    \calT_{\pfix{\bar{\xi}_\infty}}[\varphi^\prime-g_{1,l}\varphi]=-\frac{C_1}{\sqrt{2\pi}} \delta[\varphi] \quad \forall \varphi \in \calS(\bbR).
\end{align*}
Similar to the proof of Lemma \ref{lemma: distribution_function}, we consider
\begin{align*}
    \tilde{\calD}(\bbR) = \left\{\psi \in \calD(\bbR) : \int_{-\infty}^{\infty} e^{-G(t)} \psi(t) dt = 0\right\},
\end{align*}
where $G(\pfix{x}) = \int_0^x g_{1,l}(t) dt$. For any $\tilde{\psi} \in \tilde{D}(\bbR)$, define $\psi \in \calD(\bbR)$ as
\begin{align*}
    \psi = e^{G(x)} \int_{-\infty}^x e^{-G(t)} \tilde{\psi}(t) dt.
\end{align*}
Now, by noting that $\psi^{\prime} - g \psi = \tilde{\psi}$, we have 
\begin{align*}
\calT_{\qinfty}[\tilde{\psi}] = -\frac{C_1}{\sqrt{2\pi}}\delta \left(e^{G(x)} \int_{-\infty}^x e^{-G(t)} \tilde{\psi}(t) dt\right) &= -\frac{C_1}{\sqrt{2\pi}}\int_{-\infty}^0 e^{-G(t)} \tilde{\psi}(t) dt. \numberthis \label{eq: subset_d_space_ssq}
\end{align*}
Now, for an arbitrary $\psi \in \calD(\bbR) \backslash \tilde{\calD}(\bbR)$, define $a = \int_{-\infty}^{\infty} e^{-G(t)} \psi(t) dt$. Note that $a \neq 0$ as $\psi \notin \tilde{\calD}(\bbR)$. Next, fix a $\psi_0 \in \calD(\bbR) \backslash \tilde{\calD}(\bbR)$ such that $\int_{-\infty}^{\infty} e^{-G(t)} \psi_0(t) dt = 1$ and define 
\begin{align*}
    \tilde{\psi} = \psi - a \psi_0,
\end{align*}
and let $C_2 = \calT_{\qinfty}[\psi_0]$, $C_3 = \int_{-\infty}^0 e^{-G(t)} \psi_0(t)dt$. Observe that $\tilde{\psi} \in \tilde{\calD}(\bbR)$. Now, using the above equation and linearity of $\calT_{\qinfty}$, we have
\begin{align*}
    &\calT_{\qinfty}[\psi] = \calT_{\qinfty}[\tilde{\psi}] + a\calT_{\qinfty}[\psi_0] \\
    \overset{\eqref{eq: subset_d_space_ssq}}{=}{}& -\frac{C_1}{\sqrt{2\pi}}\int_{-\infty}^0 e^{-G(t)} \tilde{\psi}(t) dt + a \calT_{\qinfty}[\psi_0] \\
    ={}&-\frac{C_1}{\sqrt{2\pi}}\int_{-\infty}^0 e^{-G(t)} \psi(t) dt + \left(C_2+\frac{C_1 C_3}{\sqrt{2\pi}} \right) \int_{-\infty}^{\infty} e^{-G(t)} \psi(t) dt \\
    ={}&\int_{-\infty}^\infty \left(C_2+\frac{C_1 C_3}{\sqrt{2\pi}}- \frac{C_1}{\sqrt{2\pi}}\mathbbm{1}\{t < 0\} \right) e^{-G(t)} \psi(t) dt.
\end{align*}
So we essentially have
\begin{align*}
    \rho_{\qinfty}(x)=  e^{-G(x)}\left( C_2+\frac{C_1C_3}{\sqrt{2\pi}}-\frac{C_1}{\sqrt{2\pi}}\mathbbm{1}\{x < 0\}\right),
\end{align*}
with $C_2 + C_1C_3/\sqrt{2\pi} = C_1/\sqrt{2\pi}$ as $\rho_{\qinfty}(x)$ has no mass for $x \leq 0$ and $C_1/\sqrt{2\pi} = 1/\int_{0}^\infty e^{-G(t)} dt$ to ensure $\rho_{\qinfty}(x)$ is a probability distribution. This argument can be formalized similar to the proof of Lemma \ref{lemma: distribution_function} but we omit the details as they are repetitive. This completes the proof.
\hfill $\square$
\endproof
\subsection{Proof of Theorem \ref{theo: classical_general}: Case III}
\proof{Proof of Theorem \ref{theo: classical_general}: $\epsilon_\eta\tau_\eta \rightarrow \infty$} By using \eqref{eq: functional_identity_single_server_queue}, replacing $\epsilon_\eta$ by $1/\tau_\eta$, and multiplying both sides by $\mathbbm{1}\left\{\frac{\barq_\eta^+}{\tau_\eta}\leq t^\star\right\}$, we get
\begin{align*}
    \mathbbm{1}\left\{\frac{\barq_\eta^+}{\tau_\eta}\leq t^\star\right\}\left(e^{j \omega \barq_\eta^+ / \tau_\eta}-e^{j \omega (\barq_\eta+\bara_\eta-\bars_\eta) / \tau_\eta}\right)=\left(1-e^{-j\omega \baru_\eta / \tau_\eta}\right)\mathbbm{1}\left\{\frac{\barq_\eta^+}{\tau_\eta}\leq t^\star\right\}.
\end{align*}
Now, by taking the expectation with respect to the stationary distribution of $\{q_\eta(k): k \in \bbZ_+\}$ and noting that $\E{e^{j\omega \barq_\eta/\tau_\eta}\mathbbm{1}\left\{\frac{\barq_\eta}{\tau_\eta}\leq t^\star\right\}}=\E{e^{j\omega\barq_\eta^+/\tau_\eta}\mathbbm{1}\left\{\frac{\barq_\eta^+}{\tau_\eta}\leq t^\star\right\}}$ by stationarity, we get
\begin{align}
    &\E{e^{j \omega \barq_\eta / \tau_\eta}\mathbbm{1}\left\{\frac{\barq_\eta}{\tau_\eta}\leq t^\star\right\}}-\E{e^{j \omega (\barq_\eta+\bara_\eta-\bars_\eta) / \tau_\eta}\mathbbm{1}\left\{\frac{\barq_\eta^+}{\tau_\eta}\leq t^\star\right\}} \nonumber\\
    ={}& \E{\left(1-e^{-j\omega \baru_\eta / \tau_\eta}\right)\mathbbm{1}\left\{\frac{\barq_\eta^+}{\tau_\eta}\leq t^\star\right\}} \label{eq: drift_regime_3}
\end{align}
Now, we simplify the LHS and RHS separately. Write LHS as the sum of $\calT_3$ and $\calT_4$ defined as follows:
\begin{align*}
    \calT_{3} = {}& \E{e^{j\omega(\barq_\eta + \bara_\eta - \bars_\eta) / \tau_\eta}\mathbbm{1}\left\{\frac{\barq_\eta}{\tau_\eta} \leq t^\star\right\} - e^{j\omega(\barq_\eta + \bara_\eta - \bars_\eta) / \tau_\eta}\mathbbm{1}\left\{\frac{\barq_\eta^+}{\tau_\eta} \leq t^\star\right\}} \\
    \calT_{4} ={}& \E{e^{j \omega \barq_\eta / \tau_\eta}\mathbbm{1}\left\{\frac{\barq_\eta}{\tau_\eta}\leq t^\star\right\} - e^{j\omega(\barq_\eta + \bara_\eta - \bars_\eta) / \tau_\eta}\mathbbm{1}\left\{\frac{\barq_\eta}{\tau_\eta} \leq t^\star\right\}} 
\end{align*}
We can now simplify $\calT_3$ analogous to $\calT_1$ in the proof of Lemma \ref{lemma: regime_3_inside}. We omit some of the steps that are similar to the proof of Lemma \ref{lemma: regime_3_inside}.
\begin{align*}
    \calT_{3} &= e^{j\omega t^\star}\E{e^{j\omega(\barq_\eta + \bara_\eta - \bars_\eta - \tau_\eta t^\star) / \tau_\eta}\left(\mathbbm{1}\left\{\frac{\barq_\eta}{\tau_\eta} \leq t^\star\right\}-\mathbbm{1}\left\{\frac{\barq_\eta^+}{\tau_\eta} \leq t^\star\right\}\right)} \\
    &\overset{(a)}{=} j\omega e^{j\omega t^\star}\E{\frac{\barq_\eta + \bara_\eta - \bars_\eta}{ \tau_\eta}\left(\mathbbm{1}\left\{\frac{\barq_\eta}{\tau_\eta} \leq t^\star\right\}-\mathbbm{1}\left\{\frac{\barq_\eta^+}{\tau_\eta} \leq t^\star\right\}\right)} + \omega^2 o\left(\frac{1}{\tau_\eta^2}\right) + o\left(\frac{1}{\tau_\eta^2}\right)o(\omega^2) \\
    &\overset{(b)}{=} \frac{j\omega}{\tau_\eta} e^{j\omega t^\star}\E{\barq_\eta^+\left(\mathbbm{1}\left\{\frac{\barq_\eta}{\tau_\eta} \leq t^\star\right\}-\mathbbm{1}\left\{\frac{\barq_\eta^+}{\tau_\eta} \leq t^\star\right\}\right)} + \omega^2 o\left(\frac{1}{\tau_\eta^2}\right) + o\left(\frac{1}{\tau_\eta^2}\right)o(\omega^2) \\
    &\overset{(c)}{=}\frac{j\omega}{\tau_\eta} e^{j\omega t^\star}\E{\mathbbm{1}\left\{\frac{\barq_\eta}{\tau_\eta}\leq t^\star\right\}\left(\barq_\eta^+ - \barq_\eta\right)}+ \omega^2 o\left(\frac{1}{\tau_\eta^2}\right) + o\left(\frac{1}{\tau_\eta^2}\right)o(\omega^2) \\
    &\overset{(d)}{=}\frac{j\omega}{\tau_\eta} e^{j\omega t^\star}\E{\mathbbm{1}\left\{\frac{\barq_\eta}{\tau_\eta}\leq t^\star\right\}\left(\bara_\eta - \bars_\eta + \baru_\eta\right)}+ \omega^2 o\left(\frac{1}{\tau_\eta^2}\right) + o\left(\frac{1}{\tau_\eta^2}\right)o(\omega^2) \\
     &\overset{(e)}{=}\frac{j\omega}{\tau_\eta} e^{j\omega t^\star}\E{\mathbbm{1}\left\{\frac{\barq_\eta}{\tau_\eta}\leq t^\star\right\} \baru_\eta}+ \omega^2 o\left(\frac{1}{\tau_\eta^2}\right) + o\left(\frac{1}{\tau_\eta^2}\right)o(\omega^2) \\
     &\overset{(f)}{=}\frac{j\omega}{\tau_\eta} e^{j\omega t^\star}\E{ \baru_\eta}+ \omega^2 o\left(\frac{1}{\tau_\eta^2}\right) + o\left(\frac{1}{\tau_\eta^2}\right)o(\omega^2) \\
     &\overset{(g)}{=}\frac{j\omega\epsilon_\eta}{\tau_\eta} e^{j\omega t^\star}\E{\phi^s\left(\frac{
    \barq_\eta}{\tau_\eta}\right)-\phi^c\left(\frac{
    \barq_\eta}{\tau_\eta}\right)}+ \omega^2 o\left(\frac{1}{\tau_\eta^2}\right) + o\left(\frac{1}{\tau_\eta^2}\right)o(\omega^2),
\end{align*}
where $(a)$ follows by Taylor's series expansion and shows that the second order term is $\omega^2o(1/\tau_\eta^2)$ similar to the proof of Claim \ref{claim: t_11_regime_3_inside}. Next, $(b)$ follows by the queue evolution equation given by \eqref{eq: queue_evolution} and noting that $\mathbbm{1}\left\{\frac{\barq_\eta}{\tau_\eta} \leq t^\star\right\}-\mathbbm{1}\left\{\frac{\barq_\eta^+}{\tau_\eta} \leq t^\star\right\} \neq 0$ only when $|\barq_\eta^+ - \tau_\eta t^\star| \leq A_{\max}$ which implies $\barq_\eta^+ > 0 \Rightarrow \baru_\eta = 0$ for $\eta$ large enough. Further, $(c)$ follows by a result analogous to Lemma \ref{lemma: z_sgnz} for a single server queue. In particular, 
\begin{align*}
    0 &= \E{\barq_\eta^+\mathbbm{1}\left\{\frac{\barq_\eta^+}{\tau_\eta}\leq t^\star\right\} -\barq_\eta\mathbbm{1}\left\{\frac{\barq_\eta}{\tau_\eta}\leq t^\star\right\}} \\
    &= \E{\mathbbm{1}\left\{\frac{\barq_\eta}{\tau_\eta}\leq t^\star\right\}\left(\barq_\eta^+-\barq_\eta\right)} + \E{\barq_\eta^+\left(\mathbbm{1}\left\{\frac{\barq_\eta^+}{\tau_\eta}\leq t^\star\right\}-\mathbbm{1}\left\{\frac{\barq_\eta}{\tau_\eta}\leq t^\star\right\}\right)}.
\end{align*}
Next, $(d)$ follows by the queue evolution equation given by \eqref{eq: queue_evolution}, $(e)$ follows by noting that $\E{\bara_\eta |\barq_\eta}=\E{\bars_\eta | \barq_\eta}$ for $\barq_\eta \leq \tau_\eta t^\star$, and $(f)$ follows as $\mathbbm{1}\{\barq_\eta / \tau_\eta > t^\star\}\baru_\eta = 0$ as $\barq_\eta / \tau_\eta > t^\star$ implies that $\barq_\eta^+ / \tau_\eta > t^\star - A_{\max}/\tau_\eta > 0$ for $\eta$ large enough. Lastly, $(g)$ follows by \eqref{eq: expectation_unused_service}. Now, we simplify $\calT_{4}$ below which is analogous to $\calT_2$ in Lemma \ref{lemma: regime_3_inside}.
\begin{align*}
    \calT_{4} &= \E{\mathbbm{1}\left\{\frac{\barq_\eta}{\tau_\eta} \leq t^\star\right\}e^{j\omega \barq_\eta / \tau_\eta}\left(1-e^{j\omega(\bara_\eta - \bars_\eta) / \tau_\eta}\right)} \\
    &\overset{(a)}{=} -\frac{j\omega}{\tau_\eta}\E{\mathbbm{1}\left\{\frac{\barq_\eta}{\tau_\eta} \leq t^\star\right\}e^{j\omega \barq_\eta / \tau_\eta}(\bara_\eta-\bars_\eta)} + \frac{\omega^2}{2\tau_\eta^2}\E{\mathbbm{1}\left\{\frac{\barq_\eta}{\tau_\eta} \leq t^\star\right\}e^{j\omega \barq_\eta / \tau_\eta}\left(\bara_\eta - \bars_\eta\right)^2} \\
    &+ o\left(\frac{1}{\tau_\eta^2}\right)o(\omega^2) \\
    &\overset{(b)}{=}\frac{\omega^2}{2\tau_\eta^2}\E{\mathbbm{1}\left\{\frac{\barq_\eta}{\tau_\eta} \leq t^\star\right\}e^{j\omega \barq_\eta / \tau_\eta}\left(\bara_\eta - \bars_\eta\right)^2} + o\left(\frac{1}{\tau_\eta^2}\right)o(\omega^2) \\
    &\overset{(c)}{=} \frac{\omega^2\left(\sigma^c(\tilde{\lambda}^\star) + \sigma^s(\tilde{\mu}^\star)\right)}{2\tau_\eta^2}\E{\mathbbm{1}\left\{\frac{\barq_\eta}{\tau_\eta} \leq t^\star\right\}e^{j\omega \barq_\eta / \tau_\eta}} + \omega^2o\left(\frac{1}{\tau_\eta^2}\right) + o\left(\frac{1}{\tau_\eta^2}\right)o(\omega^2),
\end{align*}
where $(a)$ follows Taylor's series expansion with the precise argument given by Lemma \ref{lemma: taylor_series}. Next, $(b)$ follows by noting that the first term is $0$ as $\E{\bara_\eta | \barq_\eta} = \E{\bars_\eta | \barq_\eta}$ for $\barq_\eta \leq\tau_\eta t^\star$. Lastly, $(c)$ follows similar to $(f)$ in the proof of Claim \ref{claim: t_2_regime_3_inside}. As $\phi^c(\cdot)$ and $\phi^s(\cdot)$ are non-increasing and increasing respectively, there exists $c^\star$ such that $\phi^c(q) = \phi^s(q)= c^\star$ for all $q \leq t^\star$. Now, we define $\tilde{\lambda}^\star=\lambda^\star+c^\star \lim_{\eta \uparrow \infty} \epsilon_\eta$ and $\tilde{\mu}^\star=\mu^\star+c^\star \lim_{\eta \uparrow \infty} \epsilon_\eta$. As these steps are analogous to the proof of Lemma \ref{lemma: regime_3_inside}, we omit the details. Now, we simplify the RHS of \eqref{eq: drift_regime_3} below.
\begin{align*}
    &\E{\left(1-e^{-j\omega \baru_\eta / \tau_\eta}\right)\mathbbm{1}\left\{\frac{\barq_\eta^+}{\tau_\eta}\leq t^\star\right\}} \\
    \overset{(a)}{=}{}&  \E{\left(\frac{j\omega \baru_\eta}{\tau_\eta} + \frac{\omega^2 \baru_\eta^2}{2\tau_\eta^2} + o\left(\frac{1}{\tau_\eta^2}\right)o(\omega^2) \right)\mathbbm{1}\left\{\frac{\barq_\eta^+}{\tau_\eta}\leq t^\star\right\}} \\
    \overset{(b)}{=}{}& \E{\frac{j\omega \baru_\eta}{\tau_\eta} + \frac{\omega^2 \baru_\eta^2}{2\tau_\eta^2}} + o\left(\frac{1}{\tau_\eta^2}\right)o(\omega^2) \\
    \overset{(c)}{=}{}& \E{\frac{j\omega \baru_\eta}{\tau_\eta}} + \omega^2o\left(\frac{1}{\tau_\eta^2}\right) + o\left(\frac{1}{\tau_\eta^2}\right)o(\omega^2) \\
    \overset{(d)}{=}{}&\frac{j\omega\epsilon_\eta}{\tau_\eta}\E{\phi^s\left(\frac{
    \barq_\eta}{\tau_\eta}\right)-\phi^c\left(\frac{
    \barq_\eta}{\tau_\eta}\right)}+ \omega^2 o\left(\frac{1}{\tau_\eta^2}\right) + o\left(\frac{1}{\tau_\eta^2}\right)o(\omega^2),
\end{align*}
where $(a)$ follows Taylor's series expansion with the precise argument given by Lemma \ref{lemma: taylor_series}. Next, $(b)$ follows as $\E{\baru_\eta \mathbbm{1}\left\{\frac{\barq_\eta^+}{\tau_\eta} >t^\star\right\}} = 0$ for large enough $\eta$. Now, $(c)$ holds due to the following.
\begin{align*}
    0&\leq \E{\baru_\eta^2} \leq S_{\max}\E{\baru_\eta} = \epsilon_\eta S_{\max} \E{\phi^s\left(\frac{\barq_\eta}{\tau_\eta}\right)-\phi^c\left(\frac{\barq_\eta}{\tau_\eta}\right)} \\
    &\leq \frac{\epsilon_\eta S_{\max}}{t^\star \tau_\eta} \E{\barq_\eta\left(\phi^s\left(\frac{\barq_\eta}{\tau_\eta}\right)-\phi^c\left(\frac{\barq_\eta}{\tau_\eta}\right)\right)} = o(1),
\end{align*}
where the last equality follows by showing a result similar to Proposition \ref{prop: pos_rec}. Lastly, $(d)$ follows by \eqref{eq: expectation_unused_service}. Substituting the above along with $\calT_3$ and $\calT_4$ in \eqref{eq: drift_regime_3}, we get
\begin{align*}
    &\E{\mathbbm{1}\left\{\frac{\barq_\eta}{\tau_\eta} \leq t^\star\right\}e^{j\omega \barq_\eta / \tau_\eta}} \\
    ={}& \frac{2j\epsilon_\eta \tau_\eta}{\omega(\sigma^c(\tilde{\lambda}^\star) + \sigma^s(\tilde{\mu}^\star))}\E{\phi^s\left(\frac{
    \barq_\eta}{\tau_\eta}\right)-\phi^c\left(\frac{
    \barq_\eta}{\tau_\eta}\right)}\left(1-e^{j\omega t^\star}\right)+ o_{\eta}\left(1\right) + o_{\eta}\left(1\right)o_{\omega}(1).
\end{align*}
Now, by letting $\omega \rightarrow 0$, we get
\begin{align*}
    \P{\frac{\barq_\eta}{\tau_\eta} \leq t^\star} = \frac{2\epsilon_\eta \tau_\eta t^\star}{(\sigma^c(\tilde{\lambda}^\star) + \sigma^s(\tilde{\mu}^\star))}\E{\phi^s\left(\frac{
    \barq_\eta}{\tau_\eta}\right)-\phi^c\left(\frac{
    \barq_\eta}{\tau_\eta}\right)}+ o_{\eta}(1)
\end{align*}
Substituting that back again, we get
\begin{align}
    \E{\mathbbm{1}\left\{\frac{\barq_\eta}{\tau_\eta} \leq t^\star\right\}e^{j\omega \barq_\eta / \tau_\eta}} = \P{\frac{\barq_\eta}{\tau_\eta} \leq t^\star} \left(\frac{e^{j\omega t^\star} - 1}{j \omega t^\star}\right) + o_{\eta}(1) + o_{\eta}(1)o_{\omega}(1). \label{eq: inside}
\end{align}
Similar to Lemma \ref{lemma: regime_3_outside}, we can further show that
\begin{align}
    \limsup_{\eta \uparrow \infty} \P{\frac{\barq_\eta}{\tau_\eta} > t^\star} = 0. \label{eq: outside}
\end{align}
As the steps are analogous to the proof of Lemma \ref{lemma: regime_3_outside}, we omit the details here. Now, by using \eqref{eq: inside} and \eqref{eq: outside}, we get
\begin{align*}
    \lim_{\eta \uparrow \infty}\E{e^{j \omega \barq_\eta / \tau_\eta}} &= \lim_{\eta \uparrow \infty}\E{e^{j \omega \barq_\eta / \tau_\eta} \mathbbm{1}\left\{\frac{\barq_\eta}{\tau_\eta }\leq t^\star\right\}} + \lim_{\eta \uparrow \infty}\E{e^{j \omega \barq_\eta / \tau_\eta}\mathbbm{1}\left\{\frac{\barq_\eta}{\tau_\eta }> t^\star\right\}} \\
    &=\lim_{\eta \uparrow \infty}\E{e^{j \omega \barq_\eta / \tau_\eta} \mathbbm{1}\left\{\frac{\barq_\eta}{\tau_\eta }\leq t^\star\right\}} \\
    &= \lim_{\eta \uparrow \infty} \P{\frac{\barq_\eta}{\tau_\eta} \leq t^\star} \left(\frac{e^{j\omega t^\star} - 1}{j \omega t^\star}\right) \\
    &= \frac{e^{j\omega t^\star} - 1}{j \omega t^\star}. 
\end{align*}
By noting that the above is a characteristic function of a uniform distribution, the proof is complete by Levy's continuity theorem (\citet[Chapter 18]{levy_book}). \hfill $\square$
\endproof
\newpage
\end{APPENDICES}

\end{document}